\documentclass[10pt,letterpaper,reqno]{amsart}
\usepackage{amssymb,a4wide,amsthm,amsmath,amsfonts,xcolor}
\usepackage{mathrsfs}

\usepackage{verbatim}
\usepackage[colorlinks,linkcolor={blue},
citecolor={blue},urlcolor={red},]{hyperref}
\usepackage{hyperref}
\usepackage[utf8]{inputenc}
\usepackage[T1]{fontenc}

\numberwithin{equation}{section}
\newcommand{\embed}{\hookrightarrow}

\newtheorem{theorem}{Theorem}[section]
\newtheorem{lemma}[theorem]{Lemma}
\newtheorem{proposition}[theorem]{Proposition}
\newtheorem{assumption}[theorem]{Assumption}
\newtheorem{corollary}[theorem]{Corollary}
\newtheorem{definition}[theorem]{Definition}
\newtheorem{main result}{Main Result}

\newtheorem{remark}[theorem]{Remark}

\allowdisplaybreaks

\newcommand\adda[1]{{\color{blue} #1}}
\newcommand\coma[1]{{\color{red} {#1}}}
\newcommand\dela[1]{}

\def\m{m^{\prime}}
\def\u{u^{\prime}}

\def\l{\left}
\def\r{\right}
\def\b{\big}
\def\p{\prime}

\begin{document}
	
	\title[Stochastic control of the Landau-Lifshitz-Gilbert equation]
	{Stochastic control of the Landau-Lifshitz-Gilbert equation}

	\author[Z.~Brze\'zniak]{Zdzis{\l}aw Brze\'zniak}
	\address{Department of Mathematics\\
		The University of York\\
		Heslington, York YO10 5DD, UK} \email{zb500@york.ac.uk}
	\author[S. Gokhale]{Soham Gokhale}
	\address{School of Mathematics\\ Indian Institute of Science Education and Research Thiruvananthapuram\\ Trivandrum 695551, INDIA}	\email{gokhalesoham16@iisertvm.ac.in}
	\author[U.~Manna]{Utpal Manna}
	\address{School of Mathematics\\ Indian Institute of Science Education and Research Thiruvananthapuram\\ Trivandrum 695551, INDIA}
	\email{manna.utpal@iisertvm.ac.in}

	\thanks{}
	\dela{\begin{abstract}
			The stochastic Landau-Lifshitz-Gilbert equation, perturbed by Gaussian noise is considered in dimension 1. Along with noise, a control is added to the effective field. Existence of a weak martingale solution for the resulting equation is shown using the classical Faedo-Galerkin approximation. Pathwise uniqueness is coupled with the theory of Yamada and Watanabe to give the existence of a unique strong solution. Moreover, the obtained solution is shown to satisfy the maximum regularity. The existence of an optimal control is then shown.
		\end{abstract}
	}
	\begin{abstract}
		We consider the stochastic Landau-Lifshitz-Gilbert equation in dimension 1. A control process is added to the effective field. We show the existence of a weak martingale solution for the resulting controlled equation. The proof uses the classical Faedo-Galerkin approximation, along with the Jakubowski's version of the Skorohod Theorem. We then show pathwise uniqueness for the obtained solution, which is then coupled with the theory of Yamada and Watanabe to give the existence of a unique strong solution. We then show, using some semigroup techniques that the obtained solution satisfies the maximum regularity. We then show the existence of an optimal control is then shown. A main ingredient of the proof is using the compact embedding of a space into itself, albeit with the weak topology. That is, we use the compact embedding of the space $L^2(0,T:L^2)$ into the space $L^2_w(0,T:L^2)$.
	\end{abstract}
	\date{\today}
	\maketitle
	\date{\today}
	\maketitle

	\section{Introduction}
	Magnetic storage devices are widely used to store, process data. A common example would be a hard disk that is used to store computer data. Such media are nowadays comprehensively used to store large amounts of data.  Another example is that of a ferromagnetic nanowire ($d = 1$) separating domains of almost uniform magnetization $m$. The speed of reading or writing depends upon the magnetization switching time. Hence, understanding magnetization processes and the corresponding mechanisms can help in an optimal design for the storage media and lead to faster and better devices.
	
	Magnetic devices are made up of several ferromagnetic particles. Each of them has the capacity to be magnetized in two directions. These can be hence used to store one bit of data each. Therefore it is beneficial to get more efficient switching of the magnetization states of the particles in order to get better data storage and processing. This can be controlled by adding an external control (field pulses)\dela{ $u$}.



	Weiss initiated the study of the theory of ferromagnetism, see \cite{brown1963micromagnetics} and references therein. Landau and Lifshitz \cite{landau1992theory} and Gilbert \cite{Gilbert} developed it further. Let $\mathcal{O}\subset \mathbb{R}$ be a bounded interval.
	For a temperature below the Curie temperature, the magnetization $m$ satisfies the Landau-Lifshitz-Gilbert (LLG) equation
	
	\begin{align*}
		\begin{cases}
			& \frac{\partial m}{\partial t} = \alpha_{\text{g}}(m\times H_{\text{eff}}) - \alpha_{\text{g}}\alpha_{\text{d}} m\times(m\times H_{\text{eff}})  \  \text{in}\ \mathcal{O}_T=(0,T)\times \mathcal{O},\\
			\nonumber & \frac{\partial m}{\partial \nu} =\, 0\ \text{on}\ \partial \mathcal{O}_T = [0,T] \times \partial \mathcal{O},\\
			\nonumber & m(0,\cdot) = m_0\  \text{on}\ \mathcal{O}.
		\end{cases}
	\end{align*}
	Here $\times$ denotes the vector product in $\mathbb{R}^3$. The constant $\alpha_{\text{g}}$ and $\alpha_{\text{d}}$ are the gyromagnetic ratio and the damping parameter respectively \cite{Cimrak_2008}. $H_{\text{eff}}$ denotes the effective field, which is described after a short remark and some simplifying assumptions.\\
	For ferromagnetic materials at temperature below the Curie temperature, the modulus of the magnetization remains constant. For simplicity, we assume the constant to be 1.

	For simplicity, let us assume that $\alpha_{\text{g}} = 1$ and denote the constant $\alpha_{\text{d}}$ by $\alpha$.\dela{ Note that $0 < \alpha        \ll 1$.} We do not use the smallness of the parameter $\alpha$ anywhere in the calculations. The resulting equation is
	
	\begin{align*}
		\frac{\partial m}{\partial t} = (m\times H_{\text{eff}})  - \alpha \, m\times(m\times H_{\text{eff}}) .
	\end{align*}

	\dela{ After suitably normalizing, the equation can be rewritten as
		\begin{align*}
			dm = (m\times H_{\text{eff}}) \, dt - \alpha_{\text{d}} m\times(m\times H_{\text{eff}}) \, dt.
		\end{align*}
		For simplicity of notation, we replace $\alpha_{\text{d}}$ by $\alpha$.
	}

	\dela{Try to write something as follows: For simplicity, let us assume that $\alpha_{\text{g}} = 1$ and denote the constant $\alpha_{\text{d}}$ by $\alpha$. Note that $0 < \alpha \, \ll 1$.}
	
	The effective field $H_{\text{eff}}=H_{\text{eff}}(m)= - D\mathcal{E}(m)$ is deduced from the Landau–Lifshitz energy $\mathcal{E}(m)$. It consists of three parts: the exchange energy, the anistropy energy and the magnetostatic energy \cite{Kruzik+Prohl}. Visintin \cite{Visintin_1985} studied the equation with contribution from all three energies to the total magnetic energy. For simplicity, we consider only the exchange energy. Similar simplifying assumptions can be seen in \cite{Michiel_EtAl_2001}, \cite{ZB+BG+TJ_Weak_3d_SLLGE} to name a few. DeSimone \cite{DeSimone_HysterisisSmallFerromagneticParticles}, Gioia and James \cite{Gioia1997micromagnetics} studied the LLGE (with exchange energy only) as a limiting case of the models that include the other types of energy. The exchange energy is given by
	
	\begin{equation*}
		\mathcal{E}(m)=\int_{\mathcal{O}}\frac{A_{\text{exc}}}{2}|\nabla m|_{\mathbb{R}^3}^2 \,dx,
	\end{equation*}
	where $A_{\text{exc}}$ is a material constant  \cite{Cimrak_2008}. In the calculations that follow, we assume that $A_{\text{exc}} = 1$.
	Therefore, we have the effective field $H_{\text{eff}} = \Delta m$. The resulting equation is then
	\begin{align}\label{deterministic equation 1.1}
		\begin{cases}
			&\frac{\partial m}{\partial t} = ( m \times \Delta m )  - \alpha \, m \times ( m \times \Delta m)  \  \text{in}\ \mathcal{O}_T=(0,T)\times \mathcal{O},\\
			& \frac{\partial m}{\partial \nu} =\, 0\ \text{on}\ \partial \mathcal{O}_T,\\
			& m(0,\cdot) = m_0\  \text{on}\ \mathcal{O}.
		\end{cases}
	\end{align}
	\dela{\textbf{A brief history of LLGE:} \\}
	\dela{The authors in \cite{Alouges+Soyeur_GlobalWeakSolutionsLLGE} establish necessary criterion for the existence of a global weak solution for the LLG equation. Fang and Li in \cite{Fang+Li_GlobalWeakSolutionLLGE3D} consider the equation in dimension 3 and study the existence of a global weak solution. On similar lines, \cite{Pu+Wang+Wang_GlobalLLG} shows the global existence and uniqueness of weak solutions. Melcher in \cite{Melcher_GlobalSolvabilityCauchyProblemLLGE} shows the existence and uniqueness of a mild solution for the Cauchy problem for the LLG equations. In \cite{Bartels+Prohl_ImplicitFiniteElementLLG}, the authors present an implicit finite element scheme for the LLG equation and show the convergence of the approximations to a weak solution, along with some numerical results.
		\dela{See also \cite{Ammari+Laurence_AsymptoticBehaviourThinFerromagnet}} \cite{Bartels+Joy+Prohl_NumericalAnalysisLLGE} shows convergence of a scheme for approximating weak solution of LLG equation. The authors also study (numerically) discrete finite time blow up in dimension 2. In \cite{Pu+Guo_FractionalLLGE}, the authors show the existence of global weak solutions to the periodic fractional LLG equation.
	}

	\dela{\textbf{Motivation for introducing Noise:}} The solutions of \eqref{deterministic equation 1.1} correspond to the equilibrium states of the ferromagnet. We aim to study the phase transition between equilibrium states that is induced by the thermal fluctuations in the effective field $H_{\text{eff}}$. In order to do this, we must modify the equation \eqref{deterministic equation 1.1} to include the random fluctuations affecting the dynamics of $m$. The study of analysis of noise induced transitions was started by Neel \cite{Neel_1946} and further studied in \cite{Brown_Thermal_Fluctuations}, \cite{Kamppeter+FranzEtAl_StochasticVortexDynamics}, among others. The introduction of a stochastic forcing term in \eqref{equation 1.1} is due to Brown, see \cite{brown1963micromagnetics} ,  \cite{Brown_Thermal_Fluctuations}.
	
	Following \cite{ZB+BG+TJ_Weak_3d_SLLGE}, we perturb the equation by adding Gaussian noise to the effective field. That is we replace $H_{\text{eff}}$ by $H_{\text{eff}} + \zeta$, where $\zeta$ (informally) denotes a white noise. The resulting equation is then
	
	\begin{equation}
		\frac{\partial m}{\partial t} = \l( m \times \l(\Delta m + \zeta \r) \r)  - \alpha \, m \times \l( m \times \l(\Delta m + \zeta \r)\r)
	\end{equation}

	The noise term in the resulting equation will then be of the form
	\dela{\begin{equation*}
			m \times h - \alpha \, m \times (m \times h).
		\end{equation*}
		for some given function $h : \mathcal{O} \to \mathbb{R}^3$.
	}
	
	\begin{equation*}
		m \times \zeta - \alpha \, m \times (m \times \zeta).
	\end{equation*}
	The noise needs to be invariant under the coordinate transformation. Hence it is suggested \cite{Kamppeter+FranzEtAl_StochasticVortexDynamics}, \cite{Brown_Thermal_Fluctuations} that we consider the noise in the Stratonovich form.

	Therefore the stochastic counterpart (stochastic LLG equation) of equation \eqref{deterministic equation 1.1} is
	
	\begin{align}\label{equation 1.1}
		\begin{cases}
			&dm = m\times H_{\text{eff}} \, dt - \alpha \, m\times(m\times H_{\text{eff}}) \, dt \\
			& \qquad \ + \bigl(m \times h - \alpha \, m \times (m \times h) \bigr) \circ \, dW(t)\ \text{in}\ \mathcal{O}_T=(0,T)\times \mathcal{O},\\
			& \frac{\partial m}{\partial \nu} =\, 0\ \text{on}\ \partial \mathcal{O}_T,\\
			& m(0,\cdot) = m_0\  \text{on}\ \mathcal{O}.
		\end{cases}
	\end{align}
	Here $h : \mathcal{O} \to \mathbb{R}^3$ is a given bounded function. $W$ is a real valued Wiener process on a given probability space $ \l( \Omega , \mathcal{F} , \mathbb{P} \r) $. Also, $\circ \,  dW$ denotes the Stratonovich differential.\\
	\dela{ See the article by Cimrak \cite{MR2430351}, $\alpha_1$ the gyromagnetic ration and $\alpha_2$ is the damping parameter. Note that $\alpha_2$ is not just the damping parameter. It is the product of the gyromagnetic ration $\alpha_1$ and the damping parameter. The article by Cimrak \cite{MR2430351} also mentions some saturation magnetization $M_s$, which is possibly 1. Confirm this. are constants.
		The system is driven by a Hilbert space valued $Q$-Wiener process $W$ (Real valued Wiener process should be sufficient as the general case has not been considered by us) and the integral is understood in the Stratonovich sense. The choice of Stratonovich integral  has been argued in \cite{Brown_Thermal_Fluctuations}.
		Write more about the existing results for existence of solutions, precisely \cite{ZB+BG+TJ_Weak_3d_SLLGE}, \cite{ZB+BG+TJ_LargeDeviations_LLGE}.	The orientation of the magnetization vector can vary from point to point. Exchange energy will attempt to make the magnetic moments in the immediately surrounding space lie parallel to one another. \adda{Give reference for this statement. Is this statement needed?}}
	
	The existence of a weak martingale solution to the stochastic LLG equation (with $\alpha \, = 0$) in dimension 1,2,3 has been shown in \cite{ZB+BG+TJ_Weak_3d_SLLGE}. Brze\'zniak, Goldys and Jegaraj in \cite{ZB+BG+TJ_LargeDeviations_LLGE} consider the problem for dimension 1 (and non-zero anisotropy) and show that existence of a weak martingale solution. They further show the existence of a strong solution with maximal regularity, along with establishing large deviations principle. Brze\'zniak and Li in \cite{ZB+Li_Weak_Solution_SLLGE_Anisotropy} show the existence of a weak solution for the stochastic LLG equation with multidimensional noise and non-zero anisotropy energy. The authors in \cite{Kohn+Reznikoff+Vanden_MagneticElementsLargeDeviation} use large deviation theory and stochastic resonance to study thermally activated phenomenon in micromagnetics. Perturbing the model by a multiplicative noise, Pu and Guo in \cite{Pu+Guo_2010} prove the existence of regular martingale solutions in dimension 2, followed by some finite time blow-up criterion. The work \cite{Guo+Xueke_SLLGE_Global} establishes the existence of a global weak solution to the stochastic LLG equation for any dimension $d>0$, also showing that for dimension 1, the associated Cauchy problem admits a unique global smooth solution. Brze\'zniak and Manna in \cite{ZB+UM_SLLGE_JumpNoise} discuss the stochastic LLG equation in dimension 3, which is driven by pure jump noise. They show the existence of a weak martingale solution, which satisfies the required constraint condition. See also \cite{Li_Thesis}.
	Brze\'zniak, Manna and Mukherjee in \cite{ZB+UM+DM_WongZakai} show the existence of a strong solution.  Manna, Mukherjee and Panda in \cite{UM+AAP+DM_SLLGE_WongZakai_AnisotropyEnergy} show the existence of a strong solution with non-zero anisotropy energy. The key ingredients for both the previous results are the Doss-Sussmann transform and the Wong-Zakai approximations.
	
	Some important numerical studies include, but are not limited to Banas, Brze\'zniak, Neklyudov and Prohl  \cite{banas2013stochastic}, \cite{Banas+ZB+Neklyudov+Prohl_ConergentFiniteElementSLLGE}, see also \cite{Banas+ZB+Prohl_2013_ComputationalStudiesSLLGE}, Goldys, Le and Tran in \cite{GB+Le+Tran_Finite_Element_Scheme_SLLGE}, Brze\'zniak, Grotowski and Le in  \cite{BG+Grotowski+Le_Weak_martingale_SLLGE_multidimensional_noise}\dela{, Alouges, de Bouard and Hocquet in \cite{Alouges+Bouard+Hocquet_SemiDiscreteSchemeSLLEquation}}.\\
	
	A natural question here is about what happens when the temperature is above the Curie temperature. In that case the model can replaced by the Landau-Lifshitz-Bloch (LLB) equation. The model was proposed by Geranin in \cite{garanin1997fokker} in 1997. The LLB equation essentially interpolates between the LLG equation at low temperatures and the Ginzburg-Landau theory of phase transitions. Le in \cite{LE_Deterministic_LLBE} shows the existence of weak solution to the LLB equation.
	The same author with Brze\'zniak and Goldys in \cite{ZB+BG+Le_SLLBE} show the existence and uniqueness of a solution for the stochastic LLB equation, along with the existence of invariant measures for dimensions 1 and 2. On similar lines, Jiang, Ju and Wang in \cite{Jiang+Ju+Wang_MartingaleWeakSolnSLLBE} showed the existence of a weak martingale solution to the stochastic LLB equation. In  \cite{Qiu+Tang_Wang_AsymptoticBehaviourSLLBE}, the authors establish the large deviation principle and the central limit theorem for the 1 dimensional stochastic Landau-Lifshitz-Bloch equation.

	\dela{The effective field $H_{\text{eff}}=H_{\text{eff}}(m)=D\mathcal{E}(m)$ is deduced from the Landau–Lifshitz energy $\mathcal{E}(m)=\mathcal{E}_u(m)$ for some given external field $u: \mathcal{O}_T\times\Omega\rightarrow\mathbb{R}^3$. It consists of three parts: the exchange energy, the anistropy energy and the magnetostatic energy, see \cite{Kruzik+Prohl}. For simplicity, we consider in this work only the exchange and	external field energies
		
		\begin{equation}
			\mathcal{E}(m)=\int_{\mathcal{O}}\frac{a}{2}|\nabla m|_{\mathbb{R}^3}^2-\l\langle m,u\r\rangle_{\mathbb{R}^3} \,dx,
		\end{equation}
	}

	\dela{Similar assumptions \adda{That only exchange energy can be considered has also been assumed in deterministic analogues ....} have been made in the deterministic cases (see for example \cite{MR1848770}). \adda{ The paper \cite{MR1848770} has dimension 3 and specifically mentions that while not considering the other forms of energy. Thus, we can say that the exchange energy is considered for simplicity and similar assumptions can be found in \cite{MR1848770}, \cite{ZB+BG+TJ_Weak_3d_SLLGE}, etc.}
	}

	\dela{What is $\l\langle m , u \r\rangle$?? Is it some kind of inner product?? It is the $\mathbb{R}^3$ inner product. Write it in the equality for $\mathcal{E}$. Write it as $\l\langle m , u \r\rangle_{\mathbb{R}^3}$.}

	As indicated earlier, our aim is to try and optimize the switching of magnetization by giving some external input. Following the works \cite{Dunst+Klein}, \cite{D+M+P+V}, we add the control to the effective field. We now describe the inclusion of the control process in the stochastic LLG equation. Let $u: \Omega \times [0,T] \times \mathcal{O} \rightarrow \mathbb{R}^3$ denote a control process. We denote by $\mathcal{E}_u$, the external field energy, see \cite{D+M+P+V} corresponding to the control process $u$, which is given by
	\begin{equation}
		\mathcal{E}_u(m) = - \int_{\mathcal{O}} \l\langle m , u \r\rangle \, dx.
	\end{equation}
	The energy now considered is the sum of the exchange energy and the external field energy. Therefore the resulting effective field is
	\begin{equation}
		H_{\text{eff}} = \Delta m + u + \zeta.
	\end{equation}	
	\dela{
		The equation corresponding to the above mentioned effective field is
		\begin{align*}
			\nonumber dm &= \bigg[ m\times \Delta m- \alpha \, m\times(m\times \Delta m) + m \times u - \alpha \, m \times (m \times u)  \bigg] \, dt \\
			& + \big( m \times h - \alpha \, m \times (m \times h) \big)  \circ \, dW(t).
		\end{align*}
		
		Try not to write the equation written above.
	}
	Summarizing, the equation considered in this paper is the following:
	\begin{align}\label{problem considered introduction equation}
		\begin{cases}
			& dm = \big[ m\times \Delta m- \alpha \, m\times(m\times \Delta m) + m \times u - \alpha \, m \times (m \times u)  \big] \, dt \\
			& \qquad + \big(m \times h - \alpha \, m \times (m \times h) \big) \circ \, dW(t) \dela{\ \text{in}\ \mathcal{O}_T=(0,T)\times \mathcal{O},},\ t\in [0,T],\\
			& \frac{\partial m}{\partial \nu} =\, 0,\ \text{on}\ \partial \mathcal{O}_T,\\
			& m(0,\cdot) = m_0\  \text{on}\ \mathcal{O}.
		\end{cases}
	\end{align}
	Here $h:\mathcal{O}\to\mathbb{R}^3$ is a bounded function and $W$ is a real valued Wiener process.\\	
	\dela{ We will be considering the above equation in the paper.\\ Try to talk about the existing literature on stochastic optimal control, see [30] in D+M+P+V. Try not to write this. Seems difficult.}
	Kr\"uzik and Prohl in \cite{Kruzik+Prohl} give an overview of some of the developments in analysis and numerics of ferromagnetism. The initial studies for deterministic optimal control of ferromagnetic dynamics were done in \cite{Agarwal+Carbou+Labbe+Prieur},\cite{Alouges+Beauchard}. Dunst, Klein, Prohl and Sch\"afer in \cite{Dunst+Klein} show the existence of an optimal control subject to a one dimensional Landau-Lifshitz-Gilbert equation. Dunst, Majee, Prohl and Vallet in \cite{D+M+P+V} consider the stochastic counterpart of the above problem. In \cite{D+M+P+V}, the authors show the existence of weak optimal control for \eqref{equation 1.1} along with the following cost functional \eqref{eqn D+M+P+V Cost Functional} (for arbitrary but fixed $p\geq 2
	$, $K > 0$)
	
	\begin{equation}\label{eqn D+M+P+V Cost Functional}
		J(\pi) = \mathbb{E}\l[ \int_{0}^{T} \left(|m(t) - \bar{m}(t)|_{L^2}^2 + |u(t)|_{H^{1}}^{2p}\right)dt + \Psi\bigl(m(T)\bigr)\r],
	\end{equation}
	subject to \eqref{equation 1.1} and
	\begin{equation}\label{eqn-1.3}
		\l| u(t) \r|_{L^2}^2 \leq K\text{ for a.a. } t\in[0,T], \mathbb{P}-a.s.
	\end{equation}

	Here $\bar{m}$ is a given desired state and $\Psi$ is a given Lipschitz continuous function on $L^2$. The existence has been shown for any $p\geq 2$ (as given above) in dimension $d = 1,2,3$. They use the smallness of the parameter $\alpha$ and hence do not consider the terms $m \times \l(m \times u\r)$ and $ m \times \l( m \times h \r)$.\dela{The paper \cite{D+M+P+V} shows the existence of a weak optimal control in dimension d=1,2,3. The existence and large deviations results are possibly in dimension 1.}

	We consider the following problem in this paper.\\
	\dela{Minimizing the cost functional is not the control problem. The control problem is the controlled equation. Try to write the controlled equation separately. The control problem is to solve the controlled equation subject to minimizing the cost functional over a prescribed set.}
	\textbf{Control problem:} Let $\bar{m}\in L^2(\Omega ; L^2(0,T; H^1))$ be a given desired state that takes values on the unit sphere $\mathcal{S}^2$. The terminal cost is given by the function $\Psi : \mathbb{S}^2 \to [0,\infty)$.
	For a fixed $0 < T < \infty$, our aim is to minimize the cost functional
	
	\begin{align}\label{definition of cost functional introduction}
		J(\pi) = \mathbb{E} \l[ \int_{0}^{T} \l( \l|m(t) - \bar{m}(t)\r|_{H^1}^2 + \l| u(t) \r|_{L^2}^2 \r) + \Psi(m(T))\r]
	\end{align}
	over the space of admissible solutions $\pi = \l(\Omega , \mathcal{F} , \l\{ \mathcal{F}_t \r\}_{t\in[0,T]} , \mathbb{P} , W , m , u \r)$ to the problem \eqref{problem considered introduction equation}. \\
	\textbf{Contribution of the paper:}
	\begin{itemize}
		\item The full equation has been considered. That is, we have considered the triple product term both in the drift term and the noise coefficient. Also, the control is added in the Landau-Lifshitz energy $\mathcal{E}$, and hence the effective field $H_{\text{eff}}$, making the control operator is non-linear (in $m$). The noise has been added to the effective field and the full form of noise (including the triple product term) has been considered. Similar to the previous argument, the noise coefficient is no longer linear. In a sense, the complete problem has been considered. Note that the only energy considered is the exchange energy. Anistropy and stray energy have not been considered.
		
		\item \dela{The cost functional contains the $H^1$ norm of the solution.
		}
		We believe that our cost functional is more natural than of Prohl \textit{et. al.}, because firstly the solution $m$ takes values in the $H^1$ space and not $L^2$, so controlling its $H^1$-norm is what matters.
		And secondly, the natural control set is the $L^2$-space and not $H^1$ and hence again controlling its $L^2$-norm is what matters. Moreover, if, as one would hope to prove in the future there exists a solution to our problem with the space-time white noise, then considering the control as an $L^2$-valued function is natural from  Large Deviations point of view. \\
		On the other hand Prohl \textit{et. al.} \cite{D+M+P+V} consider a constraint \eqref{eqn-1.3} on the control. We will return to a similar problem in a forthcoming paper.

		Prohl \textit{et. al.} in \cite{D+M+P+V} show the existence of an optimal relaxed control for the relaxed version of the problem. The control is constructed using compactness properties of random Young measures on a chosen Polish space. The work is inspired mainly from \cite{ZB+RS}, wherein the authors study an optimal relaxed control problem for a semilinear stochastic partial differential equation on a Banach space. In this work, we use the compact embedding of the space $L^2(0,T;L^2)$ into the space $L^2_w(0,T;L^2)$, which is the space $L^2(0,T;L^2)$ endowed with the weak topology.
	\end{itemize}

	\dela{
		\begin{remark}
			It is suggested (see for example \cite{Dunst+Klein}) that the control operator should be the tangent to the solution $m$. Let $a,b\in\mathbb{R}^3$ be non-zero vectors. Then $a,a \times b, a \times (a \times b)$ form an orthogonal basis for $\mathbb{R}^3$. In particular, the space of all vectors that are tangent to $a$ can be spanned by $a \times b, a \times (a \times b)$. For a control process $u$, it is therefore sufficient to consider the linear operator as a linear combination of $m \times u$ and $m \times (m \times u)$.
		\end{remark}
	}
	\textbf{Structure of the paper:} The paper is organized as follows. Before showing the existence of an optimal control, we show that the set of admissible solutions to the problem \eqref{problem considered} is non-empty. This is done in Theorem \ref{Theorem Existence of a weak solution}. It shows the existence of a weak martingale solution for the given problem. The proof uses the Faedo-Galerkin approximation, see Section \ref{Section Faedo Galerkin approximation}, followed by some compact embeddings in order to use the Jakubowski version of the Skorohod Theorem in Section \ref{Section Proof of existence of a solution}. The basic motivation for the proof is taken from \cite{Flandoli_Gatarek}, see also \cite{ZB+BG+TJ_Weak_3d_SLLGE}. Following the resuts in \cite{ZB+BG+TJ_LargeDeviations_LLGE}, we show the pathwise uniqueness for the obtained weak martingale solution in Section \ref{Section Pathwise uniqueness}. This combined with the theory of Yamada and Watanabe, see \cite{Ikeda+Watanabe} give the existence of a strong solution to the problem \eqref{problem considered}. This is followed by showing maximal regularity in Theorem \ref{Theorem Further regularity}. The proof uses the analytic properties of the semigroup generated by the operator $(-\Delta)$.
	
	The existence of an optimal control is then shown in Section \ref{Section Optimal control}. Therein we take a minimizing sequence of admissible solutions and then show that they converge to another admissible solution, which is a minimizer of the cost functional \eqref{definition of cost functional introduction}.

	\section{Notations and Preliminaries}\label{Section Notation}
	Let the domain $\mathcal{O}\subset \mathbb{R}$ be a bounded interval. We fix this domain throughout the paper. The letter $C$ is used for denoting a generic constant, whose value may change from line to line.
	$L^p$ denotes the space $L^p(\mathcal{O}:\mathbb{R}^3)$ and $W^{k,p}$ denotes the space $W^{k,p}(\mathcal{O} : \mathbb{R}^3)$. $\l\langle \cdot , \cdot \r\rangle$, $|\cdot|$ respectively denote the standard inner product and norm in $\mathbb{R}^3$.

	Let $v\in L^{\infty}$, $q\in[1,\infty]$.
	
	$G(v):L^{q}\rightarrow L^{q}$ is defined as follows
	\dela{\begin{equation}
			G(v)k:=v\times k-\alpha \, v\times(v\times k)\ \text{for}\ k\in L^q.
	\end{equation}}
	
	\begin{equation}\label{definition of G}
		G(v) : L^q \ni k \mapsto v \times k - \alpha \, v \times (v \times k) \in L^q.
	\end{equation}
	
	The following lemma shows that the above defined $G$ is a polynomial map. Moreover, the restriction of $G$ to the space $L^{\infty} \cap H^1$ is also a polynomial map.

	\dela{\adda{Check the definition properly. Is the statement really needed?? Is the definition written above not sufficient??}
		
		Fix $\adda{v??} h\in L^{\infty}\cap H^1$.
		$G\adda{(h)}: L^{\infty} \rightarrow L^{\infty}$ and $G\adda{(h)}: L^{\infty}\cap H^1 \rightarrow L^{\infty}\cap H^1$ given by
		\begin{equation}
			G(v) h :=v\times h-\alpha \, v\times(v\times h)
		\end{equation}
		is well defined. \adda{Why well defined??}
		
		\adda{	Fix $v\in L^{\infty}\cap H^1$.
			$G(v): L^{\infty} \rightarrow L^{\infty}$ and $G(v): L^{\infty}\cap H^1 \rightarrow L^{\infty}\cap H^1$ given by
			\begin{equation}
				G(v) h :=v\times h-\alpha \, v\times(v\times h)
			\end{equation}
			is well defined.
			
		}
	}

	\begin{lemma}\label{Lemma G is a polynomial map}
		Let $q\in[1,\infty]$. The map $G:L^{\infty}\rightarrow \mathcal{L}(L^q)$ is a polynomial map of degree 2. Hence it is of polynomial growth and is Lipschitz on balls, that is there exists a constant $C_0>0$ such that
		\begin{equation}\label{eqn-2.3}
			|G(v)k|_{L^q}\leq C_0(1+|v|_{L^{\infty}})|v|_{L^{\infty}}|k|_{L^q}.
		\end{equation}
		Thus,
		\begin{equation*}
			|G(v)|_{\mathcal{L}(L^q)}\leq C_0(1+|v|_{L^{\infty}})|v|_{L^{\infty}}.
		\end{equation*}
		Moreover for every $r>0$, there exists a constant $C_r>0$ such that for all $v_i\in L^{\infty}$ with $|v|_{L^{\infty}}\leq r$, we have
		\[
		|G(v_1)k-G(v_2)k|_{L^q}\leq C_r|k|_{L^q}|v_1-v_2|_{L^{\infty}},\ k\in L^q.
		\]
		Thus,
		\begin{equation*}
			|G(v_1)-G(v_2)|_{\mathcal{L}(L^q)}\leq C_r|v_1-v_2|_{L^{\infty}} ,\ \text{if}\ v_1,v_2\in L^{\infty}, |v_i|_{L^{\infty}}\leq r\ \text{for}\ i=1,2.
		\end{equation*}
		Finally, the restriction of the map $G$ to the space $H^1\cap L^{\infty}$ takes values in the space $\mathcal{L}(H^1\cap L^{\infty})$, i.e.
		\[
		G:H^1\cap L^{\infty}\rightarrow \mathcal{L}( H^1 \cap L^{\infty}).
		\]
		As such this map $G$ is also a polynomial of degree $2$ and hence of $C^\infty$ class and of quadratic growth.
		
	\end{lemma}
	\begin{proof}[Proof of Lemma \ref{Lemma G is a polynomial map}] Let us choose and fix  $q\in[1,\infty]$. \begin{itemize}
			\item[\textbf{Step 1}] The map $G:L^{\infty}\rightarrow \mathcal{L}(L^q)$ is a polynomial of degree $2$, because $G$ is a linear combination   of a linear map $G_1: L^{\infty} \ni v \mapsto \{ h \mapsto  v\times h \} \in  \mathcal{L}(L^q)$ and of a  homogenous polynomial of degree $2$, $G_2:  L^{\infty} \ni v \mapsto  \{ h \mapsto  \alpha \, v\times(v\times h)\} \in  \mathcal{L}(L^q)$. Note that $G_1$ is continuous because for
			$v\in L^{\infty},h\in L^q$, we have the following.
			\begin{align*}
				\l| G_1(v) h \r|_{L^{q}} & = \l| v \times h \r|_{L^{q}} \leq \l| v \r|_{L^{\infty}} \l| h \r|_{L^{q}}.
			\end{align*}
			We define the bilinear map corresponding to $G_2$ by
			\begin{equation*}
				\tilde{G}_2:  L^{\infty}\times L^{\infty} \ni (v_1,v_2) \mapsto  \{ h \mapsto  \alpha \, v_1\times(v_2\times h)\} \in  \mathcal{L}(L^q).
			\end{equation*}
			Note that $G_2$ is a continuous bilinear map because for $v_1,v_2\in L^{\infty},k\in L^q$ we have
			\begin{align*}
				|\tilde{G}_2(v_1,v_2)k|_{L^{q}}&
				\leq | v_1\times(v_2\times k)|_{L^{q}}\\
				&\leq | v_1|_{L^{\infty}} |(v_2\times k)|_{L^{q}}
				\leq |v_1|_{L^{\infty}}|v_2|_{L^{\infty}}|k|_{L^q}.
			\end{align*}
			\item[\textbf{Step 2}]
			So we proved that map $G:L^{\infty}\rightarrow \mathcal{L}(L^q)$ is a polynomial of degree $2$. It follows that $G$ is a $C^\infty$ function, Lipschitz on balls and of quadratic growth.
			Concerning the last claim, the inequality \eqref{eqn-2.3} is also a consequence of the above proof. 		
			\item[\textbf{Step 3}] We have already proved that map $G$ Lipschitz on balls but we can make this assertion more precise. \dela{\adda{The calculations below can be made easier and simpler by doing them separately for $G_1$ and $G_2$.}}
			Let $v_1,v_2\in L^{\infty}$ with $|v_i|_{L^{\infty}}\leq r$ for $i=1,2$ and $k\in L^q$.
			\begin{align*}
				|G_1(v_1)k - G_1(v_2)k|_{L^{q}}&=|v_1\times k - v_2\times k|_{L^{q}}\\
				& = |(v_1 - v_2)\times k|_{L^q}\\
				& \leq |v_1-v_2|_{L^{\infty}}|k|_{L^q}.
			\end{align*}
			The second term $G_2$ is also Lipschitz on balls can be shown as follows.
			

			Let $v_1,v_2\in L^{\infty}$ with $|v_i|_{L^{\infty}}\leq r$ for $i=1,2$ and $k\in L^q$.

			\begin{align*}
				|G_2(v_1)k - G_2(v_2)k|_{L^{q}}&=\alpha \, | v_1\times(v_1\times k) - \alpha \, v_2\times(v_2\times k)|_{L^{q}}\\
				& = \alpha|v_1\times(v_1\times k) - v_2\times(v_1\times k) + v_2\times(v_1\times k) - v_2\times(v_2\times k)|_{L^{q}}|_{L^{q}}\\
				& \leq |(v_1 - v_2) \times (v_1 \times k)|_{L^q} + |v_2 \times
				((v_1 - v_2) \times k)|_{L^q}\\
				& \leq |v_1 - v_2|_{L^{\infty}}|v_1|_{L^{\infty}}|k|_{L^q} + |v_1 - v_2|_{L^{\infty}}|v_2|_{L^{\infty}}|k|_{L^q}\\
				& \leq 2r |v_1 - v_2|_{L^{\infty}}|k|_{L^q}.
			\end{align*}
			
			Thus,
			\begin{equation*}
				|G(v_1)k - G(v_2)k|_{L^q} \leq C_r |v_1 - v_2|_{L^{\infty}}|k|_{L^q},
			\end{equation*}
			for some constant $C_r$ depending on $r$.
			
			
			\item[\textbf{Step 4}] To prove the last part one can deal separately with the maps $G_1$ and $G_2$. We begin with $G_1$.
			Let $v, h \in H^1 \cap L^{\infty}$.  By Step 1 we have
			\begin{align*}
				|G_1(v)h|_{L^\infty } \leq  |v|_{L^\infty }|h|_{L^\infty },
			\end{align*}
			and
			\begin{align*}
				|G_1(v)h|_{L^2 } \leq  |v|_{L^\infty }|h|_{L^2 }.
			\end{align*}
			Moreover,
			\begin{align*}
				|\nabla [G_1(v)h]|_{L^2} &= |\nabla (v \times h)|_{L^2}
				\leq |\nabla v \times h|_{L^2} + |v \times \nabla h|_{L^2} \\
				& \leq |\nabla v|_{L^{2}}|h|_{L^{\infty}} + |v|_{L^{\infty}}|\nabla h|_{L^2}. 		
			\end{align*}
			Summing up, we have proved that for some constant $C>0$
			\begin{align*}
				|G_1(v)h|_{L^\infty }+|G_1(v)h|_{H^1} &\leq C |v|_{L^\infty } \bigl( |h|_{L^\infty }+ |h|_{H^1 } \bigr)+  |v|_{H^1}|h|_{L^{\infty}}\\
				&\leq C \bigl( |v|_{L^\infty }+ |v|_{H^1}\bigr) \times \bigl( |h|_{L^\infty }+ |h|_{H^1 } \bigr).
			\end{align*}
			We used here that
			\begin{align*}
				|u|_{H^1}^2  := |u|_{L^2}^2 + |\nabla u|_{L^2}^2.
			\end{align*}
			
			In the same way one can prove that
			\begin{align*}
				|G_2(v)h|_{L^\infty }+|G_2(v)h|_{H^1}
				&\leq C \bigl( |v|^2_{L^\infty }+ |v|_{L^\infty }|v|_{H^1}\bigr)^2 \times C\bigl( |h|_{L^\infty }+ |h|_{H^1 } \bigr).
			\end{align*}
			
			
		\end{itemize}
		The proof is complete.
	\end{proof}
	We have proved above that $G$ is of $C^\infty$ class. We can calculate the Fr\'echet derivative $DG(v),v\in L^{\infty}$, of $G$ in the following way.  	
	\begin{proposition}\label{prop-derivative}
		Assume that $q \in [1,\infty]$.
		For $v,w\in L^{\infty}$ we have
		\begin{align}\label{eqn-derivative of DG}
			DG(v)(w) = \{ L^q \ni  h \mapsto  w \times h - \alpha \, \left[ v \times (w \times h) + w \times (v \times h) \right] \in L^q\} \in \mathcal{L}(L^q).
		\end{align}
		Moreover, the same holds when the spaces $L^{\infty}$ and $L^{q}$ are replaced by the $H^1 \cap L^{\infty}$ as in the last part of the previous Lemma.
	\end{proposition}
	
	\begin{proof}[Proof of Proposition \ref{prop-derivative}] We only consider the first part since the second is completely analogous. \\
		Since $G = G_1 - \alpha \, G_2$ by the proof of previous Lemma, it is sufficient to consider $G_1$ and $G_2$ separately. Since $G_1$  is bounded linear map, by Proposition 2.4.2,
		\cite{Cartan_Calculus},
		\[
		DG_1(v)(w)=G_1(w).
		\]
		Concerning $G_2$, we observed that $G_2(v)=\tilde{G}_2(v,v)$ for all $v \in L^\infty$, where $\tilde{G}_2$ is the corresponding continuous bilinear map. Hence, see Theorem 2.4.3, \cite{Cartan_Calculus},
		\[
		DG_2(v)(w)=\tilde{G}_2(v,w)+\tilde{G}_2(w,v).
		\]
		The result follows.
	\end{proof}

	\section{Statements of the Main results}\label{Section Statements of the main Results}
	
	\dela{Change section to Statements of the main results
	}
	\dela{Also change the name of Section \ref{Section Proof of existence of a solution} to Proof of Theorem \ref{Theorem Existence of a weak solution}: Proof of existence of a solution.}
	We first state some assumptions that will be required in the proof for Theorem \ref{Theorem Existence of a weak solution}.
	
	\begin{assumption}\label{Assumptions for existence of weak martingale solution}
		
		\begin{enumerate}
			\item\label{Assumptions for existence of weak martingale solution point 1} Let $\l( \Omega , \mathcal{F} , \mathbb{F} , \mathbb{P} \r)$ be a probability space which satisfies the usual hypotheses. That is,
			\begin{enumerate}
				\item $\mathbb{P}$ is complete on $(\Omega , \mathcal{F})$.
				\item For every $t\geq 0$, $\mathcal{F}_t$ contains every $(\mathcal{F} , \mathbb{P})$-null set.
				\item The filtration $\mathbb{F} = \{ \mathcal{F}_{t} \}_{t \in [0,T]}$ is right continuous.
			\end{enumerate}
			
			\item\label{Assumptions for existence of weak martingale solution point 2} $W$ is a real valued Wiener process defined on the above probability space with the filtration $\mathbb{F}$.
			
			\item\label{Assumptions for existence of weak martingale solution point 3} The given function $h$ is assumed to be in the space $H^1$, and is independent of time.

			\item\label{Assumptions for existence of weak martingale solution point 4} 
			A process $u$ is an $\mathbb{F}$- progressively measurable process such that
			the following inequality holds for each $p\geq 1$,
			\begin{equation}\label{assumption on u}
				K_p:=\mathbb{E} \l( \int_{0}^{T} \l| u (t) \r|_{L^2}^2 \, dt \r)^p < \infty.
			\end{equation}
			In particular, the  trajectories of $u$ take values in $L^2(0,T;L^2)$.
		\end{enumerate}
	\end{assumption}
	Define the following operator
	\begin{align*}
		\Delta : &H^1 \rightarrow (H^1)^\p \\
		&	v  \mapsto \Delta v,
	\end{align*}
	where for $w\in H^1,v\in H^1$, the linear map $\Delta v\in (H^1)^\p$ is defined by
	
	\begin{equation}\label{Interpretation eqn 1}
		\ _{(H^1)^\p}\l\langle \Delta v , w \r\rangle_{H^1} = - \l\langle \nabla v , \nabla w \r\rangle_{L^2}.
	\end{equation}
	
	\begin{equation*}
		\l( \Delta v \r) \l( w \r) : = - \l\langle \nabla v , \nabla w \r\rangle_{L^2}.
	\end{equation*}
	Since $v,w\in H^1$,
	\begin{align*}
		\l| 	\ _{(H^1)^\p}\l\langle \Delta v , w \r\rangle_{H^1} \r| &= \l| \l\langle \nabla v , \nabla w \r\rangle_{L^2} \r| \\
		& \leq \l|\nabla v\r|_{L^2}\l|\nabla w\r|_{L^2} \\
		& \leq \l|v\r|_{H^1}\l|w\r|_{H^1}<\infty.
	\end{align*}
	Hence for $v\in H^1$, $\Delta v$ as defined above is in $(H^1)^{\prime}$.
	
	Let $v_1,v_2, w\in H^1$, $a_1,a_2\in\mathbb{R}$.
	\begin{align*}
		\ _{(H^1)^\p}\l\langle \Delta \l( a_1v_1 + a_2v_2 \r) , w \r\rangle_{H^1} &= - \l\langle \nabla \l( a_1v_1 + a_2v_2 \r) , \nabla w \r\rangle_{L^2} \\
		& = - \l\langle \nabla  a_1v_1  , \nabla w \r\rangle_{L^2} + \l\langle \nabla  a_2v_2  , \nabla w \r\rangle_{L^2} \\
		& = a_1 \l\langle  \nabla  v_1  , \nabla w \r\rangle_{L^2} + a_2 \l\langle \nabla  v_2  , \nabla w \r\rangle_{L^2} \\
		& = a_1\ _{(H^1)^\p}\l\langle \Delta  v_1  , w \r\rangle_{H^1} + a_2\ _{(H^1)^\p}\l\langle \Delta  v_2  , w \r\rangle_{H^1}.
	\end{align*}
	Hence $\Delta$ is a linear operator from the space $H^1$ to its dual $(H^1)^\p$.\\
	Define an operator $A$ (Neumann Laplacian) on its domain $D(A)\subset L^2$ to $L^2$ as follows.
	\begin{align}\label{Definition of Neumann Laplacian}
		\begin{cases}
			D(A) &:= \l\{v\in H^2 : \nabla (v)(x) = 0, \ \text{for}\ x\in\partial \mathcal{O}\r\}, \\
			Av &:= -\Delta v, \ \text{for}\ v\in D(A).
		\end{cases}	
	\end{align}
	It is known that the operator $A$ is a self-adjoint operator in $L^2$. We define another operator $A_1$ by
	\begin{equation}\label{Definition of A_1}
		A_1 = I_{L^2} + A,
	\end{equation}
	where $I_{L^2}$ denotes the identity operator on $L^2$. It is also known that $\l(A_1\r)^{-1}$ is compact. Also,
	the space $D(A_1^{\frac{1}{2}})$ equipped with the graph norm coincides with the space $H^1$.
	
	For $v_1,v_2,v_3\in H^1$, we interpret terms $\Delta v_1$, $v_1 \times \Delta v_2$ and $v_1 \times (v_2 \times \Delta v_3)$ as elements of $(H^1)^\p$. Let $\phi\in H^1$.
	\begin{enumerate}
		\dela{	\item
			\begin{equation}
				\ _{(H^1)^\p}\l\langle \Delta v_1 , \phi \r\rangle_{H^1} = - \l\langle \nabla v_1 , \nabla \phi \r\rangle_{L^2}.
		\end{equation}}
		\item
		\begin{equation}\label{Interpretation eqn 2}
			\ _{(H^1)^\p}\l\langle v_2 \times \Delta v_1 , \phi \r\rangle_{H^1} = -\l\langle \nabla \l( \phi \times v_2 \r) , \nabla v_1 \r\rangle_{L^2}.
		\end{equation}
		\item
		\begin{equation}\label{Interpretation eqn 3}
			\ _{(H^1)^\p}\l\langle v_3 \times \l( v_2 \times \Delta v_1 \r) , \phi \r\rangle_{H^1} = -\l\langle \nabla \l(\l(
			\phi \times v_3  \r) \times v_2 \r) , \nabla v_1 \r\rangle_{L^2}.
		\end{equation}
	\end{enumerate}
	The above equalities can be obtained from the divergence theorem in case $v_1\in D(A)$.

	\dela{Here the function $h$ needs to be given. The regularity may be told later. But the function needs to be fixed.}

	\textbf{Note:} We require that the magnetization is saturated at each point $t\in[0,T]$, that is, the constraint condition \eqref{eqn-constraint condition} is satisfied by the process $m$.\dela{ We assume the constraint to be $1$ for simplicity.} For that, the initial data $m_0$, see \eqref{problem considered} should lie on the unit sphere $\mathbb{S}^2$. Towards that, we denote by $W^{1,2}(\mathcal{O}:\mathbb{S}^2)$ the space of all $v\in W^{1,2}(\mathcal{O}:\mathbb{R}^3)$ such that $\l| v(x) \r|_{\mathbb{R}^3} = 1$ for Leb.-a.a. $x\in\mathcal{O}$.\\
	We now recall the problem that will be considered (that is \eqref{problem considered introduction equation}).
	\begin{align}\label{problem considered}
		\begin{cases}
			& dm = \big[ m\times\Delta m + m\times u -  \alpha \, m\times\l(m\times u\r)   - \alpha \, m\times\l(m\times \Delta m\r)\big]dt\\
			& \qquad + G(m)\circ dW(t)\dela{\ \text{in}\ \mathcal{O}_T=(0,T)\times \mathcal{O}} , \ t \in [0,T],\\
			& \frac{\partial m}{\partial \nu} = 0,\  \text{on } \partial \mathcal{O}_T,\\
			& m(0,\cdot) = m_0\ \text{on}\ \mathcal{O}.
		\end{cases}
	\end{align}
	
	\dela{	\begin{align}
			\nonumber dm(t) &=\left[m(t) \times \Delta m(t) - \alpha \, m(t)\times\left(m(t)\times u(t)\right) - \alpha \, m(t) \times \left(m(t) \times \Delta m(t)\right) + m(t)\times u(t)\right] \, dt\\
			& + G(m(t)) \circ dW(t).
		\end{align}
	}
	Here 
	\begin{equation*}
		G(m) = m \times h - \alpha \, m \times (m \times h).
	\end{equation*}\\	
	Note that the stochastic term is understood in the Stratonovich sense. It can also be understood in the It\^o sense by adding a correction term, see, for example, \cite{Prato+Zabczyk}, \cite{Oksendal_Book}. The resulting equation is
	\begin{align}\label{definition of solution}
		\nonumber	dm(t) = & \bigg[m(t) \times \Delta m(t) - \alpha \, m(t) \times (m(t) \times \Delta m(t)) + m(t)\times u(t) - \alpha \, m(t)\times(m(t)\times u(t))   \\
		& + \frac{1}{2}\bigl[DG(m(t))\bigr]\big[G\big(m(t)\big)\big]\bigg] dt + G\big(m(t)\big) \,  dW(t).
	\end{align}

	\begin{definition}[Weak martingale solution]\label{Definition of Weak martingale solution}

		\dela{
			Let $T>0$ and $m_0\in W^{1,2}(\mathcal{O}: \mathbb{S}^2)$ be fixed. Given a stochastic process $u$,  a weak martingale solution of \eqref{problem considered} is a tuple
		}
		Assume $T>0$. Let the function $h$ and a control $u$ be given as in Assumption \ref{Assumptions for existence of weak martingale solution}. A weak martingale solution of \eqref{problem considered} is a tuple
		\begin{equation*}
			\pi^\p = (\Omega^\p, \mathcal{F}^\p, \mathbb{F}^\p, \mathbb{P}^\p, W^\p, m^\p, u^\p)
		\end{equation*} such that

		\begin{enumerate}
			\item $(\Omega^\p, \mathcal{F}^\p, \mathbb{P}^\p)$ is a probability space satisfying the usual hypotheses.
			$W^\p$ is a real valued $\mathbb{F}^\p$-adapted Wiener process.

			\item $ \m $ is $H^1$-valued progressively measurable process such that for $\mathbb{P}^\p$-a.s. $\omega \in \Omega^\p$,
			$\m(\omega, \cdot)\in C([0,T]; L^2)$.
			
			\item The process $\m$ satisfies the constraint condition. That is
			\begin{equation}\label{eqn-constraint condition}
				\l|\m(t, x)\r|_{\mathbb{R}^3} = 1,\ \text{for Leb. a.a. }x\in \mathcal{O},\ \text{for all }t\in[0,T],  \ \mathbb{P}^\p\text{-a.s.}
			\end{equation}

			\item \dela{The control process $\u$ is progressively measurable taking values in the space $L^{2}(0,T;L^2)$.} $\u$ is a control process satisfying the assumptions in Assumption \ref{Assumptions for existence of weak martingale solution}  and has the same law on the space $L^2(0,T ; L^2)$ as that of the process $u$.

			\item There exist constants $C_1,C_2>0$ such that for each $p\geq 1$,

			\begin{enumerate}
				\item\begin{equation}\label{bound using u 1}
					\mathbb{E}^{\prime}  \sup_{t\in [0,T]} \l| \m(t) \r|_{H^1}^{2p} \leq C_1 + KC_2,
				\end{equation}
				
				\item
				\begin{equation}\label{bound using u 2}
					\mathbb{E}^{\prime} \left(\int_{0}^{T} \l| \m(t) \times \Delta \m(t) \r|_{L^2}^2\, dt \right)^p  \leq C_1 + KC_2.
				\end{equation}
				Here $\mathbb{E}^\p$ denotes the expectation in $(\Omega^\p, \mathcal{F}^\p, \mathbb{P}^\p)$.
			\end{enumerate}

			\item The paths of $\m(\omega^{\prime})$
			are  continuous taking values in $X^{\beta}$ for any $\beta<\frac{1}{2}$ for $\mathbb{P}^{\prime}$-a.s. $\omega^{\prime}\in\Omega^{\prime}$. The space $X^{\beta}$ is the domain of the operator $D(A_1^\beta)$. More about this is given in Section \ref{Section Faedo Galerkin approximation}.
			
			\item \dela{Is $\phi \in$ $H^1(\mathcal{O})$ sufficient?? Is $\phi$ required to be a process??}
			
			For every $\phi \in$ $H^1(\mathcal{O})$ and every $t\in[0,T]$, the following equality holds $\mathbb{P}^\p$-a.s.
			
			\begin{align}
				\nonumber \l\langle \m(t) , \phi \r\rangle_{L^2} =& \l\langle m_0 , \phi \r\rangle_{L^2} +
				\int_{0}^{t} \l\langle \nabla \m(s) , \m(s) \times \nabla \phi \r\rangle_{L^2} \, ds \\
				\nonumber & + \int_{0}^{t} \l\langle \nabla \m(s) , \nabla (\phi \times \m(s)) \times \m(s) \r\rangle_{L^2} \,ds \\
				\nonumber &  + \int_{0}^{t} \l\langle   \m(s) \times \u(s) , \phi \r\rangle_{L^2} \,ds \\
				\nonumber & - \alpha \, \int_{0}^{t} \l\langle \m(s) \times ( \m(s)\times \u(s)) , \phi
				\r\rangle_{L^2}\, ds \\
				\nonumber & + \frac{1}{2} \int_{0}^{t} \l\langle \l[DG(m(s))\r]\big(G\big(m(s)\big)\big]  ,\phi \r\rangle_{L^2} \, ds \\
				& + \int_{0}^{t} \l\langle G(\m(s))  , \phi \r\rangle_{L^2} \circ dW^\p(s).
			\end{align}
		\end{enumerate}
	\end{definition}
	\dela{Write the equation in It\^o form. Weak formulation and Stratonovich integral should be written carefully.}
	Now we state the existence theorem for a weak martingale solution for the problem \eqref{problem considered}.
	
	\begin{theorem}[Existence of a weak martingale solution]\label{Theorem Existence of a weak solution}

		Let the assumptions in Assumption \ref{Assumptions for existence of weak martingale solution} hold. Let $u$ be a given control process satisfying Assumption \ref{Assumptions for existence of weak martingale solution}. Let the initial data $m_0$ be in $ W^{1,2}(\mathcal{O},\mathbb{S}^2)$. Then the problem \eqref{problem considered} admits a weak martingale solution
		\begin{equation*}
			\pi^\p = \l( \Omega^{\prime} , \mathcal{F}^\p , \mathbb{F}^\p , \mathbb{P}^\p , W^\p , \m , \u \r)
		\end{equation*}
		as in Definition \ref{Definition of Weak martingale solution}\dela{, with the obtained control process $u^\p$ having the same law as the process $u$}.

		\dela{
			
			Let $\l(\Omega , \mathcal{F} , \mathbb{P}\r)$ be a probability space satisfying the usual hypotheses. Let $p\geq 1$. Suppose that the law of the initial data $m_0$ is such that $m_0\in W^{1,2}(\mathcal{O},\mathbb{S}^2)$ (that is $m_0\in W^{1,2}(\mathcal{O})$ with $|m_0(x)|_{\mathbb{R}^3} = 1$ for Leb. a.a. $x \in \mathcal{O}$) $\mathbb{P}-a.s.$ Let $h\in H^1$ be given. Let $(\Omega, \mathcal{F}, \mathbb{F} , \mathbb{P})$ be a filtered probability space satisfying the usual hypotheses. Let $u$ be a given control process having values in the space $L^2(0,T;L^2)$ and satisfying

			\begin{equation}
				\mathbb{E}\left[\left(\int_{0}^{T} |u(t)|^{2}_{L^{2}}\,dt\right)^p\right] \leq K
			\end{equation}

			for a given constant $K$.

			Then the problem \eqref{problem considered} has a weak martingale solution \\
			$$(\Omega^{\prime}, \mathcal{F}^{\prime}, \mathbb{F}^{\prime}, \mathbb{P}^{\prime}, W^{\prime}, \m, u^{\prime}).$$ \\
			The processes $\m , u^{\prime}: \Omega^{\prime} \times \mathcal{O} \times [0,T]\rightarrow \mathbb{R}^{3}$ are such that

			\begin{enumerate}

				\item $u^{\prime}$ is a progressively measurable stochastic process such that

				\dela{	\begin{equation}
						\mathbb{E}^{\prime}|u^{\prime}|_{L^{2p}(0,T;L^2)}^{2p} = \mathbb{E^{\prime}}\left[\left(\int_{0}^{T} |u^{\prime}(t)|^{2p}_{L^{2}}\,dt\right)\right] \leq K.
					\end{equation}	
				}

				\adda{No need of writing this here. It has already been mentioned in the definition.}
				
				\begin{equation}\label{assumption on u_n}
					\mathbb{E}^{\prime}|u^{\prime}|_{L^{2}(0,T;L^2)}^{2p} = \mathbb{E^{\prime}}\left[\left(\int_{0}^{T} |u^{\prime}(t)|_{L^{2}}^2 \, dt\right) ^{p} \right] \leq K.
				\end{equation}

				
				\adda{Write this in the definition. Write both the bounds separately. Also write it in terms of $K$. The term in $\u$ need not be mentioned.}
				
				\dela{\item There exist constants $C_1,C_2>0$ such that
					\begin{equation}
						\mathbb{E}^{\prime} \left[ \sup_{t\in [0,T]} |\m(t)|_{H^1}^{2p} + \left(\int_{0}^{T}|\m(s) \times \Delta \m(s) |^2\, ds\right)^p \right] \leq C_1 + C_2\mathbb{E^{\prime}}\left[ \left(\int_{0}^{T} |u^{\prime}(t)|_{L^2}\, dt \right)^{p} \right].
				\end{equation}}
				
				\item Moreover, the paths of $\m(\omega^{\prime})$
				are  continuous taking values in $X^{\beta}$ for any $\beta<\frac{1}{2}$ for $\mathbb{P}^{\prime}$-a.s. $\omega^{\prime}\in\Omega^{\prime}$.

			\end{enumerate}
		}

	\end{theorem}

	The proof of Theorem \ref{Theorem Existence of a weak solution} has its motivation mainly from the papers \cite{ZB+BG+TJ_Weak_3d_SLLGE} and \cite{Flandoli_Gatarek}. The way to deal with the triple product term in the noise is similar to the work \cite{ZB+BG+TJ_LargeDeviations_LLGE}. The proof begins with the Faedo-Galerkin approximation (Section \ref{Section Faedo Galerkin approximation}), followed by obtaining uniform energy estimates. Then some compactness results, see \cite{Simon_Compact_Sets}, \cite{Simons}, etc. are used to show the tightness of the laws of the finite dimensional approximations on appropriate spaces. Prokhorov's theorem, followed by the Jakubowski version of the Skorohod Theorem are then applied (in Section \ref{Section Proof of existence of a solution}) to show the convergence of the sequence of approximates (possibly along a subsequence) and hence the existence of a weak martingale solution. An application of the It\^o Lemma in Section \ref{sec-The constraint condition section} yields the constraint condition \eqref{eqn-constraint condition}.
	
	As a corollary of the existence result (Theorem \ref{Theorem Existence of a weak solution}), we can prove that equation \eqref{problem considered} makes sense in the strong (P.D.E.) sense in $L^2$. This is formally stated in Corollary \ref{Strong form of weak martingale solution}.

	The next theorem states that the solutions to the problem \eqref{problem considered} are pathwise unique. Using this and the theory of Yamada and Watanabe, we also show the existence and uniqueness of a strong solution. The main result of the section is Theorem \ref{thm-pathwise uniqueness}.
	
	\begin{theorem}[Pathwise uniqueness]\label{thm-pathwise uniqueness}
		Let us assume that  process  $u$ is a  control process such that the Assumption \ref{assumption on u} holds. Let $(\Omega , \mathcal{F} , \mathbb{P} , W , m_1 , u)$ and $(\Omega , \mathcal{F} , \mathbb{P} , W , m_2 , u)$ be two weak martingale solutions to \eqref{problem considered} (with the same initial data $m_0$), corresponding to a given control process  $u$, as in Definition \ref{Definition of Weak martingale solution} and satisfying the properties stated in Theorem \ref{Theorem Existence of a weak solution}. Then
		
		\begin{equation*}
			m_1(t) = m_2(t)\ \mathbb{P}-a.s.
		\end{equation*}
		for each $t\in [0,T]$.
	\end{theorem}
	
	The existence of a unique strong solution to the problem \eqref{problem considered} is shown as a consequence, see Theorem \ref{thm-existence of a strong solution} to the above theorem. It follows from the pathwise uniqueness and the theory of Yamada and Watanabe, see \cite{Ikeda+Watanabe}.

	Section \ref{Section Further regularity} deals with the proof of Theorem \ref{Theorem Further regularity}. Here we show that the obtained solution takes values in $D(A_1)$. We first write the obtained equation in the mild form. Towards this, Corollary \ref{Corollary m times m times Delta m equals Delta m plus gradient m squared m} shows that whenever $m$ satisfies the constraint condition, we have the following equality $(H^1)^\p$.
	\begin{equation*}m \times (m \times \Delta m) = -\Delta m - |\nabla m|_{\mathbb{R}^3}^2m.
	\end{equation*}
	Therefore the equations \eqref{problem considered} and \eqref{problem considered for optimal control part} are equivalent. The proof mainly follows by using the ultracontractivity and the maximal regularity properties of the semigroup generated by the operator $A$.
	\begin{theorem}[Maximal regularity]\label{Theorem Further regularity}
		Let $\mathcal{O}\subset\mathbb{R}$ be bounded. Let the probability space and initial data, along with the given control process $u$ be as given in Theorem \ref{Theorem Existence of a weak solution} . Also assume that the process $u$ satisfies the assumption \eqref{assumption on u} for $p=2$. Then there exists a unique strong solution $m$ which satisfies the properties mentioned in Theorem \ref{Theorem Existence of a weak solution}. Moreover,  there exists a constant $C>0$ such that
		\begin{align}\label{Further regularity}
			\mathbb{E} \left( \int_{0}^{T} | \nabla m(t) |_{L^4}^4\, dt  + \int_{0}^{T} |A_1 m(t)|_{L^2}^2\, dt\right) \leq C.
		\end{align}
	\end{theorem}

	Section \ref{Section Optimal control} shows that the problem \eqref{problem considered} admits an optimal control corresponding to the cost functional \eqref{definition of cost functional introduction}. Let $\mathcal{U}_{ad}(m_0,T)$ denote the space of all admissible solutions of the problem \eqref{problem considered} (to be detailed in Section \ref{Section Optimal control}). The idea for the proof is to show that the space $\mathcal{U}_{ad}(m_0,T)$ is non-empty. This gives us a minimizing sequence of admissible solutions. This sequence is then shown to converge (in a suitable sense) to an admissible solution, which is a minimizer of the cost functional \eqref{definition of cost functional}.

	\begin{definition}[Optimal control]\label{definition of optimal control}
		Let the law of the initial data $m_0$ be as in Theorem \ref{Theorem Existence of a weak solution}. Let $h\in H^1$ be fixed. An admissible solution of the problem \eqref{problem considered}
		$$\pi^\p = (\Omega^\p , \mathcal{F}^\p , \mathbb{F}^\p , \mathbb{P}^\p , W^\p , m^\p , u^\p) \in \mathcal{U}_{ad}(m_0,T)$$
		is said to be an optimal control for the problem \eqref{problem considered}\dela{\eqref{Optimal control problem}} if and only if
		\begin{equation}\label{eqn-?}
			J(\pi^\p) = \inf_{\pi \in \mathcal{U}_{ad}(m_0,T)} J(\pi),
		\end{equation}
		i.e.
		\begin{equation}
			J(\pi^\p) = \min_{\pi \in \mathcal{U}_{ad}(m_0,T)} J(\pi),
		\end{equation}
	\end{definition}
	The infimum in \eqref{eqn-?} easily exists. The main difficulty lies in showing the existence of a minimum.

	\textbf{Note:} Since we also show the existence of a strong solution, one can also consider the formulation of the control problem using strong solutions instead of weak martingale solutions. We choose to use strong martingale solutions (Definition \ref{Definition Strong martingale solution}) instead. Some more details are given in Remark \ref{Remark Strong optimal control}.

	\begin{theorem}[Existence of optimal control]\label{Theorem existence of optimal control}
		Let $\l( \Omega , \mathcal{F}, \mathbb{F} , \mathbb{P} \r)$ be a probability space satisfying the usual conditions, see Assumption \ref{Assumptions for existence of weak martingale solution}. Let $W$ be a real valued Wiener process on the space $\l( \Omega , \mathcal{F}, \mathbb{F} , \mathbb{P} \r)$. Let the initial data $m_0$ be as in Theorem \ref{Theorem Existence of a weak solution} and the function $h$ satisfy Assumption \ref{Assumptions for existence of weak martingale solution}. Let $u$ be a given control process satisfying \eqref{Assumptions for existence of weak martingale solution point 4} of Assumption \ref{Assumptions for existence of weak martingale solution}, in particular satisfying \eqref{assumption on u} for $p = 4$. Then there exists an optimal control for the problem \eqref{problem considered} according to Definition \ref{definition of optimal control}.
	\end{theorem}
	

	\section{Faedo-Galerkin approximation}\label{Section Faedo Galerkin approximation}
	
	Let $A = -\Delta$ be the Neumann Laplacian operator defined in \eqref{Definition of Neumann Laplacian}.
	Let $\{e_i\}_{i\in\mathbb{N}}$ be an orthonormal basis of $L^2$ consisting of eigen functions of $A$ (for example refer to \cite{Evans} page 335 Theorem 1).
	
	Recall the operator $A_1 = I_{L^2} + A$ defined in \eqref{Definition of A_1}.	
	For $\beta \geq 0$, let us define the space
	\begin{equation}\label{definition of the space X beta}
		X^{\beta} = \text{dom}\l(A_1^{\beta}\r).
	\end{equation}
	Without the loss of generality, if $\mathcal{O} = (0,1)$, then it is known that
	
	\begin{align}
		X^{\beta} = \begin{cases}
			\l\{ v\in H^{2\beta}(0,1 ; \mathbb{R}^3) \, ;
			v^\p(1) = v^\p(0) = 0  \r\}, \text{ if } 2\beta > \frac{1}{2},\\
			H^{2\beta}, \text{ if } 2\beta \leq \frac{1}{2}. 
		\end{cases}       
	\end{align}

	Its dual space is denoted by $X^{-\beta}$.
	For $\beta = 0$, we have
	$$X^0 = L^2.$$
	
	\dela {Note that for $0\leq \beta < \frac{3}{4}$, $X^{\beta} = H^{2\beta}$. (See \cite{ZB+UM+DM_WongZakai})
	}

	Let $H_n$ denote the linear span of $\{e_1,\dots,e_n\}$.
	Let $P_n$ denote the orthogonal projection
	\begin{equation*}
		P_n:L^2\to H_n.
	\end{equation*}
	Since $e_i$ is an eigen function of the operator $A$ for each $i\in\mathbb{N}$, we can prove that $e_i\in D(A)$. Therefore, the space $H_n\subset D(A)$.

	The given control process $u$ induces the following measurable map, see, for example, Proposition 3.19, Section 3.7 in \cite{Prato+Zabczyk}
	\begin{equation*}
		u:\Omega \rightarrow L^2\l(0,T;L^2\r).
	\end{equation*}

	Define $u_n$ as the projection of $u$ under the projection $P_n$. That is for a.a. $t\in[0,T]$,
	\begin{equation*}
		u_n(t) := P_n\bigl(u(t)\bigr),\ \mathbb{P}-\text{a.s.}
	\end{equation*}
	
	\dela{
		Hence note that for $p\geq 1$
		\begin{align}\label{bound on projection u_n}
			\nonumber\mathbb{E}\left(\int_{0}^{T}|u_n(t)|_{L^2} dt\right)^{p}  &= \mathbb{E}\left(\int_{0}^{T}\left|P_n(u(t))\right|_{L^2} dt \right)^{p}  \\
			& \leq \mathbb{E}\left(\int_{0}^{T}|u|_{L^2} dt\right)^{p} \leq K.
		\end{align}
		Change the above according to the assumption \eqref{assumption on u}.
	}

	Hence note that for $p\geq 1$
	\begin{align}
		\nonumber\mathbb{E}\left(\int_{0}^{T}|u_n(t)|_{L^2}^{2} \,  dt\right)^p  &= \mathbb{E}\left(\int_{0}^{T}\left|P_n(u(t))\right|_{L^2}^{2} \, dt \right)^p  \\
		& \leq \mathbb{E}\left(\int_{0}^{T}|u(t)|_{L^2}^{2} \,  dt\right)^p \leq K_p.
	\end{align}

	\begin{remark}\label{Remark meaning of norms are equivalent on Hn}[Equivalence of norms on $H_n$]
		Let us recall that $H_n$ is a finite dimensional vector space. 	On this vector space we can consider many norms. 	
		For instance, the space $H_n$ is  a subspace of $L^2$. Hence we can endow the space $H_n$ with the norm inherited from the space $L^2$.
		This norm on $H_n$ will be denoted by $\vert \cdot \vert_{L^2}$.
		
		More generally, we can endow the space $H_n$ with the norm inherited from every  space $X$ provided $H_n$ is a subspace of $X$.
		This norm on $H_n$ will be denoted by $\vert \cdot \vert_{X}$.
		
		For instance, we can take $X$ to be $D(A)$ or any Sobolev space $H^{\theta,p}$ for $\theta \geq 0$ and $p\geq 2$. Since,
		$H^1=H^{1,2} \embed L^\infty $ and $H_n \embed H^1$, $H_n$ is also a subspace of the space $L^\infty$ and hence we can consider on $H_n$ the norm
		inherited from the space $L^\infty$.
		This norm on $H_n$ will be denoted by $\vert \cdot \vert_{L^\infty}$.
		
		Finally, let us point out that since all norms on a finite dimensional  vector space are equivalent, see e.g. Exercise 1 in section V.9 of the book \cite{Dieudonne_1969-vol-1},
		we infer that for any two norms $\vert \cdot \vert_{X}$ and $\vert \cdot \vert_{Y}$ on $H_n$ there exists a constant $C_n=C_n(X,Y)$ such that
		\[
		\frac{1}{C_n} \vert m \vert_{X} \leq \vert m \vert_{Y} \leq C_n \vert m \vert_{X},  \:\: \mbox{ for every } m \in H_n.
		\]
		In particular, this holds for the norms   $\vert \cdot \vert_{L^2}$ and  $\vert \cdot \vert_{L^\infty}$.

		Also, for $v\in H_n$,
		\begin{equation}
			v = \sum_{i=1}^{n} \l\langle v , e_i \r\rangle_{L^2} e_i.
		\end{equation}
		Let $\|\cdot\|$ denote some ($L^p$,$H^1$, etc.) norm on the space $H_n$. Then there exists a constant $C_n>0$ such that
		\begin{align*}
			\| v \|  &= \| \sum_{i=1}^{n} \l\langle v , e_i \r\rangle_{L^2} e_i \| \\
			\leq & \sum_{i=1}^{n} \l| \l\langle v , e_i \r\rangle_{L^2} \r| \| e_i \| \\
			\leq & \sum_{i=1}^{n} \l| v \r|_{L^2} \l| e_i \r|_{L^2} \| e_i \| \\
			\leq & C_n \l| v \r|_{L^2}.
		\end{align*}
		Also,
		\begin{align*}
			\Delta v = & \Delta \l( \sum_{i=1}^{n} \l\langle v , e_i \r\rangle_{L^2} e_i \r) \\
			= &   \sum_{i=1}^{n} \l\langle v , e_i \r\rangle_{L^2} \Delta e_i \\
			= &   \sum_{i=1}^{n} \l\langle v , e_i \r\rangle_{L^2} \l( - \lambda_i\r) e_i.
		\end{align*}
		Here $\lambda_i$ denote the eigen values corresponding to $e_i$. Here, by $\Delta$ we mean the operator $A$ defined in \eqref{Definition of Neumann Laplacian}. Therefore from the above equality, there exists a constant $C_n>0$ such that
		\begin{align*}
			\| \Delta v \| \leq &  \sum_{i=1}^{n} \l\langle v , e_i \r\rangle_{L^2} \l|  \lambda_i \r| \| e_i\| \\
			\leq & C_n \l| v \r|_{L^2}.
		\end{align*}
	\end{remark}

	\begin{lemma}\label{lemma Lipschitz continuity}
		\begin{enumerate}
			\item Consider a function $f:H_n \to H_n$. If $f$ is Lipschitz continuous on each closed, bounded ball $B\subset H_n$, then $f$ is locally Lipshitz continuous on $H_n$.
			\item For a locally Lipschitz continuous function $f:H_n \to H_n$, the image of a closed, bounded ball $B\subset H_n$ under $f$ is compact, and as a result, bounded in $H_n$.
		\end{enumerate}
	\end{lemma}
	
	\begin{proof}[Proof of Lemma \ref{lemma Lipschitz continuity}]
		\textbf{Proof of (1)}:\\
		The space $H_n$ is a finite dimensional vector space. Hence $H_n$ is locally compact. Therefore, $f$ is locally Lipschitz on $H_n$ if and only if $f$ is Lipschitz on all compact subsets of $H_n$. 
		For a given compact subset $K$ of $H_n$, we can choose a closed and bounded ball $B\subset H_n$ (which is again compact) large enough so that $K \subset B \subset H_n$. Therefore in order to show the local Lipschitz continuity of $f$ on $H_n$, it suffices to show that $f$ is Lipschitz continuous on all closed and bounded balls in $H_n$. This concludes the proof of (1).
		\\
		\textbf{Proof of (2):}\\
		That $f$ is locally Lipschitz continuous implies that $f$ is also continuous on $H_n$. As mentioned in the proof of (1) above, closed and bounded balls are compact in $H_n$. Therefore the image of $B$ under $f$ is a compact, and hence also a bounded subset of $H_n$.
		
	\end{proof}

	\begin{remark}\label{Remark on locally Lipschitz functions}[A Remark on locally Lipschitz functions on $H_n$]
		The aim of the following calculations is to show that if $f,g$ are locally Lipschitz continuous functions on the finite dimensional space $H_n$, then their product $h_1 := fg$ and composition $h_2 := f \circ g$ is also locally Lipshitz continuous on $H_n$. Using Lemma \ref{lemma Lipschitz continuity}, to show that the functions $h_1,h_2$ are locally Lipschitz continuous on $H_n$, it suffices to show that they are Lipschitz continuous on all closed and bounded balls $B\subset H_n$.
		Towards that, let the functions $f,g:H_n\to H_n$ be locally Lipschitz continuous on $H_n$. Fix an arbitrary closed and bounded ball $B\subset H_n$. Since $g$ is continuous and $B$ is compact, the set $g(B) \subset H_n$ is compact. Being locally Lipschitz continuous on $H_n$, both $f$ and $g$ are Lipschitz continuous on the ball $B$. Let $C_{f,B}$ and $C_{g,B}$ denote their respective Lipschitz constants. Similarly, let $C_{f,g(B)}$ denote the Lipschitz constant for the function $f$ on $g(B)$.
		\begin{enumerate}
			\item[\textbf{Claim 1:}] The map
			\begin{align*}
				h_1 := fg: H_n \ni v \mapsto f(v)g(v) \in H_n
			\end{align*}
			is Lipschitz on $B$.\\
			\textbf{Brief proof of Claim 1: \dela{\adda{of what?}}:} Let $v_1,v_2\in B$. Therefore
			\begin{align*}
				\l| f(v_1)g(v_1) - f(v_2)g(v_2) \r|_{L^2} \leq & \l| f(v_1)g(v_1) - f(v_2)g(v_1) \r|_{L^2} + \l| f(v_2)g(v_1) - f(v_2)g(v_2) \r|_{L^2} \\
				\leq & \l| \l[f(v_1) - f(v_2) \r]g(v_1)  \r|_{L^2} + \l| f(v_2)\l[g(v_1) - g(v_2)\r] \r|_{L^2} \\
				\leq & \l| f(v_1) - f(v_2)   \r|_{L^2} \l| g(v_1) \r|_{L^{\infty}}  + \l| f(v_2) \r|_{L^{\infty}} \l| g(v_1) - g(v_2) \r|_{L^2} \\
				\leq & C_{f,B} \l| g(v_1) \r|_{L^{\infty}}  \l| v_1 - v_2 \r|_{L^2} + C_{g,B} \l| f(v_2) \r|_{L^{\infty}}\l| v_1 - v_2 \r|_{L^2} \\
				\leq & C \l[ \l| f(v_2) \r|_{L^{\infty}} + \l| g(v_1) \r|_{L^{\infty}} \r] \l| v_1 - v_2 \r|_{L^2} \\
				\leq & C \l| v_1 - v_2 \r|_{L^2}.
			\end{align*}
			The constant $C$ here can depend on the ball $B$, the functions $f,g$ and $n$.
			\item[\textbf{Claim 2:}] The composition
			\begin{align*}
				h_2 := f\circ g: H_n \ni v \mapsto f\bigl( g(v) \bigr) \in H_n
			\end{align*}
			is Lipschitz on $B$.\\
			\textbf{Brief proof of Claim 2:} Let $v_1,v_2\in B$. Therefore
			\begin{align*}
				\l| \bigl(f\circ g\bigr)(v_1) - \bigl(f\circ g\bigr)(v_2) \r|_{L^2} = & \l| f\bigl( g(v_1) \bigr) - f\bigl( g(v_2) \bigr) \r|_{L^2} \\
				\leq & C_{f,g(B)} \l| g(v_1) - g(v_2) \r|_{L^2} \\
				\leq & C_{f,g(B)}C_{g,B} \l| v_1 - v_2 \r|_{L^2}.
			\end{align*}
			The constants here can depend on the ball $B$, the functions $f,g$ and $n$.
		\end{enumerate}
		Since the calculations above hold true for any such ball $B$, the product and the composition are locally Lipschitz
	\end{remark}	
	We define the following:
	\begin{align*}
		F_n^1:H_n&\ni m\mapsto P_n\l(m\times\Delta m\r)\in H_n,\\
		F_n^2:H_n&\ni m\mapsto P_n\l(m\times\l( m \times\Delta m\r)\r)\in H_n,\\
		F_n^3:H_n&\ni m\mapsto P_n\l(m\times u_n\r)\in H_n,\\
		F_n^4:H_n&\ni m\mapsto P_n \l( m \times (m \times u_n)\r)\in H_n,\\
		G_n:H_n&\ni m\mapsto P_n(G(m))\in H_n.
	\end{align*}
	
	\begin{lemma}\label{Lemma locally Lipschitz coefficients problem considered}
		The maps $F^i_n,i=1,2,3,4$ and $G_n$ defined above are locally Lipschitz on $H_n$.
	\end{lemma}
	
	\begin{proof}[Proof of Lemma \ref{Lemma locally Lipschitz coefficients problem considered}]
		We show the proofs for $F_n^1$ and $G_n$. Others can be proved in a similar way.
		We start with the proof for $F_n^1$. Define a map $f: H_n \times H_n \to H_n$, given by
		\begin{align*}
			f(v,w) = P_n\l(v \times \Delta w\r).
		\end{align*}
		The map $f$ is bilinear. Moreover, by Remark \ref{Remark meaning of norms are equivalent on Hn} there exists a constant $C_n>0$ such that
		\begin{align*}
			\l| f(v,w) \r|_{L^2} & = \left| P_n\l(v \times \Delta w\r) \right|_{L^2} \\
			& \leq \left| v \times \Delta w \right|_{L^2} \\
			& \leq \left| v \right|_{L^4} \left| \Delta w \right|_{L^4} \\
			& \leq C_n \left| v \right|_{L^2} \left|  w \right|_{L^2}.
		\end{align*}
		Therefore the map
		$$H_n\ni v \mapsto f(v,v) \in H_n$$
		is a polynomial of degree 2 on $H_n$, and as a consequence, is locally Lipschitz. We observe that $F_n^1(v) = f(v,v)$. Therefore $F_n^1$ is locally Lipschitz.

		The proof for $G_n$ can be given as follows. Consider the map $f_1 : H_n \to H_n$, given by
		\begin{equation}
			f_1(v) = P_n(v \times h).
		\end{equation}
		\begin{align*}
			\l|f_1(v)\r|_{L^2} \leq & \l|P_n(v \times h)\r|_{L^2} \\
			\leq & \l| v \times h \r|_{L^2} \\
			\leq & \l| h \r|_{L^{\infty}} \l| v \r|_{L^2}.
		\end{align*}
		$h\in H^1 \hookrightarrow L^{\infty}$ implies that the map $f_1$ is a bounded linear map. As a consequence, $f_1$ is Lipschitz continuous on $H_n$.
		\dela{We observe the following. For $w \in L^2$, $P_n(w) = \sum_{i=1}^{n}\l\langle w , e_i \r\rangle_{L^2} e_i$.
			Therefore,
			\begin{align*}
				\l|f_1(v)\r|_{L^2} \leq & \sum_{i=1}^{n}\l| \l\langle v \times h , e_i \r\rangle_{L^2}\r| \l| e_i \r|_{H_n}\\
				\leq & \sum_{i=1}^{n}  \l| v \times h \r|_{L^2}  \l| e_i \r|_{L^2} \l| e_i \r|_{H_n}\\
				\leq & C  \l| v \times h \r|_{L^2} \\
				\leq & C \l|h\r|_{L^{\infty}} \l| v \r|_{L^2} \\
				\leq & C \l|h\r|_{H^1} \l| v \r|_{H_n}.
			\end{align*}
			Since $h\in H^1$, we can say that the map $f_1$ is a polynomial. Hence $f_1$ is locally Lipschitz.}
		
		Now, define the map $f_2:H_n\times H_n \to H_n$, given by
		\begin{align*}
			f_2(v,w) = P_n\bigl( v \times \l( w \times h \bigr) \r).
		\end{align*}
		Clearly, $f_2$ is a bilinear map. Further, by Remark \ref{Remark meaning of norms are equivalent on Hn}, there exists a constant $C_n>0$ such that
		\begin{align*}
			\l|f_2(v,w)\r|_{L^2} = & \l| v \times \l( w \times h \r) \r|_{L^2} \\
			\leq & \l| v \r|_{L^4} \l| w \r|_{L^4} \l| h \r|_{L^{\infty}} \\
			\leq & C_n \l| v \r|_{L^2} \l| w \r|_{L^2} \l| h \r|_{L^{\infty}}.
		\end{align*}
		Therefore $h\in H^1\hookrightarrow L^{\infty}$ implies that $f_2$ is bilinear bounded.\\
		The map
		$$H_n\ni v \mapsto f_2(v,v)\in H_n$$
		is therefore a homogeneous polynomial of degree 2 on $H_n$, and as a consequence, is locally Lipschitz.\\
		We now observe that $G_n(v)$ is a linear combination of $f_1(v)$ and $f_2(v,v)$. Therefore $G_n$ is locally Lipschitz on $H_n$.
		
		\dela{
			\textbf{Another way to show That $f_2$ is locally Lipschitz:}\\
			Let $v_1,v_2,w_1,w_2\in H_n$.
			\begin{align*}
				\l| f_2(v_1,w_1) - f_2(v_2,w_2) \r|_{L^2} = & \l| v_1 \times \l( w_1 \times h \r)  -  v_2 \times \l( w_2 \times h \r) \r|_{L^2} \\
				\leq & \l| \l(v_1 - v_2\r) \times \l( w_1 \times h \r)  \r|_{L^2} + \l| v_2 \times \bigl( \l(w_1 - w_2\r) \times h \bigr)  \r|_{L^2} \\
				\leq & 	 \l| v_1 - v_2 \r|_{L^4}  \l| w_1 \r|_{L^4} \l| h \r|_{L^{\infty}}  + \l| v_2 \r|_{L^4}  \l|w_1 - w_2\r|_{L^4} \l| h \r|_{L^{\infty}} \\
				\leq & C \bigl[ \l| v_1 - v_2 \r|_{L^2}  \l| w_1 \r|_{L^4} \l| h \r|_{L^{\infty}}  + \l| v_2 \r|_{L^4}  \l|w_1 - w_2\r|_{L^2} \l| h \r|_{L^{\infty}} \bigr]
			\end{align*}
			$h\in H^1\hookrightarrow L^{\infty}$ implies that the function $f_2$ is locally Lipschitz on $H_n$.
		}
		
	\end{proof}
	
	Let $\psi_0:\mathbb{R} \rightarrow [0,1]$ be a function of $C_c^{1}(\mathbb{R})$ class such that
	\begin{align}\label{definition of bump function}
		\psi_0(x) =
		\begin{cases}
			1\ \text{if}\ |x|\leq |h|_{L^{\infty}} + 1,\\
			0\ \text{if}\ |x|\geq |h|_{L^{\infty}} + 2.
		\end{cases}
	\end{align}
	
	Define a function $\psi_n:H_n\to\mathbb{R}$  the following formula.
	\begin{align}\label{definition of bump function psi n}
		\psi_n(v) = \psi_0\bigl(\l| v \r|_{L^{\infty}}\bigr) \, \psi_0\bigl(\l| P_n\l(  v \times h \r)  \r|_{L^{\infty}}\bigr) \, \psi_0\bigl(\l| P_n \l( v \times \l( v \times h \r) \r) \r|_{L^{\infty}}\bigr),\:\:v\in H_n.
	\end{align}
	
	\begin{lemma}\label{Lemma psi n is Lipschitz}
		The function $$\psi_n:H_n\ni v \mapsto \psi_0\bigl(\l| v \r|_{L^{\infty}}\bigr) \, \psi_0\bigl(\l| P_n\l(  v \times h \r)  \r|_{L^{\infty}}\bigr) \, \psi_0\bigl(\l| P_n \l( v \times \l( v \times h \r) \r) \r|_{L^{\infty}}\bigr) \in \mathbb{R}$$ is locally Lipschitz.
	\end{lemma}
	\begin{proof}[Proof of Lemma \ref{Lemma psi n is Lipschitz}]
		We define the following auxiliary functions.
		\begin{align*}
			f_1&:H_n \ni v \mapsto P_n\l(v \times h\r) \in H_n\\
			f_2&:H_n \times H_n \ni (v,w) \mapsto P_n\bigl(v \times \l(w \times h\r)\bigr) \in H_n.\\
			f_3&:H_n \ni v \mapsto \l| v \r|_{L^{\infty}} \in \mathbb{R}\\
			\beta &: \mathbb{R}^3 \ni (x_1,x_2,x_3) \mapsto x_1 x_2 x_3 \in \mathbb{R}
		\end{align*}
		\dela{The maps $f_2,\tilde{\psi}$ are not really needed
			We observe that
			\[
			f_1(v) = \l( f_5(v) , f_5\circ f_3(v) , f_5\circ f_4(v) \r), \:\: v\in H_n.
			\]
		}
		
		The map $f_1$ is a linear map from $H_n$ to $H_n$. Moreover, for $v\in H_n$, since $P_n$ is orthonormal projection from  $L^2$ to $H_n$ and hence
		\textbf{a contraction} and $ H^1 \hookrightarrow L^{\infty}$.
		\begin{align*}
			\l| f_1(v) \r|_{L^2} = & \l| P_n(v \times h) \r|_{L^2} \leq   \l| v \times h \r|_{L^2} \\
			\leq & \l| v \r|_{L^2} \l| h \r|_{L^{\infty}}
			\leq  C \l| v \r|_{L^2} \l| h \r|_{H^1}.
		\end{align*}
		Therefore $f_1$ is a bounded linear map, and hence also Lipschitz continuous on $H_n$.\\
		Similarly we can treat  the map $f_2$. Clearly this map is bilinear. It is also well defined. To see this, let $v,w\in H_n$. Since $P_n$ is a contraction, by H\"older's inequality, Remark \ref{Remark meaning of norms are equivalent on Hn} and the continuous embedding $H^1\hookrightarrow L^{\infty}$ we have the following sequence of inequalities.
		\begin{align*}
			\l| f_2(v,w) \r|_{L^2} = & \l| P_n\bigl(v \times \l( w \times h \r)\bigr) \r|_{L^2} \\
			\leq & \l| v \times \l( w \times h \r) \r|_{L^2} \\
			\leq & \l| v \r|_{L^4} \l| w \r|_{L^4} \l|  h\r|_{L^{\infty}}\\
			\leq & C_n \l| v \r|_{L^2} \l| w \r|_{L^2} \l|  h\r|_{H^1}.
		\end{align*}
		$v,w\in H_n$ and $h\in H_1$ implies that the right hand side of the above inequality is finite. Therefore the map $$H_n  \ni v \mapsto f_2(v,v) \in H_n$$ is a homogeneous polynomial of degree 2 on $H_n$. Therefore it is analytic, and in particular is locally Lipschitz.\\		
		By Remark \ref{Remark meaning of norms are equivalent on Hn}, the map $f_3$ is well defined and  there exists a constant $C_n>0$ such that the following inequality holds.
		\begin{align*}
			\l|f_3(v)\r|_{\mathbb{R}} = \l| v \r|_{L^{\infty}} \leq C_n \l|v\r|_{L^2}, \:\: v\in H_n.
		\end{align*}
		Moreover, for $v_1,v_2\in H_n$, triangle inequality and also by  Remark \ref{Remark meaning of norms are equivalent on Hn},
		\begin{align*}
			\l| f_3(v_1) - f_3(v_2) \r|_{\mathbb{R}} =& \l| v_1 - v_2 \r|_{L^{\infty}}
			\leq    C_n \l|v_1 - v_2\r|_{L^2}.
		\end{align*}
		Therefore $f_3$ is Lipschitz continuous on $H_n$.\\
		The map $\beta$ is a trilinear map from $\mathbb{R}^3$ to $\mathbb{R}$, with
		\begin{equation*}
			\l| \beta (x_1,x_2,x_3) \r|_{\mathbb{R}} = \l|x_1\r|_{\mathbb{R}} \l|x_2\r|_{\mathbb{R}} \l|x_3\r|_{\mathbb{R}},  \ \text{for}\ (x_1,x_2,x_3)\in\mathbb{R}^3.
		\end{equation*}
		Therefore $\beta$ is 		
		\dela{continuous. Therefore we can conclude that $\beta$ is a polynomial of degree 3 and hence is} locally Lipschitz.\\
		\dela{$\beta$ is not a polynomial. You wrote "$\beta$ is a trilinear map". This is true. You get a polynomial by
			\[
			\mathbb{R} \ni x \mapsto \beta(x,x,x) \in \mathbb{R}.
			\]
			
		}

		\dela{\textbf{Not Required}Observe that the map $\tilde{\psi}$ can also be given by
			$$\tilde{\psi}(x_1,x_2,x_3) = \beta (\psi_0(x_1) , \psi_0(x_2) , \psi_0(x_3)), \ \text{for}\ (x_1,x_2,x_3)\in\mathbb{R}^3.$$
			Since by definition (see \eqref{definition of bump function}), $\psi_0$ is of $C_0^\infty$-class,  the function  $\tilde{\psi}$  is also
			of $C_0^\infty$-class. In particular, the Frechet derivative of the function  $\tilde{\psi}$ is bounded and hence  we can conclude that $\tilde{\psi}$ is globally  Lipschitz.
			Moreover, since the range of the function $\psi_0$ is compact and $\beta$ is continuous, the
			the range of the function $\tilde{\psi}$ is also compact and hence bounded. Thus the function $\tilde{\psi}$ is bounded.

			Since composition of bounded and Lipschitz functions is again bounded and Lipschitz, each component of $\tilde{\psi}$ is locally Lipschitz. Therefore $\tilde{\psi}$ is locally Lipschitz.\\}
		So far we have shown that the maps $f_1,f_2,f_3$ are locally Lipschitz.\dela{ Therefore, the composition of these functions, as given in the above expression for $\psi_n$ are also locally Lipschitz. \dela{Since this holds for each component of $\psi_n$, we can conclude that the map}}
		The map $\psi_0$ is assumed to be of $C_0^{\infty}$ class, and hence is bounded and locally Lipschitz. We can therefore conclude that the map $\tilde{\psi}_n$, given by
		\begin{equation*}
			\tilde{\psi}_n : H_n\ni v\mapsto \bigl( \l( \psi_0\circ f_3\r)(v) , \l( \psi_0\circ f_3\circ  f_1 \r) (v), \l( \psi_0\circ f_3\circ f_2 \r) (v,v)\bigr) \in \mathbb{R}^3
		\end{equation*}
		is locally Lipschitz.
		\dela{Now we show that the map $\psi_n$ is locally Lipschitz. }The map $\psi_n$ can be written as a composition of the functions described so far, as follows. $$\psi_n(v) = \l(\beta\circ\tilde{\psi}_n\r)(v) = \beta \bigl( \l( \psi_0\circ f_3\r)(v) , \l( \psi_0\circ f_3\circ  f_1 \r) (v), \l( \psi_0\circ f_3\circ f_2 \r) (v,v)\bigr).$$Hence $\psi_n$, as a composition of $\beta$ with  $\tilde{\psi}_n$, is locally Lipschitz.
		
	\end{proof}

	\dela{

		\coma{The Lemma can be removed. It is an alternate proof for the same result in Lemma \ref{Lemma psi n is Lipschitz}.}
		\begin{lemma}\label{Lemma locally Lipschitz coefficients truncated problem considered}
			The maps
			\begin{align*}
				\psi_n\: F_n^4 : H_n \ni v \mapsto \psi_n(v) F_n^4(v) \in H_n,
			\end{align*}
			\begin{align*}
				\psi_n\: \bigl[DG_n\bigr]\l(G_n\r) : H_n \ni v \mapsto \psi_n(v) \bigl[DG_n(v)\bigr]\l(G_n(v)\r) \in H_n
			\end{align*}
			are locally Lipschitz.
			\dela{The maps $\psi(\cdot)F_n^4(\cdot), \psi(\cdot)G_n(\cdot)$ and $\psi(\cdot)\l[DG_n(\cdot)\r]\l( G_n(\cdot)\r)$ are locally Lipschitz on $H_n$.
				
				\[
				\psi F_n^4: H_n\to H_n
				\]
				\[
				\l(\psi F_n^4\r)(v) = \psi(v) F_n^4(v),\ \text{for}\ v\in H_n.
				\]
			}
			
		\end{lemma}
		
		\begin{proof}[Proof of Lemma \ref{Lemma locally Lipschitz coefficients truncated problem considered}]
			The constant $C$ is used here to denote a generic constant that can depend on $n\in\mathbb{N}$. It's value may vary from line to line.
			We first show that the function $\psi_n$ is locally Lipschitz on $H_n$.
			
			Although $\psi_n$ depends on $n\in\mathbb{N}$, we use $\psi$ here instead of $\psi_n$ to simplify writing.		The function $\psi_0$ is a compactly supported smooth function. Therefore there exists a constant $C>0$ such that
			\begin{align*}
				\l|\psi_0(x) - \psi_0(y)\r| \leq C \l| x - y \r|,\ \forall x,y\in\mathbb{R}.
			\end{align*}
			Let $v_1,v_2\in H_n$. Therefore $v_1 - v_2\in H_n$.
			\begin{align*}
				\l| \psi_0(\l|v_1\r|_{L^{\infty}}) - \psi_0(\l|v_2\r|_{L^{\infty}}) \r|
				&\leq C\bigl| \l|v_1\r|_{L^{\infty}} - \l|v_1\r|_{L^{\infty}} \bigr| \\
				&\leq C \l| v_1 - v_2 \r|_{L^{\infty}}\ (\text{By triangle inequality})\\
				& = C \l|\sum_{i=1}^{n} \l\langle v_1 - v_2 , e_i \r\rangle_{L^2} e_i \r|_{L^{\infty}}\ (\text{Since}\ v_1-v_2\in H_n) \\
				& \leq C \sum_{i=1}^{n} \l| \l\langle v_1 - v_2 , e_i \r\rangle_{L^2} \r| \l| e_i \r|_{L^{\infty}} \\
				& \leq C \sum_{i=1}^{n} \l| v_1- v_2 \r|_{L^2} \l| e_i \r|_{L^2} \l| e_i \r|_{L^{\infty}} \\
				& \leq C \l| v_1- v_2 \r|_{L^2}.
			\end{align*}
			Similarly, since $H^1\hookrightarrow L^{\infty}$, we have the following sequence of inequalities.
			\dela{
				\begin{align*}
					\l| \psi_0(\l|P_n\l( v_1 \times h\r)\r|_{L^{\infty}}) - \psi(\l|P_n\l( v_2 \times h \r)\r|_{L^{\infty}}) \r|
					&\leq C\bigl| \l|P_n\l(v_1 \times h\r)\r|_{L^{\infty}} - \l|P_n\l( v_2 \times h \r)\r|_{L^{\infty}} \bigr| \\
					(\text{By triangle inequality})\ &\leq C \l| P_n \l(v_1 \times h\r) - P_n \l(v_2 \times h\r) \r|_{L^{\infty}} \\
					(\text{Since}\ L^{\infty}\hookrightarrow H^1) \ & \leq C \l| P_n \l(v_1 \times h\r) - P_n \l(v_2 \times h\r) \r|_{H^1}\\
					&\leq C \l| P_n \l(v_1 \times h - v_2 \times h\r) \r|_{H^1} \\
					(\text{$\{P_n\}_{n\in\mathbb{N}}$ is orthogonal basis in $H^1$}) \ &\leq C \l| v_1 \times h - v_2 \times h \r|_{H^1}\\
					&\leq C \l| \l(v_1 - v_2\r) \times h \r|_{H^1} \\
					& \leq C \l[\l| \nabla \l(v_1 - v_2\r) \times h \r|_{L^{2}} + \l|  \l(v_1 - v_2\r) \times \nabla h \r|_{L^{2}}\r] \\
					(\text{By H\"older's inequality})\ & \leq C \l[\l| \nabla \l(v_1 - v_2\r)\r|_{L^{2}} \l| h \r|_{L^{\infty}} + \l|  \l(v_1 - v_2\r) \r|_{L^{\infty}} \l| \nabla h \r|_{L^{2}}\r]\\
					(\text{Since}\ L^{\infty}\hookrightarrow H^1) \ & \leq C  \l[\l| \nabla \l(v_1 - v_2\r)\r|_{L^{2}} \l| h \r|_{H^1} + \l|  \l(v_1 - v_2\r) \r|_{L^{\infty}} \l| h \r|_{H^1}\r]\\
					& \leq C \l| h \r|_{H^1} \l|  \l(v_1 - v_2\r) \r|_{H^1} \\
					(\text{Since all norms on $H_n$ are equivalent})\ & \leq C(h) \l| v_1 - v_2 \r|_{H_n}.
				\end{align*}
			}
			\begin{align*}
				\l| \psi_0(\l|P_n\l( v_1 \times h\r)\r|_{L^{\infty}}) - \psi_0(\l|P_n\l( v_2 \times h \r)\r|_{L^{\infty}}) \r|
				&\leq C\bigl| \l|P_n\l(v_1 \times h\r)\r|_{L^{\infty}} - \l|P_n\l( v_2 \times h \r)\r|_{L^{\infty}} \bigr| \\
				(\text{By triangle inequality})\ &\leq C \l| P_n \l(v_1 \times h\r) - P_n \l(v_2 \times h\r) \r|_{L^{\infty}} \\
				& \leq C \l| P_n \l(v_1 \times h\r) - P_n \l(v_2 \times h\r) \r|_{H^1}\\
				&\leq C \l| P_n \l(v_1 \times h - v_2 \times h\r) \r|_{H^1} \\
				& \leq C \sum_{i=1}^{n}\l|\l\langle \l(v_1 - v_2\r) \times h , e_i \r\rangle_{L^2}\r| \l|e_i\r|_{e_i} \\
				& \leq C \sum_{i=1}^{n} \l| v_1 - v_2 \r|_{L^2} \l| h \r|_{L^{\infty}} \l| e_i \r|_{L^2}  \l|e_i\r|_{e_i} \\
				& \leq C \l| v_1 - v_2 \r|_{L^2} \sum_{i=1}^{n}  \l| h \r|_{L^{\infty}} \l| e_i \r|_{L^2}  \l|e_i\r|_{e_i} \\
				& \leq C \l| v_1 - v_2 \r|_{L^2}.
				\dela{
					(\text{$\{P_n\}_{n\in\mathbb{N}}$ is orthogonal basis in $H^1$}) \ &\leq C \l| v_1 \times h - v_2 \times h \r|_{H^1}\\
					&\leq C \l| \l(v_1 - v_2\r) \times h \r|_{H^1} \\
					& \leq C \l[\l| \nabla \l(v_1 - v_2\r) \times h \r|_{L^{2}} + \l|  \l(v_1 - v_2\r) \times \nabla h \r|_{L^{2}}\r] \\
					(\text{By H\"older's inequality})\ & \leq C \l[\l| \nabla \l(v_1 - v_2\r)\r|_{L^{2}} \l| h \r|_{L^{\infty}} + \l|  \l(v_1 - v_2\r) \r|_{L^{\infty}} \l| \nabla h \r|_{L^{2}}\r]\\
					(\text{Since}\ L^{\infty}\hookrightarrow H^1) \ & \leq C  \l[\l| \nabla \l(v_1 - v_2\r)\r|_{L^{2}} \l| h \r|_{H^1} + \l|  \l(v_1 - v_2\r) \r|_{L^{\infty}} \l| h \r|_{H^1}\r]\\
					& \leq C \l| h \r|_{H^1} \l|  \l(v_1 - v_2\r) \r|_{H^1} \\
					(\text{Since all norms on $H_n$ are equivalent})\ & \leq C(h) \l| v_1 - v_2 \r|_{H_n}
				}
			\end{align*}
			
			Arguing similarly, we can show that there exists a constant $C>0$ such that
			\begin{align*}
				\l| \psi_0(\l|P_n\bigl( v_1 \times \l(v_1 \times h\r)\bigr)\r|_{L^{\infty}}) - \psi(\l|P_n\bigl( v_2 \times \l(v_2 \times h\r)\bigr)\r|_{L^{\infty}}) \r|
				&\leq C \bigl[ \l| v_1 \r|_{L^{\infty}} + \l| v_2 \r|_{L^{\infty}} \bigr]\l| v_1 - v_2 \r|_{L^2}\\
				&\leq C \bigl[ \l| v_1 \r|_{L^{\infty}} + \l| v_2 \r|_{L^{\infty}} \bigr]\l| v_1 - v_2 \r|_{L^2}.
			\end{align*}
			Now,
			\begin{align*}
				&\psi(v_1) - \psi(v_2) =  \psi_0(\l| v_1 \r|_{L^{\infty}}) \, \psi_0(\l| P_n\l(  v_1 \times h \r)  \r|_{L^{\infty}}) \, \psi_0(\l| P_n \l( v_1 \times \l( v_1 \times h \r) \r) \r|_{L^{\infty}}) \\
				& - \psi_0(\l| v_2 \r|_{L^{\infty}}) \, \psi_0(\l| P_n\l(  v_2 \times h \r)  \r|_{L^{\infty}}) \, \psi_0(\l| P_n \l( v_2 \times \l( v_2 \times h \r) \r) \r|_{L^{\infty}}) \\
				= & \psi_0(\l| v_1 \r|_{L^{\infty}}) \, \psi_0(\l| P_n\l(  v_1 \times h \r)  \r|_{L^{\infty}}) \, \psi_0(\l| P_n \l( v_1 \times \l( v_1 \times h \r) \r) \r|_{L^{\infty}}) \\
				& - \psi_0(\l| v_2 \r|_{L^{\infty}}) \, \psi_0(\l| P_n\l(  v_1 \times h \r)  \r|_{L^{\infty}}) \, \psi_0(\l| P_n \l( v_1 \times \l( v_1 \times h \r) \r) \r|_{L^{\infty}}) \\
				& + \psi_0(\l| v_2 \r|_{L^{\infty}}) \, \psi_0(\l| P_n\l(  v_1 \times h \r)  \r|_{L^{\infty}}) \, \psi_0(\l| P_n \l( v_1 \times \l( v_1 \times h \r) \r) \r|_{L^{\infty}}) \\
				& - \psi_0(\l| v_2 \r|_{L^{\infty}}) \, \psi_0(\l| P_n\l(  v_2 \times h \r)  \r|_{L^{\infty}}) \, \psi_0(\l| P_n \l( v_1 \times \l( v_1 \times h \r) \r) \r|_{L^{\infty}}) \\
				& + \psi_0(\l| v_2 \r|_{L^{\infty}}) \, \psi_0(\l| P_n\l(  v_2 \times h \r)  \r|_{L^{\infty}}) \, \psi_0(\l| P_n \l( v_1 \times \l( v_1 \times h \r) \r) \r|_{L^{\infty}}) \\
				& - \psi_0(\l| v_2 \r|_{L^{\infty}}) \, \psi_0(\l| P_n\l(  v_2 \times h \r)  \r|_{L^{\infty}}) \, \psi_0(\l| P_n \l( v_2 \times \l( v_2 \times h \r) \r) \r|_{L^{\infty}}) \\
				= & \bigg[\psi_0(\l| v_1 \r|_{L^{\infty}}) - \psi_0(\l| v_2 \r|_{L^{\infty}})\bigg] \, \psi_0(\l| P_n\l(  v_1 \times h \r)  \r|_{L^{\infty}}) \, \psi_0(\l| P_n \l( v_1 \times \l( v_1 \times h \r) \r) \r|_{L^{\infty}}) \\
				& + \psi_0(\l| v_2 \r|_{L^{\infty}}) \, \bigg[\psi_0(\l| P_n\l(  v_1 \times h \r)  \r|_{L^{\infty}}) - \psi_0(\l| P_n\l(  v_2 \times h \r)  \r|_{L^{\infty}}) \bigg] \, \psi_0(\l| P_n \l( v_1 \times \l( v_1 \times h \r) \r) \r|_{L^{\infty}}) \\
				& + \psi_0(\l| v_2 \r|_{L^{\infty}}) \, \psi_0(\l| P_n\l(  v_2 \times h \r)  \r|_{L^{\infty}}) \, \bigg[ \psi_0(\l| P_n \l( v_1 \times \l( v_1 \times h \r) \r) \r|_{L^{\infty}}) - \psi_0(\l| P_n \l( v_2 \times \l( v_2 \times h \r) \r) \r|_{L^{\infty}})\bigg].
			\end{align*}
			We want to show locally Lipschitz property on $H_n$. It therfore suffices to show Lipschitz property on closed balls in $H_n$. In that direction, assume that for some $0<R<\infty$, we have  $\l|v_1\r|_{L^{\infty}},\l|v_2\r|_{L^{\infty}} \leq R$. By Remark \ref{Remark meaning of norms are equivalent on Hn}, this is equivalent to saying that $v_1,v_2$ lie in a closed ball in the space $H_n$.\dela{ By this we mean that since the space $H_n$ is embedded inside the space $L^{\infty}$, $v_1,v_2$ lie in a closed ball in the space $H_n$ endowed with the norm inherited from the space $L^{\infty}$.}
			
			\dela{
				We split the above calculation into three parts. For the first term, we have the following. Since $\psi_0$ is Lipschitz continuous and takes values in $[0,1]$,
				\begin{align*}
					&\l| \bigg[\psi_0(\l| v_1 \r|_{L^{\infty}}) - \psi_0(\l| v_2 \r|_{L^{\infty}})\bigg] \, \psi_0(\l| P_n\l(  v_1 \times h \r)  \r|_{L^{\infty}}) \, \psi_0(\l| P_n \l( v_1 \times \l( v_1 \times h \r) \r) \r|_{L^{\infty}}) \r| \\
					\leq & 	\l| \psi_0(\l| v_1 \r|_{L^{\infty}}) - \psi_0(\l| v_2 \r|_{L^{\infty}}) \r| \\
					\leq & C \bigl| \l| v_1 \r|_{L^{\infty}} - \l| v_2 \r|_{L^{\infty}} \bigr| \\
					\leq & C \l| v_1 - v_2 \r|_{L^{\infty}} \\
					\leq & CC_n \l| v_1 - v_2 \r|_{L^2}.
				\end{align*}
				By Lemma \ref{Lemma locally Lipschitz coefficients problem considered} and Lemma \ref{Lemma DGn Gn is polynomial map}, the maps $F_n^4,G_n,\l[DG\r]\l(G\r)$ are Lipschitz on balls. Therefore, the products $\psi(\cdot)F_n^4(\cdot), \psi(\cdot)G_n(\cdot), \psi(\cdot)\l[DG(\cdot)\r]\l(G(\cdot)\r)$ is also locally Lipschitz on $H_n$.\\
				The second term can be handled similarly. For the third term,
				\begin{align*}
					&\l| \psi_0(\l| v_2 \r|_{L^{\infty}}) \, \psi_0(\l| P_n\l(  v_2 \times h \r)  \r|_{L^{\infty}}) \, \bigg[ \psi_0(\l| P_n \l( v_1 \times \l( v_1 \times h \r) \r) \r|_{L^{\infty}}) - \psi_0(\l| P_n \l( v_2 \times \l( v_2 \times h \r) \r) \r|_{L^{\infty}})\bigg] \r| \\
					\leq & \bigl|  \psi_0(\l| P_n \l( v_1 \times \l( v_1 \times h \r) \r) \r|_{L^{\infty}}) - \psi_0(\l| P_n \l( v_2 \times \l( v_2 \times h \r) \r) \r|_{L^{\infty}}) \bigr| \\
					\leq & C \bigl| \l| P_n \l( v_1 \times \l( v_1 \times h \r) \r) \r|_{L^{\infty}} - \l| P_n \l( v_2 \times \l( v_2 \times h \r) \r) \r|_{L^{\infty}} \bigr| \\
					\leq & C  \l| P_n \l( v_1 \times \l( v_1 \times h \r) -  v_2 \times \l( v_2 \times h \r) \r) \r|_{L^{\infty}}  \\
					\leq & C  \l| P_n \l( v_1 \times \l( v_1 \times h \r) -  v_2 \times \l( v_1 \times h \r) \r) \r|_{L^{\infty}}
					+ C \l| P_n \l( v_2 \times \l( v_1 \times h \r) -  v_2 \times \l( v_2 \times h \r) \r) \r|_{L^{\infty}}\\
					\leq & C_n  \l| P_n \l( v_1 \times \l( v_1 \times h \r) -  v_2 \times \l( v_1 \times h \r) \r) \r|_{L^2}
					+ C_n \l| P_n \l( v_2 \times \l( v_1 \times h \r) -  v_2 \times \l( v_2 \times h \r) \r) \r|_{L^2}\\
					\leq & C_n  \l| v_1 \times \l( v_1 \times h \r) -  v_2 \times \l( v_1 \times h \r)  \r|_{L^2}
					+ C_n \l| v_2 \times \l( v_1 \times h \r) -  v_2 \times \l( v_2 \times h \r)  \r|_{L^2}\\
					= & C_n  \l| \l(v_1 - v_2\r) \times \l( v_1 \times h \r)   \r|_{L^2}
					+ C_n \l| v_2 \times \l( \l( v_1 - v_2\r) \times h \r)   \r|_{L^2}\\
					\leq & C_n   \l|v_1 - v_2\r|_{L^2}  \l| v_1\r|_{L^{\infty}} \l| h \r|_{L^{\infty}}
					+ C_n \l| v_2 \r|_{L^{\infty}}  \l| v_1 - v_2\r|_{L^2}  \l|h\r|_{L^{\infty}}   \\
					\leq & C_n C \l|v_1 - v_2\r|_{L^2} .
				\end{align*}
			}
			
			Then from the above arguments, there exists a constant $C>0$ (that may depend on $n,R$) such that
			\begin{align}
				\l| \psi(v_1) - \psi(v_2) \r| \leq C \l| v_1 - v_2 \r|_{L^2}.
			\end{align}
			Therefore the function $\psi:H_n\to\mathbb{R}$ is Lipschitz on balls, and hence locally Lipschitz.\\
			Observe that for $v\in H_n$, we have $\psi F_n^4(v) = \psi(v) F_n^4(v)$. By Lemma \ref{Lemma locally Lipschitz coefficients problem considered}, $F_n^4$ is locally Lipschitz\dela{ and by Lemma \ref{Lemma psi n is Lipschitz}, $\psi$ is locally Lipschitz}. Therefore by Remark \ref{Remark on locally Lipschitz functions}, their product is also locally Lipschitz.\\
			A similar argument using\dela{ Lemma \ref{Lemma DGn Gn is polynomial map} and } Lemma \ref{Lemma locally Lipschitz coefficients problem considered} implies that the second map is also locally Lipschitz, thus concluding the proof of the lemma.		
		\end{proof}

	}

	\dela{Lemma similar to Lemma \ref{Lemma psi n is Lipschitz}.
		\begin{lemma}\label{lem-aux-01 ZB}
			Assume  the function $h \in H^1$.\\
			Then the  following function
			\[
			f: H_n \ni m  \mapsto \psi(|m|_{L^{\infty}})\psi(|P_n(m \times h)|_{L^{\infty}}) \psi(|P_n(m \times (m \times h))|_{L^{\infty}})P_n(P_n(m \times (m \times h)) \times h) \in H_n
			\]
			is well defined and locally Lipschitz.
		\end{lemma}
		\begin{proof}[Proof of Lemma \ref{lem-aux-01 ZB}]
			Let us define two auxiliary functions
			\begin{align*}
				f_1&: H_n \ni m  \mapsto \bigl( |m|_{L^{\infty}}, |P_n(m \times h)|_{L^{\infty}}, |P_n(m \times (m \times h))|_{L^{\infty}}\bigr) \in \mathbb{R}^3\\
				\tilde{\psi}&: \mathbb{R}^3 \ni (y_1,y_2,y_3) \mapsto \psi(y_1)\psi(y_2)\psi(y_3) \in \mathbb{R},\\
				f_2&:= \tilde{\psi}\circ f_1: H_n \to \mathbb{R}\\
				f_3&: H_n \ni m  \mapsto P_n(P_n(m \times (m \times h)) \times h) \in H_n \\
				\beta&:\mathbb{R} \times H_n \ni (y,m) \mapsto ym \in H_n,
			\end{align*}
			and observe that
			\[
			f=\beta\circ(f_2,f_3).
			\]
			About $f_3$. We defined a bilinear map
			\[
			\tilde{f}_3: H_n \times H_n \ni (m_1,m_2)  \mapsto P_n(P_n(m_1 \times (m_2 \times h)) \times h) \in H_n \\
			\]
			Then, by ...
			\[
			\vert \tilde{f}_3 (m_1,m_2)\vert_{H_n} \leq .... \vert m_1\vert  \vert m_2\vert
			\]
			Thus $\tilde{f}_3$ is well defined. It is obviously bilinear and the above inequality implies that it is bilinear bounded. Since
			\[
			f_3(m)=\tilde{f}_3 (m,m), \:\: m \in H_n,
			\]
			we infer that $f_3$ is a polynomial and hence an analytic function. In particular, it is locally Lipschitz, i.e. in this context Lipschitz on balls.

		\end{proof}
	}

	\dela{
		\begin{align}
			dm_n(t) = \l[F_n^1(m_n) + F_n^2(m_n) + F_n^3(m_n) + \frac{1}{2}\l[DG_n(m_n)\r]\l(G_n(m_n)\r)\r] \, dt + G(m_n) \, dW(t)
		\end{align}
		with the initial condition $m_n(0) = P_n(m_0)$. For reference, see for example \cite{Prato+Zabczyk}, \cite{Ikeda+Watanabe}.\\
	}
	
	\begin{lemma}\label{Lemma DGn Gn is polynomial map}
		For each $n$, the map
		\begin{align*}
			\l[ DG_n \r]\l( G_n \r): H_n \ni v \mapsto \bigl[ DG_n (v) \bigr]\bigl( G_n (v) \bigr) \in H_n
		\end{align*}
		is locally Lipschitz.	
	\end{lemma}
	\begin{proof}[Proof of Lemma \ref{Lemma DGn Gn is polynomial map}]
		We show the local Lipschitz continuity by first observing that the map $\l[ DG_n \r]\l( G_n \r)$ is a composition of two maps $DG_n$ and $G_n$, and then showing that both are locally Lipshitz in $H_n$. That $G_n$ is locally Lipschitz has been shown in the previous lemma (Lemma \ref{Lemma locally Lipschitz coefficients problem considered}).
		
		Define an auxiliary map $f_1: H_n\to H_n$, given by
		\begin{equation}
			f_1(v) = P_n\l( v \times h \r).
		\end{equation}
		Also define the maps $f_2,f_3: H_n  \times H_n \to H_n$, given by
		\begin{equation}
			f_2(v,w) = P_n\l(v \times \l( w \times h \r)\r)
		\end{equation}
		and
		\begin{equation}
			f_3(v,w) =  P_n\l(w  \times \l( v \times h \r)\r).
		\end{equation}
		Note that the map $f_2$ is similar to the map $f_2$ defined in Lemma \ref{Lemma locally Lipschitz coefficients problem considered} and is therefore locally Lipschitz on $H_n$. Although $f_3$ is not exactly the same as the map $f_2$, following the similar line of calculations we can show that $f_3$ is also locally Lipschitz on $H_n$.\\
		We now observe, see Proposition \ref{prop-derivative}, that $DG_n(v)$ is a linear combination of $f_1(v), f_2(v,v)$ and $f_3(v,v)$, and is therefore locally Lipschitz.\\
		By Remark \ref{Remark on locally Lipschitz functions}, the map $[DG](G)$, being a composition of two locally Lipschitz maps is locally Lipschitz.
		\dela{\textbf{Why the composition is locally Lipschitz:} Let $f,g$ be two locally Lipschitz maps. That is, $f,g$ are Lipschitz on balls. Since both $f,g$ are also continuous, they are bounded on closed and bounded balls. Therefore for a closed and bounded ball, we can consider the functions $f,g$ being Lipschitz continuous and bounded. Therefore their composition $f\circ g$ is also bounded and Lipschitz continuous on the ball.}
	\end{proof}

	\textbf{The Approximated Equation.}
	The approximated equation in $H_n$ is as follows. For $n\in\mathbb{N}$,
	\begin{align}\label{definition of solution Faedo Galerkin approximation 1}
		\nonumber m_n(t) =& P_n(m_0) + \int_{0}^{t} P_n \bigl( m_n(s) \times \Delta m_n(s) \bigr)\, ds - \alpha \, \int_{0}^{t} P_n \bigl[m_n(s) \times \bigl( m_n(s) \times \Delta m_n(s) \bigr)\bigr] ds\\
		\nonumber & + \int_{0}^{t} P_n \l( m_n(s) \times u_n(s) \r) ds - \alpha \, \int_{0}^{t} P_n \l[ m_n(s) \times \bigl( m_n(s) \times u_n(s) \bigr) \r]\, ds \\
		& + \frac{1}{2} \int_{0}^{t} \bigl[ DG_n \big( m_n(s) \big) \bigr] \big[ G_n \big( m_n(s) \big) \big] \, ds + \int_{0}^{t} G_n \big( m_n(s) \big) \, dW(s),\ t\in[0,T].
	\end{align}
	Let $s\in[0,T]$ and $n\in\mathbb{N}$. Using Proposition \ref{prop-derivative}, we can compute the correction term $DG(m_n)[G(m_n)]$ in \eqref{definition of solution Faedo Galerkin approximation 1} as given by the following. Note that we suppress the argument $(s)$ for brevity.
	
	
	\begin{align}\label{Stratonivich to Ito correction term}
		\nonumber	DG_n\big(m_n\big)\bigl(G_n(m_n)\bigr) =& P_n\big(P_n(m_n \times h) \times h\big) - \alpha \, P_n\bigg(P_n\big(m_n \times (m_n \times h)\big) \times h\bigg)\\
		\nonumber& - \alpha \, P_n(P_n\l(m_n \times h\r) \times (m_n \times h)) \\
		\nonumber & + P_n\bigg(m_n \times \big(P_n(m_n \times h) \times h\big)\bigg)\\
		\nonumber&-\alpha \, P_n\bigg(P_n\big(m_n \times (m_n(s) \times h)\big) \times (m_n \times h)\bigg)\\
		& - \alpha \, P_n\bigg(m_n \times \big(P_n(m_n \times \big(m_n \times h)\big)\times h\big)\bigg).
	\end{align}
	\textbf{The Truncated Approximated Equation:} Let $n\in\mathbb{N}$. We incorporate the cut-off $\psi_n$ defined in \eqref{definition of bump function psi n} into the equation \eqref{definition of solution Faedo Galerkin approximation 1}.\\
	\textbf{Note:} In the sections that follow, we will replace the notation $\psi_n$ with $\psi$ for brevity.
	
	\begin{align}\label{definition of solution Faedo Galerkin approximation}
		\nonumber m_n(t) = & P_n(m_0) + \int_{0}^{t} P_n \bigl( m_n(s) \times \Delta m_n(s) \bigr) \, ds - \alpha \, \int_{0}^{t} P_n \l[ m_n(s) \times \bigl( m_n(s) \times \Delta m_n(s) \bigr) \r] ds\\
		\nonumber & + \int_{0}^{t} P_n \bigl (m_n(s) \times u_n(s) \bigr) ds
		\ - \alpha \, \int_{0}^{t} \psi\big(m_n(s)\big) P_n \bigl[ m_n(s) \times \big( m_n(s) \times u_n(s) \big) \bigr] \, ds  \\
		\nonumber & + \frac{1}{2} \int_{0}^{t} \psi\big(m_n(s)\big)^2\bigl[ DG_n\big(m_n(s)\big) \bigr]\big[G_n\big(m_n(s)\big)\big]\, ds \\
		&	+ \int_{0}^{t} \psi\big(m_n(s)\big)G_n\big(m_n(s)\big)\, dW(s),\ t\in[0,T].
	\end{align}
	By Lemma \ref{Lemma locally Lipschitz coefficients problem considered}, Lemma \ref{Lemma psi n is Lipschitz}\dela{ Lemma \ref{Lemma locally Lipschitz coefficients truncated problem considered}} and Lemma \ref{Lemma DGn Gn is polynomial map}, we infer that the coefficients $F_n^i,\ i=1,\dots 4$ and $G_n$ are locally Lipschitz on $H_n$, for each $n\in\mathbb{N}$. Note that in order to prove that the solution to a stochastic differential equation is global, it is not sufficient that the coefficients are only locally Lipschitz.
	Global existence can be given by the one-sided linear growth, which in turn is given by the following inequality. For all $v \in H_n$, we have the following.
	\begin{equation*}
		\l\langle F_n^i (v) , v \r\rangle_{L^2} = 0 = \l\langle G_n (v) , v \r\rangle_{L^2}.
	\end{equation*}
	Combining local Lipschitz regularity and one sided linear growth with Theorem 10.6 in \cite{Chung+Williams_Introduction_To_Stochastic_Integration}, the problem \eqref{definition of solution Faedo Galerkin approximation} admits a unique solution.
	
	\dela{Proposition \ref{Proposition Global solution for FG approximation} outlines a proof for the existence. }

	\dela{
		\begin{align}
			\nonumber m_n(t) =& P_n(m_0) + \int_{0}^{t} P_n(m_n(s) \times \Delta m_n(s))\, ds - \alpha \, \int_{0}^{t} P_n (m_n(s) \times (m_n(s) \times \Delta m_n(s)))\, ds\\
			\nonumber  &+  \int_{0}^{t} P_n(m_n(s) \times u_n(s))\, ds \\
			\nonumber  &-  \alpha \, \int_{0}^{t} \psi(|m_n(s)|_{L^{\infty}}) P_n (m_n(s) \times(m_n(s) \times u_n(s)))  \, ds\\
			\nonumber  &+  \frac{1}{2} \bigg[ \int_{0}^{t} P_n(P_n(m_n(s) \times h) \times h) \, ds \\
			\nonumber &-   \alpha \, \int_{0}^{t} \psi(|m_n(s)|_{L^{\infty}})\psi(|P_n(m_n(s) \times h)|_{L^{\infty}}) \psi(|P_n(m_n(s) \times (m_n(s) \times h))|_{L^{\infty}}) \centerdot \\
			\nonumber &\quad P_n(P_n(m_n(s) \times (m_n(s) \times h)) \times h)  \, ds \\
			\nonumber &-   \alpha \, \int_{0}^{t} \psi(|m_n(s)|_{L^{\infty}})\, \psi(|P_n(m_n(s) \times h)|_{L^{\infty}}) \, \psi(|P_n(m_n(s) \times (m_n(s) \times h))|_{L^{\infty}}) \centerdot \\
			\nonumber &\quad P_n(P_n(m_n(s) \times h) \times (m_n(s) \times h)) ds \\
			\nonumber  &- \alpha \, \int_{0}^{t} P_n\bigl(m_n(s) \times (P_n(m_n(s) \times h) \times h)\bigr)  \, ds\\
			\nonumber &+   \alpha^2 \int_{0}^{t} \psi^2(|m_n(s)|_{L^{\infty}}) \psi(|P_n(m_n(s) \times h)|_{L^{\infty}})^2 \psi(|P_n(m_n(s) \times (m_n(s) \times h))|_{L^{\infty}})^2 \centerdot \\
			\nonumber &\quad P_n(P_n(m_n(s) \times (m_n(s) \times h)) \times (m_n(s) \times h))  \, ds\\
			\nonumber &+   \alpha^2 \int_{0}^{t} P_n(m_n(s) \times (P_n(m_n(s) \times (m_n(s) \times h))\times h)) \, ds \bigg] \\
			\nonumber &+   \int_{0}^{t}  P_n\l( m_n(s) \times h \r) \, dW(s)\\
			\nonumber  &-  \alpha \, \int_{0}^{t} \psi(|m_n(s)|_{L^{\infty}}) \psi(|P_n(m_n(s) \times h)|_{L^{\infty}}) \psi(|P_n(m_n(s) \times (m_n(s) \times h))|_{L^{\infty}}) \centerdot \\
			& \quad P_n \l( m_n(s) \times \l(m_n(s) \times h\r) \r)\, dW(s) .		
		\end{align}
		
	}

	\dela{
		
		We fix the following notation. For $s\in[0,T],n\in\mathbb{N}$, let
		\begin{align}\label{Definition of Vn}
			\nonumber  V_n(m(s)) =   &  \frac{1}{2} \bigg[  P_n(P_n(m_n(s) \times h) \times h)   \\
			\nonumber &-   \alpha \,  \psi(|m_n(s)|_{L^{\infty}})\psi(|P_n(m_n(s) \times h)|_{L^{\infty}}) \psi(|P_n(m_n(s) \times (m_n(s) \times h))|_{L^{\infty}}) \centerdot \\
			\nonumber &\quad P_n(P_n(m_n(s) \times (m_n(s) \times h)) \times h)   \\
			\nonumber &-   \alpha \,  \psi(|m_n(s)|_{L^{\infty}}) \psi(|P_n(m_n(s) \times h)|_{L^{\infty}}) \psi(|P_n(m_n(s) \times (m_n(s) \times h))|_{L^{\infty}}) \centerdot \\
			\nonumber &\quad P_n(P_n(m_n(s) \times h) \times (m_n(s) \times h))   \\
			\nonumber  &- \alpha \,  P_n(m_n(s) \times (P_n(m_n(s) \times h) \times h)) \\
			\nonumber &+   \alpha^2  \psi^2(|m_n(s)|_{L^{\infty}}) \psi(|P_n(m_n(s) \times h)|_{L^{\infty}})^2 \psi(|P_n(m_n(s) \times (m_n(s) \times h))|_{L^{\infty}})^2 \centerdot \\
			\nonumber &\quad P_n(P_n(m_n(s) \times (m_n(s) \times h)) \times (m_n(s) \times h))  \\
			&+   \alpha^2  P_n(m_n(s) \times (P_n(m_n(s) \times (m_n(s) \times h))\times h))   \bigg].
		\end{align}
		\adda{Replace $V_n$ by $\psi(m_n(s))^2 \l[DG(m_n(s))\r]\l( G(m_n(s)) \r)$}
		Here the $\centerdot$ represents the multiplication of two quantities on the left and right hand side of the dot. This is used since the term was not fitting on one line.\\
		Essentially, $V_n$ will be used to denote the correction term arising from the Stratonovich to It\^o conversion.

	}
	We now state a lemma that will be used in the calculations that follow.
	\begin{lemma}\label{Technical lemma bound on DGn Gn}
		\begin{enumerate}
			\item The following equality holds for all $w\in H_n$.
			\begin{equation}
				\l\langle \bigl[DG_n(w)\bigr]\bigl(G_n(w)\bigr) , w \r\rangle_{L^2} = - \l| G(w) \r|_{L^2}^2.
			\end{equation}
			
			\item There exists a constant $C>0$ such that for all $n\in\mathbb{N}$, the following inequality holds.
			\begin{equation}
				\l| \l\langle \psi(w)^2 \, \bigl[DG_n(w)\bigr]\bigl(G_n(w)\bigr) , \Delta w \r\rangle_{L^2} \r| \leq C \bigl( 1 + \l| w \r|_{H^1} \bigr)\l| v \r|_{H^1},\ w\in H_n.
			\end{equation}
		\end{enumerate}		
	\end{lemma}
	\begin{proof}[Proof of Lemma \ref{Technical lemma bound on DGn Gn}]
		For a proof of (1), we refer the reader to Corollary B.3 in \cite{ZB+BG+TJ_LargeDeviations_LLGE}.\\
		\textbf{Proof of (2):} From the proof of Lemma \ref{Lemma DGn Gn is polynomial map}, we observe that the map $\l[DG\r]\l(G\r)$ is a sum of polynomial maps of degree 2 and 3. The equality in \eqref{Stratonivich to Ito correction term} gives the precise form of the term considered. The main idea of the proof is the following .
		\begin{equation}
			\l| \l\langle \psi(v)^2 \, \l[DG_n(v)\r]\bigl(G(v)\bigr) , \Delta v \r\rangle_{L^2} \r| =  	\l| \l\langle \psi(v)^2 \, \nabla\bigl[\l[DG_n(v)\r]\bigl(G(v)\bigr) \bigr] , \nabla v \r\rangle_{L^2} \r|
		\end{equation}
		By using the product rule for derivatives, the term $\nabla\bigl[\bigl[DG_n(v)\bigr]\bigl(G(v)\bigr) \bigr]$ can be split into terms of two types. The first type, wherein the derivative is on $v$ and the second type where the derivative is on the term $h$. For the first type, we use H\"older's inequality with $L^2$ norm on the gradient term and $L^{\infty}$ norm on the remaining term/s. For the second type, the $L^2$ norm is applied on one of the terms containing $v$, while all other terms get the $L^{\infty}$ norm. The cut-off function $\Psi$ ensures that the $L^{\infty}$ norm is taken care of. As an example, we show the calculations here for the first term from \eqref{Stratonivich to Ito correction term}. Rest follow suite. Let $w\in H^1$.
		\dela{
			\begin{align*}
				\l| \psi(v)^2 \l\langle \nabla P_n\l( P_n \l(v \times h\r) \times h \r) , \nabla v \r\rangle_{L^2} \r| = & \l| \psi(v)^2 \l\langle  P_n \nabla\l( P_n \l(v \times h\r) \times h \r) , \nabla v \r\rangle_{L^2} \r| \\
				= & \l| \psi(v)^2 \l\langle \nabla \l( P_n \l(v \times h\r) \times h \r) , \nabla v \r\rangle_{L^2} \r| \\
				\leq & \l| \psi(v)^2 \l\langle   P_n \l[\nabla\l(v \times h\r) \r] \times h  , \nabla v \r\rangle_{L^2} \r| \\
				&+ \l| \psi(v)^2 \l\langle  \l( P_n \l(v \times h\r) \times \nabla h \r) , \nabla v \r\rangle_{L^2} \r| \\
				\leq & \l| \psi(v)^2 \l\langle   P_n \l[\nabla v \times h \r] \times h  , \nabla v \r\rangle_{L^2} \r| \\
				& + \l| \psi(v)^2 \l\langle   P_n \l[v \times \nabla h \r] \times h  , \nabla v \r\rangle_{L^2} \r| \\
				& + \l| \psi(v)^2 \l\langle   P_n \l[\l(v \times h\r) \r] \times \nabla h  , \nabla v \r\rangle_{L^2} \r| \\
				(\text{By H\"older's inequality})\ \leq &  \psi(v)^2   \l| P_n \l[\nabla v \times h \r] \r|_{L^2}\l| h \r|_{L^{\infty}}  \l| \nabla v \r|_{L^2} \\
				& + \psi(v)^2    \l| P_n \l[\l(v \times \nabla h\r) \r] \r|_{L^{2}} \l|  h \r|_{L^{\infty}}  \l| \nabla v \r|_{L^2} \\
				& + \psi(v)^2    \l| P_n \l[\l(v \times h\r) \r] \r|_{L^{\infty}} \l| \nabla h \r|_{L^2}   \l| \nabla v \r|_{L^2} \\
				\leq & \psi(v)^2 \l|\nabla v\r|_{L^2}^2 \l|h\r|_{L^{\infty}}^2 \\
				& +  \psi(v)^2 \l|v\r|_{L^{\infty}} \l| \nabla h \r|_{L^{2}} \l|h\r|_{L^{\infty}} \l|\nabla v\r|_{L^2} \\
				& + \psi(v)^2    \l| P_n \l[\l(v \times h\r) \r] \r|_{L^{\infty}} \l| \nabla h \r|_{L^2}   \l| \nabla v \r|_{L^2}.
			\end{align*}
		}
		First we observe the following. There exists a constant $C>0$ such that
		\begin{align*}
			\l|w \times h\r|_{H^1} \leq C \l| w \r|_{H^1} \l| h \r|_{H^1}.
		\end{align*}
		Now, let $v\in H_n$. Since $H_n \subset H^1$, the above inequality also holds for $w$ replaced by $v$. In the following sequence of inequalities, we use $C$ to denote a generic constant that is positive and independent of $n\in\mathbb{N}$. The constant $C$ can depend on $\l| h \r|_{H^1}$, and the value of $C$ may change from line to line.
		\begin{align*}
			\l| \psi(v)^2 \l\langle  P_n\l( P_n \l(v \times h\r) \times h \r) , \Delta v \r\rangle_{L^2} \r|
			= & \l| \psi(v)^2 \l\langle   P_n \l(v \times h\r) \times h  , \Delta v \r\rangle_{L^2} \r| \\
			= &  \l| \psi(v)^2 \l\langle   \nabla\l( P_n \l(v \times h\r) \times h \r) , \nabla v \r\rangle_{L^2} \r| \\
			\leq & \l| \psi(v)^2 \l\langle   \nabla\bigl( P_n \l(v \times h\r) \bigr) \times h  , \nabla v \r\rangle_{L^2} \r| \\
			&+ \l| \psi(v)^2 \l\langle  P_n \l(v \times h\r) \times \nabla h  , \nabla v \r\rangle_{L^2} \r| \\
			\leq &  \psi(v)^2 \l| \nabla\bigl( P_n \l(v \times h\r) \bigr) \r|_{L^2}  \l| h \r|_{L^{\infty}} \l| \nabla v \r|_{L^2}  \\
			&+  \psi(v)^2   \l|P_n \l(v \times h\r)\r|_{L^{\infty}}  \l| \nabla h \r|_{L^2} \l| \nabla v  \r|_{L^2} \\
			\leq &  C\: \psi(v)^2  \l|  P_n \l(v \times h\r)  \r|_{H^1}  \l| h \r|_{H^1} \l| v \r|_{H^1}  \\
			&+ \psi(v)^2   \l|P_n \l(v \times h\r)\r|_{L^{\infty}}  \l| h \r|_{H^1} \l| \nabla v  \r|_{L^2} \\
			\leq &  C\: \psi(v)^2  \l| v \times h  \r|_{H^1}  \l| v \r|_{H^1}  \\
			&+  C\: \psi(v)^2   \l|P_n \l(v \times h\r)\r|_{L^{\infty}}   \l| \nabla v  \r|_{L^2} \\
			\leq &   C\: \psi(v)^2 \l| v \r|_{H^1}^2   +  \psi(v)^2   \l|P_n \l(v \times h\r)\r|_{L^{\infty}}   \l| v  \r|_{H^1}.
		\end{align*}		
		By the definition of the cut-off function $\psi_0$, we have
		\begin{equation}
			\psi(v)^2    \l| P_n \l[\l(v \times h\r) \r] \r|_{L^{\infty}} \leq \l| h \r|_{L^{\infty}} + 2.
		\end{equation}
		Therefore, combining the calculations given above, there exists a constant $C>0$ such that
		\begin{equation}
			\l| \psi(v)^2 \l\langle  P_n\l( P_n \l(v \times h\r) \times h \r) , \Delta v \r\rangle_{L^2} \r| \leq C \l( 1 + \l| v \r|_{H^1} \r) \l| v \r|_{H^1}.
		\end{equation}
	\end{proof}

	For $n\in\mathbb{N}$, let $m_n = \l( m_n(t) \r)_{t\in[0,T]}$ be the solution to the problem \eqref{definition of solution Faedo Galerkin approximation}. We now obtain some uniform energy estimates which will be used to show tightness of the laws of the processes $m_n$ on a suitable space.

	\begin{lemma}\label{bounds lemma 1 without p}
		We have the following bounds.
		\begin{enumerate}
			\item The following equality holds for each $n\in\mathbb{N}$,	\begin{equation}\label{bound 1}
				\l|m_n(t)\r|_{L^2}^2 = \l|m_n(0)\r|_{L^2}^2 \ \text{for each}\ t\in[0,T],\ \mathbb{P}-a.s.
			\end{equation}
			
			\item 	There exists a constant $C>0$ such that for all $n\in\mathbb{N}$, the following inequalities hold.
			\begin{enumerate}
				\item \begin{equation}\label{bound 2 without p}
					\mathbb{E}\l[\sup_{t\in[0,T]}\l|m_n(t)\r|_{H^1}^{2}\r]\leq C,
				\end{equation}
				\item \begin{equation}\label{bound 3 without p}
					\mathbb{E} \l[ \int_{0}^{T} \l| m_n(t) \times \Delta m_n(t) \r|_{L^2}^2 \, dt \r] \leq C.
				\end{equation}
			\end{enumerate}

		\end{enumerate}
	\end{lemma}
	
	\begin{proof}[Proof of Lemma \ref{bounds lemma 1 without p}]
		
		The bounds are obtained by applying the It\^o formula followed by using the Burkholder-Davis-Gundy inequality (Lemma \ref{BDG inequality upper bound}) for the terms with the stochastic integral. Then we apply the Gronwall lemma to obtain the required bounds. Similar ideas have been used in \cite{ZB+BG+TJ_LargeDeviations_LLGE}, see Lemma 3.3, \cite{ZB+BG+TJ_Weak_3d_SLLGE} among others.\\
		\textbf{Proof of the bound \eqref{bound 1}}:\\
		Let us choose and fix $n\in\mathbb{N}$. We define a function $\phi_1 : H_n\rightarrow \mathbb{R}$ by
		\begin{equation}\label{definition of phi 1}
			\phi_1(v) = \frac{1}{2}|v|^2_{L^2},\ \text{for}\ v\in H_n.
		\end{equation}
		Note that the It\^o formula applied is for finite dimensional (Eucildean spaces) and hence $H_n$ is required instead of $L^2$.		
		For $v,v_1,v_2\in H_n$, we have
		$$\phi_1^{\prime}(v)(v_2)=\l\langle v , v_2 \r\rangle_{L^2},$$ 
		and
		$$\phi_1^{\prime\prime}(v)(v_1 , v_2)=\l\langle v_1 , v_2 \r\rangle_{L^2}.$$
		Let $m_n = \l( m_n(t) \r)_{t\in[0,T]}$ be the solution of \eqref{definition of solution Faedo Galerkin approximation}. Applying the It\^o formula to $\phi_1$ gives us the following equation for all $t\in[0,T]$, $\mathbb{P}$-a.s.

		\dela{
			
			\begin{align}
				\nonumber \phi_1(m_n(t))  = & \phi_1(P_n(m_0)) + \int_{0}^{t} \l\langle P_n(m_n(s) \times \Delta m_n(s)) , m_n(s) \r\rangle_{L^2}\, ds \\
				\nonumber  &-  \alpha \, \int_{0}^{t} \l\langle P_n (m_n(s) \times (m_n(s) \times \Delta m_n(s)))    , m_n(s) \r\rangle_{L^2} \, ds\\
				\nonumber &+   \int_{0}^{t} \l\langle P_n(m_n(s) \times u_n(s))  ,  m_n(s)  \r\rangle_{L^2} \, ds \\
				\nonumber  &-  \alpha \, \int_{0}^{t} \psi(|m_n(s)|_{L^{\infty}}) \l\langle P_n (m_n(s) \times(m_n(s) \times u_n(s)))  ,  m_n(s)  \r\rangle_{L^2} ds \\
				\nonumber & + \int_{0}^{t} \l\langle \psi(m_n(s))^2 \, \l[DG_n(m_n(s))\r]\l(G_n(m_n(s))\r) , m_n(s) \r\rangle_{L^2} \, ds \\
				\dela{\nonumber &+   \frac{1}{2} \bigg[ \int_{0}^{t} \l\langle P_n(P_n(m_n(s) \times h) \times h)  ,  m_n(s)  \r\rangle_{L^2} ds \\
					\nonumber  &- \alpha \, \int_{0}^{t} \psi(|m_n(s)|_{L^{\infty}})\psi(|P_n(m_n(s) \times h)|_{L^{\infty}}) \psi(|P_n(m_n(s) \times (m_n(s) \times h))|_{L^{\infty}})\centerdot \\
					\nonumber &\quad\l\langle P_n(P_n(m_n(s) \times (m_n(s) \times h)) \times h)  ,  m_n(s)  \r\rangle_{L^2} ds \\
					\nonumber  &-  \alpha \, \int_{0}^{t} \psi(|m_n(s)|_{L^{\infty}}) \psi(|P_n(m_n(s) \times h)|_{L^{\infty}}) \psi(|P_n(m_n(s) \times (m_n(s) \times h))|_{L^{\infty}})\centerdot \\
					\nonumber &\quad\l\langle P_n(P_n(m_n(s) \times h) \times (m_n(s) \times h)) ,  m_n(s)  \r\rangle_{L^2} ds \\
					\nonumber  &- \alpha \, \int_{0}^{t} \l\langle P_n(m_n(s) \times (P_n(m_n(s) \times h) \times h))  ,  m_n(s)  \r\rangle_{L^2} ds\\
					\nonumber &+   \alpha^2 \int_{0}^{t} \psi^2(|m_n(s)|_{L^{\infty}}) \psi(|P_n(m_n(s) \times h)|_{L^{\infty}})^2 \psi(|P_n(m_n(s) \times (m_n(s) \times h))|_{L^{\infty}})^2 \centerdot \\
					\nonumber& \quad\l\langle P_n(P_n(m_n(s) \times (m_n(s) \times h)) \times (m_n(s) \times h)) ,  m_n(s)  \r\rangle_{L^2} ds\\
					\nonumber  &+  \alpha^2 \int_{0}^{t} \l\langle P_n(m_n(s) \times (P_n(m_n(s) \times (m_n(s) \times h))\times h)) ,  m_n(s)  \r\rangle_{L^2} ds \bigg]  \\}
				\nonumber  & + \int_{0}^{t} \l\langle P_n\l( m_n(s) \times h \r) ,  m_n(s)  \r\rangle_{L^2} \, dW(s)\\
				\nonumber  & + \int_{0}^{t} \psi(|m_n(s)|_{L^{\infty}}) \psi(|P_n(m_n(s) \times h)|_{L^{\infty}}) \psi(|P_n(m_n(s) \times (m_n(s) \times h))|_{L^{\infty}})\centerdot \\
				\nonumber &\quad \l\langle P_n\l( m_n(s) \times \l(m_n(s) \times h\r)\r) ,  m_n(s)  \r\rangle_{L^2}\, dW(s) \\
				\nonumber  &+  \frac{1}{2} \int_{0}^{t} \bigg| P_n\big(  m_n(s) \times h - \alpha \, \psi(|m_n(s)|_{L^{\infty}}) \psi(|P_n(m_n(s) \times h)|_{L^{\infty}}) \centerdot \\
				\nonumber& \quad\psi(|P_n\l(m_n(s) \times (m_n(s) \times h)\r)|_{L^{\infty}})
				P_n \l( m_n(s) \times \l( m_n(s) \times h \r) \r) \big) \bigg|_{L^2}^2 \, ds \\
				= &  \frac{1}{2}\l| P_n(m_0) \r|_{L^2}^2 + \sum_{i=1}^{8} C_i I_i(t).
			\end{align}	
		}
		
		\begin{align}\label{equation Ito formula applied to phi 1}
			\nonumber \phi_1(m_n(t))  = & \phi_1(P_n(m_0)) + \int_{0}^{t} \l\langle P_n\big(m_n(s) \times \Delta m_n(s)\big) , m_n(s) \r\rangle_{L^2}\, ds \\
			\nonumber  &-  \alpha \, \int_{0}^{t} \l\langle P_n \bigg(m_n(s) \times \big(m_n(s) \times \Delta m_n(s)\big)\bigg)    , m_n(s) \r\rangle_{L^2} \, ds\\
			\nonumber &+   \int_{0}^{t} \l\langle P_n\big(m_n(s) \times u_n(s)\big)  ,  m_n(s)  \r\rangle_{L^2} \, ds \\
			\nonumber  &-  \alpha \, \int_{0}^{t} \psi\big(m_n(s)\big) \l\langle P_n \bigg(m_n(s) \times\big(m_n(s) \times u_n(s)\big)\bigg)  ,  m_n(s)  \r\rangle_{L^2} ds \\
			\nonumber & + \frac{1}{2} \int_{0}^{t} \l\langle \psi\big(m_n(s)\big)^2 \, \bigl[DG_n\big(m_n(s)\big)\bigr]\bigl(G_n\big(m_n(s)\big)\bigr) , m_n(s) \r\rangle_{L^2} \, ds \\
			\dela{\nonumber &+   \frac{1}{2} \bigg[ \int_{0}^{t} \l\langle P_n(P_n(m_n(s) \times h) \times h)  ,  m_n(s)  \r\rangle_{L^2} ds \\
				\nonumber  &- \alpha \, \int_{0}^{t} \psi(|m_n(s)|_{L^{\infty}})\psi(|P_n(m_n(s) \times h)|_{L^{\infty}}) \psi(|P_n(m_n(s) \times (m_n(s) \times h))|_{L^{\infty}})\centerdot \\
				\nonumber &\quad\l\langle P_n(P_n(m_n(s) \times (m_n(s) \times h)) \times h)  ,  m_n(s)  \r\rangle_{L^2} ds \\
				\nonumber  &-  \alpha \, \int_{0}^{t} \psi(|m_n(s)|_{L^{\infty}}) \psi(|P_n(m_n(s) \times h)|_{L^{\infty}}) \psi(|P_n(m_n(s) \times (m_n(s) \times h))|_{L^{\infty}})\centerdot \\
				\nonumber &\quad\l\langle P_n(P_n(m_n(s) \times h) \times (m_n(s) \times h)) ,  m_n(s)  \r\rangle_{L^2} ds \\
				\nonumber  &- \alpha \, \int_{0}^{t} \l\langle P_n(m_n(s) \times (P_n(m_n(s) \times h) \times h))  ,  m_n(s)  \r\rangle_{L^2} ds\\
				\nonumber &+   \alpha^2 \int_{0}^{t} \psi^2(|m_n(s)|_{L^{\infty}}) \psi(|P_n(m_n(s) \times h)|_{L^{\infty}})^2 \psi(|P_n(m_n(s) \times (m_n(s) \times h))|_{L^{\infty}})^2 \centerdot \\
				\nonumber& \quad\l\langle P_n(P_n(m_n(s) \times (m_n(s) \times h)) \times (m_n(s) \times h)) ,  m_n(s)  \r\rangle_{L^2} ds\\
				\nonumber  &+  \alpha^2 \int_{0}^{t} \l\langle P_n(m_n(s) \times (P_n(m_n(s) \times (m_n(s) \times h))\times h)) ,  m_n(s)  \r\rangle_{L^2} ds \bigg]  \\}
			\nonumber & + \int_{0}^{t} \l\langle \psi\big(m_n(s)\big) G_n\big(m_n(s)\big) ,  m_n(s)  \r\rangle_{L^2} \, dW(s)\\
			\nonumber  &+  \frac{1}{2} \int_{0}^{t} \psi\big(m_n(s)\big)^2 \l| G_n\big(m_n(s)\big) \r|_{L^2}^2 \, ds \\
			\dela{
				\nonumber  & + \int_{0}^{t} \l\langle P_n\l( m_n(s) \times h \r) ,  m_n(s)  \r\rangle_{L^2} \, dW(s)\\
				\nonumber  & + \int_{0}^{t} \psi(m_n(s))  \l\langle P_n\l( m_n(s) \times \l(m_n(s) \times h\r)\r) ,  m_n(s)  \r\rangle_{L^2}\, dW(s) \\
				\nonumber  &+  \frac{1}{2} \int_{0}^{t} \bigg| P_n\big(  m_n(s) \times h - \alpha \, \psi(m_n(s))		P_n \l( m_n(s) \times \l( m_n(s) \times h \r) \r) \big) \bigg|_{L^2}^2 \, ds \\}
			= &  \frac{1}{2}\l| P_n(m_0) \r|_{L^2}^2 + \sum_{i=1}^{7} C_i I_i(t),\ t\in[0,T].
		\end{align}	
		Here $C_i, i=1,\dots,7$ are the constants accompanying the integrals.
		
		For each $n\in\mathbb{N}$, the projection operator is self-adjoint. That is for $v_1,v_2\in L^2$, the following holds.
		
		\begin{equation}\label{projection operator is self-adjoint}
			\l\langle P_n v_1 , v_2 \r\rangle _{L^2}  = \l\langle v_1 , P_n v_2 \r\rangle _{L^2}.
		\end{equation}
		Also, $P_n^2 = P_n$ along with the self-adjoint property implies that
		
		\begin{equation}\label{projection operator property 2}
			\l\langle P_n v_1 , P_n v_2 \r\rangle _{L^2}  =  \l\langle P_n^2 v_1 , v_2 \r\rangle _{L^2}  =   \l\langle P_n v_1 , v_2 \r\rangle _{L^2} .
		\end{equation}
		\dela{	We also have that
			\begin{align}
				\l\langle \Delta P_n v_1 , v_2 \r\rangle_{L^2} =  \l\langle P_n \Delta v_1 , v_2 \r\rangle_{L^2} ,\ \text{for}\ v_1,v_2\in H^1.
			\end{align}
		}
		The above mentioned properties will be frequently used in the calculations that follow.
		Another property of vectors that will be used frequently is the following:
		\begin{equation}
			\l\langle a , a \times b \r\rangle_{\mathbb{R}^3} = 0\ \text{for}\ a,b\in\mathbb{R}^3.
		\end{equation}		
			%
			%
			%
		Using the properties mentioned above, we can show that for $i=1, 2, 3, 4, 6$.
		\begin{equation}
			I_i(t) = 0.
		\end{equation}
		In particular, we observe that the terms $I_3,I_4$ with the control operator do not contribute to the calculations since they are $0$.
		We show the calculations for $I_1$. Rest follow suite. Let $v\in H_n$.
		\begin{align*}
			\l\langle P_n \l(v \times \Delta v\r) , v \r\rangle_{L^2} & = \l\langle v \times \Delta v , P_n v \r\rangle_{L^2} \\
			& = \l\langle v \times \Delta v , v \r\rangle_{L^2} \\
			&= \int_{\mathcal{O}} \l\langle v(x) \times \Delta v(x) , v(x) \r\rangle_{\mathbb{R}^3} \, dx = 0.
		\end{align*}
		Replacing $v$ in the above setup by $m_n(s)$ and then integrating gives us the desired result. Notice that we have used only the properties of the projection operator $P_n$ on $L^2$ and the $\mathbb{R}^3$ inner product properties. $\psi$ is a scalar valued function and hence does not contribute to the $\mathbb{R}^3$ inner product. Moreover, $\psi$ is not a function of the space variable, and hence can be taken out of the $L^2$ inner product as well. Therefore a similar result follows for $I_i,i=2,3,4,6$.\\		
		\textbf{Calculation for $I_5$ and $I_7$:}\\
		By equality (1) of Lemma \ref{Technical lemma bound on DGn Gn}, we have
		\begin{align*}
			I_5(t) = & \int_{0}^{t} \l\langle \psi\big(m_n(s)\big)^2 \, \l[DG_n\big(m_n(s)\big)\r]\l[G_n\big(m_n(s)\big)\r] , m_n(s) \r\rangle_{L^2} \, ds \\
			= & - \int_{0}^{t} \psi\big(m_n(s)\big)^2 \l| G_n\big(m_n(s)\big) \r|_{L^2}^2 \, ds = -I_7(t).
		\end{align*}
		Observe that in the equality \eqref{equation Ito formula applied to phi 1}, we have $C_5 = C_7 = \frac{1}{2}$. Note that the term $C_7 I_7(t)$ in \eqref{equation Ito formula applied to phi 1} is the term arising from the application of the It\^o formula. Therefore, using equality (1) in Lemma \ref{Technical lemma bound on DGn Gn}, we have the following.
		\begin{equation}
			C_5I_5(t) + C_7I_7(t) = 0.
		\end{equation}
		\dela{
			\textbf{Calculations for $I_1, I_2$.}\\
			For each $s\in[0,T]$ and $n\in\mathbb{N}$,
			\begin{align*}
				\l\langle m_n(s),P_n\b(m_n(s)\times\Delta m_n(s)\b) \r\rangle_{L^2}=& \l\langle m_n(s),m_n(s)\times \Delta m_n(s) \r\rangle_{L^2} = 0.
			\end{align*}
			Similarly,
			\begin{align*}
				\l\langle m_n(s), P_n\b(m_n(s)\times \l( m_n\times \Delta m_n(s) \r) \b) \r\rangle_{L^2} &= \l\langle m_n(s), m_n(s)\times \l( m_n\times \Delta m_n(s) \r) \r\rangle_{L^2} = 0.
			\end{align*}
			
			Therefore for each $t\in[0,T]$,
			\begin{align}
				I_1(t) = I_2(t) = 0.
			\end{align}
		}
		\dela{
			
			\textbf{Calculations for $I_1, I_2$.}\\
			Let $v\in H_n$. The operator $P_n$ is self adjoint on $L^2$. Therefore
			\begin{align*}
				\l\langle v , P_n\l( v \times \Delta v \r) \r\rangle_{L^2} = \l\langle P_n v ,  v \times \Delta v  \r\rangle_{L^2}.
			\end{align*}
			$v\in H_n$ implies that $P_n v = v$. Also, for $a,b\in\mathbb{R}^3$,
			\begin{equation*}
				\l\langle a , a \times b \r\rangle_{\mathbb{R}^3} = 0.
			\end{equation*}
			Therefore,
			\begin{align*}
				\l\langle P_n v ,  v \times \Delta v  \r\rangle_{L^2} & =  \l\langle  v ,  v \times \Delta v  \r\rangle_{L^2}\\
				& = \int_{\mathcal{O}}  \l\langle v(x) ,  v(x) \times \Delta v(x)  \r\rangle_{\mathbb{R}^3} \, dx = 0.
			\end{align*}
			Working similarly, we can show that
			\begin{align*}
				\l\langle v , P_n\bigl(v \times \l( v \times \Delta v \r) \bigr) \r\rangle_{L^2} = \l\langle P_n v , v \times \l( v \times \Delta v \r)  \r\rangle_{L^2} = 0.
			\end{align*}
			Therefore replacing $v$ by $m_n(s)$, for $s\in[0,T]$ and integrating gives
			\begin{equation}
				I_1(t) = I_2(t) = 0.
			\end{equation}
			
			\dela{
				\textbf{Calculations for $I_3,I_4$.}\\
				Similar to the above calculations, we have

				\begin{align*}
					\l\langle m_n(s) , P_n \l( m_n(s) \times u_n(s) \r) \r\rangle_{L^2} & = \l\langle P_n m_n(s) , m_n(s) \times u_n(s)  \r\rangle_{L^2} \\
					& = \l\langle m_n(s) , m_n(s) \times u_n(s) \r\rangle_{L^2} = 0.
				\end{align*}
				Similarly,
				\begin{align*}
					\l\langle m_n(s) , P_n \l( m_n(s) \times(m_n(s) \times u_n(s)) \r) \r\rangle_{L^2} = \l\langle m_n(s) , m_n(s) \times(m_n(s) \times u_n(s)) \r\rangle_{L^2} = 0.
				\end{align*}
				Therefore for each $t\in[0,T]$.
				\begin{align}
					I_3(t) = I_4(t) = 0.
				\end{align}
				
			}
			\textbf{Calculations for $I_3,I_4$.}\\
			Let $v_1,v_2\in H_n$. Then,
			\begin{align*}
				\l\langle v_1 , P_n\l( v_1 \times v_2 \r) \r\rangle_{L^2} & = \l\langle v_1 ,  v_1 \times v_2  \r\rangle_{L^2}\\
				& = \int_{\mathcal{O}} \l\langle v_1(x) ,  v_1(x) \times v_2(x)  \r\rangle_{L^2} \, dx = 0.
			\end{align*}
			
			\textbf{Calculations for $I_{11},I_{12},I_{13}$.}\\
			The terms $I_{11},I_{12}$ together constitute the noise coefficient. The term $I_{13}$ is the correction term that arises due to the application of the It\^o formula.\\
			For each $s\in[0,T]$,
			\begin{align*}
				\l\langle m_n(s) , P_n\l(m_n(s) \times h - \alpha \, m_n(s) \times \l( m_n(s) \times h \r)\r) \r\rangle_{L^2} = 0.
			\end{align*}
			Hence for each $t\in[0,T]$,
			\begin{align}
				I_{11}(t) = I_{12}(t) = 0.
			\end{align}
			Note that for each $s\in[0,T]$, $\psi$ is a constant (i.e. $\psi$ does not depend on the space variable). Therefore to simplify the presentation, the cut-off function has not been mentioned in the above calculation.

			Note that
			\begin{align*}
				\l\langle m_n(s) \times h , m_n(s) \times \l( m_n(s) \times h \r) \r\rangle_{L^2} = 0.
			\end{align*}
			Therefore
			\begin{align}
				\nonumber I_{13}(t) =& \int_{0}^{t}\l| m_n(s) \times h \r|_{L^2}^2 \, ds \\
				\nonumber& + \int_{0}^{t} \psi^2(|m_n(s)|_{L^{\infty}}) \psi^2(|P_n(m_n(s) \times h)|_{L^{\infty}}) \psi^2(|P_n(m_n(s) \times (m_n(s) \times h))|_{L^{\infty}})\centerdot \\
				& \quad\l| m_n(s) \times m \l( m_n(s) \times h \r) \r|_{L^2}^2 \, ds
			\end{align}
			\textbf{Calculation for $I_5$.}
			\begin{align*}
				\l\langle m_n(s) , P_n\b(P_n(m_n(s) \times h) \times h\b) \r\rangle_{L^2} &= \l\langle P_n(m_n(s)) , P_n\b(m_n(s) \times h\b) \times h \r\rangle_{L^2} \ \text{By}\ \eqref{projection operator is self-adjoint}\\
				& = \l\langle m_n(s) , P_n\b(m_n(s) \times h\b) \times h \r\rangle_{L^2} \ \text{Since}\ m_n(s)\in H_n\\
				& =  - \l\langle m_n(s) \times h , P_n\b(m_n(s) \times h\b) \r\rangle_{L^2} \\
				& =   - \l\langle P_n\b(m_n(s) \times h\b) , P_n\b(m_n(s) \times h\b) \r\rangle_{L^2} \ \text{By}\ \eqref{projection operator property 2} \\
				& = -\l|P_n\b(m_n(s) \times h\b)\r|_{L^2}^2.
			\end{align*}
			Therefore
			\begin{align*}
				I_5(t) = -\int_{0}^{t}\l|P_n\b(m_n(s) \times h\b)\r|_{L^2}^2 \, ds.
			\end{align*}
			This term (along with $I_9$) effectively cancels the correction term due to the It\^o formula.\\
			\textbf{Calculation for $I_6$.}\\
			By \eqref{projection operator is self-adjoint}, followed by \eqref{projection operator property 2}, we have
			\dela{\begin{align*}
					\l\langle m_n(s) &, P_n\l(P_n(m_n(s) \times (m_n(s) \times h)) \times h\r) \r\rangle_{L^2} \\
					&= \l\langle P_n(m_n(s)) , P_n(m_n(s) \times (m_n(s) \times h)) \times h \r\rangle_{L^2} \ \text{By}\ \eqref{projection operator is self-adjoint}\\
					& = \l\langle m_n(s) , P_n\l(m_n(s) \times (m_n(s) \times h)\r) \times h \r\rangle_{L^2} \ \text{By}\ m_n(s)\in H_n\\
					& = - \l\langle m_n(s) \times h , P_n\l(m_n(s) \times (m_n(s) \times h)\r) \r\rangle_{L^2} \\
					& = - \l\langle P_n\l(m_n(s) \times h\r) , P_n\l(m_n(s) \times \l(m_n(s) \times h\r)\r) \r\rangle_{L^2}.\ \text{By}\ \eqref{projection operator property 2}
				\end{align*}
				Therefore}
			\begin{align*}
				I_6(t) = - \int_{0}^{t} \l\langle P_n\l(m_n(s) \times h\r) , P_n\l(m_n(s) \times \l(m_n(s) \times h\r)\r) \r\rangle_{L^2} \, ds.
			\end{align*}
			\textbf{Calculation for $I_7$.}
			Similarly,
			\dela{\begin{align*}
					\l\langle m_n(s) ,& P_n(P_n(m_n(s) \times h) \times (m_n(s) \times h)) \r\rangle_{L^2} \\
					& = \l\langle m_n(s) , P_n(m_n(s) \times h) \times (m_n(s) \times h) \r\rangle_{L^2} \ \text{By}\ \eqref{projection operator is self-adjoint},\ m_n\in H_n \\
					& =  \l\langle m_n(s) \times (m_n(s) \times h) , P_n(m_n(s) \times h) \r\rangle_{L^2} \\
					& =  \l\langle P_n(m_n(s) \times (m_n(s) \times h)) , P_n(m_n(s) \times h) \r\rangle_{L^2}. \ \text{By}\ \eqref{projection operator property 2}
				\end{align*}
				Hence}
			\begin{align*}
				I_7(t) = \int_{0}^{t} \l\langle P_n(m_n(s) \times (m_n(s) \times h)) , P_n(m_n(s) \times h) \r\rangle_{L^2} \, ds.
			\end{align*}
			$C_6 = C_7$ together with the above calculations for $I_6$ and $I_7$ impliy that $I_6(t) + I_7(t) = 0$.\\
			Again by \eqref{projection operator is self-adjoint}, we have the following two sets of equalities.\\
			\textbf{Calculation for $I_8$}
			\dela{\begin{align*}
					\l\langle m_n(s) ,& P_n(m_n(s) \times (P_n(m_n(s) \times h) \times h)) \r\rangle_{L^2} \\
					& = \l\langle P_n(m_n(s)) , m_n(s) \times (P_n(m_n(s) \times h)) \r\rangle_{L^2}  \ \text{By}\ \eqref{projection operator is self-adjoint} \\
					& = \l\langle m_n(s) , m_n(s) \times (P_n(m_n(s) \times h) \times h) \r\rangle_{L^2} \ \text{Because}\ m_n(s) \in H_n \\
					& = 0.
			\end{align*}}
			\begin{align*}
				\l\langle m_n(s) ,& P_n(m_n(s) \times (P_n(m_n(s) \times h) \times h)) \r\rangle_{L^2} = 0.
			\end{align*}
			\textbf{Calculation for $I_9$.}
			\dela{\begin{align*}
					\l\langle m_n(s) ,& P_n(P_n(m_n(s) \times (m_n(s) \times h)) \times (m_n(s) \times h)) \r\rangle_{L^2} \\
					& = \l\langle m_n(s) , P_n(m_n(s) \times (m_n(s) \times h)) \times (m_n(s) \times h) \r\rangle_{L^2} \\
					& =  - \l\langle m_n(s) \times (m_n(s) \times h) , P_n(m_n(s) \times (m_n(s) \times h))  \r\rangle_{L^2} \\
					& = - \l\langle P_n(m_n(s) \times (m_n(s) \times h)) , P_n(m_n(s) \times (m_n(s) \times h))  \r\rangle_{L^2} \\
					& = - |P_n(m_n(s) \times (m_n(s) \times h))|_{L^2}^2.
			\end{align*}}
			\begin{align*}
				\l\langle m_n(s) ,& P_n(P_n(m_n(s) \times (m_n(s) \times h)) \times (m_n(s) \times h)) \r\rangle_{L^2}  = - |P_n(m_n(s) \times (m_n(s) \times h))|_{L^2}^2.
			\end{align*}\\
			\textbf{Calculation for $I_{10}$.}
			\dela{\begin{align*}
					\l\langle m_n(s) ,& P_n(m_n(s) \times P_n(m_n(s) \times (m_n(s)\times h)) \times h )  \r\rangle_{L^2} \\
					& = \l\langle P_n(m_n(s)) , m_n(s) \times P_n(m_n(s) \times (m_n(s)\times h)) \times h   \r\rangle_{L^2}\\
					& = \l\langle m_n(s) , m_n(s) \times P_n(m_n(s) \times (m_n(s)\times h)) \times h   \r\rangle_{L^2} = 0.
			\end{align*}}
			\begin{align*}
				& \l\langle m_n(s) , P_n(m_n(s) \times P_n(m_n(s) \times (m_n(s)\times h)) \times h )  \r\rangle_{L^2}  \\
				= & \l\langle m_n(s) , m_n(s) \times P_n(m_n(s) \times (m_n(s)\times h)) \times h   \r\rangle_{L^2} = 0.
			\end{align*}
			The exponents and coefficients of $\psi$ have been written so that some terms cancel each other once the It\^o Lemma is applied. Hence combining the above calculations with the equation \eqref{equation Ito formula applied to phi 1}, we get
			\dela{	\begin{equation*}
					\frac{1}{2}|m_n(s)|_{L^2} = \frac{1}{2}|m_n(0)|_{L^2}
				\end{equation*}
				That is
			}
			
		}
		Combining the above calculations with the equation \eqref{equation Ito formula applied to phi 1}, we get
		\begin{equation*}
			\frac{1}{2}|m_n(t)|_{L^2}^2 = \frac{1}{2}|m_n(0)|_{L^2}^2.
		\end{equation*}
		That is
		\begin{equation*}
			|m_n(t)|_{L^2}^2 = |m_n(0)|_{L^2}^2.
		\end{equation*}
		This holds for each $t\in[0,T]$ and $n\in\mathbb{N}$.\\
		This concludes the proof of the bound \eqref{bound 1}.
		Note that by the definition of $m_n(0)$ and the projection operator $P_n$,
		\begin{align}
			\nonumber	|m_n(0)|_{L^2} = |P_n(m(0))|_{L^2} \leq |m(0)|_{L^2}.
		\end{align}
		Thus, there exists a constant $C>0$ such that
		\dela{
			\begin{equation*}
				\sup_{t\in [0,T]} |m_n(t)|_{L^2} \leq|m(0)|_{L^2}.
			\end{equation*}
		}
		
		\begin{equation}\label{bound on m_n}
			\mathbb{E}\sup_{n\in\mathbb{N}}\sup_{t\in [0,T]} |m_n(s)|_{L^2}^2  \leq \mathbb{E} \l|m(0)\r|_{L^2}^2 \leq C.
		\end{equation}
		\textbf{Proof of bound \eqref{bound 2 without p}}:\\
		This can be shown by applying the It\^o formula to the function $\phi_2 : H_n\rightarrow \mathbb{R}$ defined by
		
		\begin{equation}\label{definition of phi 2}
			\phi_2(v)=\frac{1}{2}|\nabla v|_{L^2}^2\ \text{for}\ v\in H^1.
		\end{equation}		
		For each $v,v_1,v_2\in H_n$,
		\begin{align}
			\phi_2^\p(v_1)(v_2) = \l\langle \nabla v_1 , \nabla v_2 \r\rangle_{L^2} = \l\langle v_1 , -\Delta v_2  \r\rangle_{L^2} = \l\langle v_1 , A v_2 \r\rangle_{L^2}.
		\end{align}
		
		\begin{align*}
			\phi_2^{\prime \prime}(v)(v_1,v_2) = \l\langle \nabla v_1 , \nabla v_2 \r\rangle_{L^2}.
		\end{align*}		
		Application of the It\^o formula gives the following equation for all $t\in[0,T]$ $\mathbb{P}$-a.s.

		\begin{align}\label{equation Ito formula applied to phi 2}
			\nonumber \phi_2\big(m_n(t)\big) =\,& \phi_2\big(P_n(m_0)\big) + \int_{0}^{t} \l\langle P_n\big(m_n(s) \times \Delta m_n(s)\big) , (-\Delta ) m_n(s) \r\rangle_{L^2}\, ds \\
			\nonumber & -   \alpha \, \int_{0}^{t} \l\langle P_n \bigg(m_n(s) \times \big(m_n(s) \times \Delta m_n(s)\big)\bigg)    ,  (-\Delta ) m_n(s) \r\rangle_{L^2} \, ds\\
			\nonumber & +   \int_{0}^{t} \l\langle P_n\big(m_n(s) \times u_n(s)\big)  ,  (-\Delta )  m_n(s)  \r\rangle_{L^2} \, ds \\
			\nonumber  & -  \alpha \, \int_{0}^{t} \psi\big(|m_n(s)|_{L^{\infty}}\big) \l\langle P_n \bigg(m_n(s) \times \big(m_n(s) \times u_n(s)\big)\bigg)  ,  (-\Delta )  m_n(s)  \r\rangle_{L^2} ds \\
			\nonumber & + \frac{1}{2} \int_{0}^{t} \l\langle \psi\big(m_n(s)\big)^2 \l[DG_n\big(m_n(s)\big)\r]\l(G_n\big(m_n(s)\big)\r) , ( -\Delta )m_n(s) \r\rangle_{L^2} \, ds \\
			\dela{
				\nonumber & +   \frac{1}{2} \bigg[ \int_{0}^{t} \l\langle P_n(P_n(m_n(s) \times h) \times h)  ,  (-\Delta )  m_n(s)  \r\rangle_{L^2} ds \\
				\nonumber & -  \alpha \, \int_{0}^{t} \psi(|m_n(s)|_{L^{\infty}})\psi(|P_n(m_n(s) \times h)|_{L^{\infty}}) \psi(|P_n(m_n(s) \times (m_n(s) \times h))|_{L^{\infty}})\centerdot \\
				\nonumber &\quad \l\langle P_n(P_n(m_n(s) \times (m_n(s) \times h)) \times h)  ,  (-\Delta )  m_n(s)  \r\rangle_{L^2} ds \\
				\nonumber & -   \alpha \, \int_{0}^{t} \psi(|m_n(s)|_{L^{\infty}}) \psi(|P_n(m_n(s) \times h)|_{L^{\infty}}) \psi(|P_n(m_n(s) \times (m_n(s) \times h))|_{L^{\infty}})\centerdot \\
				\nonumber &\quad \l\langle P_n(P_n(m_n(s) \times h) \times (m_n(s) \times h)) ,  (-\Delta )   m_n(s)  \r\rangle_{L^2} ds \\
				\nonumber  & - \alpha \, \int_{0}^{t} \l\langle P_n(m_n(s) \times (P_n(m_n(s) \times h) \times h))  ,  (-\Delta )  m_n(s)  \r\rangle_{L^2} ds\\
				\nonumber & +   \alpha^2 \int_{0}^{t} \psi^2(|m_n(s)|_{L^{\infty}}) \psi(|P_n(m_n(s) \times h)|_{L^{\infty}})^2 \psi(|P_n(m_n(s) \times (m_n(s) \times h))|_{L^{\infty}})^2 \centerdot \\
				\nonumber&\quad \l\langle P_n(P_n(m_n(s) \times (m_n(s) \times h)) \times (m_n(s) \times h)) ,  (-\Delta )  m_n(s)  \r\rangle_{L^2} ds\\
				\nonumber & +   \alpha^2 \int_{0}^{t} \l\langle P_n(m_n(s) \times (P_n(m_n(s) \times (m_n(s) \times h))\times h)) ,  (-\Delta )  m_n(s)  \r\rangle_{L^2} ds \bigg] \\
				\nonumber & +   \int_{0}^{t} \psi(|m_n(s)|_{L^{\infty}}) \psi(|P_n(m_n(s) \times h)|_{L^{\infty}}) \psi(|P_n(m_n(s) \times (m_n(s) \times h))|_{L^{\infty}})\centerdot \\
				\nonumber &\quad\l\langle P_n\l( m_n(s) \times \l(m_n(s) \times h\r)\r) ,  (-\Delta )  m_n(s)  \r\rangle_{L^2}\, dW(s) \\
				\nonumber & +   \frac{1}{2} \int_{0}^{t} \bigg| \nabla P_n\bigg(  m_n(s) \times h  \\
				\nonumber & -   \alpha \, \psi(|m_n(s)|_{L^{\infty}}) \psi(|P_n(m_n(s) \times h)|_{L^{\infty}}) \psi(|P_n\l(m_n(s) \times (m_n(s) \times h)\r)|_{L^{\infty}}) \centerdot \\
				\nonumber & \quad P_n \l( m_n(s) \times \l( m_n(s) \times h \r) \r) \bigg) \bigg|_{L^2}^2 \, ds \\}
			\nonumber & +  \int_{0}^{t} \psi\big(m_n(s)\big)\l\langle G_n\big(m_n(s)\big) ,  (-\Delta )  m_n(s)  \r\rangle_{L^2} \, dW(s)\\
			\nonumber & + \frac{1}{2}\int_{0}^{t} \l| \nabla G_n\big(m_n(s)\big) \r|_{L^2}^2 \, ds\\
			= &  \frac{1}{2}\l| m_0 \r|_{L^2}^2 + \sum_{i=1}^{7} C_i J_i(t).		\end{align}	
		Here $C_i, i=1,\dots,7$ are the constants accompanying the integrals. We show induvidual calculations for the terms.\\	
		\textbf{Calculation for $J_1$.}
		\dela{
			\begin{align*}
				\l\langle \nabla m_n(s) , \nabla P_n(m_n(s) \times \Delta m_n(s)) \r\rangle_{L^2} &= - \l\langle \Delta m_n(s) , P_n(m_n(s) \times \Delta m_n(s)) \r\rangle_{L^2} \\
				& = - \l\langle \Delta m_n(s) , m_n(s) \times \Delta m_n(s) \r\rangle_{L^2} = 0.
			\end{align*}
		}
		Let $v\in H_n$. Then
		\begin{align*}
			\l\langle P_n \l( v \times \Delta v \r) , \Delta v \r\rangle_{L^2} &= \l\langle v \times \Delta v  , P_n \Delta v \r\rangle_{L^2} \\
			&= \l\langle v \times \Delta v  ,  \Delta v \r\rangle_{L^2}\\
			&= \int_{\mathcal{O}} \l\langle v(x) \times \Delta v(x)  ,  \Delta v(x) \r\rangle_{\mathbb{R}^3} \, dx = 0.
		\end{align*}
		Therefore replacing $v$ by $m_n(s)$ and integrating gives us the following.
		\begin{align}
			\int_{0}^{t} \l\langle   P_n\big(m_n(s) \times \Delta m_n(s)\big) , \l( -\Delta \r) m_n(s) \r\rangle_{L^2} \, ds =  0.
		\end{align}
		\dela{
			Therefore for each $t\in[0,T]$,
			\begin{align}
				J_1(t) = 0
			\end{align}
		}
		Working similar to the previous calculation, we have\\
		\textbf{Calculation for $J_2$.}
		\dela{\begin{align*}
				\l\langle \nabla m_n(s) &, \nabla P_n(m_n(s) \times (m_n(s) \times \Delta m_n(s))) \r\rangle_{L^2} \\
				&= - \l\langle \Delta m_n(s) , P_n(m_n(s) \times (m_n(s) \times \Delta m_n(s))) \r\rangle_{L^2} \\
				& = - \l\langle \Delta m_n(s) , m_n(s) \times (m_n(s) \times \Delta m_n(s)) \r\rangle_{L^2} \\
				& =   \l\langle m_n(s) \times \Delta m_n(s) ,  m_n(s) \times \Delta m_n(s) \r\rangle_{L^2}\\
				& = |m_n(s) \times \Delta m_n(s)|_{L^2}^2.
		\end{align*}}
		\begin{align*}
			\int_{0}^{t}& \l\langle P_n\bigg(m_n(s) \times \big(m_n(s) \times \Delta m_n(s)\big)\bigg) , \l( - \Delta \r) m_n(s) \r\rangle_{L^2} \, ds \\
			& =   \int_{0}^{t} \l\langle m_n(s) \times \Delta m_n(s) ,  m_n(s) \times \Delta m_n(s) \r\rangle_{L^2} \, ds\\
			& = \int_{0}^{t} |m_n(s) \times \Delta m_n(s)|_{L^2}^2 \, ds.
		\end{align*}
		\dela{
			Therefore for each $t\in[0,T]$,
			\begin{align*}
				J_2(t) = \int_{0}^{t} |m_n(s) \times \Delta m_n(s)|_{L^2}^2 \, ds.
			\end{align*}
		}
		Note that in the equation \eqref{definition of solution Faedo Galerkin approximation}, the coefficient $(-\alpha)$ is negative. This along with the above equality can enable us to take this term on the left hand side after applying the It\^o Lemma.\\
		\textbf{Calculation for $J_3$.}\\
		Let $\varepsilon>0$. For $t\in[0,T]$, we have the following by H\"older's inequality followed by Young's inequality.
		\dela{\begin{align*}
				|\l\langle \nabla m_n(s) , \nabla P_n(m_n(s) \times u_n(s)) \r\rangle_{L^2}| & = |- \l\langle \Delta m_n(s) , P_n(m_n(s) \times u_n(s)) \r\rangle_{L^2}| \\
				& = |- \l\langle \Delta m_n(s) , m_n(s) \times u_n(s) \r\rangle_{L^2}| \\
				& = | \l\langle m_n(s) \times \Delta m_n(s) , u_n(s) \r\rangle_{L^2}| \\
				& \leq \frac{\varepsilon}{2} | m_n(s) \times \Delta m_n(s)|_{L^2}^2 + \frac{C(\varepsilon)}{2} |u_n(s)|_{L^2}^2. \\
				& \text{(By Young's $\varepsilon$ inequality)}
			\end{align*}
			\begin{align*}
				|\l\langle \nabla m_n(s) , \nabla P_n(m_n(s) \times u_n(s)) \r\rangle_{L^2}| & \leq |- \l\langle \Delta m_n(s) , P_n(m_n(s) \times u_n(s)) \r\rangle_{L^2}|
			\end{align*}
			We choose $\varepsilon > 0$ later. The last inequality uses Young's $\varepsilon$ inequality.\\
			Thus for $t\in[0,T]$,}
		\dela{
			\begin{align*}
				J_3(t) = & \int_{0}^{t} \l|\l\langle P_n(m_n(s) \times u_n(s)) , \l( - \Delta \r) m_n(s)  \r\rangle_{L^2}\r| \, ds \\
				\nonumber = &\int_{0}^{t} \l|\l\langle \nabla m_n(s) , \nabla P_n(m_n(s) \times u_n(s)) \r\rangle_{L^2}\r| \, ds \\
				\leq & \frac{\varepsilon}{2} \int_{0}^{t} | m_n(s) \times \Delta m_n(s)|_{L^2}^2 \, ds + \frac{C(\varepsilon)}{2} \int_{0}^{t} |u_n(s)|_{L^2}^2 \, ds.
			\end{align*}	
		}
		\begin{align*}
			\l| J_3(t) \r| = & \l| \int_{0}^{t} \l\langle P_n\big(m_n(s) \times u_n(s)\big) , \l( - \Delta \r) m_n(s)  \r\rangle_{L^2} \, ds \r| \\
			\leq & \int_{0}^{t} \l|\l\langle P_n\big(m_n(s) \times u_n(s)\big) , \l( - \Delta \r) m_n(s)  \r\rangle_{L^2}\r| \, ds \\
			\nonumber =& \int_{0}^{t} \l|\l\langle m_n(s) \times u_n(s) , P_n \Delta  m_n(s)  \r\rangle_{L^2}\r| \, ds \\
			\nonumber =& \int_{0}^{t} \l|\l\langle m_n(s) \times u_n(s) ,  \Delta  m_n(s)  \r\rangle_{L^2}\r| \, ds \\
			\nonumber =& \int_{0}^{t} \l|\l\langle u_n(s) , m_n(s) \times  \Delta  m_n(s)  \r\rangle_{L^2}\r| \, ds \\
			\leq & \frac{\varepsilon}{2} \int_{0}^{t} | m_n(s) \times \Delta m_n(s)|_{L^2}^2 \, ds + \frac{C(\varepsilon)}{2} \int_{0}^{t} |u_n(s)|_{L^2}^2 \, ds.
		\end{align*}
		\textbf{Calculation for $J_4$.}\\
		Working similar to the above calculation, there exists a constant $C(\varepsilon)$ such that for each $t\in[0,T]$,
		
		\begin{align*}
			\dela{\int_{0}^{t} |\psi(|m_n(s)|_{L^{\infty}})\l\langle \l( - \Delta \r) m_n(s) &, \nabla P_n(m_n(s) \times (m_n(s) \times u_n(s))) \r\rangle_{L^2}| \, ds \\}
			& \int_{0}^{t} | \l\langle  P_n\bigg(m_n(s) \times \big(m_n(s) \times u_n(s)\big)\bigg) \psi\big(m_n(s)\big) , \Delta m_n(s) \r\rangle_{L^2}|  \, ds\\
			& \leq \int_{0}^{t} | \l\langle  m_n(s) \times \big(m_n(s) \times u_n(s)\big) \psi(m_n(s)) , \Delta m_n(s) \r\rangle_{L^2}| \, ds \\
			& \leq \int_{0}^{t} | \l\langle m_n(s) \times  \Delta m_n(s), (m_n(s) \times u_n(s))) \psi\big(m_n(s)\big) \r\rangle_{L^2}| \, ds \\
			&\leq  \frac{\varepsilon}{2} \frac{C(\varepsilon)}{2} \int_{0}^{t} | m_n(s) \times \Delta m_n(s)|_{L^2}^2 \, ds + \int_{0}^{t} \psi(m_n(s))^2 |m_n(s) \times u_n(s)|_{L^2}^2 \, ds \\
			& \leq \frac{\varepsilon}{2} \int_{0}^{t} | m_n(s) \times \Delta m_n(s)|_{L^2}^2 \, ds + \frac{C(\varepsilon)C(h)}{2} \int_{0}^{t} |u_n(s)|_{L^2}^2 \, ds.
		\end{align*}
		\dela{
			\begin{align*}
				|\psi(|m_n(s)|_{L^{\infty}})\l\langle \nabla m_n(s) &, \nabla P_n(m_n(s) \times (m_n(s) \times u_n(s))) \r\rangle_{L^2}| \\
				& \leq | \l\langle \Delta m_n(s), P_n(m_n(s) \times (m_n(s) \times u_n(s))) \psi(|m_n(s)|_{L^{\infty}}) \r\rangle_{L^2}| \\
				& \leq \frac{\varepsilon}{2} | m_n(s) \times \Delta m_n(s)|_{L^2}^2 + \frac{C(\varepsilon)C(h)}{2} |u_n(s)|_{L^2}^2.
			\end{align*}
		}
		The second last inequality follows from Young's inequality.
		We observe that
		\begin{equation*}
			\l|m_n(s) \times u_n(s)\r|_{L^2}^2 \leq \l|m_n(s)\r|_{L^{\infty}}^2 \l|u_n(s)\r|_{L^2}^2.
		\end{equation*}
		Also by the definition of the bump function (cut-off) $\psi$ in \eqref{definition of bump function}, we have
		\begin{equation*}
			\psi(m_n(s))^2 \l|m_n(s)\r|_{L^{\infty}}^2 \leq \l(\l|h\r|_{L^{\infty}} + 2\r)^2.
		\end{equation*}
		\dela{
			The second inequality follows from the above mentioned observations.
			Thus for $t\in[0,T]$,
			
			\begin{align*}
				J_4(t) = &\int_{0}^{t} |\l\langle \nabla m_n(s) , \nabla P_n(m_n(s) \times(m_n(s) \times u_n(s))) \r\rangle_{L^2}| ds \\
				\leq & \frac{\varepsilon}{2} \int_{0}^{t} | m_n(s) \times \Delta m_n(s)|_{L^2}^2 ds  \\
				+ & \frac{C(\varepsilon)C(h)}{2} \int_{0}^{t} |u_n(s)|_{L^2}^2 ds.
			\end{align*}
		}
		\textbf{Calculation for $J_5$.}\\
		\dela{
			Observe that the term $V_n$ is a polynomial of order 3 in $m_n$. The idea to estimate $J_5$ is the following. Let $v_1,v_2\in H_n$.
			\begin{equation*}
				\l\langle \l(-\Delta\r) v_1 , v_2 \r\rangle_{L^2} =  \l\langle \nabla v_1 , \nabla v_2 \r\rangle_{L^2} .
			\end{equation*}
			Therefore
			\begin{align*}
				\int_{0}^{t}  \l\langle V_n(s) , \l(-\Delta\r) m_n(s) \r\rangle_{L^2} \, ds = \int_{0}^{t} \l\langle \nabla V_n , \nabla m_n(s) \r\rangle_{L^2} \, ds.
			\end{align*}
			By using the product rule for derivatives, the term $\nabla V_n$ can be split into terms of two types. The first type, wherein the derivative is on $m_n$ and the second type where the derivative is on the term $h$. For the first type, we use H\"older's inequality with $L^2$ norm on the gradient term and $L^{\infty}$ norm on the remaining term/s. For the second type, the $L^2$ norm is applied on one of the terms containing $m_n$, while all other terms get the $L^{\infty}$ norm. The cut-off function $\Psi$ ensures that the $L^{\infty}$ norm is taken care of.
		}
		\dela{Since $h\in W^{1,\infty}$, there exists a constant $C>0$ which is independent of $n$ such that
			\begin{equation}
				\l|J_5(t)\r| \leq \int_{0}^{t} \l| \l\langle \nabla V_n(s) , \nabla m_n(s) \r\rangle_{L^2}\r| \, ds
				\leq C \int_{0}^{t} \l| m_n(s) \r|_{H^1}^2 \, ds
		\end{equation}}
		By Lemma \ref{Technical lemma bound on DGn Gn}, we have the following.
		\begin{align*}
			\l|J_5(t)\r| \leq &  \int_{0}^{t} \l| \l\langle \psi\big(m_n(s)\big)^2 \, \l[DG_n\big(m_n(s)\big)\r]\l[G_n\big(m_n(s)\big)\r] , \Delta m_n(s) \r\rangle_{L^2} \r| \, ds \\
			\leq & C \int_{0}^{t} \l( 1 + \l|  m_n(s) \r|_{H^1}\r)\l|  m_n(s) \r|_{H^1}  \, ds.
		\end{align*}
		Since $T<\infty$, the term on the right hand side of the above inequality can be replaced by square of $\l|  m_n(s) \r|_{H^1}$. That is, there exists a constants $C_1,C_2>0$ (which may depend on $T$, but not on $n\in\mathbb{N}$) such that
		\begin{align}
			\l|J_5(t)\r| \leq C_1 + C_2\int_{0}^{t} \l|  m_n(s) \r|_{H^1}^2  \, ds.
		\end{align}
		
		\dela{
			\textbf{Calculation for $J_5$.}\\
			By H\"older's inequality along with the continuous embedding $H^1\hookrightarrow L^{\infty}$, there exists a constant $C(h)>0$ such that for each $t\in[0,T]$,
			\dela{\begin{align*}
					|\l\langle \nabla m_n(s) ,& \nabla P_n(P_n(m_n(s) \times h) \times h) \r\rangle_{L^2}| \\
					&= 	|\l\langle \nabla m_n(s) , \nabla P_n(m_n(s) \times h) \times h \r\rangle_{L^2}| \\
					&\leq |\l\langle \nabla m_n(s) ,  \nabla P_n( m_n(s) \times h) \times h \r\rangle_{L^2}| \\
					&+ |\l\langle \nabla m_n(s) ,  P_n(m_n(s) \times  h) \times \nabla h \r\rangle_{L^2}| \\
					& \leq |\nabla m_n(s)|_{L^2} |\nabla P_n( m_n(s) \times h)|_{L^2} |h|_{L^{\infty}} \\
					&+ |\nabla m_n(s)|_{L^2} |P_n( m_n(s) \times h)|_{L^{\infty}} |\nabla h|_{L^2}\ \text{By H\"older's inequality}\\
					&\leq  |\nabla m_n(s)|_{L^2} |P_n \nabla ( m_n(s) \times h)|_{L^2} |h|_{L^{\infty}} \\
					&+ C |\nabla m_n(s)|_{L^2} |P_n( m_n(s) \times h)|_{H^1} |\nabla h|_{L^2} \ \text{By}\ H^1\hookrightarrow L^{\infty}\\
					&\leq  |\nabla m_n(s)|_{L^2} |\nabla ( m_n(s) \times h)|_{L^2} |h|_{L^{\infty}}
					(\text{By}\ \l|P_n \centerdot\r|_{L^2} \leq \l|\centerdot\r|_{L^2}) \\
					&+ C |\nabla m_n(s)|_{L^2} | m_n(s) \times h|_{H^1} |\nabla h|_{L^2}  \ \text{By}\ \l|P_n \centerdot\r|_{H^1} \leq \l|\centerdot\r|_{H^1} \\
					& \leq C(h)  |m_n(s)|_{H^1}^2.
				\end{align*}
				Thus for $t\in[0,T]$, we get}
			\begin{align*}
				J_5(t) = \int_{0}^{t} |\l\langle \nabla m_n(s) , \nabla P_n(P_n(m_n(s) \times h) \times h) \r\rangle_{L^2}| ds \leq C(h) \int_{0}^{t} |m_n(s)|_{H^1}^2 ds.
			\end{align*}
			Similarly, \\
			\textbf{Calculation for $J_6$.}
			\dela{\begin{align*}
					& |\l\langle \nabla m_n(s) , \psi(|m_n(s)|_{L^{\infty}}) \psi(\l|P_n(m_n(s) \times h)\r|_{L^{\infty}}) \psi(\l|P_n(m_n(s) \times (m_n(s) \times h))\r|_{L^{\infty}}) \centerdot \\
					& \nabla P_n(P_n(m_n(s) \times (m_n(s) \times h)) \times h) \r\rangle_{L^2} |  \\
					= & |\l\langle \nabla m_n(s) , \psi(|m_n(s)|_{L^{\infty}}) \psi(|P_n(m_n(s) \times h)|_{L^{\infty}}) \psi(|P_n(m_n(s) \times (m_n(s) \times h))|_{L^{\infty}})  \centerdot \\
					& \nabla (P_n(m_n(s) \times (m_n(s) \times h)) \times h) \r\rangle_{L^2}|\\
					\leq &  |\l\langle \nabla m_n(s) , \psi(|m_n(s)|_{L^{\infty}}) \psi(|P_n(m_n(s) \times h)|_{L^{\infty}}) \psi(|P_n(m_n(s) \times (m_n(s) \times h))|_{L^{\infty}}) \centerdot \\
					& \quad \nabla P_n(m_n(s) \times (m_n(s) \times h)) \times h \r\rangle_{L^2}|\\
					+ & |\l\langle \nabla m_n(s) , \psi(|m_n(s)|_{L^{\infty}}) \psi(|P_n(m_n(s) \times h)|_{L^{\infty}}) \psi(|P_n(m_n(s) \times (m_n(s) \times h))|_{L^{\infty}}) \centerdot \\
					& \quad P_n(m_n(s) \times (m_n(s) \times h)) \times \nabla h \r\rangle_{L^2}| \\
					\leq &  |\nabla m_n(s)|_{L^2}|\nabla P_n(m_n(s) \times (m_n(s) \times h))|_{L^2} \psi(|m_n(s)|_{L^{\infty}}) |h|_{L^{\infty}} \psi(|P_n(m_n(s) \times h)|_{L^{\infty}}) \centerdot \\
					&  \psi(|P_n(m_n(s) \times (m_n(s) \times h))|_{L^{\infty}}) \\
					+ &  |\nabla m_n(s)|_{L^2}||P_n(m_n(s) \times (m_n(s) \times h))|_{L^{\infty}}|\nabla h|_{L^2}  \psi(|m_n(s)|_{L^{\infty}}) \psi(|P_n(m_n(s) \times h)|_{L^{\infty}})  \centerdot \\
					& \psi(|P_n(m_n(s) \times (m_n(s) \times h))|_{L^{\infty}}) \\
					\leq &  |\nabla m_n(s)|_{L^2}^2|m_n(s)|_{L^{\infty}}|h|_{L^{\infty}}\psi(|m_n(s)|_{L^{\infty}}) \psi(|P_n(m_n(s) \times h)|_{L^{\infty}})\centerdot \\
					& \quad  \psi(|P_n(m_n(s) \times (m_n(s) \times h))|_{L^{\infty}})\\
					+ &  |\nabla m_n(s)|_{L^2}||P_n(m_n(s) \times (m_n(s) \times h))|_{L^{\infty}}|\nabla h|_{L^2} |h|_{L^{\infty}} \psi(|m_n(s)|_{L^{\infty}})\centerdot \\
					& \psi(|P_n(m_n(s) \times h)|_{L^{\infty}}) \psi(|P_n(m_n(s) \times (m_n(s) \times h))|_{L^{\infty}}) \\
					\leq & C(h)  |m_n(s)|_{H^1}^2.
				\end{align*}
				
			}
			\begin{align*}
				& |\l\langle \nabla m_n(s) , \psi(|m_n(s)|_{L^{\infty}}) \psi(\l|P_n(m_n(s) \times h)\r|_{L^{\infty}}) \psi(\l|P_n(m_n(s) \times (m_n(s) \times h))\r|_{L^{\infty}}) \centerdot \\
				& \nabla P_n(P_n(m_n(s) \times (m_n(s) \times h)) \times h) \r\rangle_{L^2} |  \\
				\leq & C(h)  |m_n(s)|_{H^1}^2.
			\end{align*}
			Here the $\centerdot$ in the first line represents multiplication of the two quantities on the first two lines. The reason for this is that the whole expression was not fitting in the same line.
			\dela{If the collection of equations does not fit on one page then it shifts to the next page and the preceeding equations are "expanded" so that it fills the entire previous page.}

			Hence for $t\in[0,T]$,
			\begin{align*}
				\int_{0}^{t} |\l\langle \nabla m_n(s) &, \psi(|m_n(s)|_{L^{\infty}}) \psi(|P_n(m_n(s) \times h)|_{L^{\infty}}) \psi(|P_n(m_n(s) \times (m_n(s) \times h))|_{L^{\infty}}) \centerdot \\
				& \nabla P_n(P_n(m_n(s) \times (m_n(s) \times h)) \times h) \r\rangle_{L^2}|\, ds \\
				&\leq CC(h) \int_{0}^{t} |m_n(s)|_{H^1}^2 ds.
			\end{align*}
			The calculations for $J_7$ can be done similarly.\\
			\textbf{Calculation for $J_8$.}\\
			Working on the lines of the calculations for $J_3,J_4$, we can show that for $\varepsilon>0$, there exists a constants $C(\varepsilon),C(h)$ such that
			\dela{\begin{align*}
					|\l\langle \nabla m_n(s) &, \nabla P_n(m_n(s) \times (P_n(m_n(s) \times h) \times h))\r\rangle_{L^2}| \\
					&= |\l\langle \nabla m_n(s) , \nabla m_n(s) \times (P_n(m_n(s) \times h) \times h) \r\rangle_{L^2}| \\
					& = |\l\langle \Delta m_n(s) ,  m_n(s) \times (P_n(m_n(s) \times h) \times h) \r\rangle_{L^2}| \\
					& = |\l\langle  m_n(s) \times \Delta m_n(s) ,  P_n(m_n(s) \times h) \times h \r\rangle_{L^2}| \\
					& \leq \frac{\varepsilon}{2} |m_n(s) \times \Delta m_n(s)|_{L^2}^2 + \frac{C(\varepsilon)}{2}|P_n(m_n(s) \times h) \times h|_{L^2}^2\\
					& \leq  \frac{\varepsilon}{2} |m_n(s) \times \Delta m_n(s)|_{L^2}^2 + \frac{C(\varepsilon)}{2}|P_n(m_n(s)\times h)|_{L^2}^2  |h|_{L^{\infty}}\\
					& \leq \frac{\varepsilon}{2} |m_n(s) \times \Delta m_n(s)|_{L^2}^2 + \frac{C(\varepsilon)}{2}|m_n(s)\times h|_{L^2}^2  |h|_{L^{\infty}}\ \text{Since}\ \l|P_n \centerdot\r|_{L^2} \leq \l|\centerdot\r|_{L^2}\\
					& \leq \frac{\varepsilon}{2} |m_n(s) \times \Delta m_n(s)|_{L^2}^2 + \frac{C(\varepsilon)}{2}|m_n(s)|_{L^2}^2  |h|^3_{L^{\infty}} \\
					& \leq \frac{\varepsilon}{2} |m_n(s) \times \Delta m_n(s)|_{L^2}^2 + \frac{C(\varepsilon)C(h)}{2}|m_n(s)|_{L^2}^2.
				\end{align*}
				Thus,
			}
			\begin{align*}
				\int_{0}^{t} |\l\langle \nabla m_n(s) &, \nabla P_n(m_n(s) \times (P_n(m_n(s) \times h) \times h))\r\rangle_{L^2}| ds \\
				&\leq \frac{\varepsilon}{2} \int_{0}^{t} |m_n(s) \times \Delta m_n(s)|_{L^2}^2 ds + \frac{C(\varepsilon)C(h)}{2} \int_{0}^{t} |m_n(s)|_{L^2}^2 ds.
			\end{align*}
			A combination of H\"older's inequality and Young's inequality along gives us the following set of inequalities (for $J_9,J_{10}$) for each $t\in[0,T]$.\\
			\textbf{Calculation for $J_9$.}
			\begin{align*}
				|\l\langle \nabla m_n(s) &, \psi(|m_n(s)|_{L^{\infty}}) \psi(|P_n(m_n(s) \times h)|_{L^{\infty}}) \centerdot \\
				& \nabla (P_n(P_n(m_n(s) \times (m_n(s) \times h))\times (m_n(s) \times h))) \r\rangle_{L^2}| \\
				= & |\l\langle \nabla m_n(s) , \psi(|m_n(s)|_{L^{\infty}}) \psi(|P_n(m_n(s) \times h)|_{L^{\infty}}) \centerdot \\
				&  \nabla (P_n(m_n(s) \times (m_n(s) \times h))\times (m_n(s) \times h)) \r\rangle_{L^2}| \\
				\leq &  |\l\langle \nabla m_n(s) , \psi(|m_n(s)|_{L^{\infty}}) \psi(|P_n(m_n(s) \times h)|_{L^{\infty}}) \centerdot \\
				&  \nabla P_n(m_n(s) \times (m_n(s) \times h))\times (m_n(s) \times h) \r\rangle_{L^2}| \\
				+&  |\l\langle \nabla m_n(s) , \psi(|m_n(s)|_{L^{\infty}}) \psi(|P_n(m_n(s) \times h)|_{L^{\infty}})  \centerdot \\
				& P_n(m_n(s) \times (m_n(s) \times h))\times  \nabla(m_n(s) \times h) \r\rangle_{L^2}|\\
				\leq &  |\nabla m_n(s)|_{L^2}  |\nabla P_n(m_n(s) \times (m_n(s) \times h))|_{L^2}  \centerdot \\
				& \psi(|m_n(s)|_{L^{\infty}}) \psi(|P_n(m_n(s) \times h)|_{L^{\infty}}) |m_n(s) \times h|_{L^{\infty}}\\
				+&  |\nabla m_n(s)|_{L^2} | P_n(m_n(s) \times (m_n(s) \times h))|_{L^{\infty}}  \centerdot \\
				& \psi(|m_n(s)|_{L^{\infty}}) \psi(|P_n(m_n(s) \times h)|_{L^{\infty}}) |\nabla (m_n(s) \times h)|_{L^2} \\
				\leq &  C(h) |m_n(s)|^2_{H^1}.
			\end{align*}
			Thus for $t\in[0,T]$,
			\begin{align*}
				J_9(t) & = \int_{0}^{t} |\l\langle \nabla m_n(s) , \psi(|m_n(s)|_{L^{\infty}}) \nabla (P_n(P_n(m_n(s) \times (m_n(s) \times h))\times (m_n(s) \times h))) \r\rangle_{L^2}| ds \\
				&\leq C(h) \int_{0}^{t} |m_n(s)|_{H^1}^2 ds.
			\end{align*}
			\textbf{Calculation for $J_{10}$.}
			\dela{\begin{align*}
					|\l\langle \nabla m_n(s) &, \psi(|m_n(s)|_{L^{\infty}}) \psi(|P_n(m_n(s) \times h)|_{L^{\infty}})\psi(|P_n(m_n(s) \times (m_n(s) \times h))|_{L^{\infty}})\centerdot \\
					&\quad \nabla P_n(m_n(s) \times P_n(m_n(s) \times (m_n(s) \times h)) \times h)  \r\rangle_{L^2}| \\
					& \leq 	|\l\langle \nabla m_n(s) , \psi(|m_n(s)|_{L^{\infty}}) \nabla (m_n(s) \times P_n(m_n(s) \times (m_n(s) \times h)) \times h) \r\rangle_{L^2}|\\
					& \leq |\l\langle \nabla m_n(s) , \psi(|m_n(s)|_{L^{\infty}}) \nabla m_n(s) \times P_n(m_n(s) \times (m_n(s) \times h)) \times h \r\rangle_{L^2}| \\
					&\quad + |\l\langle \nabla m_n(s) , \psi(|m_n(s)|_{L^{\infty}})  (m_n(s) \times \nabla P_n(m_n(s) \times (m_n(s) \times h)) \times h) \r\rangle_{L^2}| \\
					& \leq |\l\langle \nabla m_n(s) , \psi(|m_n(s)|_{L^{\infty}})  (m_n(s) \times \nabla P_n(m_n(s) \times (m_n(s) \times h)) \times h) \r\rangle_{L^2}| \\
					& \leq |\nabla m_n(s)|_{L^2} \psi(|m_n(s)|_{L^{\infty}}) |m_n(s)|_{L^{\infty}} |\nabla P_n(m_n(s) \times (m_n(s) \times h)) \times h|_{L^2} \\
					& \leq |\nabla m_n(s)|_{L^2} \psi(|m_n(s)|_{L^{\infty}}) |m_n(s)|_{L^{\infty}}\big[|m_n(s)|_{L^{\infty}} 2|\nabla m_n(s)|_{L^2}  |h|_{L^{\infty}} \\
					& \quad+ |m_n(s)|_{L^{\infty}}^2|\nabla h|_{L^2}\big] \\
					& \leq 2 |\nabla m_n(s)|^2_{L^2} \psi(|m_n(s)|_{L^{\infty}}) |m_n(s)|^2_{L^{\infty}} |h|_{L^{\infty}} \\
					&\quad + |\nabla m_n(s)|_{L^2} \psi(|m_n(s)|_{L^{\infty}}) |m_n(s)|_{L^{\infty}}^3|\nabla h|_{L^2} \\
					& \leq 2 |m_n(s)|^2_{H^1} \psi(|m_n(s)|_{L^{\infty}}) |m_n(s)|^2_{L^{\infty}} |h|_{L^{\infty}} \\
					&\quad+ |m_n(s)|_{H^1} C\l|m_n(s)\r|_{H^1} \psi(|m_n(s)|_{L^{\infty}}) |m_n(s)|_{L^{\infty}}^2|\nabla h|_{L^2} \\
					& \leq C(h)C |m_n(s)|_{H^1}^2.
				\end{align*}
				Here the $\centerdot$ in the first line represents multiplication of the two quantities on the first two lines. The reason for this is that the whole expression was not fitting in the same line.
				
				The second last inequality uses the following inequalities:
				\begin{align*}
					\l|\nabla m_n(s)\r|_{L^2} \leq \l|m_n(s)\r|_{H^1}.
				\end{align*}
				The embedding $H^1 \hookrightarrow L^{\infty}$ is continuous. Hence there exists a constant $C>0$ such that
				\begin{align*}
					\l|\cdot\r|_{L^{\infty}} \leq C \l|\cdot\r|_{H^1}.
				\end{align*}
				For obtaining the last inequality, we observe that by the definition \eqref{definition of bump function} of $\psi$,
				\begin{align*}
					\psi(|m_n(s)|_{L^{\infty}}) |m_n(s)|_{L^{\infty}}^2 \leq \l(\l| h \r|_{L^{\infty}} + 2\r)^2.
				\end{align*}
				Thus for any $t\in[0,T]$, we get
			}
			\begin{align}
				\nonumber J_{10}(t)& = \int_{0}^{t} |\l\langle \nabla m_n(s) , \psi(|m_n(s)|_{L^{\infty}}) \psi(|P_n(m_n(s) \times h)|_{L^{\infty}})\psi(|P_n(m_n(s) \times (m_n(s) \times h))|_{L^{\infty}}) \centerdot \\
				\nonumber & \quad \nabla P_n(m_n(s) \times P_n(m_n(s) \times (m_n(s) \times h)) \times h)  \r\rangle_{L^2}| \, ds \\
				& \leq  C(h) \int_{0}^{t} |m_n(s)|_{H^1}^2 \, ds .
			\end{align}
		}

		The integral $J_{7}$ can be bounded similarly. What remain now are the terms that constitute the noise.\\
		\textbf{Calculations related to the noise term $J_{6}$.}
		The idea for bounding these terms is to use the Burkholder-Davis-Gundy inequality. With that in view, we present some calculations that will be required when bounding the terms (more precisely their expectation)\dela{ $J_{11}$ and $J_{12}$}.
		
		The map $G_n$ can be expressed as a sum of two maps $G_n^1,G_n^2$ on $H_n$.
		\begin{equation}
			G_n^1(v) = P_n(v \times h),
		\end{equation}
		\begin{equation}
			G_n^2(v) = P_n\bigl(v \times \l( v \times h\r)\bigr).
		\end{equation}
		\begin{equation}
			G_n = G_n^1 - \alpha \, G_n^2
		\end{equation}		
		Now,
		\begin{align*}
			J_6(t) = & \int_{0}^{t} \psi(m_n(s))\l\langle G_n\big(m_n(s)\big), \l( - \Delta \r) m_n(s) \r\rangle_{L^2} \, dW(s) \\
			= & \int_{0}^{t} \psi\big(m_n(s)\big)\l\langle G_n^1\big(m_n(s)\big) - \alpha \, G_n^2\big(m_n(s)\big) , \l( - \Delta \r) m_n(s) \r\rangle_{L^2} \, dW(s) \\
			= & \int_{0}^{t} \psi\big(m_n(s)\big)\l\langle G_n^1\big(m_n(s)\big) , \l( - \Delta \r) m_n(s) \r\rangle_{L^2} \, dW(s) \\
			& - \alpha \,   \int_{0}^{t} \psi\big(m_n(s)\big)\l\langle  G_n^2\big(m_n(s)\big) , \l( - \Delta \r) m_n(s) \r\rangle_{L^2} \, dW(s) \\
			=& \int_{0}^{t} \psi\big(m_n(s)\big) \l\langle P_n \l(m_n(s) \times h \r) , \l( - \Delta \r) m_n(s) \r\rangle_{L^2} \, dW(s) \\
			& - \alpha \,   \int_{0}^{t} \psi\big(m_n(s)\big)\l\langle  P_n\bigl(m_n(s) \times \l(m_n(s) \times h \r)\bigr) , \l( - \Delta \r) m_n(s) \r\rangle_{L^2} \, dW(s) \\
			=& \int_{0}^{t} \psi\big(m_n(s)\big) \l\langle m_n(s) \times h  , \l( - \Delta \r) m_n(s) \r\rangle_{L^2} \, dW(s) \\
			& - \alpha \,   \int_{0}^{t} \psi\big(m_n(s)\big)\l\langle  m_n(s) \times \l(m_n(s) \times h \r) , \l( - \Delta \r) m_n(s) \r\rangle_{L^2} \, dW(s) \\
			=& \int_{0}^{t} \psi\big(m_n(s)\big) \l\langle \nabla\l( m_n(s) \times h\r)  , \nabla  m_n(s) \r\rangle_{L^2} \, dW(s) \\
			& - \alpha \,   \int_{0}^{t} \psi\big(m_n(s)\big)\l\langle  \nabla \l[m_n(s) \times \bigl(m_n(s) \times h \bigr) \r] , \nabla m_n(s) \r\rangle_{L^2} \, dW(s) .
		\end{align*}
		The plan now is to apply the Burkholder-Davis-Gundy inequality. Prior to that, we will establish the following estimates: By the product rule for derivatives, followed by the use of H\"older's inequality, we get a constant $C(h)>0$ such that for any $t\in[0,T]$,
		\dela{\begin{align*}
				|\l\langle \nabla m_n(s) , \nabla P_n(m_n(s) \times h) \r\rangle_{L^2}| & = |\l\langle \nabla m_n(s) , \nabla (m_n(s) \times h) \r\rangle_{L^2}| \ \text{By}\ \eqref{projection operator is self-adjoint} , \nabla P_n = P_n \nabla\\
				& \leq |\l\langle \nabla m_n(s) , \nabla m_n(s) \times h \r\rangle_{L^2}| + |\l\langle \nabla m_n(s) , m_n(s) \times \nabla h \r\rangle_{L^2}| \\
				& \leq |h|_{L^{\infty}} |\nabla m_n(s)|_{L^2}^2 + |\nabla m_n(s)|_{L^2}\l|m_n(s)\r|_{L^2}^2|\nabla h|_{L^2}^2 \\
				&\quad \text{(By H\"older's inequality)} \\
				&\leq C(h) |m_n(s)|_{H^1}^2 + C |m_n(s)|_{H^1}\ \text{By}\ \eqref{bound on m_n}
			\end{align*}
			Since $a \leq 1 + a^2$, for $a\in\mathbb{R}^+$, there exists constants $C,C^\p >0$ such that for $t\in[0,T]$,
			\begin{align*}
				C^\p \int_{0}^{t} |m_n(s)|_{H^1}^2 + |m_n(s)|_{H^1}\, ds \leq C\l(1 +  \int_{0}^{t} |m_n(s)|_{H^1}^2\, ds\r).
			\end{align*}
		}
		\dela{\begin{align*}
				|\l\langle \nabla m_n(s) , \nabla P_n(m_n(s) \times h) \r\rangle_{L^2}| \leq C(h) |m_n(s)|_{H^1}^2
			\end{align*}
			Hence
		}
		\begin{align*}
			\int_{0}^{t}  \bigl|\psi(m_n(s))\l\langle \nabla m_n(s) , \nabla (m_n(s) \times h) \r\rangle_{L^2}\bigr|^2\, ds  \leq C(h) \int_{0}^{t} |m_n(s)|_{H^1}^4\, ds.
		\end{align*}
		Similarly,
		\dela{\begin{align*}
				|\l\langle \nabla m_n(s) ,& \psi(|m_n(s)|_{L^{\infty}}) \psi(|P_n(m_n(s) \times h)|_{L^{\infty}})\psi(|P_n(m_n(s) \times (m_n(s) \times h))|_{L^{\infty}}) \centerdot \\
				&\nabla P_n(m_n(s) \times (m_n(s) \times h)) \r\rangle_{L^2}| \\
				\leq & |\l\langle \nabla m_n(s) , \psi(|m_n(s)|_{L^{\infty}}) \nabla (m_n(s) \times (m_n(s) \times h)) \r\rangle_{L^2}|\ ( \text{Since}\ \l|\psi(x)\r| \leq 1, x\in \mathbb{R}) \\
				\leq & |\l\langle \nabla m_n(s) , \psi(|m_n(s)|_{L^{\infty}}) \nabla m_n(s) \times (m_n(s) \times h) \r\rangle_{L^2}| \\
				&+ |\l\langle \nabla m_n(s) ,  \psi(|m_n(s)|_{L^{\infty}}) m_n(s) \times (\nabla m_n(s) \times h) \r\rangle_{L^2}| \\
				&+  \l|\l\langle \nabla m_n(s) ,  \psi(|m_n(s)|_{L^{\infty}}) m_n(s) \times (m_n(s) \times \nabla h) \r\rangle_{L^2}\r| \\
				\leq &  C(h)  |m_n(s)|_{H^1}^2.
			\end{align*}
			
			\begin{align*}
				|\psi(m_n(s)) \l\langle \nabla m_n(s) ,& \psi(|m_n(s)|_{L^{\infty}}) \psi(|P_n(m_n(s) \times h)|_{L^{\infty}})\psi(|P_n(m_n(s) \times (m_n(s) \times h))|_{L^{\infty}}) \centerdot \\
				&\nabla P_n(m_n(s) \times (m_n(s) \times h)) \r\rangle_{L^2}| \\
				\leq &  C(h)  |m_n(s)|_{H^1}^2.
			\end{align*}
		}
		The following inequality holds for $t\in[0,T]$.
		\begin{align*}
			\int_{0}^{t} | \psi\bigl(m_n(s)\bigr) \langle \nabla m_n(s) ,&  \nabla \bigl[m_n(s) \times \bigl(m_n(s) \times h\bigr) \bigr] \rangle_{L^2} |^2 \, ds \\
			& \leq C(h) \int_{0}^{t} \l| m_n(s)\r|_{H^1}^4 \, ds.
		\end{align*}
		The details for these calculations are similar to the proof of Lemma \ref{Technical lemma bound on DGn Gn}.\\
		Effectively, we have shown that there exists a constant $C>0$ such that
		\begin{align}
			\l|\int_{0}^{t} \l\langle \psi\big(m_n(s)\big) G_n\big(m_n(s)\big) , \Delta m_n(s) \r\rangle_{L^2}^2 \, ds\r| \leq C(h) \int_{0}^{t} | m_n(s)|_{H^1}^4 \, ds.
		\end{align}
		Let $\varepsilon>0$. By the Burkholder-Davis-Gundy inequality, see Lemma \ref{BDG inequality upper bound}, we deduce
		\begin{align*}
			&\mathbb{E}\sup_{t\in[0,T]} \l| \int_{0}^{t} \l\langle \psi\big(m_n(s)\big) G_n\big(m_n(s)\big) , \Delta m_n(s) \r\rangle_{L^2}^2 \, dW(s) \r| \\
			\leq & C \mathbb{E} \l(\int_{0}^{t} \l\langle \psi\big(m_n(s)\big) G_n\big(m_n(s)\big) , \Delta m_n(s) \r\rangle_{L^2}^2 \, ds\r)^\frac{1}{2} \\
			\leq & C(h) \mathbb{E} \l(\int_{0}^{T} |\nabla m_n(s)|_{H^1}^4 \, ds\r)^\frac{1}{2}\\
			\leq &  C(h) \mathbb{E} \l[ \int_{0}^{T} \l|m_n(s)\r|_{H^1}^2  \l|m_n(s)\r|_{H^1}^2 \, ds \r]^{\frac{1}{2}} \\
			\leq &  C(h) \mathbb{E}\l[ \l(\sup_{t\in [0,T]}\l|m_n(t)\r|_{H^1}^2\r)^{\frac{1}{2}}\l( \int_{0}^{T} \l|m_n(s)\r|_{H^1}^2   \, ds \r)^{\frac{1}{2}} \r]  (\text{By the Cauchy-Schwartz inequality}) \\
			\leq &  \frac{\varepsilon }{2}\mathbb{E} \l[ \sup_{t\in [0,T]} \l|m_n(t)\r|_{H^1}^{2}   \r] + \frac{C(\varepsilon)C(h)^2}{2} \mathbb{E} \l[ \int_{0}^{T} \l|m_n(s)\r|_{H^1}^{2}  \, ds \r]. \text{(By the Young's inequality)}
		\end{align*}
		\dela{
			\adda{Why is this required here??}Similarly,
			\begin{align*}
				\int_{0}^{t} \l\langle \nabla m_n(s) &, \psi(|m_n(s)|_{L^{\infty}}) \psi(|P_n(m_n(s) \times h)|_{L^{\infty}})\psi(|P_n(m_n(s) \times (m_n(s) \times h))|_{L^{\infty}}) \centerdot \\
				& \quad \nabla P_n(m_n(s) \times P_n(m_n(s) \times (m_n(s) \times h)) \times h)  \r\rangle_{L^2}^2 \, ds \\
				& \leq  C(h) \int_{0}^{t} |m_n(s)|_{H^1}^2 \, ds .
			\end{align*}
		}
		\dela{In particular,
			\begin{align}
				\nonumber \int_{0}^{T} \l\langle \nabla m_n(s) &, \psi(|m_n(s)|_{L^{\infty}}) \psi(|P_n(m_n(s) \times h)|_{L^{\infty}})\psi(|P_n(m_n(s) \times (m_n(s) \times h))|_{L^{\infty}}) \centerdot \\
				\nonumber & \quad \nabla P_n(m_n(s) \times P_n(m_n(s) \times (m_n(s) \times h)) \times h)  \r\rangle_{L^2}^2 \, ds \\
				& \leq  C(h) \int_{0}^{T} |m_n(s)|_{H^1}^4 \, ds .
			\end{align}
			
			By the Burkholder-Davis-Gundy inequality, ( see Lemma \ref{BDG inequality upper bound} ) followed by the Cauchy-Schwartz inequality, there exists a constant $C>0$ such that
			\begin{align*}
				&\mathbb{E}\sup_{t\in[0,T]}\bigg|\int_{0}^{t} \l\langle \nabla m_n(s) , \psi(m_n(s))  \nabla P_n(m_n(s) \times P_n(m_n(s) \times (m_n(s) \times h)) \times h)  \r\rangle_{L^2} \, dW(s) \bigg| \\
				\leq &  \mathbb{E} \bigg[\bigg(\int_{0}^{T} \l\langle \nabla m_n(s) , \psi(m_n(s)) \nabla P_n(m_n(s) \times P_n(m_n(s) \times (m_n(s) \times h)) \times h)  \r\rangle_{L^2}^2 \, ds \bigg)^{\frac{1}{2}}\bigg] \\
				\leq &  C \mathbb{E} \l[ \int_{0}^{T} \l|m_n(s)\r|_{H^1}^4  \, ds \r]^{\frac{1}{2}} (\text{By H\"older's inequality})  \\
				\leq &  C \mathbb{E} \l[ \int_{0}^{T} \l|m_n(s)\r|_{H^1}^2  \l|m_n(s)\r|_{H^1}^2 \, ds \r]^{\frac{1}{2}} \\
				\leq &  C \mathbb{E}\l[ \l(\sup_{t\in [0,T]}\l|m_n(t)\r|_{H^1}^2\r)^{\frac{1}{2}}\l( \int_{0}^{T} \l|m_n(s)\r|_{H^1}^2   \, ds \r)^{\frac{1}{2}} \r]  (\text{By Cauchy-Schwartz inequality}) \\
				\leq &  \frac{\varepsilon }{2}\mathbb{E} \l[ \sup_{t\in [0,T]} \l|m_n(t)\r|_{H^1}^{2}   \r] + \frac{C(\varepsilon)C^2}{2} \mathbb{E} \l[ \int_{0}^{T} \l|m_n(s)\r|_{H^1}^{2}  \, ds \r] \text{(By Young's inequality)} .
			\end{align*}
		}
		\dela{The last step follows by Young's inequality.}
		The first term on the right hand side of the above inequality has the coefficient $\frac{\varepsilon}{2}$. This $\varepsilon > 0$ is chosen later such that the coefficient on the left hand side remains positive.\\
		\dela{
			Similarly,
			\begin{align*}
				\mathbb{E}&\sup_{t\in[0,T]}\bigg|\int_{0}^{t} \l\langle \nabla m_n(s) , P_n(m_n(s) \times (m_n(s) \times h))  \r\rangle_{L^2} \, dW(s) \bigg| \\
				& \leq \mathbb{E} \bigg[\bigg(\int_{0}^{T} \l\langle \nabla m_n(s) , \nabla P_n(m_n(s) \times (m_n(s) \times h)) \r\rangle_{L^2}^2 \, ds \bigg)^{\frac{1}{2}}\bigg] \\
				& \leq C \mathbb{E} \l[ \int_{0}^{T} \l|m_n(s)\r|_{H^1}^4  \, ds \r]^{\frac{1}{2}} \\
				& \leq \frac{\varepsilon }{2}\mathbb{E} \l[ \sup_{t\in [0,T]} \l|m_n(t)\r|_{H^1}^{2}   \r] + \frac{C(\varepsilon )C^2}{2} \mathbb{E} \l[ \int_{0}^{T} \l|m_n(s)\r|_{H^1}^{2}  \, ds \r].
			\end{align*}
			The first term on the right hand side of the above inequalities will be taken to the left hand side when doing the final estimates.
		}
		\dela{
			Here and in the previous calculation, $\varepsilon > 0$ is chosen such that
			\begin{equation}\label{choice of epsilon inequality 1}
				\varepsilon < 1.
			\end{equation}
			Hence the first term on the right hand side of the above inequality, along with the similar term in the preceding calculation can be taken to the left hand side of the inequality once all the calculations are combined. The choice of $\varepsilon$ will ensure that the coefficient remains positive.
		}
		
		We combine all the above inequalities (except the calculations for $J_6$) with the equation \eqref{equation Ito formula applied to phi 2} to get
		\begin{align}\label{bound 2 intermediate inequality}
			\nonumber \l| \nabla m_n(t) \r|^2_{L^2} +& (\alpha \, - \varepsilon) \int_{0}^{t} |m_n(s) \times \Delta m_n(s)|_{L^2}^2\, ds \leq |m_n(0)|_{H^1}^2 + \frac{C(\varepsilon)}{2}[C(h) + 1]\int_{0}^{t}|u_n(s)|_{L^2}^2\, ds \\
			+ &  C(h)\int_{0}^{t}|m_n(s)|_{H^1}^2\, ds + C(h)\bigg|\int_{0}^{t}\l\langle \psi\big(m_n(s)\big)  G_n\big(m_n(s)\big) , \Delta m_n(s) \r\rangle_{L^2}\, dW(s)\bigg|.
		\end{align}
		Choose $\varepsilon$ small enough such that $\alpha \, - \varepsilon > 0$. For instance, $\varepsilon = \frac{\alpha}{2}$ works here. This implies that the second term on the left hand side of the above inequality is non-negative. Therefore we can remove that term, still keeping the inequality intact.\\
		We take $\sup_{t\in[0,T]}$ of both sides of the resulting inequality, followed by taking the expectation to get
		\dela{
			
			\begin{align*}
				\nonumber \sup_{t\in[0,T]}&\l| \nabla m_n(t) \r|^2_{L^2} + (\alpha \, - \frac{3\varepsilon}{2}) \sup_{t\in[0,T]}\int_{0}^{t} |m_n(s) \times \Delta m_n(s)|_{L^2}^2\, ds \\
				\leq & |m_n(0)|_{H^1}^2 + \frac{C(\varepsilon)}{2}[C(h) + 1]\int_{0}^{T}|u_n(s)|_{L^2}^2\, ds \\
				&+   C(h)  \int_{0}^{T}|m_n(s)|_{H^1}^2\, ds + C(h) \sup_{t\in[0,T]} \bigg|\int_{0}^{t}\l\langle \nabla\l( m_n(s) \times h\r) , \nabla m_n(s) \r\rangle_{L^2}\, dW(s)\bigg|\\
				&+   C(h) \sup_{t\in[0,T]} \bigg|\int_{0}^{t} \psi(|m_n(s)|_{L^{\infty}}) \psi(|P_n(m_n(s) \times h)|_{L^{\infty}})\psi(|P_n(m_n(s) \times (m_n(s) \times h))|_{L^{\infty}}) \centerdot \\
				&\l\langle \nabla\l( m_n(s) \times \l(m_n(s) \times h\r)\r) , \nabla m_n(s) \r\rangle_{L^2}\, dW(s)\bigg|.
			\end{align*}
			\dela{Add a constant to the right hand side??
				Here, choose $\varepsilon > 0$ such that \eqref{choice of epsilon inequality 1} holds and
				\begin{equation}\label{choice of epsilon inequality 2}
					\alpha \, - \frac{3\varepsilon}{2} > 0.
				\end{equation}  \adda{Verify the coefficient of $\varepsilon$.}
				Thus the term $(\alpha \, - \frac{3\varepsilon}{2}) \sup_{t\in[0,T]}\int_{0}^{t} |m_n(s) \times \Delta m_n(s)|_{L^2}^2\, ds$ is non-negative $\mathbb{P}$-a.s.
			}
			
			Hence
			\begin{align*}
				\nonumber &\sup_{t\in[0,T]}\l| \nabla m_n(t) \r|^2_{L^2} \leq   |m_n(0)|_{H^1}^2 + \frac{C(\varepsilon)}{2}[C(h) + 1]\int_{0}^{T}|u_n(s)|_{L^2}^2\, ds \\
				& +   C(h)  \int_{0}^{T}|m_n(s)|_{H^1}^2\, ds + C(h) \sup_{t\in[0,T]} \bigg|\int_{0}^{t}\l\langle \nabla\l( m_n(s) \times h\r) , \nabla m_n(s) \r\rangle_{L^2}\, dW(s)\bigg|\\
				& +   C(h) \sup_{t\in[0,T]} \bigg|\int_{0}^{t} \psi(|m_n(s)|_{L^{\infty}}) \psi(|P_n(m_n(s) \times h)|_{L^{\infty}}) \centerdot \\
				& \psi(|P_n(m_n(s) \times (m_n(s) \times h))|_{L^{\infty}}) \l\langle \nabla\l( m_n(s) \times \l(m_n(s) \times h\r)\r) , \nabla m_n(s) \r\rangle_{L^2}\, dW(s)\bigg|.
			\end{align*}
			Taking the expectation of both sides of the above inequality gives
			
		}
		\begin{align*}
			\nonumber &\mathbb{E}\sup_{t\in[0,T]}\l| \nabla m_n(t) \r|^2_{L^2} \leq  \mathbb{E}|m_n(0)|_{H^1}^2 + \frac{C(\varepsilon)}{2}[C(h) + 1]\mathbb{E}\int_{0}^{T}|u_n(s)|_{L^2}^2\, ds \\
			& + C(h) \mathbb{E} \int_{0}^{T}|m_n(s)|_{H^1}^2\, ds + C(h) \mathbb{E}\sup_{t\in[0,T]} \bigg|\int_{0}^{t}\l\langle \psi(m_n(s))G_n(m_n(s)) , \Delta m_n(s) \r\rangle_{L^2}\, dW(s)\bigg|.
			\dela{+ C(h) \mathbb{E}\sup_{t\in[0,T]} \bigg|\int_{0}^{t}\l\langle \nabla\l( m_n(s) \times h\r) , \nabla m_n(s) \r\rangle_{L^2}\, dW(s)\bigg|\\
				& +   C(h) \mathbb{E}\sup_{t\in[0,T]} \bigg|\int_{0}^{t} \psi\big(|m_n(s)|_{L^{\infty}}\big) \psi\big(|P_n(m_n(s) \times h)|_{L^{\infty}}\big) \centerdot \\
				&\quad \psi\big(|P_n\big(m_n(s) \times (m_n(s) \times h)\big)|_{L^{\infty}}\big) \l\langle \nabla \bigl( m_n(s) \times \l(m_n(s) \times h\r) \bigr) , \nabla m_n(s) \r\rangle_{L^2}\, dW(s)\bigg|}
		\end{align*}
		Combining all the constants into a suitable constant $C>0$ and replacing
		$\l|m_n(s)\r|_{H^1}$ (inside the integrals) by $\sup_{r\in[0,s]} \l|m_n(r)\r|_{H^1}$ gives
		\begin{align*}
			\mathbb{E}\sup_{t\in[0,T]}\l| \nabla m_n(t) \r|^2_{L^2} \leq \mathbb{E}|m_n(0)|_{H^1}^2 + C(K) + \varepsilon \mathbb{E}\sup_{t\in[0,T]}\l|  m_n(t) \r|^2_{H^1} + C \mathbb{E} \int_{0}^{T} \sup_{r\in[0,s]} \l|m_n(r)\r|_{H^1}^2 \, ds.	
		\end{align*}
		Here again, $\varepsilon = \min\{\frac{\alpha}{2} , \frac{1}{2}\}$ is chosen in order to keep the coefficient on the left hand side positive.
		We observe the following. For $v\in H^1$,
		\begin{align*}
			\l| v \r|_{H^1}^2 = \l|v\r|_{L^2}^2 + \l| \nabla v \r|_{L^2}^2.
		\end{align*}
		Therefore applying the bound \eqref{bound on m_n}, we replace the left hand side by the full $H^1$ norm by adding a constant (using Lemma \ref{bounds lemma 1} and finiteness of $T$) to the right hand side. The resulting inequality is
		\begin{align*}
			\l( 1 - \varepsilon \r)\mathbb{E}\sup_{t\in[0,T]}\l| m_n(t) \r|^2_{H^1} \leq \mathbb{E}|m_n(0)|_{H^1}^2 + C(K) + C \mathbb{E} \int_{0}^{T} \sup_{r\in[0,s]} \l|m_n(r)\r|_{H^1}^2 \, ds.	
		\end{align*}
		\dela{Here again, $\varepsilon = \min\{\frac{\alpha}{2} , \frac{1}{2}\}$ is chosen in order to keep the coefficient on the left hand side positive.}

		\dela{
			Thus,
			\begin{align*}
				\mathbb{E}\sup_{t\in[0,T]}\l| m_n(t) \r|^2_{H^1} &\leq \mathbb{E}\sup_{t\in[0,T]}\l| m_n(t) \r|^2_{L^2} + \mathbb{E}\sup_{t\in[0,T]}\l| \nabla m_n(t) \r|^2_{L^2} \\
				& \leq  \mathbb{E}\sup_{t\in[0,T]}\l| m_n(t) \r|^2_{L^2} + C(T,m_0) \\
				& \leq \mathbb{E}|m_n(0)|_{H^1}^2 + C(K) + C \mathbb{E} \int_{0}^{T} \sup_{r\in[0,s]} \l|m_n(r)\r|_{H^1} \, ds + C(T,m_0) .
			\end{align*}
			\adda{What is this for?}
		}

		Using Fubini's theorem and then applying the Gronwall Lemma, the assumption on the initial data $m_0$ implies that there exists a constant $C>0$ such that
		\begin{align}
			\mathbb{E}\sup_{t\in[0,T]}\l| m_n(t) \r|^2_{H^1} \leq C.
		\end{align}
		This concludes the proof of bound \eqref{bound 2 without p}.\\
		\textbf{Proof of bound \eqref{bound 3 without p}}:\\
		Going back to the inequality \eqref{bound 2 intermediate inequality} (with $\varepsilon = \min\{\frac{\alpha}{2} , \frac{1}{2}\}$), we observe that the first term on the right hand side is non-negative. Hence the term can be neglected without changing the inequality. We take the supremum over $[0,T]$, followed by the expectation of both sides. In particular, the bound \eqref{bound 2 without p} implies that there exists a constant $C>0$ such that
		\begin{equation}
			\mathbb{E} \int_{0}^{T} \l| m_n(t) \times \Delta m_n(t) \r|_{L^2}^2 \, dt \leq C.
		\end{equation}
		This concludes the proof of bound \eqref{bound 3 without p}, and hence the proof of Lemma \ref{bounds lemma 1 without p}.
	\end{proof}
	Having shown some energy estimates, we show that the $p$-th order moments for the approximations are also bounded.
	\begin{lemma}\label{bounds lemma 1}
		Assume $p\geq 1$. There exists a constant $C>0$ such that for all $n\in\mathbb{N}$, the following bounds hold
		\begin{equation}\label{bound 2}
			\mathbb{E}\l[\sup_{r\in[0,T]}|m_n(r)|_{H^1}^{2p}\r]\leq C_{p},
		\end{equation}
		\begin{align}\label{bound 3}
			\mathbb{E}\l[\int_{0}^{T} |m_n(s)\times \Delta m_n(s)|_{L^2}^2 \, ds \r]^p \leq C_{p},
		\end{align}
		\begin{align}\label{bound 4}
			\mathbb{E}\l[\l(\int_{0}^{T}|m_n(s)\times\l(m_n(s)\times\Delta m_n\l(s\r)\r)|_{L^{2}}^2  \, ds  \r)^{p} \r]\leq C_{2p},
		\end{align}
		\begin{equation}\label{bound 5}
			\mathbb{E}\l[\l(\int_{0}^{T}|m_n(s)\times u_n(s)|_{L^{2}}^2 \, ds \r)^{p} \r]\leq C_{2p},
		\end{equation}
		\begin{equation}\label{bound 6}
			\mathbb{E}\l[\l(\int_{0}^{t}|m_n(s) \times (m_n(s) \times u_n(s))|_{L^{2}}^2 \, ds\r)^{p} \r]\leq C_{4p}.
		\end{equation}
		
	\end{lemma}
	\dela{How the constants $C$ above depend on $K_p$ and which $p$? \\}
	The quantity written at the base of the constants $C$ represents the regularity of $u$ that is required for the bound with $p$ to hold on the left hand side. For instance in the bound \eqref{bound 5}, the constant $C$ depends on $K_{2p}$ and in the bound \eqref{bound 6} the constant depends on $K_{4p}$. In particular for $p=1$, we require the bound $K_4$.
	
	\begin{proof}[Proof of Lemma \ref{bounds lemma 1}]
		
		The proof uses the calculations from the previous lemma (Lemma \ref{bounds lemma 1 without p}). The idea is to again apply It\^o's formula and obtain \eqref{bounds lemma 1 intermediate equation}. The same bounds as before will be used, except for the stochastic integral term. We write some more calculations regarding the individual terms that are not given previously.\\
		\textbf{Proof of the bound \eqref{bound 2}}:\\
		Let $p\geq 1$. As before, let $\varepsilon>0$.
		\dela{
			\begin{align*}
				&\mathbb{E}\sup_{t\in[0,T]}\bigg|\int_{0}^{t} \l\langle \nabla m_n(s) , \psi(|m_n(s)|_{L^{\infty}}) \psi(|P_n(m_n(s) \times h)|_{L^{\infty}})\psi(|P_n(m_n(s) \times (m_n(s) \times h))|_{L^{\infty}}) \centerdot \\
				& \quad \nabla P_n(m_n(s) \times P_n(m_n(s) \times (m_n(s) \times h)) \times h)  \r\rangle_{L^2} \, dW(s) \bigg|^p \\
				&\leq  \mathbb{E} \bigg[\bigg(\int_{0}^{T} \l\langle \nabla m_n(s) , \psi(|m_n(s)|_{L^{\infty}}) \psi(|P_n(m_n(s) \times h)|_{L^{\infty}})\psi(|P_n(m_n(s) \times (m_n(s) \times h))|_{L^{\infty}}) \centerdot \\
				& \quad \nabla P_n(m_n(s) \times P_n(m_n(s) \times (m_n(s) \times h)) \times h)  \r\rangle_{L^2}^2 \, ds \bigg)^{\frac{p}{2}}\bigg] \\
				&\leq  C \mathbb{E} \l[ \int_{0}^{T} \l|m_n(s)\r|_{H^1}^4  \, ds \r]^{\frac{p}{2}} \\
				&\leq  C \mathbb{E} \l[ \int_{0}^{T} \l|m_n(s)\r|_{H^1}^2 \l|m_n(s)\r|_{H^1}^2    \, ds \r]^{\frac{p}{2}} \\
				&\leq  C \mathbb{E} \l[ \l(\sup_{t\in [0,T]} \l|m_n(s)\r|_{H^1}^2 \r)^{\frac{1}{2}}  \l( \int_{0}^{T} \l|m_n(s)\r|_{H^1}^2   \, ds \r)^{\frac{1}{2}} \r]^p\ \text{By Cauchy-Schwartz inequality}\\
				&\leq   \frac{C\varepsilon}{2}\mathbb{E} \l[ \sup_{t\in [0,T]} \l|m_n(s)\r|_{H^1}^2    \r] + \frac{C(\varepsilon)}{2} \mathbb{E} \l[ \int_{0}^{T} \l|m_n(s)\r|_{H^1}^2   \, ds\r]^p\ \text{By Young's inequality}\\
				&\leq  \frac{C\varepsilon}{2}\mathbb{E} \l[ \sup_{t\in [0,T]} \l|m_n(s)\r|_{H^1}^2    \r] + \frac{C(\varepsilon)}{2} \mathbb{E} \l[ \int_{0}^{T} \l|m_n(s)\r|_{H^1}^{2p}  \, ds \r].
			\end{align*}
		}
		We recall the following inequality established in the calculations for $J_6$ in the proof for \eqref{bound 2 without p} in the previous lemma.
		\begin{align}
			\l|\int_{0}^{t} \l\langle \psi\big(m_n(s)\big) G_n\big(m_n(s)\big) , \Delta m_n(s) \r\rangle_{L^2}^{2} \, ds\r| \leq C \int_{0}^{t} \l| m_n(s) \r|_{H^1}^{4} \, ds.
		\end{align}
		By Lemma \ref{BDG inequality upper bound}, followed by Cauchy-Schwartz inequality and Young's inequality and then Jensen's inequality, there exists constant $C,C(\varepsilon)$ (which can depend on $h$) such that
		\begin{align*}
			&\mathbb{E}\sup_{t\in[0,T]} \l| \int_{0}^{t} \l\langle \psi(m_n(s)) G_n\big(m_n(s)\big) , \Delta m_n(s) \r\rangle_{L^2} \, dW(s) \r|^{p} \\
			\leq & C \mathbb{E} \l(\int_{0}^{t} \l\langle \psi\big(m_n(s)\big) G_n\big(m_n(s)\big) , \Delta m_n(s) \r\rangle_{L^2}^2 \, ds\r)^\frac{p}{2} \\
			\leq & C \mathbb{E} \l(\int_{0}^{T} \l| m_n(s) \r|_{H^1}^4 \, ds\r)^\frac{p}{2}\\
			\leq &  C \mathbb{E} \l[ \int_{0}^{T} \l|m_n(s)\r|_{H^1}^2  \l|m_n(s)\r|_{H^1}^2 \, ds \r]^\frac{p}{2} \\
			\leq &  C \mathbb{E}\l[ \l(\sup_{t\in [0,T]}\l|m_n(t)\r|_{H^1}^2\r)^\frac{p}{2}
			\l( \int_{0}^{T} \l|m_n(s)\r|_{H^1}^2   \, ds \r)^\frac{p}{2} \r]  (\text{By Cauchy-Schwartz inequality}) \\
			\leq &  \frac{\varepsilon }{2}\mathbb{E} \l[ \sup_{t\in [0,T]} \l|m_n(t)\r|_{H^1}^{2p}   \r] + \frac{C(\varepsilon)C^2}{2} \mathbb{E} \l[ \int_{0}^{T} \l|m_n(s)\r|_{H^1}^{2p}  \, ds \r]. \text{(By Young's inequality)}
		\end{align*}
		\dela{\begin{align*}
				&\mathbb{E}\sup_{t\in[0,T]}\bigg|\int_{0}^{t} \l\langle \nabla m_n(s) , \psi(m_n(s)) \nabla G_n(m_n(s)) \r\rangle_{L^2} \, dW(s) \bigg|^p \\
				&\leq  \frac{\varepsilon}{2}\mathbb{E} \l[ \sup_{t\in [0,T]} \l|m_n(s)\r|_{H^1}^2    \r] + \frac{C(\varepsilon)}{2} \mathbb{E} \l[ \int_{0}^{T} \l|m_n(s)\r|_{H^1}^{2p}  \, ds \r].
			\end{align*}
		}
		\dela{The last inequality follows due to Jensen's inequality. $\varepsilon > 0$ is to be chosen later.\\}
		\dela{Omit this. Note Needed Similarly,
			\dela{
				\begin{align*}
					\mathbb{E}\sup_{t\in[0,T]}\bigg|\int_{0}^{t} \l\langle \nabla m_n(s) &, P_n(m_n(s) \times (m_n(s) \times h))  \r\rangle_{L^2} \, dW(s) \bigg|^p \\
					& \leq \mathbb{E} \bigg[\bigg(\int_{0}^{T} \l\langle \nabla m_n(s) , \nabla P_n(m_n(s) \times (m_n(s) \times h)) \r\rangle_{L^2}^2 \, ds \bigg)^{\frac{p}{2}}\bigg] \\
					& \leq C \mathbb{E} \l[ \int_{0}^{T} \l|m_n(s)\r|_{H^1}^4  \, ds \r]^{\frac{p}{2}} \\
					& \leq \frac{\varepsilon}{2}\mathbb{E} \l[ \sup_{t\in [0,T]} \l|m_n(s)\r|_{H^1}^2    \r] + \frac{\varepsilon}{2} \mathbb{E} \l[ \int_{0}^{T} \l|m_n(s)\r|_{H^1}^{2p}  \, ds \r].
				\end{align*}
			}
			\begin{align*}
				\mathbb{E}\sup_{t\in[0,T]}\bigg|\int_{0}^{t} \l\langle \nabla m_n(s) , P_n(m_n(s) \times (m_n(s) \times h))  \r\rangle_{L^2} \, dW(s) \bigg|^p
				\leq & \frac{\varepsilon}{2}\mathbb{E} \l[ \sup_{t\in [0,T]} \l|m_n(s)\r|_{H^1}^2    \r] \\
				& + \frac{C(\varepsilon)}{2} \mathbb{E} \l[ \int_{0}^{T} \l|m_n(s)\r|_{H^1}^{2p}  \, ds \r].
		\end{align*}}
		\dela{The last inequality follows due to Jensen's inequality.
			$\varepsilon > 0$ is to be chosen later.
			One of the factors to be considered in the choice of $\varepsilon$ is that
			\begin{equation}\label{choice of p epsilon 1}
				C^p \varepsilon < 1.
			\end{equation}
		}
		We recall the inequality \eqref{bound 2 intermediate inequality} and restate it here.
		\begin{align}
			\nonumber \l| \nabla m_n(t) \r|^2_{L^2} +& (\alpha \, - \varepsilon) \int_{0}^{t} |m_n(s) \times \Delta m_n(s)|_{L^2}^2\, ds \leq |m_n(0)|_{H^1}^2 + \frac{C(\varepsilon)}{2}[C(h) + 1]\int_{0}^{t}|u_n(s)|_{L^2}^2\, ds \\
			+&  C(h)\int_{0}^{t}|m_n(s)|_{H^1}^2\, ds + C(h)\bigg|\int_{0}^{t}\l\langle \psi\big(m_n(s)\big) G_n\big(m_n(s)\big) , \Delta m_n(s) \r\rangle_{L^2}\, dW(s)\bigg|.
		\end{align}
		Choose $\varepsilon > 0$ such that
		\begin{equation}\label{choice of p epsilon 2}
			\alpha \, - \varepsilon > 0.
		\end{equation}
		Therefore the second term on the left hand side of the resulting inequality is non-negative. Hence the inequality remains the same even if that term is neglected.\\
		We raise both sides of the inequality \eqref{bound 2 intermediate inequality} (after choosing $\varepsilon = \min\{\frac{1}{2},\frac{\alpha}{2}\}$) to power $p\geq 1$ and use Jensen's inequality (to bring the power $p$ inside the time integral) to get
		\begin{align*}
			\nonumber &\sup_{t\in[0,T]}\l| \nabla m_n(t) \r|^{2p}_{L^2} \leq C(p)\bigg[|m_n(0)|_{H^1}^{2p} + \frac{C(\varepsilon)^2}{2}[C(h) + 1]^{p} \l(\int_{0}^{T}|u_n(s)|_{L^2} \, ds\r)^{p} \\
			&+   C(h)  \int_{0}^{T}|m_n(s)|_{H^1}^{2p}\, ds
			+ C(h)\bigg|\int_{0}^{t}\l\langle \psi\big(m_n(s)\big)\nabla G_n\big(m_n(s)\big) , \nabla m_n(s) \r\rangle_{L^2}\, dW(s)\bigg|^p.
		\end{align*}
		In the steps that follow, the constant $C(p)$ is absorbed into the existing constants. We take the expectation of both sides to get
		\dela{
			\begin{align}
				\nonumber \mathbb{E}\sup_{t\in[0,T]}\l| \nabla m_n(t) \r|^{2p}_{L^2}  \leq & C(p)^p\mathbb{E}|m_n(0)|_{H^1}^{2p} + \frac{C(\varepsilon)^p}{2^p}[C(h) + 1]^p \mathbb{E} \l(\int_{0}^{T}|u_n(s)|_{L^2}^2\, ds \r)^{p} \\
				\nonumber   &+  C(h)^p \mathbb{E}  \int_{0}^{T}|m_n(s)|_{H^1}^{2p}\, ds \\
				\nonumber &+ C(h)^p \mathbb{E} \l(\sup_{t\in[0,T]} \int_{0}^{t}\l\langle \nabla\l( m_n(s) \times h\r) , \nabla m_n(s) \r\rangle_{L^2}\, dW(s) \r)^{p}\\
				\nonumber  &+  C(h)^p \mathbb{E} \sup_{t\in[0,T]} \large\bigg[\int_{0}^{t} \psi(|m_n(s)|_{L^{\infty}}) \psi(|P_n(m_n(s) \times (m_n(s) \times h))|_{L^{\infty}}) \centerdot \\
				\nonumber & \quad \psi(|P_n(m_n(s) \times h)|_{L^{\infty}})\l\langle \nabla\l( m_n(s) \times \l(m_n(s) \times h\r)\r) , \nabla m_n(s) \r\rangle_{L^2}\, dW(s)\bigg]^p \\
				\nonumber \leq &  \mathbb{E}|m_n(0)|_{H^1}^{2p} + \frac{C(\varepsilon)^p}{2^p}[C(h) + 1]^p \mathbb{E} \l(\int_{0}^{T}|u_n(s)|_{L^2}^2\, ds \r)^{p}\\
				\nonumber  &+  C(h)^p \mathbb{E} \int_{0}^{T}|m_n(s)|_{H^1}^{2p}\, ds + C(h) \mathbb{E}   \int_{0}^{T} \l|m_n(s)\r|_{H^1}^{2p}\, ds  \\
				\nonumber \leq &  \mathbb{E}|m_n(0)|_{H^1}^{2p} + \frac{C(\varepsilon)^p}{2}[C(h) + 1]^p C(K) \\
				\nonumber  &+ C(h)^p \mathbb{E} \int_{0}^{T}|m_n(s)|_{H^1}^{2p}\, ds + C(h) \mathbb{E}   \int_{0}^{T} \l|m_n(s)\r|_{H^1}^{2p}\, ds  \\
				\leq &  C + C(h) \mathbb{E} \int_{0}^{T}|m_n(s)|_{H^1}^{2p}\, ds + C(h) \mathbb{E}   \int_{0}^{T} \l|m_n(s)\r|_{H^1}^{2p}\, ds .\\
				\nonumber &\text{(Constants condensed into $C,C(h)$.)}
			\end{align}
		}
		\begin{align}
			\nonumber \mathbb{E}\sup_{t\in[0,T]}\l| \nabla m_n(t) \r|^{2p}_{L^2}  \leq & C(p)^p\mathbb{E}|m_n(0)|_{H^1}^{2p} + \frac{C(\varepsilon)^p}{2^p}[C(h) + 1]^p \mathbb{E} \l(\int_{0}^{T}|u_n(s)|_{L^2}^2\, ds \r)^{p} \\
			\nonumber   &+  C(h)^p \mathbb{E}  \int_{0}^{T}|m_n(s)|_{H^1}^{2p}\, ds \\
			\nonumber &+ C(h)^p \mathbb{E} \bigg|\int_{0}^{t}\l\langle \psi\big(m_n(s)\big)\nabla G_n\big(m_n(s)\big) , \nabla m_n(s) \r\rangle_{L^2}\, dW(s)\bigg|^{p}.
		\end{align}
		The inequalities established at the start of this proof enable us to write the following inequality.
		\begin{align}
			\mathbb{E}\sup_{t\in[0,T]}\l| \nabla m_n(t) \r|^{2p}_{L^2}  \leq & C(p)^p\mathbb{E}\l|m_n(0)\r|_{H^1}^{2p}  + C(h) \mathbb{E} \int_{0}^{T}|m_n(s)|_{H^1}^{2p}\, ds  .
		\end{align}
		The constants $C,C(h)$ may depend on $p, K, h, \varepsilon, m_0$ but not on $n$ and may vary from line to line.

		\dela{
			The following argument has been used in the previous calculations.
			Since $\{e_n\}_{n\in\mathbb{N}}$ is an orthogonal basis of $H^1$,
			\begin{align*}
				\l|m_n(0)\r|_{H^1}^2 = \l|P_n m(0)\r|_{H^1}^2 \leq \l|m_0\r|_{H^1}^2,\ \mathbb{P}\text{-a.s.}
			\end{align*}
			Hence
		}
		Since $m_n(0) = P_nm_0$, the following holds.
		\begin{align}
			\mathbb{E} \l|m_n(0)\r|_{H^1}^{2p} \leq \mathbb{E} \l|m_0\r|_{H^1}^{2p}.
		\end{align}
		By the assumptions on the initial data $m_0$, the right hand side of the above inequality is finite, thus allowing the first term on the right hand side of the previous inequality to be bounded by a constant.

		The idea now is to add the term $\mathbb{E}\sup_{t\in[0,T]}\l|m_n(t)\r|_{L^2}^{2p}$ to both sides of the inequality, as done in the proof of the bound \eqref{bound 2 without p} in Lemma \ref{bounds lemma 1 without p}. The left hand can therefore be replaced by $\mathbb{E}\sup_{t\in[0,T]}|m_n(t)|^{2p}_{H^1}$. On the right hand side, we use the bound \eqref{bound 1} to bound the added term by a constant.\\
		Hence the above inequality \dela{(for the choosing $\varepsilon$ small enough) }implies that there exists constants $C_1,C_2>0$ such that
		
		\begin{equation}\label{bounds lemma 1 intermediate equation}
			\mathbb{E}\sup_{t\in[0,T]}|m_n(t)|^{2p}_{H^1}  \leq C_1 + C_2 \mathbb{E} \int_{0}^{T}|m_n(s)|^{2p}_{H^1} \, ds \dela{+ C(h)\int_{0}^{T}|m_n(s)|_{H^1}^{2p} \, ds } .
		\end{equation}
		We now use the Grownall's inequality to get a constant $C_p>0$ such that
		
		\begin{equation}\label{bound on nabla m_n}
			\mathbb{E} \sup_{t\in [0,T]} |m_n(t)|_{H^1}^{2p} \leq C_p,\ n\in\mathbb{N}.
		\end{equation}
		This completes the proof of the bound \eqref{bound 2}.\\
		\textbf{Proof for the bound \eqref{bound 3}}:\\
		Consider the inequality \eqref{bound 2 intermediate inequality}. The first term on the left hand side is non-negative. Therefore the following inequality also holds. \dela{Taking $\sup_{t\in [0,T]}$ and then taking expectation of both the sides we get the bound \eqref{bound 3} by using the bound stated above.} 		
		\begin{align*}
			(\alpha \, - \varepsilon) \int_{0}^{t} |m_n(s) \times \Delta m_n(s)|_{L^2}^2 ds &\leq |m_n(0)|_{H^1}^2 + \frac{C(\varepsilon)}{2}[C(h) + 1]\int_{0}^{t}|u_n(s)|_{L^2}^2 ds \\
			&\quad+ C(h)\int_{0}^{t}|m_n(s)|_{H^1}^2 ds + C(h)\int_{0}^{t}|m_n(s)|_{H^1}^2 dW(s).
		\end{align*}
		We have chosen $\varepsilon \leq \frac{\alpha}{2}$. Multiplying by a suitable constant, raising the power of both sides to $p \geq 1$, followed by taking the expectation of both sides gives
		\begin{align*}
			&\mathbb{E} \int_{0}^{T} |m_n(s) \times \Delta m_n(s)|_{L^2}^{2p} ds\\
			&\leq C(p)\bigg[\mathbb{E}|m_n(0)|_{H^1}^{2p} + \frac{C(\varepsilon)^p}{2^p}[C(h) + 1]^pC(p)^p \mathbb{E} \left( \int_{0}^{T}|u_n(s)|_{L^2}^2 \, ds \right)^p  \\
			& \quad + C(h)^p \mathbb{E} \int_{0}^{T}|m_n(s)|_{H^1}^{2p} ds + C(h)^p \mathbb{E} \l( \int_{0}^{T}|m_n(s)|_{H^1}^{2} \, dW(s) \r)^{p}\bigg] \\
			& \leq C(p)\mathbb{E}|m_n(0)|_{H^1}^{2p} + C(\varepsilon , h , p , K)  + C(h , p , T) \mathbb{E} \sup_{s\in [0,T]}|m_n(s)|_{H^1}^{2p}.
		\end{align*}
		The inequality follows from the Jensen's inequality, combining the constants and using the\\
		Burkholder-Davis-Gundy inequality inequality. Also the constant $K$ arises due to $(4)$ in Assumption \eqref{assumption on u} on the control process $u$.
		
		Thus, there exists a constant $C_p>0$ such that
		\begin{align}\label{bound on m_n times Delta m_n}
			\mathbb{E} \l(\int_{0}^{T}|m_n(s) \times \Delta m_n(s) |_{L^2}^2ds\r)^p \leq C_p,\ n\in\mathbb{N}.
		\end{align}\\
		\textbf{Proof of the bound \eqref{bound 4}:}\\
		The proof is done using the bounds \eqref{bound on nabla m_n} and \eqref{bound on m_n times Delta m_n}, along with the continuous embedding $H^1 \hookrightarrow L^{\infty}$.
		\begin{align*}
			\int_{0}^{T} |m_n(s) \times (m_n(s) \times \Delta m_n(s))|_{L^2}^2 \, ds
			&\leq \int_{0}^{T} |m_n(s)|_{L^{\infty}}^2 |m_n(s) \times \Delta m_n(s)|_{L^2}^2 \, ds \\
			& \leq C \int_{0}^{T}|m_n(s)|_{H^1}^2 |m_n(s) \times \Delta m_n(s)|_{L^2}^2 \, ds \ \dela{\l(\text{By}\ H^1\hookrightarrow L^{\infty}\r)}\\
			& \leq C \sup_{r\in [0,T]} (|m_n(r)|_{H^1}^2) \int_{0}^{T} |m_n(s) \times \Delta m_n(s)|_{L^2}^2  \, ds.
		\end{align*}
		The proof follows by raising the power to $p$, taking the expectation of both sides and using the bounds \eqref{bound on nabla m_n} and \eqref{bound on m_n times Delta m_n}.		
		\dela{Similar to the previous steps, let $s\in[0,T]$ and $n\in\mathbb{N}$.
			\begin{align*}
				|m_n(s) \times (m_n(s) \times \Delta m_n(s))|_{L^2}^2 &\leq |m_n(s)|_{L^{\infty}}^2 |m_n(s) \times \Delta m_n(s)|_{L^2}^2 \\
				& \leq C |m_n(s)|_{H^1}^2 |m_n(s) \times \Delta m_n(s)|_{L^2}^2 \ \text{By}\ H^1\hookrightarrow L^{\infty}\\
				& \leq C |m_n(s) \times \Delta m_n(s)|_{L^2}^2 \sup_{s\in [0,T]} (|m_n(s)|_{H^1}^2).
			\end{align*}
			Hence by}By the above inequality followed by the Cauchy-Schwartz inequality, we get
		\begin{align*}
			&	\mathbb{E}\left( \int_{0}^{T} |m_n(s) \times (m_n(s) \times \Delta m_n(s))|_{L^2}^2\, ds \right)^{p} \\
			&\leq C  \mathbb{E} \left( \sup_{s\in [0,T]} (|m_n(s)|_{H^1}^2) \int_{0}^{T} |m_n(s) \times \Delta m_n(s)|_{L^2}^2 \, ds\right)^{p} \\
			& \leq C \l[ \mathbb{E} \left( \int_{0}^{T} |m_n(s) \times \Delta m_n(s)|_{L^2}^2\, ds\right)^{2p} \r]^{\frac{1}{2}}
			\l[ \mathbb{E} \sup_{s\in [0,T]} |m_n(s)|_{H^1}^{4p} \r]^{\frac{1}{2}}.
		\end{align*}
		The last inequality follows by the H\"older's inequality.
		Thus there exists a constant $C>0$ such that for each $n\in\mathbb{N}$, the following inequality holds:
		\begin{align}\label{bound on m_n times m_n times Delta m_n}
			\mathbb{E}\l[\l(\int_{0}^{T}|m_n(s)\times(m_n(s)\times\Delta m_n(s))|_{L^{2}}^2\r)^p\r]
			\leq C_{2p}.
		\end{align}
		This completes the proof of the bound \eqref{bound 4}.\\
		\textbf{Proof of the bounds \eqref{bound 5} and \eqref{bound 6}:}\\
		Let $s\in[0,T]$. There exists a constant $C>0$, independent of $n,s$ such that
		\begin{align*}
			|m_n(s) \times u_n(s)|_{L^2}^2 &\leq |m_n(s)|_{L^{\infty}}^2|u_n(s)|_{L^2}^2  \leq C|m_n(s)|_{H^1}^2 |u_n(s)|_{L^2}^2.
		\end{align*}
		The above inequality uses the following continuous embedding:
		$$H^1\hookrightarrow L^{\infty}.$$
		Therefore by H\"older's inequality, followed by Jensen's inequality we have
		\dela{\begin{align*}
				\mathbb{E} \l[ \int_{0}^{t} |m_n(s) \times u_n(s)|_{L^2}^2 \, ds \r]^{\frac{p}{2}} &\leq C \mathbb{E} \l[ \int_{0}^{t} |m_n(s)|_{H^1}^2 |u_n(s)|_{L^2}^2 \, ds \r]^{\frac{p}{2}} \\
				& \leq C \mathbb{E} \sup_{s\in[0,t]}|m_n(s)|_{H^1}^p \l( \int_{0}^{t}  |u_n(s)|_{L^2}^2 \, ds \r)^{\frac{p}{2}} \\
				& \leq C \l( \mathbb{E} \sup_{s\in[0,t]}|m_n(s)|_{H^1}^{2p} \r) \l[\mathbb{E}  \l( \int_{0}^{t}  |u_n(s)|_{L^2}^2 \, ds\r)^{\frac{p}{2}} \r]
		\end{align*}}
		
		\begin{align*}
			\mathbb{E} \l[ \int_{0}^{T} |m_n(s) \times u_n(s)|_{L^2}^2 \, ds \r]^{p} &\leq C \mathbb{E} \l[ \int_{0}^{T} |m_n(s)|_{H^1}^2 |u_n(s)|_{L^2}^2 \, ds \r]^{p} \\
			& \leq C \mathbb{E}\l[ \sup_{s\in[0,T]}|m_n(s)|_{H^1}^{2p} \l( \int_{0}^{T}  |u_n(s)|_{L^2}^2 \, ds \r)^{p}\r] \\
			(\text{By Cauchy-Schwartz inequality})\ & \leq C \l( \mathbb{E} \sup_{s\in[0,T]}|m_n(s)|_{H^1}^{4p} \r)^{\frac{1}{2}} \l[\mathbb{E}  \l( \int_{0}^{T}  |u_n(s)|_{L^2}^2 \, ds\r)^{2p} \r]^{\frac{1}{2}}.
		\end{align*}
		Thus by the bound \eqref{bound 2}, there exists a constant $C_{2p}>0$ such that
		\begin{equation}\label{bound on m_n times u_n}
			\mathbb{E}\l[\l(\int_{0}^{T}|m_n(s)\times u_n(s)|_{L^{2}}^2 \, ds\r)^{p}\r]\leq C_{2p},\ \forall n\in\mathbb{N}.
		\end{equation}
		
		For the bound \eqref{bound 6},
		\begin{align*}
			\l|m_n(s) \times ( m_n(s) \times u_n(s))\r|_{L^2}^2 &\leq \l|m_n(s)\r|_{L^{\infty}}^2 \l| m_n(s) \times u_n(s) \r|_{L^2}^2 \\
			& \leq C \l|m_n(s)\r|_{H^1}^2 \l| m_n(s) \times u_n(s) \r|_{L^2}^2 \ \text{By}\ H^1\hookrightarrow L^{\infty} \\
			& \leq C \l| m_n(s) \times u_n(s) \r|_{L^2}^2 \sup_{s\in[0,t]} \l|m_n(s)\r|_{H^1}^2.
		\end{align*}
		Hence by the H\"older's inequality,
		\begin{align*}
			\mathbb{E}& \l[ \int_{0}^{T} \l|m_n(s) \times ( m_n(s) \times u_n(s))\r|_{L^2}^2 \, ds\r]^{p} \\
			& \leq C \mathbb{E} \l[  \sup_{s\in[0,T]} \l|m_n(s)\r|_{H^1}^2   \int_{0}^{T}\l| m_n(s) \times u_n(s)\r|_{L^2}^2 \, ds \r]^{p} \\
			&\leq C \mathbb{E} \l[\sup_{s\in[0,t]} \l|m_n(s)\r|_{H^1}^{2p}  \l( \int_{0}^{t}\l| m_n(s) \times u_n(s)\r|_{L^2}^2 \, ds \r)^{p}\r] \\
			& \leq C \l( \mathbb{E}   \sup_{s\in[0,T]} \l|m_n(s)\r|_{H^1}^{4p} \r)^{\frac{1}{2}} \l( \mathbb{E} \l[ \int_{0}^{T}\l| m_n(s) \times u_n(s)\r|_{L^2}^2 \, ds \r]^{2p} \r)^{\frac{1}{2}} .
		\end{align*}		
		Therefore there exists a constant $C_{4p}>0$ such that
		\begin{equation}
			\mathbb{E} \l[ \int_{0}^{T} \l| m_n(s) \times \big( m_n(s) \times u_n(s) \big) \r|_{L^2}^2 \, ds\r]^{p} \leq C_{4p}.
		\end{equation}
		\dela{
			Hence by the bound \eqref{bound 2}, there exists a constant $C>0$ such that
			\begin{equation}\label{bound on mn times mn times un}
				\mathbb{E}\l[\l(\int_{0}^{T}|m_n(s) \times (m_n(s) \times u_n(s))|_{L^{2}}^2 \, ds\r)^{p}\r]\leq C,\ \forall n\in\mathbb{N}.
			\end{equation}
		}
		This concludes the proof of Lemma \ref{bounds lemma 1}.
	\end{proof}
	Having shown some uniform energy estimates, we now proceed to show some more uniform bounds for $m_n$.

	\begin{lemma}\label{bounds lemma}
		For each $\gamma\in (0,\frac{1}{2})$ and $p\geq 2$, then there exists a constant $C>0$ such that for all $n\in\mathbb{N}$, the following estimate holds.
		\begin{equation}\label{W^gamma,p bound on m_n}
			\mathbb{E}\l[|m_n|^2_{W^{\gamma,p}\l(0,T;L^2\r)}\r]\leq C.
		\end{equation}

	\end{lemma}
	\begin{proof}[Proof of Lemma \ref{bounds lemma}]
		We show that each term on the right hand side of the approximate equation \eqref{definition of solution Faedo Galerkin approximation}  satisfy the above bound and hence the process $m$ satisfies the bound as well. \dela{The bounds established in the Theorem \ref{bounds lemma 1} along with the continuous embedding (By \cite{Simons} Corollary 18)}
		The bounds established in the Lemma \ref{bounds lemma 1} imply that each integrand (except the It\^o integrals) on the right hand side of \eqref{definition of solution Faedo Galerkin approximation} is uniformly bounded in the space $L^2(\Omega ; L^2(0,T;L^2))$, which implies that the integrals lie in the space $L^2(\Omega ; W^{1,2}(0,T;L^2))$. This combined with the continuous embedding (By Corollary 18, \cite{Simons})
		\begin{equation}\label{embedding W 1,2 into W gamma, p}
			W^{1,2}(0,T;L^2)\hookrightarrow W^{\gamma ,p}(0,T;L^2).
		\end{equation}
		concludes the inequality. What remain are the It\^o integrals.
		For those terms, we use Lemma \ref{Lemma reference W alpha p bound for stochastic integral}. Again, the bounds in Lemma \ref{bounds lemma 1} make sure that the required hypotheses are satisfied.
		\dela{\textbf{A change in notation:} In the proof for Lemma \ref{bounds lemma}, we use the following notation.
			\begin{align*}
				G_n(m_n(s)) =& P_n\b(m_n(s) \times h - \alpha \, \psi(|m_n(s)|_{L^{\infty}}) \psi(|P_n(m_n(s) \times h)|_{L^{\infty}}) \centerdot \\
				& \psi(|P_n(m_n(s) \times (m_n(s) \times h))|_{L^{\infty}}) m_n(s) \times \l(m_n(s \times h)\r)\b).
			\end{align*}
			The difference is that in the previous definition of $G_n$, the bump function $\psi$ was not involved. This notation is changed for convenience in writing and will only be restricted to the proof of Lemma \ref{bounds lemma}.
			
			The proof is mainly split into two steps. One step deals with the terms without the It\^o integral. The second step deals with the It\^o integral.\\
			\textbf{Step 1:}
			
			Define a function $f$ by
			\begin{equation*}
				f(t) := \int_{0}^{t} \l(m_n(s) \times \Delta m_n(s)\r) \, ds\ \mathbb{P}-\text{a.s.}
			\end{equation*}
			Then, by the bound \eqref{bound on m_n times Delta m_n}, there exists a constant $ C_1>0 $ such that
			\begin{align*}
				\mathbb{E} |f|_{L^2(0,T;L^2)}^2 &= \mathbb{E} \int_{0}^{T} \l| \int_{0}^{t} m_n(s) \times \Delta m_n(s) ds \r|_{L^2}^2 \,  dt \\
				& = \mathbb{E} \int_{0}^{T} \int_{\mathcal{O}} \l|\int_{0}^{t} m_n(s,x)\times \Delta m_n(s,x) \, ds\r|^2 dx dt \\
				& \leq \mathbb{E} \int_{0}^{T} \int_{\mathcal{O}} \int_{0}^{t} \l| m_n(s,x)\times \Delta m_n(s,x) \r|^2 \, ds \, dx \, dt \\
				& \leq \mathbb{E} \int_{0}^{T}  \int_{0}^{t} \int_{\mathcal{O}} \l| m_n(s,x)\times \Delta m_n(s,x) \r|^2 \, dx \, ds \, dt \\
				& \leq \mathbb{E} \int_{0}^{T} \int_{0}^{t} \l| m_n(s) \times \Delta m_n(s) \r|_{L^2}^2  \, ds \, dt \\
				& \leq \mathbb{E} \int_{0}^{T} \int_{0}^{T} \l| m_n(s) \times \Delta m_n(s) \r|_{L^2}^2  \, ds \, dt \\
				&\leq \int_{0}^{T} \mathbb{E} \int_{0}^{T} \l| m_n(s) \times \Delta m_n(s) \r |_{L^2}^2 \, ds \, dt \\
				& \leq \int_{0}^{T} C_1 \, dt \\
				& \leq TC_1.
			\end{align*}
			
			Note that $f$ is weakly differentiable on $[0,T]$ and
			\begin{equation*}
				\frac{d f}{dt}(t) = m_n(t) \times \Delta m_n(t),\ \text{for}\ t\in[0,T].
			\end{equation*}
			in the weak sense.

			Hence, by the bound \eqref{bound on m_n times Delta m_n}, there exists a constant $C_2$ such that
			\begin{align*}
				\mathbb{E} \left| \frac{d f(t)}{dt} \right|_{L^2(0,T;L^2)}^2 & = \mathbb{E} \left| \frac{d}{dt} \int_{0}^{t} m_n(s) \times \Delta m_n(s) \, ds \right|_{L^2(0,T;L^2)}^2 \\
				& = \mathbb{E} \int_{0}^{T} \left| \frac{d}{dt} \int_{0}^{t} m_n(s) \times \Delta m_n(s) ds \right|_{L^2}^2 \, dt \\
				& = \mathbb{E} \int_{0}^{T} \left| m_n(t) \times \Delta m_n(t)  \right|_{L^2}^2 \, dt \\
				&\leq C_2.
			\end{align*}
			Thus, combining the above two estimates, we get a constant $C$ such that
			\begin{equation}\label{W 1,2 bound on general function f}
				\mathbb{E}|f|^2_{W^{1,2}(0,T;L^2)} \leq C.
			\end{equation}
			Thus by the embedding \eqref{embedding W 1,2 into W gamma, p} and the above estimate \eqref{W 1,2 bound on general function f}, there exist constants $C_3,C>0$ such that
			\begin{equation}\label{W gamma, p bound on general function f}
				\mathbb{E}|f|_{W^{\gamma,p}(0,T;L^2)} \leq C_3 \mathbb{E}|f|_{W^{1,2}(0,T;L^2)} \leq C.
			\end{equation}
			Here we observe that to show the bound \eqref{W gamma, p bound on general function f}, the only property of the function $f$ that was used was that $f\in L^2(\Omega ; L^2(0,T ; L^2))$. Hence to show that each of the terms on the right hand side of the equation \eqref{definition of solution Faedo Galerkin approximation} (except the term with the It\^o integral) are uniformly bounded in $L^2(\Omega ; W^{\gamma , p}(0,T;(H^1)^\p) )$, it is sufficient to show that each of the integrands are uniformly bounded in the space $L^2(\Omega ; L^2(0,T ; L^2))$.
			
			The assumption on the initial data $m_0$ and the bounds obtained in Lemma \ref{bounds lemma 1} imply that the only terms that is left is the correction term $DG(m_n(s))(m_n(s))$. We now show that the term satisfies the required bounds.\\
			We recall that the term $DG(m_n(s))(m_n(s))$ is given by the following expression.
			\begin{align*}
				DG(m_n(s))(m_n(s)) =& \frac{1}{2}  P_n(P_n(m_n(s) \times h) \times h) \\
				&- \alpha
				\psi(|m_n(s)|_{L^{\infty}})\psi(|P_n(m_n(s) \times h)|_{L^{\infty}}) \psi(|P_n(m_n(s) \times (m_n(s) \times h))|_{L^{\infty}}) \centerdot \\
				& \quad P_n(P_n(m_n(s) \times (m_n(s) \times h)) \times h) \\
				& - \alpha \, \psi(|m_n(s)|_{L^{\infty}}) \psi(|P_n(m_n(s) \times h)|_{L^{\infty}}) \psi(|P_n(m_n(s) \times (m_n(s) \times h))|_{L^{\infty}})  \centerdot \\
				& \quad P_n(P_n(m_n(s) \times h) \times (m_n(s) \times h))  \\
				&- \alpha \,  P_n(m_n(s) \times (P_n(m_n(s) \times h) \times h)) \\
				& + \alpha^2  \psi^2(|m_n(s)|_{L^{\infty}})  \psi(|P_n(m_n(s) \times (m_n(s) \times h))|_{L^{\infty}})^2 \centerdot \\
				& \quad\psi(|P_n(m_n(s) \times h)|_{L^{\infty}})^2P_n(P_n(m_n(s) \times (m_n(s) \times h)) \times (m_n(s) \times h)) \\
				& + \alpha^2  P_n(m_n(s) \times (P_n(m_n(s) \times (m_n(s) \times h))\times h)) \\
				=& \sum_{i=1}^{6} C_i I_i(s).
			\end{align*}
			Owing to the triangle inequality, it is sufficient to show the bound for each $I_i, i=1,\dots, 6$.\\
			\textbf{Calculation for $I_1$.}
			
			Let $s\in[0,T]$ and $n\in\mathbb{N}$.
			\begin{align*}
				\l| P_n(P_n(m_n(s) \times h) \times h) \r|_{L^2} &\leq \l| P_n(m_n(s) \times h) \times h \r|_{L^2} \\
				& \leq \l| P_n(m_n(s) \times h) \r|_{L^2}\l| h \r|_{L^{\infty}} \\
				& \leq \l| m_n(s) \times h \r|_{L^2}\l| h \r|_{L^{\infty}} \\
				& \leq \l| m_n(s) \r|_{L^2}\l| h \r|^2_{L^{\infty}}.
			\end{align*}
			Hence there exists a constant $C>0$ such that
			\begin{align*}
				\mathbb{E}\int_{0}^{T} \l|I_3(t)\r|_{L^2}^2 \, dt = \mathbb{E} \int_{0}^{T} \l| P_n(P_n(m_n(s) \times h) \times h) \r|_{L^2}^2 \, ds & \leq \mathbb{E} \int_{0}^{T}\l| m_n(s) \r|^2_{L^2}\l| h \r|^4_{L^{\infty}} \, ds \\
				& \leq C(h) \mathbb{E}\sup_{s\in[0,T]} \l| m_n(s) \r|^2_{L^2} < C.
			\end{align*}
			The last inequality holds due to the bound \eqref{bound on m_n}.\\
			\textbf{Calculation for $I_2$.}
			Let $s\in[0,T]$ and $n\in\mathbb{N}$.
			\begin{align*}
				& | \psi(|m_n(s)|_{L^{\infty}})\psi(|P_n(m_n(s) \times h)|_{L^{\infty}}) \psi(|P_n(m_n(s) \times (m_n(s) \times h))|_{L^{\infty}}) \centerdot \\
				& \quad P_n(P_n(m_n(s) \times (m_n(s) \times h)) \times h) |_{L^2} \\
				& \leq \psi(|m_n(s)|_{L^{\infty}})\psi(|P_n(m_n(s) \times h)|_{L^{\infty}}) \psi(|P_n(m_n(s) \times (m_n(s) \times h))|_{L^{\infty}}) \centerdot\\
				& \quad \l| P_n(P_n(m_n(s) \times (m_n(s) \times h)) \times h) \r|_{L^2} \\
				& \leq \psi(|P_n(m_n(s) \times (m_n(s) \times h))|_{L^{\infty}}) \l| P_n(m_n(s) \times (m_n(s) \times h)) \times h \r|_{L^2} \\
				& \leq \psi(|P_n(m_n(s) \times (m_n(s) \times h))|_{L^{\infty}}) \l| P_n(m_n(s) \times (m_n(s) \times h))\r|_{L^{\infty}} \l| h \r|_{L^{\infty}} \\
				& \leq C(h).
			\end{align*}
			Hence there exists a constant $C>0$ such that
			\begin{align}
				\mathbb{E}\int_{0}^{T} \l|I_2(t)\r|_{L^2}^2\, dt \leq C.
			\end{align}
			\textbf{Calculation for $I_3$.}
			
			Let $s\in[0,T]$ and $n\in\mathbb{N}$.
			\begin{align*}
				&\big| \psi(|m_n(s)|_{L^{\infty}}) \psi(|P_n(m_n(s) \times h)|_{L^{\infty}}) \psi(|P_n(m_n(s) \times (m_n(s) \times h))|_{L^{\infty}})  \centerdot \\
				& \quad  P_n(P_n(m_n(s) \times h) \times (m_n(s) \times h)) \big|_{L^2} \\
				& \leq  \psi(|m_n(s)|_{L^{\infty}}) \psi(|P_n(m_n(s) \times h)|_{L^{\infty}}) \psi(|P_n(m_n(s) \times (m_n(s) \times h))|_{L^{\infty}})  \centerdot \\
				& \quad  \l| P_n(P_n(m_n(s) \times h) \times (m_n(s) \times h)) \r|_{L^2} \\
				& \leq  \psi(|m_n(s)|_{L^{\infty}}) \psi(|P_n(m_n(s) \times h)|_{L^{\infty}}) \psi(|P_n(m_n(s) \times (m_n(s) \times h))|_{L^{\infty}}) \centerdot \\
				& \quad  \l| P_n(m_n(s) \times h) \times (m_n(s) \times h) \r|_{L^2} \\
				& \leq \psi(|m_n(s)|_{L^{\infty}}) \l|P_n(m_n(s) \times h)\r|_{L^2} \psi(|P_n(m_n(s) \times h)|_{L^{\infty}}) \l| (m_n(s) \times h) \r|_{L^2} \\
				& \leq \psi(|m_n(s)|_{L^{\infty}}) \l| m_n(s) \times h \r|_{L^2} \psi(|P_n(m_n(s) \times h)|_{L^{\infty}}) \l| (m_n(s) \times h) \r|_{L^2} \\
				& \leq C \psi(|m_n(s)|_{L^{\infty}}) \l| m_n(s) \r|_{L^{\infty}} \l| h \r|_{L^{\infty}} \psi(|P_n(m_n(s) \times h)|_{L^{\infty}}) \l| (m_n(s) \times h) \r|_{L^{\infty}} \\
				&  \leq C(h).
			\end{align*}
			Hence there exists a constant $C>0$ such that
			\begin{align}
				\mathbb{E}\int_{0}^{T} \l|I_3(t)\r|_{L^2}^2 \, dt \leq C.
			\end{align}
			The term $I_4$ can be estimated similarly.\\
			\textbf{Calculation for $I_5$.}\\
			Let $s\in[0,T]$ and $n\in\mathbb{N}$.
			\begin{align*}
				&\b|\psi^2(|m_n(s)|_{L^{\infty}}) \psi(|P_n(m_n(s) \times h)|_{L^{\infty}})^2 \psi(|P_n(m_n(s) \times (m_n(s) \times h))|_{L^{\infty}})^2 \centerdot \\
				&\quad P_n(P_n(m_n(s) \times (m_n(s) \times h)) \times (m_n(s) \times h))\b|_{L^2} \\
				\leq& \psi^2(|m_n(s)|_{L^{\infty}}) \psi(|P_n(m_n(s) \times h)|_{L^{\infty}})^2 \psi(|P_n(m_n(s) \times (m_n(s) \times h))|_{L^{\infty}})^2 \centerdot \\
				& \quad \b|P_n(m_n(s) \times (m_n(s) \times h)) \times (m_n(s) \times h) \b|_{L^2} \\
				\leq& \psi^2(|m_n(s)|_{L^{\infty}})\l|(m_n(s) \times h) \r|_{L^{\infty}}\psi(|P_n(m_n(s) \times (m_n(s) \times h))|_{L^{\infty}})^2 \centerdot \\
				& \quad \b|P_n(m_n(s) \times (m_n(s) \times h)) \b|_{L^{\infty}} \\
				\leq& C(h).
			\end{align*}
			Hence there exists a constant $C>0$ such that
			\begin{align}
				\mathbb{E}\int_{0}^{T} \l|I_5(t)\r|_{L^2}^2 \, dt \leq C.
			\end{align}
			\textbf{Calculation for $I_6$.}
			
			Let $s\in[0,T]$ and $n\in\mathbb{N}$.
			\begin{align*}
				\l| P_n(m_n(s) \times (P_n(m_n(s) \times (m_n(s) \times h))\times h)) \r|_{L^2} &\leq \l| m_n(s) \times (P_n(m_n(s) \times (m_n(s) \times h))\times h) \r|_{L^2} \\
				& \leq  \l| m_n(s) \r|_{L^{\infty}} \l| P_n(m_n(s) \times (m_n(s) \times h)) \r|_{L^2}\l| h \r|_{L^{\infty}} \\
				& \leq  \l| m_n(s) \r|_{L^{\infty}} \l| m_n(s) \times (m_n(s) \times h) \r|_{L^2}\l| h \r|_{L^{\infty}} \\
				& \leq \l| m_n(s) \r|_{L^{\infty}}^2 \l| h \r|_{L^{\infty}}^2 \l| m_n(s)\r|_{L^2} \\
				& \leq C C(h) \l|m_n(0)\r|_{L^2} \l| m_n(s) \r|_{H^1}^2 \ \text{By}\ H^1\hookrightarrow L^{\infty}  \\
				& \leq C C(h) \l|m(0)\r|_{L^2} \sup_{s\in[0,T]} \l| m_n(s) \r|_{H^1}^2.
			\end{align*}
			Hence by the bounds \eqref{bound on m_n} and \eqref{bound on nabla m_n}, along with the assumption on the initial data $m(0)$, there exists a constant $C^\p >0$ such that
			\begin{align*}
				&\mathbb{E} \int_{0}^{T} \l| P_n(m_n(s) \times (P_n(m_n(s) \times (m_n(s) \times h))\times h)) \r|_{L^2} \\
				&\leq \l| m_n(s) \times (P_n(m_n(s) \times (m_n(s) \times h))\times h) \r|_{L^2}^2 \, ds \\ &
				\leq C^2 C(h)^2 \l|m(0)\r|_{L^2}^2\mathbb{E}  \sup_{s\in[0,T]} \l| m_n(s) \r|_{H^1}^4 \\
				& \leq  C^\p.
			\end{align*}
			That is
			\begin{align}
				\mathbb{E}\int_{0}^{T} \l|I_6(t)\r|_{L^2}^2 \, dt \leq C.
			\end{align}
			Combining the above calculations, we conclude that there exists a constant $C>0$ such that
			\begin{equation}
				\l| DG(m_n)(m_n) \r|_{L^2(\Omega ; L^2(0,T ; L^2))}^2 \leq C.
			\end{equation}
			This concludes the calculations for Step 1.\\
			\textbf{Step 2}
			
			Let us now consider the stochastic integral.
			
			The calculations presented here are for details. The Lemma \ref{Lemma reference W alpha p bound for stochastic integral}, see appendix of \cite{ZB+BG+TJ_Weak_3d_SLLGE} can be used to directly show the required result.

			Show that $G_n$(with the corresponding bump function) takes values in the required space.
			Define $F(t) := \int_{0}^{t} G_n(m_n(r)) dW(r)$ for $t\in[0,T]$. Let $p\geq 2$.

			
			\begin{align*}
				\mathbb{E}\l[\l|F(t) - F(s)\r|_{L^2}^p\r] &= \mathbb{E}\l[\l|\int_{s}^{t} G_n(m_n(r)) dW(r)\r|_{L^2}^p\r]\\
				& \leq \mathbb{E} \l[\sup_{q\in[s,t]}\l|\int_{s}^{q} G_n(m_n(r)) dW(r)\r|_{L^2}^p\r]\\
				& \leq \mathbb{E} \l[\l(\int_{s}^{t}|G_n(m_n(r))|_{L^2}\r)^p\r]\\
				& \leq \mathbb{E} \l[\l(\int_{s}^{t}|G_n(m_n(r))|_{L^2}^2\r)^{\frac{p}{2}}\r].
			\end{align*}
			Note that
			\begin{equation*}
				\int_{s\wedge t}^{s\vee t} \l|G_n(m_n(r))\r|_{L^2}^2dr = \l| \int_{0}^{t}\l|G_n(m_n(r))\r|_{L^2}^2 \, dr - \int_{0}^{s} \l|G_n(m_n(r))\r|_{L^2}^2 \, dr \r|.
			\end{equation*}.
			Here, $s \wedge t$ and $s \vee t$ denote $\min\{s,t\}$ and $\max\{s,t\}$ respectively.
			
			Now,
			\begin{align*}
				\mathbb{E} \int_{0}^{T}\int_{0}^{T} \frac{|F(t) - F(s)|_{L^2}^p}{|t-s|^{1+\gamma p}} dt ds &=  \int_{0}^{T}  \int_{0}^{T} \frac{\mathbb{E} [|F(t) - F(s)|_{L^2}^p] }{|t-s|^{1+\gamma p}} dt ds \\
				& \leq C_p \int_{0}^{T}  \int_{0}^{T} \frac{\mathbb{E} [(\int_{s \wedge t}^{s \vee t}|G_n(m_n(r))|_{L^2}^2 dr )^{\frac{p}{2}}]}{|t-s|^{1+\gamma p}} dt ds \\
				& = C_p\mathbb{E} \int_{0}^{T}  \int_{0}^{T} \frac{ (\int_{s \wedge t}^{s \vee t}|G_n(m_n(r))|_{L^2}^2 dr )^{\frac{p}{2}}}{|t-s|^{1+ (2\gamma) (\frac{p}{2})}} dt ds.
			\end{align*}
			In the above calculation, the second step results as follows:
			\begin{align*}
				\mathbb{E} \left[\left|\int_{s \wedge t}^{s \vee t}G_n(m_n(r)) \, dW(r) \right|_{L^2}^{p}\right] \leq
				\mathbb{E} \left[\sup_{ q \in [s\wedge t , s \vee t]}\left| \int_{s \wedge t}^{q}G_n(m_n(r)) \, dW(r) \right|_{L^2} ^{p}\right].
			\end{align*}
			Using the Burkholder-Davis-Gundy inequality, we get
			\begin{align*}
				\mathbb{E} \left[\sup_{ q \in [s\wedge t , s \vee t]}\left| \int_{s \wedge t}^{q}|G_n(m_n(r))|_{L^2}^2 \, dr \right| ^{\frac{p}{2}}\right] \leq C_p \mathbb{E} \left[\l(\int_{s \wedge t}^{s \vee t}|G_n(m_n(r))|_{L^2}^2 \, dr \r)^{\frac{p}{2}}\right].
			\end{align*}
			Here, we split the calculations into 2 parts: $p=2$ and $p>2$.\\
			\textbf{Part 1:} Let $p=2$.
			\begin{align*}
				&\int_{0}^{T}  \int_{0}^{T} \frac{ \mathbb{E} \l[ \l(\int_{s \wedge t}^{s \vee t}|G_n(m_n(r))|_{L^2}^2 \, dr \r)\r]}{|t-s|^{1+ 2\gamma}} \, dt \, ds	 \\
				& \leq  \int_{0}^{T}  \int_{0}^{T} \frac{ \mathbb{E} \l[ \l(\int_{0}^{T}|G_n(m_n(r))|_{L^2}^2 \, dr \r)\r]}{|t-s|^{1+ 2\gamma}} \, dt \, ds \\
				& \leq C\l(T\r) \mathbb{E} \l[ \l(\int_{0}^{T}|G_n(m_n(r))|_{L^2}^2 \, dr \r)\r] \int_{0}^{T}  \int_{0}^{T} \frac{ 1 }{|t-s|^{1+ 2\gamma}} \, dt \, ds.
			\end{align*}
			Now,
			\begin{align*}
				\int_{0}^{T} \int_{0}^{T} \frac{ 1 }{|t-s|^{1+ 2\gamma}} \, dt \, ds &=  \int_{0}^{T} \left( \lim_{\varepsilon \rightarrow 0} \int_{0}^{s-\varepsilon} \frac{ 1 }{(s-t)^{1+ 2\gamma}} \, dt \, ds + \int_{s + \varepsilon}^{T} \frac{ 1 }{(t-s)^{1+ 2\alpha}} dt  \right) \, ds \\
				& = \int_{0}^{T} \frac{1}{2\gamma} \left( \lim_{\varepsilon \rightarrow 0}\frac{-1}{\varepsilon^{2\gamma}} + \frac{1}{s^{2\gamma}} - \frac{1}{(T-s)^{2\gamma}} + \frac{1}{\varepsilon^{2\gamma}} \right) \, ds \\
				&= \int_{0}^{T} \frac{-1}{2\gamma}\left(\frac{1}{(T-s)^{2\gamma}} - \frac{1}{s^{2\gamma}}\right) \, ds .
			\end{align*}
			
			$\gamma < \frac{1}{2}$ implies that $2\gamma < 1$ and hence the right hand side of the above equation is finite.
			
			Thus,
			
			\begin{align*}
				\int_{0}^{T}  \int_{0}^{T} \frac{ \mathbb{E} \l[ \l(\int_{s \vee t}^{s \wedge t}\l|G_n(m_n(r))\r|_{L^2}^2 dr \r)\r]}{\l|t-s\r|^{1+ 2\gamma}} \, dt \, ds	 & \leq C\l(T,\gamma\r) \mathbb{E} \l[ \l(\int_{0}^{T}|G_n(m_n(r))|_{L^2}^2 \, dr \r)\r] \leq C.
			\end{align*}
			\textbf{Part 2:} Now let $p>2$. Define
			\begin{equation*}
				\mathbb{F}(t) := \int_{0}^{t}|G_n(m_n(r))|_{L^2}^2 \, dr \ \text{for}\ t\in[0,T].
			\end{equation*}
			\begin{align*}
				&	\mathbb{E} \int_{0}^{T}  \int_{s}^{T} \frac{ \l(\int_{s}^{t}\l|G_n(m_n(r))\r|_{L^2}^2 \, dr \r)^{\frac{p}{2}}}{\l|t-s\r|^{1+ (2\gamma) \l(\frac{p}{2}\r)}} \, dt \, ds \\
				&= \mathbb{E} \int_{0}^{T}  \int_{s}^{T} \frac{ \l(\int_{0}^{t}\l|G_n(m_n(r))\r|_{L^2}^2 \, dr - \int_{0}^{s}\l|G_n(m_n(r))\r|_{L^2}^2 dr\r)^{\frac{p}{2}}}{|t-s|^{1+ (2\gamma) \l(\frac{p}{2}\r)}} \, dt \, ds\\
				& \leq \mathbb{E} \l|\mathbb{F}\r|_{W^{2\gamma, \frac{p}{2}}(0,T;\mathbb{R})}^{\frac{p}{2}} \\
				& \leq \mathbb{E} \l|\mathbb{F}\r|_{W^{1, \frac{p}{2}}(0,T;\mathbb{R})}^{\frac{p}{2}} \\
				& \leq \mathbb{E} \left|\frac{d\mathbb{F}}{dt} \right|_{L^{\frac{p}{2}}(0,T;\mathbb{R})}^{\frac{p}{2}} + \mathbb{E} \l|\mathbb{F}\r|_{L^{\frac{p}{2}}(0,T;\mathbb{R})}^{\frac{p}{2}}.
			\end{align*}
			Note that as before, the $\frac{d\mathbb{F}}{dt}$ should be understood in the weak sense.
			\begin{align*}
				\mathbb{E} \left| \mathbb{F}\right|^{\frac{p}{2}}_{L^{\frac{p}{2}}(0,T;\mathbb{R})} &= \mathbb{E} \int_{0}^{T}|\int_{0}^{t} |G_n(m_n(r))|_{L^2}^2 dr|^{\frac{p}{2}} dt \\
				& \leq \mathbb{E} \int_{0}^{T}|\int_{0}^{T} |G_n(m_n(r))|_{L^2}^2 dr|^{\frac{p}{2}} dt \\
				& \leq \int_{0}^{T} C(h) dt \leq C.
			\end{align*}
			\begin{align*}
				\mathbb{E} \left|\frac{d \mathbb{F}}{dt}\right|_{L^{\frac{p}{2}}(0,T;\mathbb{R})}^{\frac{p}{2}}
				& \leq \mathbb{E} \left| |G_n(m_n(t))|_{L^2}^2 - |G_n(m_n(t))|_{L^2}^2   \right|_{L^{\frac{p}{2}}(0,T;\mathbb{R})}^{\frac{p}{2}} \\
				& \leq C.
			\end{align*}
			Thus, there exists a constant $C_1>0$ such that
			\begin{equation*}
				\mathbb{E} \int_{0}^{T}  \int_{s}^{T} \frac{ \l(\int_{s}^{t}|G_n(m_n(r))|_{L^2}^2 \, dr \r)^{\frac{p}{2}}}{|t-s|^{1+ (2\gamma) \l(\frac{p}{2}\r)}} \, dt \, ds \leq C_1.
			\end{equation*}
			Thus, there exists a constant $C>0$ such that
			\begin{equation*}
				\mathbb{E}\l|F\r|^{\frac{p}{2}}_{W^{\gamma,p}(0,T;L^2)} \leq C.
			\end{equation*}
			Thus, each term on the right hand side of the equality \eqref{definition of solution Faedo Galerkin approximation} satisfies the required bound. Hence there exists a constant $C>0$ such that

			\begin{equation}
				\mathbb{E}[|m_n|^2_{W^{\gamma,p}(0,T;L^2)}]\leq C.
			\end{equation}
			This concludes the proof of Lemma \ref{bounds lemma}.}

	\end{proof}
	We now recall a notation that will be used in the following sections.
	The space $L^2_w(0,T;L^2)$ denotes the space $L^2(0,T;L^2)$ endowed with the weak topology.

	\begin{lemma}\label{tightness lemma}
		Let $p\geq 1$ and $q\geq 2$. The set of laws $\l(\mathcal{L}(m_n,u_n)\r)_{n\in\mathbb{N}}$ is tight on the space $L^p(0,T;L^q)\cap C(0,T; L^2) \times L^2_w(0,T;L^2)$. Also, the law of the Wiener process $W$ is tight on the space $C([0,T] ; \mathbb{R})$.
		
	\end{lemma}
	
	\begin{proof}
		First we show that the laws of $m_n$ are   concentrated inside a ball in the space $L^p(0,T;H^1)\cap W^{\gamma,p}(0,T ; L^2)$  .\dela{, which is compactly embedded into the space $L^p(0,T;L^q)\cap C(0,T; L^2)$. The compact embedding is justified later.}
		
		Let $R\geq 0$. By Chebyshev's inequality and the bounds established in  Lemma \ref{bounds lemma 1 without p} and Lemma \ref{bounds lemma}\dela{(in particular for $p=\infty$)}, there exists a constant $C>0$ such that
		\dela{	\begin{align*}
				\mathbb{P}\l( |m_n|_{L^p(0,T;H^1)\cap W^{\gamma,p}(0,T;L^2)} > R \r) &\leq 	\mathbb{P}\l(|m_n|^2_{L^p(0,T;H^1)} > \frac{R}{2}\r) + 	\mathbb{P}\l(|m_n|^2_{ W^{\gamma,p}(0,T;L^2)} > \frac{R}{2}\r) \\
				(\text{By Chebyshev's inequality})\ & \leq \frac{4}{R^2}\mathbb{E}\l|m_n\r|^2_{L^p(0,T;H^1)} + \frac{4}{R^2}\mathbb{E}\l|m_n\r|^2_{W^{\gamma,p}(0,T;L^2)} \\
				(\text{By Lemma \ref{bounds lemma 1 without p}, Lemma \ref{bounds lemma}})\ & \leq \frac{C}{R^2}.
			\end{align*}
		}
		\begin{align}\label{equation 1 tightness lemma}
			\mathbb{P}\l( |m_n|_{L^{\infty}(0,T;H^1)\cap W^{\gamma,p}(0,T;L^2)} > R \r) \nonumber&\leq 	\mathbb{P}\l(|m_n|_{L^{\infty}(0,T;H^1)} > \frac{R}{2}\r) + 	\mathbb{P}\l(|m_n|_{ W^{\gamma,p}(0,T;L^2)} > \frac{R}{2}\r) \\
			(\text{By Chebyshev's inequality})\ \nonumber& \leq \frac{4}{R^2}\mathbb{E}\l|m_n\r|^2_{L^{\infty}(0,T;H^1)} + \frac{4}{R^2}\mathbb{E}\l|m_n\r|^2_{W^{\gamma,p}(0,T;L^2)} \\
			(\text{By Lemma \ref{bounds lemma 1 without p}, Lemma \ref{bounds lemma}})\ & \leq \frac{C}{R^2}.
		\end{align}
		
		The right hand side of the above inequality, and hence the left hand side can be made as small as desired by choosing $R$ large enough.\\
		\dela{By Lemma \ref{compact embedding into space of continuous functions}, the following compact embedding holds for $\gamma p >1$.
			\begin{equation}
				W^{\gamma,p}(0,T;L^2) \hookrightarrow C(0,T; L^2).
		\end{equation}}		
		\dela{Refine this argument. Where does the $L^{\infty}(0,T;H^1)$ norm come into the picture? It is possibly required for the compact embedding into the space $C([0,T];L^2)$. But that should be visible explicitly.\\}
		For $ 0 \leq \beta < \frac{1}{2}$, by Lemma \ref{compact embedding of intersection 1}, the space $L^p(0,T;H^1)\cap W^{\gamma,p}(0,T;L^2)$ is compactly embedded into the space $L^p(0,T;X^{\beta})$.
		Hence the sequence of laws of $(m_n)$ is tight on the space $L^p(0,T;X^{\beta})$. For $q\geq 1$, we can find $\beta < \frac{1}{2}$ such that $X^{\beta}$ is continuously embedded into $L^q$, 
		see Theorem 1.6.1 in \cite{HenryGeometricTheoryofSemilinearParabolicEquations}, implying that the space $L^p(0,T; X^{\beta})$ is continuously embedded into the space $L^p(0,T;L^q)$. Since tightness of laws is preserved by continuous maps, the sequence of laws $(\mathcal{L}(m_n))_{n\in\mathbb{N}}$ is tight on the space $L^p(0,T;L^q)$.\\
		By combining the above argument with \dela{ Lemma \ref{compact embedding into space of continuous functions} }Lemma \ref{compact embedding of intersection 1} and Lemma \ref{compact embedding of intersection 2} in the Appendix \ref{Section Some embeddings}, the following compact embedding holds for $\gamma p >1$.
		\begin{equation}
			L^{\infty}(0,T;H^1) \cap W^{\gamma,p}(0,T;L^2) \hookrightarrow L^p(0,T; L^q) \cap C(0,T; L^2).
		\end{equation}
		Combining the above mentioned results, we can conclude that the sequence $(\mathcal{L}(m_n))_{n\in\mathbb{N}}$ is tight
		on the space $L^p(0,T;L^q)\cap C(0,T;L^2)$.\\
		\textbf{A brief justification for tightness:} Let
		\begin{equation}
			B_R:= \l\{v\in L^{\infty}(0,T;H^1)\cap W^{\gamma,p}(0,T;L^2): \l| v \r|_{L^{\infty}(0,T;H^1)\cap W^{\gamma,p}(0,T;L^2)} \leq R\r\}.
		\end{equation}
		Then for each $R\geq0$, $B_R$ is a closed ball in $L^{\infty}(0,T;H^1)\cap W^{\gamma,p}(0,T;L^2)$. Therefore $B_R$ has compact closure in $L^p(0,T;L^q)\cap C([0,T];L^2)$.\\
		For $n\in\mathbb{N}$,
		\begin{align*}
			\l(\mathcal{L}(m_n)\r)(B_R) = \mathbb{P}\l(\omega\in\Omega:\l|m_n(\omega)\r|_{L^{\infty}(0,T;H^1)\cap W^{\gamma,p}(0,T;L^2)} \leq R\r).
		\end{align*}
		The right hand side will be denoted by $\mathbb{P}(B_R)$.
		Let $\tilde{B}_R$ denote the closure of $B_R$ in $L^p(0,T;L^q)\cap C([0,T];L^2)$. Similarly, $$\l(\mathcal{L}(m_n)\r)(\tilde{B}_R) = \mathbb{P}\l(\omega\in \Omega: m_n(\omega)\in \tilde{B}\r).$$
		The right hand side will be denoted by $\mathbb{P}(\tilde{B}_R)$. Note that $\tilde{B}_R$ is a compact subset of $L^p(0,T;L^q) \cap C([0,T];L^2)$.
		By \eqref{equation 1 tightness lemma}, we have
		\begin{align*}
			1 - \mathbb{P}(\tilde{B}_R) \leq 1 - \mathbb{P}(B_R) \leq \frac{C}{R^2}.
		\end{align*}
		
		Let $\varepsilon$ be given. Choosing $R$ large enough in \eqref{equation 1 tightness lemma} $\l(R>\sqrt{\frac{C}{\varepsilon}}\r)$, we get
		\begin{align*}
			1 - \mathbb{P}(\tilde{B}_R) < \varepsilon,
		\end{align*}
		giving the required tightness of laws.  \\

		\dela{
			Let $\varepsilon>0$ be given. In order to show that the sequence of laws $\l(\mathcal{L}(m_n)\r)_{n\in\mathbb{N}}$ is tight on the space $L^p(0,T;L^q)\cap C(0,T; L^2)$, we must show that there exists a compact set $K_{\varepsilon}\subset L^p(0,T;L^q)\cap C(0,T; L^2)$ such that
			\begin{equation*}
				\inf_{n\in\mathbb{N}} \l(\mathcal{L}(m_n)\r)(K_{\varepsilon}) \geq 1 - \varepsilon.
			\end{equation*}
			Towards this, let $$B_R = \{v\in L^p(0,T;H^1)\cap W^{\gamma,p}(0,T;L^2) : \l| v \r|_{L^p(0,T;H^1)\cap W^{\gamma,p}(0,T;L^2)} \leq R\}.$$\\
			Choose $R$ large enough such that $\frac{C}{R^2} < \varepsilon.$ Fix this $R$ and consider the set $B_r$. Let $K_{\varepsilon}$ denote the closure of $B$ in $L^p(0,T;L^q)\cap C(0,T; L^2)$. The arguments that follow show that $B$ is a compact subset of $L^p(0,T;L^q)\cap C(0,T; L^2)$. Also,
			\begin{align*}
				\l(\mathcal{L}(m_n)\r) (K_{\varepsilon}) \geq 1 - \l(\mathcal{L}(m_n)\r)(B_R^c) \geq 1 - \varepsilon.
			\end{align*}
			For $ 0 \leq \beta < \frac{1}{2}$, by Lemma \ref{compact embedding of intersection 1}, the space $L^p(0,T;H^1)\cap W^{\gamma,p}(0,T;L^2)$ is compactly embedded into the space $L^p(0,T;X^{\beta})$.\\
			Hence the sequence of laws of $(m_n)_{n\in\mathbb{N}}$ is tight on the space $L^p(0,T;X^{\beta})$. For $q\geq 1$, we can find $\beta < \frac{1}{2}$ such that $X^{\beta}$ is continuously embedded into $L^q$. Hence the sequence of laws $(\mathcal{L}(m_n))_{n\in\mathbb{N}}$ is tight on the space $L^p(0,T;L^q)$.\\
			By Lemma \ref{compact embedding into space of continuous functions} in the Appendix, the following compact embedding holds for $\gamma p >1$.
			\begin{equation}
				W^{\gamma,p}(0,T;L^2) \hookrightarrow C(0,T; L^2).
			\end{equation}
			Combining the above mentioned results, we can conclude that the sequence $(\mathcal{L}(m_n))_{n\in\mathbb{N}}$ is tight
			on the space $L^p(0,T;L^q)\cap C(0,T;L^2)$. Hence the tightness follows.\\	
		}
		For showing the tightness of the laws of $u_n$, we observe that by Chebyshev's inequality and the assumption on $u$ there exists a constant $C>0$ such that,
		\begin{align*}
			\mathbb{P}(|u_n|_{L^2(0,T;L^2)} > R) &\leq \frac{1}{R^2}\mathbb{E}(|u_n|_{L^2(0,T;L^2)}^2) 
			\leq \frac{C}{R^2}.
		\end{align*}
		The right hand side of the above inequality, and hence the left hand side can be made as small as desired by choosing $R$ large enough.
		\dela{ Arguing similarly as done above, let $R$ be large enough so that $\frac{C}{R^2} < \varepsilon$. }
		By the Banach Alaoglu theorem, a closed ball in the space $L^2(0,T;L^2)$ has compact closure in the space $L^2_w(0,T;L^2)$.
		\dela{
			Hence the closure of
			\begin{equation*}
				B^\p_R = \{v\in L^2(0,T;L^2): \l| v \r|_{L^2(0,T;L^2)} \leq R\}
			\end{equation*}
			in the space $L^2_w(0,T;L^2)$ is compact. Let $K_{\varepsilon}$ be this closure. Hence
			\begin{align*}
				\l(\mathcal{L}(u_n)\r) (K_{\varepsilon}) \geq 1 - \l(\mathcal{L}(u_n)\r)(B_R^c) \geq 1 - \varepsilon.
			\end{align*}
		}
		Thus the sequence of laws $(\mathcal{L}(u_n))_{n\in\mathbb{N}}$ is tight on the space $L^2_w(0,T;L^2)$.\\
		$\mathcal{L}(W)$ is a probability measure on $C([0,T];\mathbb{R})$ and hence is tight on $C([0,T];\mathbb{R})$.
		This concludes the proof of Lemma \ref{tightness lemma}.
	\end{proof}
	In the sections that follow we choose and fix $p=q=4$.

	\section{Proof of Theorem \ref{Theorem Existence of a weak solution}: Proof of the existence of a weak martingale solution}\label{Section Proof of existence of a solution}
	
	We have so far shown the tightness of laws. Since the space $L^4(0,T;L^4)\cap C([0,T];L^2)\times C([0,T];\mathbb{R}) \times L^2_w(0,T;L^2)$ is not a metric space, we cannot proceed by applying the Prokhorov Theorem (Theorem II.6.7, \cite{Parthasarathy_1967_Book_ProbabilityMeasuresOnMetricSpaces_Book}) followed by the Skorohod Theorem (Theorem 6.7). We, instead, obtain convergence by using the Jakubowski version of the Skorohod Theorem.

	\begin{proposition}\label{Use of Skorohod theorem}
		There exists a sequence $(m_n^{\prime},W_n^{\prime},u^{\prime}_n)$ of $L^4(0,T;L^4)\cap C([0,T];L^2)\times C([0,T];\mathbb{R})\times L^2_w(0,T;L^2)$-valued random variables defined on a probability space $(\Omega^{\prime},\mathcal{F}^{\prime},\mathbb{P}^{\prime})$ such that for each $n\in\mathbb{N}$ $(m_n^{\prime},W_n^{\prime},u^{\prime}_n)$ and $(m_n,W,u_n)$ have the same laws on $L^4(0,T;L^4)\cap C([0,T];L^2)\times C([0,T];\mathbb{R})\times L^2_w(0,T;L^2)$. Further, there exists a random variable $(m^{\prime},W^{\prime},u^{\prime})$ defined on $(\Omega^{\prime},\mathcal{F}^{\prime},\mathbb{P}^{\prime})$ such that the law of $(m^{\prime},W^{\prime},u^{\prime})$ is equal to $\nu$ on $L^4(0,T;L^4)\cap C([0,T];L^2)\times C([0,T];\mathbb{R})\times L^2_w(0,T;L^2)$. Moreover, the following convergences hold $\mathbb{P}^{\prime}$-a.s. as $n$ goes to infinity.
		
		\begin{equation}
			m_n^{\prime}\rightarrow m^{\prime}\ \text{in}\ L^4(0,T;L^4)\cap C([0,T];L^2),
		\end{equation}
		
		\begin{equation}
			W_n^{\prime}\rightarrow W^{\prime}\ \text{in}\  C([0,T];\mathbb{R}),
		\end{equation}
		
		\begin{equation}\label{Equation weak convergence un prime to u prime pointwise}
			u_n^{\prime}\rightarrow u^{\prime}\ \text{in}\ L^2_w(0,T;L^2).
		\end{equation}
		
	\end{proposition}
	
	\begin{proof}
		This result follows from the Jakubowski version of the Skorohod Theorem, see Theorem 3.11 in \cite{ZB+EM}, or Theorem A.1 in \cite{ZB+MO}.

		\dela{The details mentioned following this line need not be mentioned in the paper. It is good to have details, but this can simply be referred to the paper \cite{ZB+EM}.}			
		For using the result, it is required that the space $\mathcal{X} = L^4(0,T;L^4)\cap C([0,T];L^2)\times C([0,T];\mathbb{R})\times L^2_w(0,T;L^2)$ satisfy certain hypothesis. That is, $\mathcal{X}$ must be a topological space such that there exists a sequence $\{f_n\}_{n\in\mathbb{N}}$ of continuous functions, $f_n:\mathcal{X}\to\mathbb{R}$, that separates the points of $\mathcal{X}$. A proof for this can be done similarly to the proof of Corollary 3.2 in \cite{ZB+EM}. For the sake of completion, we present some details here.\\			
		It is now sufficient to show that each space in the Cartesian product $\mathcal{X}$ satisfies the required hypothesis. $L^4(0,T;L^4)\cap C([0,T];L^2)$ and $C([0,T];\mathbb{R})$ are separable, complete metric spaces and hence satisfy the hypothesis. Regarding the space $L^2_w(0,T;L^2)$, let $\{w_n\}_{n\in\mathbb{N}}$ be a countable dense set in $L^2(0,T;L^2)$. Such a set exists since the space $L^2(0,T;L^2)$ is separable. Now consider the following functions for $n\in\mathbb{N}$.
		
		\begin{equation*}
			f_n(v) = \int_{0}^{T} \l\langle v(t) , w_n(t) \r\rangle_{L^2} \, dt.
		\end{equation*}

		We now claim that the sequence of functions defined above separates points of $L^2_w(0,T ; L^2).$ Towards that, let us choose and fix $v_1,v_2\in L^2_w(0,T;L^2)$ such that $f_n(v_1) = f_n(v_2)$ for each $n\in\mathbb{N}$. The set $\{w_n\}_{n\in\mathbb{N}}$ is dense in $L^2(0,T;L^2)$. Therefore $f_n(v_1) = f_n(v_2)$ for each $n\in\mathbb{N}$ implies that $v_1 = v_2$. Therefore if $v_1 \neq v_2$ then there exists at least one $n\in\mathbb{N}$ such that $f_n(v_1) \neq f_n(v_2)$. Hence the functions $f_n$ separate points.

	\end{proof}

	\begin{remark}
		As a consequence of Proposition \ref{Use of Skorohod theorem}, we have that the laws of $u^\p$ and $u$ are equal. To see this, first, recall that $u_n = P_n(u)$ and that $u_n\to u,\ \mathbb{P}$-a.s. in $L^2(0,T;L^2)$. Proposition \ref{Use of Skorohod theorem} implies that the processes $u_n$ and $u^\p$ have the same laws on the space $L^2_w(0,T;L^2)$. This, combined with the $\mathbb{P}^\p$-a.s. convergence of $u^\p_n$ to $u^\p$ gives the desired result.
	\end{remark}

	As a corollary of Proposition \ref{Use of Skorohod theorem}, we have the following.
	\begin{corollary}\label{Corollary m prime continuous in t with values in L2}
		\begin{align*}
			\m\in C([0,T] ; L^2),\ \mathbb{P}^{\prime}-a.s.
		\end{align*}
	\end{corollary}
	\begin{proof}[Proof of Corollary \ref{Corollary m prime continuous in t with values in L2}]
		The proof follows from the $\mathbb{P}^{\prime}$-a.s. convergence of the processes $\m_n$ to $\m$ in the space $C([0,T] ; L^2)$.
	\end{proof}

	\begin{remark}\label{Remark m prime, u prime progressively measurable}
		The processes $\m,\u$ obtained in Proposition \ref{Use of Skorohod theorem} are Borel measurable. Let the filtration $\mathbb{F}^{\prime} = \{\mathcal{F}_t^{\prime}\}_{t\in[0,T]}$ be defined by
		
		\begin{align*}
			\mathcal{F}_t^{\prime} = \sigma \l\{ \m(s) , \u(s) , W^{\prime}(s) : 0 \leq s \leq t\r\}.
		\end{align*}
		Hence $\m , \u$ are $\mathbb{F}^{\prime}$ adapted. Thus, the processes $\m$ and $\u$ have progressively measurable modifications, see Proposition 1.12, \cite{Karatzas+Steven_BrownianMotionStochacsticCalculus}. From now on, these progressively measurable modifications will be considered.
	\end{remark}

	\dela{
		
		Note that for each $n\in \mathbb{N}$ the following embeddings are continuous. $H_n\hookrightarrow L^4$, $H_n\hookrightarrow L^2$. Thus,
		\[
		C([0,T];H_n)\hookrightarrow L^4(0,T;L^4)\cap C([0,T];L^2).
		\] The identity mapping is thus continuous, injective.
		By Kuratowski's theorem, see Theorem 1.1, \cite{Vakhania_Probability_distributions_on_Banach_spaces}, the Borel subsets of $C([0,T];H_n)$ are Borel subsets of $L^4(0,T;L^4)\cap C([0,T];L^2)$. Hence $C([0,T];H_n)$ is a Borel subset of $L^4(0,T;L^4)\cap C([0,T];L^2)$.
		Thus, $\mathbb{P}\{m_n\in C([0,T];H_n)\}=1$ implies that $\mathbb{P}\{\m_n\in C([0,T];H_n)\}=1$. Hence we may assume that $m_n^{\prime}$ and that the laws of $m_n$ and $m_n^{\prime}$ are equal on $C([0,T];H_n)$.\\		
		Similar can be said about the laws of the processes $u_n$ and $u_n^{\prime}$ on the space $L^2(0,T;L^2)$.
	}
	\begin{remark}\label{same bounds remark}
		\dela{
			Note that for each $n\in \mathbb{N}$ the following embeddings are continuous. $H_n\hookrightarrow L^4$, $H_n\hookrightarrow L^2$. Thus,
			\[
			C([0,T];H_n)\hookrightarrow L^4(0,T;L^4)\cap C([0,T];L^2).
			\] The identity mapping is thus continuous, injective.
		}
		By the Kuratowski Theorem (Lemma \ref{Lemma Kuratowski}) see Theorem 1.1, \cite{Vakhania_Probability_distributions_on_Banach_spaces}, the Borel subsets of $C([0,T];H_n)$ are Borel subsets of $L^4(0,T;L^4)\cap C([0,T];L^2)$.\dela{
			Thus, $\mathbb{P}\{m_n\in C([0,T];H_n)\}=1$ implies that $\mathbb{P}\{\m_n\in C([0,T];H_n)\}=1$.} Hence we can assume that $m_n^{\prime}$ and that the laws of $m_n$ and $m_n^{\prime}$ are equal on $C([0,T];H_n)$.\\		
		Similar can be said about the laws of the processes $u_n$ and $u_n^{\prime}$ on the space $L^2(0,T;L^2)$.
	\end{remark}
	The processes $m_n^{\prime}$ satisfy the same estimates that are satisfied by $m_n$, $n\in\mathbb{N}$. In particular, we have the following proposition.
	\begin{proposition}\label{Proposition bounds on m_n prime}
		
		For each $p\geq 1$, there exists a constant $C>0$ such that for all $n\in\mathbb{N}$, the following bounds hold:
		\begin{equation}\label{bound on m_n prime}
			|\m_n(t)|_{L^2}^2 = |\m_n(0)|_{L^2}^2 \ \text{for each}\ t\in[0,T]\  \mathbb{P}-a.s.
		\end{equation}
		\begin{equation}\label{bound on nabla m_n prime}
			\mathbb{E}^{\prime}\l[\sup_{s\in[0,T]}| \m_n(s)|_{H^1}^{2p}\r]\leq C,
		\end{equation}
		\begin{align}\label{bound on mn prime times Delta mn prime}
			\mathbb{E}^{\prime}\l(\int_{0}^{T} |\m_n(s)\times \Delta \m_n(s)|_{L^2}^2    \,    ds \r)^p \leq C,
		\end{align}
		\begin{align}\label{bound on mn prime times mn prime times Delta mn prime}
			\mathbb{E}^{\prime}\l[\l(\int_{0}^{T}|\m_n(s)\times(\m_n(s)\times\Delta \m_n(s))|_{L^{2}}^2\r)^{p}\r]\leq C,
		\end{align}
		\begin{equation}\label{bound on mn prime times un prime}
			\mathbb{E}^{\prime}\l[\l(\int_{0}^{T}|\m_n(s)\times \u_n(s)|_{L^{2}}^2 \, ds\r)^{p} \, ds \r]\leq C,
		\end{equation}
		\begin{equation}\label{bound on mn prime times mn prime times un prime}
			\mathbb{E}^{\prime}\l[\l(\int_{0}^{T}|\m_n(s) \times (\m_n(s) \times \u_n(s))|_{L^{2}}^2 \, ds\r)^{p}\r]\leq C.
		\end{equation}

	\end{proposition}
	\begin{proof}
		The proposition follows from Lemma \ref{bounds lemma 1} and Remark \ref{same bounds remark}.
		
		\dela{
			The map $$v\mapsto \l|v\r|_{H^1}$$ when extended to $L^2$ is lower semicontinuous, and hence measurable.(A proof for the statement is provided in a later section.) Thus the set $\{v\in L^2 : \l|v\r|_{H^1} \leq c\}$ is a Borel subset of $L^2$, for any constant $c\in\mathbb{R}$. Hence $\{v\in H_n : \l|v\r|_{H^1} \leq c\}$ is a Borel subset of $H_n$. Also, $\m_n \times \Delta \m_n$, $\m_n \times \l( \m_n \times \Delta \m_n \r)$ etc. are measurable functions of $m$. This along with Remark \ref{same bounds remark} and the bounds from the Lemma \ref{bounds lemma 1} proves the proposition.
		}
	\end{proof}

	By the bounds established in Proposition \ref{Proposition bounds on m_n prime} above, in particular for $p = 1$, there exists \dela{measurable }processes \dela{on $[0,T]\times \Omega^{\prime}$, }$\dela{Y,} \mathbb{Y}, \dela{Z,} \mathbb{Z} \in L^{2}(\Omega^{\prime}:L^2(0,T;L^2))$ such that
	
	\begin{equation}\label{equation weak convergence of mn prime times delta mn prime to Y}
		\m_n \times \Delta \m_n \rightarrow \mathbb{Y} \ \text{weakly in}\  L^{2}(\Omega^{\prime}:L^2(0,T;L^2)),
	\end{equation}
	
	\begin{equation}\label{equation weak convergence of mn prime times mn prime times delta mn prime to Z}
		\m_n \times (\m_n \times \Delta \m_n) \rightarrow \mathbb{Z} \ \text{weakly in}\  L^{2}(\Omega^{\prime}:L^2(0,T;L^2)),
	\end{equation}
	
	In Lemma \ref{Lemma convergencem mn times Laplacian mn and mn times mn times Laplacian mn} and Lemma \ref{Lemma convergencem mn times mn times Laplacian mn}, we show that for $\phi\in L^4\l( \Omega^\p ; L^4(0,T;H^1 )\r)$
	\begin{equation}
		\lim_{n\rightarrow\infty} \mathbb{E^{\prime}}\int_{0}^{T} \l\langle \m_n(s) \times \Delta \m_n(s) - \m(s) \times \Delta \m(s) ,  \phi(s) \r\rangle_{L^2} \, ds = 0,
	\end{equation}
	and
	\begin{equation}
		\lim_{n\rightarrow\infty} \mathbb{E^{\prime}}\int_{0}^{T} \l\langle \m_n(s) \times \bigl( \m_n(s) \times \Delta \m_n(s)\bigr) - \m(s) 
		\times \bigl( \m(s) \times \Delta \m(s) \bigr) , \phi(s) \r\rangle_{L^2} \, ds = 0.
	\end{equation}

	Since the space $L^4\l( \Omega^\p ; L^4(0,T;H^1 )\r)$ is dense in the space $L^2\l( \Omega^\p ; L^2(0,T;L^2 )\r)$, by  uniqueness of limits, we can conclude that
	\begin{equation}
		\mathbb{Y} = \m \times \Delta \m,
	\end{equation}
	and
	\begin{equation}
		\mathbb{Z} = \m \times \l( \m \times \Delta \m \r).
	\end{equation}

	\dela{We show that $Y$,$Z$, $\mathbb{Y}$ and $\mathbb{Z}$ are of the required form.\\
		\adda{Is this weak convergence required? Is it sufficient to refer the reader to \cite{ZB+BG+TJ_Weak_3d_SLLGE}, \cite{ZB+BG+TJ_LargeDeviations_LLGE}??}
		
		Following Lemma 4.5 and Lemma 4.8 in \cite{ZB+BG+TJ_Weak_3d_SLLGE}, the limits can be shown as required for the first two terms. Are the next two convergences really required? Possibly YES for the convergence of $M^\p_n$ to $M^\p$ part.}
	\dela{\textbf{Note:} All the convergences mentioned above can also be (and should be) done for the terms with projection operators. A point to be noted is that the limits ( after testing with $\phi\in H^1$) are the same for both the cases. Therefore it is sufficient to talk about the terms without the projection operators.\\}
	\dela{Prior to that we show some bounds on the process $\m$.}
	\begin{proof}[Structure of the remaining section]
		We briefly describe the contents of the reminder of the section. The broad aim is to show the convergence of each term on the right hand side of the approximated equation \eqref{definition of solution Faedo Galerkin approximation} to the corresponding term, effectively to show that the obtained limit process $\m$ satisfies the equation \eqref{problem considered}. Lemma \ref{Lemma extension of norms and lower semicontinuity} gives some bounds on the obtained limit processes $\m$ and $\u$. Lemma \ref{Lemma continuity of m prime in H1 weak and X beta} shows that the paths of the process $\m$ are continunous in $H^1$(with the weak topology) and in $X^{\beta}$, for $\beta < \frac{1}{2}$. Lemma \ref{Lemma sup in t L2 in space convergene of m_n prime} shows the convergence on $\m_n$ to $\m$ in the space of continuous functions on $L^2$. Lemma \ref{Lemma convergence of the bump function} shows some convergence results for the bump function $\psi$. Lemma \ref{Lemma convergence mn times un and mn times mn times un} shows the convergence of the terms containing the control process to the respective terms. Similarly, Lemma \ref{Lemma convergence of Gn} shows the convergence for the terms containing $G_n$. All these results are then collectively used to show that the process $\m$ is a weak martingale solution for the problem \eqref{problem considered}.
	\end{proof}

	\begin{lemma}\label{Lemma extension of norms and lower semicontinuity}
		We have the following bounds.
		\begin{enumerate}
			
			\item		\begin{equation}\label{equation bound on m prime}
				\sup_{0\leq t\leq T}\l|m^{\prime}(t)\r|_{L^2} \leq |m_0|_{L^2}\ \mathbb{P}^{\prime}-\text{a.s.}
			\end{equation}
			
			\item There exists a constant $C>0$ such that
			\begin{equation}\label{equation H1 bound on m prime}
				\mathbb{E}^{\prime}\sup_{0\leq t\leq T}\l|m^{\prime}(t)\r|^{2p}_{H^1} \leq C.
			\end{equation}

			\dela{\item
				\begin{equation}
					\mathbb{E}^{\prime} \int_{0}^{T} \l| \u(t) \r|^{2p}_{L^2} \leq K.
				\end{equation}
			}
			
			\item
			\begin{equation}
				\mathbb{E}^{\prime} \l(\int_{0}^{T} \l| \u(t) \r|^{2}_{L^2}\r)^{p} \leq K_p.
			\end{equation}

		\end{enumerate}
	\end{lemma}
	\begin{proof}
		For the first inequality, recall that the process $\m_n$ converges to $\m$ in $C([0,T];L^2)$ $\mathbb{P}^{\prime}$-a.s. Hence $\mathbb{P}^{\prime}$-a.s.,
		\begin{align}
			\sup_{0\leq t\leq T} \l|\m(t)\r|_{L^2} \leq \liminf_{n\rightarrow\infty} \sup_{0\leq t\leq T} \l|\m_n(t)\r|_{L^2} \leq \l| m_0 \r|_{L^2}.
		\end{align}
		This concludes the proof of the first inequality \eqref{equation bound on m prime}.
		
		Moreover, by the Fatou Lemma, see \cite{Rudin_RCA}
		\dela{
			\begin{align}
				\mathbb{E}^{\prime} \sup_{0\leq t\leq T} \l|\m(t)\r|_{L^2} &\leq \liminf_{n\rightarrow\infty}  \mathbb{E}^\p \sup_{0\leq t\leq T} \l|\m_n(t)\r|_{L^2} \\
				& \leq \mathbb{E} \l|m_0\r|_{L^2}.
			\end{align}
		}
		\begin{align}
			\mathbb{E}^{\prime} \sup_{0\leq t\leq T} \l|\m(t)\r|_{L^2}^2 &\leq \liminf_{n\rightarrow\infty}  \mathbb{E}^\p \sup_{0\leq t\leq T} \l|\m_n(t)\r|_{L^2}^2 \leq C.
		\end{align}
		For the second inequality, extend the definition of the norm $|\cdot|_{H^1}$ to the domain $L^2$ as follows:
		
		\begin{align}
			|v|_{H^1}=\begin{cases}
				|v|_{H^1},&\ \text{if}\ v\in H^1, \\
				\infty,&\ \text{if}\ v\in L^2, v\notin H^1.
			\end{cases}
		\end{align}
		The extended maps are lower semicontinuous.
		\dela{
			We first show that the extended maps are lower semicontinuous.
			
			Let $r > 0 $ be arbitrary but fixed. Consider the set $B = \{v\in L^2 : |v|_{H^1} \leq r\}$. Our aim now is to show that $B$ is closed in $L^2$. Let $(v_n)_n\in\mathbb{N}$ be a sequence in $B$ such that $v_n \rightarrow v$ in $L^2$.
			$(v_n)_{n\in\mathbb{N}}\subset B$ implies that the sequence is bounded in $H^1$. Thus, it has a subsequence, again denoted by $(v_n)_{n\in \mathbb{N}}$, such that
			\begin{equation*}
				v_n \rightarrow v^{\prime}\ \text{weakly in}\ H^1
			\end{equation*}
			for some $v^{\prime}\in H^1$.
			Weak convergence in $H^1$ implies weak convergence in $L^2$. Thus, by uniqueness of limit, we have $v = v^{\prime}$. Hence $v\in H^1$.
			
			By the lower semicontinuity of $|\cdot|_{H^1}$ on the space $H^1$, we have
			\begin{align*}
				|v|_{H^1} \leq \liminf_{n\rightarrow\infty} |v_n|_{H^1} \leq r.
			\end{align*}
			Thus, $v\in B$. Hence $B$ is a closed subset of $L^2$. Thus, the extended map is lower semicontinuous.

			\dela{Also, the extended maps
				\[
				v\in C([0,T];(H^1)^{\prime})\mapsto \sup_{0\leq t\leq T}|v(t)|_{L^2}
				\]
				and
				\[
				v\in C([0,T];L^2)\mapsto \sup_{0\leq t\leq T}|v(t)|_{H^1}
				\]
				are lower semicontinuous. Therefore pointwise convergence of $m_n^{\prime}$ to $m^{\prime}$ in $C([0,T];L^2)$ and the uniform bound \eqref{bound on m_n prime} implies that $\mathbb{P}^{\prime}$-a.s.
				\begin{equation}\label{bound on m prime L2}
					\sup_{0\leq t\leq T}|m^{\prime}(t)|_{L^2} \leq \liminf_{n\rightarrow\infty} \sup_{0\leq t\leq T}|m^{\prime}_n(t)|_{L^2}\leq |m_0|_{L^2}
				\end{equation}
				
				and
				\begin{align}
					\sup_{0\leq t\leq T}|m^{\prime}(t)|_{H^1} & \leq \liminf_{n\rightarrow\infty} \sup_{0\leq t\leq T}|m_n^{\prime}(t)|_{H^1}.
				\end{align}
			}

			Similarly, the extended map
			\begin{equation}
				v\in C([0,T];L^2)\mapsto \sup_{0\leq t\leq T}|v(t)|_{H^1}
			\end{equation}
			is lower semicontinuous.
		}
		
		Therefore pointwise convergence of $m_n^{\prime}$ to $m^{\prime}$ in $C([0,T];L^2)$ implies that $\mathbb{P}^{\prime}$-a.s.

		\begin{align*}
			\sup_{0\leq t\leq T}|m^{\prime}(t)|_{H^1} & \leq \liminf_{n\rightarrow\infty} \sup_{0\leq t\leq T}|m_n^{\prime}(t)|_{H^1}.
		\end{align*}
		Thus by Fatou's Lemma followed by  \eqref{bound on nabla m_n prime}, there exists a constant $C>0$ such that
		\begin{align*}
			\mathbb{E}^{\prime}\sup_{0\leq t\leq T}|m^{\prime}(t)|^{2p}_{H^1} & \leq \mathbb{E^{\prime}}\liminf_{n\rightarrow\infty} \sup_{0\leq t\leq T}|m_n^{\prime}(t)|^{2p}_{H^1} \\
			&\leq  \liminf_{n\rightarrow\infty} \mathbb{E^{\prime}} \sup_{0\leq t\leq T}|m_n^{\prime}(t)|^{2p}_{H^1} \leq C.
		\end{align*}
		That is,
		\begin{align}\label{bound on m prime H1}
			\mathbb{E}^{\prime}\sup_{0\leq t\leq T}|m^{\prime}(t)|^{2p}_{H^1} \leq C.
		\end{align}
		\dela{ Here the second step follows by Fatou's lemma and the bounds \eqref{bound on nabla m_n prime} and \eqref{bound on m_n prime}. }
		
		\dela{
			The third inequality can be shown similarly. Extend the norm $\l|\cdot\r|_{L^{2p}(0,T;L^2)}$ to the space $\l|\cdot\r|_{L^{2}(0,T;L^2)}$ as follows.
			\begin{align*}
				\l| v \r|_{L^{2p}(0,T;L^2)} =
				\begin{cases}
					\l| v \r|_{L^{2p}(0,T;L^2)},&\ \text{if}\ v\in L^{2p}(0,T;L^2), \\
					\infty,&\ \text{if}\ v\in L^{2}(0,T;L^2),\ v\notin L^{2p}(0,T;L^2).
				\end{cases}
			\end{align*}
			As done previously, let $r \geq 0$ be fixed and define set $B=\{v\in L^{2}(0,T;L^2) : \l| v \r|_{L^{2p}(0,T;L^2)} \leq r\}$. Let $(v_n)_{n\in\mathbb{N}}$ be a sequence in $B$ such that $v_n \rightarrow v$ in $L^{2}(0,T;L^2)$ for some $v\in L^{2}(0,T;L^2)$. We show that $v\in B$. The sequence is in the set $B$ implies that there exists a subsequence, say $(v_{n_k})_{k\in\mathbb{N}}$ and some $v^{\prime}\in L^{2p}(0,T;L^2)$ such that $v_{n_k} \rightarrow v^{\prime}$ weakly in $L^{2p}(0,T;L^2)$. The continuous embedding $L^{2p}(0,T;L^2) \hookrightarrow L^{2}(0,T;L^2)$ implies that $v_{n_k}\rightarrow v$ in $L^2(0,T;L^2)$. Hence by the uniqueness of limits, $v = v^{\prime}$. Convergence of $v_{n_k}$ to $v$ implies that
			\begin{align*}
				\l| v \r|_{L^{2p}(0,T':L^2)} \leq \liminf_{k\rightarrow\infty} \l| v_{n_k}\r|_{L^{2p}(0,T;L^2)} \leq r.
			\end{align*}
			Hence $v\in B$. Thus the extended norm is lower semicontinuous.
			
			Hence,
			\begin{align*}
				\l| \u \r|_{L^{2p}(0,T;L^2)} \leq \liminf_{n\rightarrow\infty} \l| \u_n \r|_{L^{2p}(0,T;L^2)}.
			\end{align*}
			Therefore by Fatou's Lemma,
			\begin{align*}
				\mathbb{E}^{\prime} \l| \u \r|_{L^{2p}(0,T;L^2)}^{2p} &\leq \mathbb{E}^{\prime} \liminf_{n\rightarrow\infty} \l| \u_n \r|_{L^{2p}(0,T;L^2)}^{2p} \\
				& \leq \liminf_{n\rightarrow\infty} \mathbb{E}^{\prime} \l| \u_n \r|_{L^{2p}(0,T;L^2)}^{2p} \\
				& \leq K.
			\end{align*}	
			
		}

		The sequence $\u_n$ converges to $\u$ in $L^2(0,T;L^2)$ $\mathbb{P}^{\prime}$-a.s.
		Hence by the Fatou Lemma,
		\begin{align}
			\mathbb{E}^{\prime} \l|\u\r|_{L^2(0,T;L^2)}^{2p} \leq \liminf_{n\rightarrow\infty} \mathbb{E}^{\prime} \l|\u_n\r|_{L^2(0,T;L^2)}^{2p} \leq K_p.
		\end{align}
		This concludes the proof of the Lemma \eqref{Lemma extension of norms and lower semicontinuity}.

	\end{proof}
	

	\dela{
		We have that $\m_n \rightarrow \m$ in $L^4(0,T;L^4)$ $\mathbb{P}^{\prime}$-a.s. This convergence can be extended to convergence in $L^4(\Omega^{\prime}:L^4(0,T;L^4))$. This can be done as follows.
		\begin{align*}
			|\m_n|_{L^4(0,T;L^4)}  &\leq |\m_n - \m|_{L^4(0,T;L^4)}  + |\m|_{L^4(0,T;L^4)}
		\end{align*}
		Thus,
		\begin{align*}
			|\m_n|_{L^4(0,T;L^4)} - |\m|_{L^4(0,T;L^4)}  &\leq |\m_n - \m|_{L^4(0,T;L^4)}.
		\end{align*}
		Similarly,
		\begin{align*}
			|\m|_{L^4(0,T;L^4)} - |\m_n|_{L^4(0,T;L^4)}  &\leq |\m_n - \m|_{L^4(0,T;L^4)}.
		\end{align*}
		Thus,
		\begin{align*}
			\left| |\m|_{L^4(0,T;L^4)} - |\m_n|_{L^4(0,T;L^4)} \right|  &\leq |\m_n - \m|_{L^4(0,T;L^4)}.
		\end{align*}
		The right hand side of the above inequality goes to 0 as $n$ goes to infinity.
		
		Hence as $n$ goes to infinity, we have
		\begin{equation*}
			|\m_n|_{L^4(0,T;L^4)} \rightarrow |\m|_{L^4(0,T;L^4)}
		\end{equation*}
		$\mathbb{P}^{\prime}$-a.s.

		By Bounded convergence theorem, we get
		\begin{equation}
			\mathbb{E}^{\prime}\int_{0}^{T}|m_n^{\prime}(t)-m^{\prime}|_{L^4}^4dt\rightarrow 0
		\end{equation}
		as $n$ goes to infinity.
	}

	The uniform bound \eqref{bound on nabla m_n prime}, along with the continuous embedding $H^1 \hookrightarrow L^4$ implies that the sequence $\{\m_n\}_{n\in\mathbb{N}}$ is uniformly integrable in $L^4(\Omega^{\prime} ; L^4(0,T ; L^4))$. Hence by the Vitali Convergence Theorem \dela{\cite{BogachevBook_MeasureTheor_2007}}, see for example Section 6, Exercise 10 in \cite{Rudin_RCA},
	\begin{equation}\label{eqn convergence of L4L4 norm of m_n prime}
		\mathbb{E}^{\prime}\int_{0}^{T}|m_n^{\prime}(t)-m^{\prime}|_{L^4}^4dt\rightarrow 0.
	\end{equation}
	Note that
	\begin{equation}\label{bound on u_n prime}
		\mathbb{E^{\prime}} \l[\int_{0}^{T} |\u_n(t)|_{L^2}^2 dt\r]^4 \leq C,
	\end{equation}
	for some constant $C$ independent of $n$. \dela{Thus, there exists a subsequence, labelled by the same index, such that}
	
	Weak convergence of $\u_n$ to $\u$ pointwise (from \eqref{Equation weak convergence un prime to u prime pointwise}) implies that for any $\phi\in L^2(\Omega^\p ; L^2(0,T;L^2))$, we have
	\begin{equation}
		\int_{0}^{t}\l\langle \u_n(s,\omega^\p) - \u(s,\omega^\p), \phi \r\rangle_{L^2} \, ds \to 0,
	\end{equation}
	as $n$ tends to infinity.
	
	Using the bounds in \eqref{assumption on u} for $p = 4$, we can prove (see the proof of \eqref{eqn-1001} in Section \ref{Section Optimal control} for a similar calculation) that there exists a constant $C>0$ such that
	\begin{equation}
		\mathbb{E^{\prime}}\l| \int_{0}^{t}\l\langle \u_n(s) - \u(s), \phi \r\rangle_{L^2} \, ds \r|^{\frac{4}{3}} \leq C.
	\end{equation}
	In fact, the following holds for any $1 \leq q \leq \frac{4}{3}$
	\begin{equation}
		\mathbb{E^{\prime}}\l| \int_{0}^{t}\l\langle \u_n(s) - \u(s), \phi \r\rangle_{L^2} \, ds \r|^{q} \leq C.
	\end{equation}
	Therefore in particular for $q = 1$, we have
	\begin{equation}
		\mathbb{E^{\prime}}\l| \int_{0}^{t}\l\langle \u_n(s) - \u(s), \phi \r\rangle_{L^2} \, ds \r| \leq C,
	\end{equation}
	giving \eqref{weak convergence of u_n prime to u in L2L2L2}. Therefore, we have the following convergence as a result of the pointwise convergence of $\u_n$ to $\u$ and the Vitali Convergence Theorem.
	\begin{equation}\label{weak convergence of u_n prime to u in L2L2L2}
		\u_n \rightarrow \u\ \text{weakly in}\ L^2(\Omega^{\prime}: L^2(0,T ; L^2)),
	\end{equation}
	as $n$ goes to infinity.
	\dela{Although \eqref{weak convergence of u_n prime to u in L2L2L2} is correct, the justification is not exactly correct.
	}
	
	\dela{By the same logic used in the convergence \eqref{eqn convergence of L4L4 norm of m_n prime}, why is this convergence \eqref{weak convergence of u_n prime to u in L2L2L2} not strong? Answer: Because the tightness is shown with the weak topology on $H^1$ and hence the pointwise convergence will also be in the weak topology.
		Another Question is that why the limit is same? An answer can be: Let $f_n = \int_{0}^{T} \l\langle \u_n(s) , \phi(s) \r\rangle_{L^2} \, ds$, for $\phi\in L^2(0,T;L^2)$. Therefore $f_n\to f$ pointwise (for $\phi\in L^2(\Omega^\p:L^2(0,T;L^2))$). Also, $\sup_{n\in\mathbb{N}}\mathbb{E}^\p\l|f_n\r| \leq C$ implies uniform integrability. Hence the convergence follows by the Vitali Theorem.}

	\begin{lemma}\label{Lemma continuity of m prime in H1 weak and X beta}
		There exists an event $\Omega^{\prime\prime} \subset \Omega^{\prime}$ of full $\mathbb{P}^{\prime}$-measure such that
		for  every $\omega^{\prime}\in\Omega^{\prime\prime}$, the following assertions hold:
		\begin{enumerate}
			\item\label{item 1 Lemma continuity of m prime in H1 weak and X beta} The path of $\m(\omega^{\prime})$ is continuous in $H^1$ with the weak topology.
			
			\item\label{item 2 Lemma continuity of m prime in H1 weak and X beta} The path of $\m(\omega^{\prime})$ is continuous in $X^{\beta}$ (with the norm topology) for $\beta < \frac{1}{2}$.
		\end{enumerate}
	\end{lemma}
	\begin{proof}[Proof of Lemma \ref{Lemma continuity of m prime in H1 weak and X beta}]
		\dela{How the set $\Omega^{\prime\prime}$ is chosen?}
		The inequality \eqref{bound on m prime H1} holds in $L^2(\Omega^{\prime})$. Hence there exists a subset $\Omega^{\prime\prime}\in\mathcal{F}_0^\prime$ of full $\mathbb{P}^\p$-measure such that
		$$\sup_{t\in [0,T]}|\m(t)(\omega^{\prime})|_{H^1}< \infty,\ \forall \, \omega^{\prime} \in \Omega^{\prime\prime}.$$
		Hence the quantity $\sup_{t\in [0,T]}|\m(t)(\omega^{\prime})|_{H^1}$ is finite, on an event $\Omega^{\p\p}\subset\Omega^{\p}$ of full $\mathbb{P}^\p$-measure.		
		\dela{Fix $\omega^{\prime}\in \Omega^{\prime}$ such that
			\begin{equation*}
				\sup_{t\in [0,T]}|\m(t)(\omega^{\prime})|_{H^1} < \infty.
			\end{equation*}
		}
		Fix $\omega^{\prime}\in \Omega^{\prime\p}$.\\
		\textbf{Idea of the proof:} To show the continuity of the process $\m(\omega^{\prime})$ at $t\in[0,T]$, we consider a sequence $\l(t_n\r)_{n\in\mathbb{N}}$ in $ [0,T] $ that converges to $t$. We then show that the sequence $\l(\m(t_n)(\omega^{\prime})\r)_{n\in\mathbb{N}}$ converges in the appropriate topology to $\m(t)(\omega^{\prime})$ as $n$ goes to infinity.\\		
		\textbf{Proof of \eqref{item 1 Lemma continuity of m prime in H1 weak and X beta}:} Let $(t_n)_{n\in\mathbb{N}}$ be a sequence in $[0,T]$ such that 
		\begin{equation*}
			t_n\rightarrow t \in [0,T],
		\end{equation*}
		in $[0,T]$.
		The sequence $(\m(t_n)(\omega^{\prime}))_{n\in\mathbb{N}}$ is bounded in $H^1$ and hence has a weakly convergent subsequence, say $(\m(t_{n_k})(\omega^{\prime}))_{n\in\mathbb{N}}$ such that \begin{equation}
			\m(t_{n_k})(\omega^{\prime}) \rightarrow z
		\end{equation}\dela{$$\m(t_{n_k})(\omega^{\prime}) \rightarrow z$$} weakly in $H^1$ for some $z\in H^1$.\\
		The space $H^1$ is compactly embedded in the space $(H^1)^{\prime}$. Hence as $k$ goes to infinity (possibly along a subsequence),\dela{$$\m(t_{n_k})(\omega^{\prime}) \rightarrow z\ \text{in}\ (H^1)^{\prime}.$$}
		\begin{equation}
			\m(t_{n_k})(\omega^{\prime}) \rightarrow z\ \text{in}\ (H^1)^{\prime}.
		\end{equation}
		That $\m(\omega^{\prime})$ is continuous with values in $(H^1)^{\prime}$ implies that \dela{$$\m(t_{n_k})(\omega^{\prime}) \rightarrow \m(t)(\omega^{\prime})$$}
		\begin{equation}
			\m(t_{n_k})(\omega^{\prime}) \rightarrow \m(t)(\omega^{\prime})
		\end{equation} in $(H^1)^{\prime}$.\\
		Hence by the uniqueness of limit,
		\begin{equation*}
			z = \m(t)(\omega^{\prime}).
		\end{equation*}
		From the above argument we can conclude that every subsequence of $\m(t_{n})(\omega^{\prime})$ thus has a further subsequence which converges to the same limit $ \m(t)(\omega^{\prime}) $. Hence
		\begin{equation*}
			\m(t_{n})(\omega^{\prime}) \rightarrow \m(t)(\omega^{\prime})\ \text{weakly in}\ H^{1}.
		\end{equation*}
		The sequence of arguments can be repeated for any $t\in[0,T]$. Hence $\m(\omega^{\prime})$ is continuous in $H^1$ with the weak topology.\\
		\textbf{Sketch of a proof of \eqref{item 2 Lemma continuity of m prime in H1 weak and X beta}.} $H^1$ is compactly embedded in $X^{\beta}$ for $\beta < \frac{1}{2}$, see Lemma \ref{X gamma compact embedding} in Appendix \ref{Section Some embeddings}.
		Hence
		
		\begin{equation*}
			\m(t_{n_k})(\omega^{\prime}) \rightarrow \m(t)(\omega^{\prime})\ \text{in}\ X^{\beta}.
		\end{equation*}
		Replicating the above arguments, the continuity of $\m(\omega^{\prime})$ in $X^{\beta}$ can be shown.

	\end{proof}

	\begin{lemma}\label{Lemma sup in t L2 in space convergene of m_n prime}
		We have the following convergence.
		\begin{align*}
			\mathbb{E^{\prime}} \sup_{t\in [0,T]}|\m_n(t) - \m(t)|^8_{L^2} \rightarrow 0\ \text{as}\ n\rightarrow\infty.
		\end{align*}
	\end{lemma}
	\begin{proof}[Proof of Lemma \ref{Lemma sup in t L2 in space convergene of m_n prime}]
		We use the following inequality in the subsequent calculations. For $v\in H^1$,\dela{ A brief sketch of the proof is also given.}
		\begin{equation}\label{Inequality L^2 leq H1 H-1}
			|v|_{L^2} \leq |v|_{H^1}^{\frac{1}{2}} |v|_{(H^1)^{\prime}}^{\frac{1}{2}}
		\end{equation}
		
		\textbf{Sketch of a proof of the inequality \eqref{Inequality L^2 leq H1 H-1}.} Use the Gelfand triple $H^{1} \hookrightarrow L^2 \hookrightarrow (H^1)^{\prime}$ and hence $$\l\langle a , b \r\rangle_{L^2} = \ _{H^1}\l\langle a , b \r\rangle_{(H^1)^{\prime}}$$
		for $a\in H^1$ and $b\in L^2$.\\
		Thus
		$$| \l\langle a , b \r\rangle_{L^2} | = |\ _{H^1}\l\langle a , b \r\rangle_{(H^1)^{\prime}}| \leq |a|_{H^1}^{\frac{1}{2}} |b|_{(H^1)^{\prime}}^{\frac{1}{2}}.$$
		
		\begin{align*}
			\mathbb{E^{\prime}} \sup_{t\in [0,T]}|\m_n(t) - \m(t)|^8_{L^2} &\leq C \mathbb{E^{\prime}} \sup_{t\in [0,T]} \left[ |\m_n(t) - \m(t)|^4_{(H^1)^{\prime}} |\m_n(t) - \m(t)|^4_{H^{1}} \right]\\
			& \leq C \mathbb{E^{\prime}} \sup_{t\in [0,T]} \left[ (|\m_n(t)|_{H^1}^4 + |\m(t)|^4_{H^{1}}) |\m_n(t) - \m(t)|^4_{(H^1)^{\prime}} \right] \\
			& \leq C \left[ \left(\mathbb{E^{\prime}} \sup_{t\in [0,T]}|\m_n(t)|_{H^1}^8 + \mathbb{E^{\prime}} \sup_{t\in [0,T]}|\m(t)|_{H^1}^8\right)^{\frac{1}{2}}  \right]\centerdot \\
			& \quad\left(\mathbb{E^{\prime}} \sup_{t\in [0,T]}|\m_n(t) - \m(t)|^4_{(H^1)^{\prime}}\right)^{\frac{1}{2}} \\
			& \leq C \left(\mathbb{E^{\prime}} \sup_{t\in [0,T]}|\m_n(t) - \m(t)|^4_{(H^1)^{\prime}}\right)^\frac{1}{2}.
		\end{align*}
		The above calculation uses the inequality \eqref{Inequality L^2 leq H1 H-1} followed by Cauchy-Schwartz inequality and concluding with applying the bounds \eqref{bound on m_n prime} and  \eqref{bound on nabla m_n prime}.\\
		Moreover, there exists a constant $C>0$ such that $\mathbb{P}^\p$-a.s.
		\begin{align*}
			\sup_{t\in [0,T]}|\m_n(t) - \m(t)|^8_{(H^1)^{\prime}} &\leq \sup_{t\in [0,T]}|\m_n(t)|^8_{(H^1)^{\prime}}  + \sup_{t\in [0,T]}|\m_n(t)|^8_{(H^1)^{\prime}} \\
			& \leq C \l(\sup_{t\in [0,T]}|\m_n(t)|^8_{H^{1}}  + \sup_{t\in [0,T]}|\m_n(t)|^8_{H^{1}}\r).
		\end{align*}
		The last step along with the bound \eqref{bound on nabla m_n prime} gives us a bound for using the Lebesgue convergence theorem, thus concluding the proof.
	\end{proof}

	\begin{lemma}\label{Lemma convergence of the bump function}
		We have the following convergences.
		\begin{enumerate}
			\item \begin{equation*}
				\lim_{n\rightarrow\infty} \mathbb{E^{\prime}} \sup_{t\in [0,T]} \l|\psi_0\l(\l|\m_n(t)\r|_{L^{\infty}}\r) - \psi_0\l(\l|\m(t)\r|_{L^{\infty}}\r)\r|^4 = 0.
			\end{equation*}
			
			\item \begin{equation*}
				\lim_{n\rightarrow\infty} \mathbb{E^{\prime}} \sup_{t\in [0,T]} \l| \psi_0\bigl(\l|P_n\l(\m_n(t) \times h\r)\r|_{L^{\infty}}\bigr) -  \psi_0\l(\l|\m(t) \times h\r|_{L^{\infty}}\r)\r|^4 = 0.
			\end{equation*}
			
			\item \begin{equation*}
				\lim_{n\rightarrow\infty} \mathbb{E^{\prime}} \sup_{t\in [0,T]} \l|\psi_0\bigl(\l|P_n\bigl(\m_n(t) \times \l(\m_n(t) \times h\r)\bigr)\r|_{L^{\infty}}) - \psi_0\bigl(\l|\m(t) \times \l(\m(t) \times h\r)\r|_{L^{\infty}}\bigr)\r|^4 = 0.
			\end{equation*}
		\end{enumerate}
		In particular,
		\begin{equation}\label{equation convergence of the bump function}
			\lim_{n\rightarrow\infty} \mathbb{E^{\prime}} \sup_{t\in [0,T]} \l|\psi\big(\m_n(t)\big) - \psi\big(\m(t)\big)\r|^4 = 0.
		\end{equation}
	\end{lemma}
	\begin{proof}[Proof of Lemma \ref{Lemma convergence of the bump function}]
		We state a couple of results that will be used in the proof that follows.
		\begin{enumerate}
			\item Since, $\psi_0 \in C_c^{1}(D)$, the following holds for any $x,y\in \mathbb{R}$.
			\begin{align*}
				|\psi_0(x) - \psi_0(y)| \leq \sup_{z\in\mathbb{R}}|D\psi(z)| |x - y|.
			\end{align*}
			\dela{
				\textbf{Sketch of proof for the above statement (1):}
				Let $x,y\in\mathbb{R}$. Without the loss of generality, let $x\leq y$. The case $x \geq y$ can be shown similarly/ Then there exists $c\in[x,y]$ such that
				\begin{align*}
					(\psi(x) - \psi(y)) = D\psi(c)(x - y)
				\end{align*}
				Thus,
				\begin{align*}
					|\psi(x) - \psi(y)| &\leq \sup_{c\in[x,y]}D\psi(c)|x - y| \\
					& \leq \sup_{c\in\mathbb{R}}D\psi(c)|x - y|.
				\end{align*}
			}
			
			\item \dela{ Another inequality that will be used is as follows.} We refer the reader to Theorem 5.8, \cite{Adams+Fournier}, for the following inequality.
			\begin{equation}\label{Inequality L infinity leq L2 H1}
				|v|_{L^{\infty}} \leq C |v|_{L^2}^{\frac{1}{2}} |v|_{H^1}^{\frac{1}{2}} ,\  v\in H^1.
			\end{equation}
		\end{enumerate}
		We show the first convergence in Lemma \ref{Lemma convergence of the bump function}. The other two can be shown similarly.
		
		\begin{align*}
			\mathbb{E^{\prime}} \sup_{t\in [0,T]} \l|\psi_0\l(|\m_n(t)|_{L^{\infty}}\r) - \psi_0\bigl(\l|\m(t)\r|_{L^{\infty}}\bigr)\r|^4 &\leq C \mathbb{E^{\prime}} \sup_{t\in [0,T]} \l|\l|\m_n(t)\r|_{L^{\infty}} - \l|\m(t)\r|_{L^{\infty}}\r|^4 \\
			& \leq C \mathbb{E^{\prime}} \sup_{t\in [0,T]} \l|\m_n(t) - \m(t)\r|_{L^{\infty}}^4 \\
			& \leq C \mathbb{E^{\prime}} \sup_{t\in [0,T]} \l(\l|\m_n(t) - \m(t)\r|_{L^2}^2 \l|\m_n(t) - \m(t)\r|_{H^1}^2\r) \\
			& \leq C \mathbb{E^{\prime}} \sup_{t\in [0,T]} \l|\m_n(t) - \m(t) \r|_{L^2}^2 \l( \l|\m_n(t)\r|_{H^1}^2 + \l|\m(t)\r|_{H^1}^2\r) \\
			& \leq \left(\mathbb{E^{\prime}} \sup_{t\in [0,T]} |\m_n(t) - \m(t)|_{L^2}^4\right)^{\frac{1}{2}} \centerdot\\
			& \quad\left( \mathbb{E^{\prime}} \sup_{t\in [0,T]} \l( \l|\m_n(t)\r|_{H^1}^4 + \l|\m(t)\r|_{H^1}^4 \r) \right)^{\frac{1}{2}}.
		\end{align*}
		The above calculation uses triangle inequality in the first step followed by \eqref{Inequality L infinity leq L2 H1} and concluding with the H\"older inequality.
		By the bounds \eqref{bound on m_n prime}, \eqref{bound on nabla m_n prime} and Lemma \ref{Lemma sup in t L2 in space convergene of m_n prime}, the right hand side of the above inequality goes to 0 as $n$ goes to infinity.
	\end{proof}
	
	\dela{	By the bounds established in Proposition \ref{Proposition bounds on m_n prime}, there exists a measurable process $Y \in L^{2}(\Omega^{\prime}:L^2(0,T;L^2))$ such that
		
		\begin{equation}\label{convergence m_n times Delta m_n}
			m_n^{\prime}\times\Delta m_n^{\prime}\rightarrow Y \ \text{weakly in}\  L^{2}(\Omega^{\prime}:L^2(0,T;L^2)),
		\end{equation}
		The next lemma helps to identify the limit $Y$ with $\m \times \Delta \m$.
	}
	\begin{lemma}\label{Lemma convergencem mn times Laplacian mn and mn times mn times Laplacian mn}
		Let $\phi\in L^4(\Omega^\p ;  L^4(0,T ; H^1))$. Then
		\dela{Choose $\phi\in L^4(\Omega^p: L^4(0,T;W^{1,4}))$. See Lemma 4.5 in \cite{ZB+BG+TJ_Weak_3d_SLLGE}.}
		
		\begin{align}
			\lim_{n\rightarrow\infty} \mathbb{E^{\prime}} \int_{0}^{T} \l\langle \m_n(s) \times \Delta \m_n(s) , \phi(s) \r\rangle_{L^2}\, ds = 			\mathbb{E^{\prime}} \int_{0}^{T} \l\langle \m(s) \times \Delta \m(s) , \phi(s) \r\rangle_{L^2}\, ds.
		\end{align}

	\end{lemma}
	\begin{proof}[Proof of Lemma \ref{Lemma convergencem mn times Laplacian mn and mn times mn times Laplacian mn}]
		

		\dela{By the uniform bounds \eqref{bound on m_n prime} and \eqref{bound on nabla m_n prime}, there exists a subsequence of $(\m_n)_{n\in\mathbb{N}}$ (denoted by the same sequence) which converges  to $\m$ weakly in .
			
			some element of $L^2(\Omega^{\prime}:L^2(0,T);L^2)$. That it converges to $\m$ can be shown by the following argument and the uniqueness of limit.
			By Proposition \ref{Use of Skorohod theorem} for $ v \in L^2(\Omega^{\prime}:L^2(0,T);L^2) $,
			\begin{equation}
				\mathbb{E^{\prime}}\int_{0}^{T} | \l\langle  \m_n -  \m , v \r\rangle_{L^2} |\, ds \to 0 \ \text{as}\ n\rightarrow \infty.
			\end{equation}
			
			We observe that
			\begin{align}
				\mathbb{E^{\prime}}\int_{0}^{T} \l| \l\langle \nabla \m_n - \nabla \m , \phi \r\rangle_{L^2} \r|\, ds &=  \mathbb{E^{\prime}}\int_{0}^{T} \l| \l\langle \m_n - \m, \nabla \phi \r\rangle_{L^2}\r|\, ds
			\end{align}
			The assumption $\phi\in H^1$, it implies that $\nabla \phi \in L^2$. The right hand side of the above equality thus goes to 0 as $n$ goes to infinity. Hence by the uniqueness of limits and the bound \eqref{bound on m prime H1} implies that the limit is indeed $\nabla \m$.
		}
		
		By the uniform bounds \eqref{bound on m_n prime} and \eqref{bound on nabla m_n prime}, there exists a subsequence of $(\m_n)_{n\in\mathbb{N}}$ (denoted by the same sequence) such that
		\begin{equation}
			\nabla \m_n \rightarrow \nabla \m \ \text{weakly in}\ L^2\bigl(\Omega^{\prime} ; (L^2(0,T) ; L^2)\bigr),
		\end{equation}
		as $n$ goes to infinity.\\
		Now, the use of H\"older's inequality and Agmon's inequality gives us the following set of inequalities.
		\dela{
			\begin{align*}
				\mathbb{E^{\prime}} \int_{0}^{T} & |\l\langle \m_n(s) \times \Delta \m_n(s) , \phi \r\rangle_{L^2}   - 	  \l\langle \m(s) \times \Delta \m(s) , \phi \r\rangle_{L^2}|\, ds \\
				\leq& \mathbb{E^{\prime}} \int_{0}^{T} |\l\langle \m_n(s) \times \Delta \m_n(s) , \phi \r\rangle_{L^2} - \l\langle \m(s) \times \Delta \m_n(s) , \phi \r\rangle_{L^2} |\, ds \\
				& + \mathbb{E^{\prime}} \int_{0}^{T} |\l\langle \m(s) \times \Delta \m_n(s) , \phi \r\rangle_{L^2}  - \l\langle \m(s) \times \Delta \m(s) , \phi \r\rangle_{L^2}|\, ds \\
				=& \mathbb{E^{\prime}} \int_{0}^{T} |\l\langle \nabla \m_n , \nabla \phi \times \m_n(s) \r\rangle_{L^2}  - \l\langle \nabla \m_n(s) , \nabla \phi \times \m(s) \r\rangle_{L^2}|\, ds \\
				&+ \mathbb{E^{\prime}} \int_{0}^{T} |\l\langle \nabla \m_n(s) , \nabla \phi \times \m(s) \r\rangle_{L^2} - \l\langle \nabla \m , \nabla \phi \times \m \r\rangle_{L^2}|\, ds \\
				=&  \mathbb{E^{\prime}} \int_{0}^{T} |\l\langle \nabla \m_n(s) , \nabla \phi \times (\m_n(s) - \m(s)) \r\rangle_{L^2}|\, ds \\
				&+ \mathbb{E^{\prime}} \int_{0}^{T} |\l\langle \nabla (\m_n(s) - \m(s)) , \nabla \phi \times \m(s) \r\rangle_{L^2}|\, ds \\	
				\leq&  \mathbb{E^{\prime}} \int_{0}^{T} |\nabla \m_n(s)|_{L^2} |\nabla \phi|_{L^2} |\m_n(s) - \m(s)|_{L^{\infty}} ds \\
				&+ \mathbb{E^{\prime}} \int_{0}^{T} |\l\langle \nabla (\m_n(s) - \m(s)) , \nabla \phi \times \m(s) \r\rangle_{L^2}|\, ds \\
				\leq&  C \mathbb{E^{\prime}} \sup_{t\in [0,T]}  |\m_n(s)|_{H^1} \int_{0}^{T}   |\m_n(s) - \m(s)|_{H^1}^{\frac{1}{2}}  |\m_n(s) - \m(s)|_{L^2}^{\frac{1}{2}}\, ds \\
				&+ \mathbb{E^{\prime}} \int_{0}^{T} |\l\langle \nabla (\m_n(s) - \m(s)) , \nabla \phi \times \m(s) \r\rangle_{L^2}| \\
				\leq&  C \mathbb{E^{\prime}} \sup_{t\in [0,T]}  |\m_n(s)|_{H^1} \int_{0}^{T}   (|\m_n(s)|_{H^1} + \m(s)|_{H^1})^{\frac{1}{2}}  |\m_n(s) - \m(s)|_{L^2}^{\frac{1}{2}}\, ds \\
				&+ \mathbb{E^{\prime}} \int_{0}^{T} |\l\langle \nabla (\m_n(s) - \m(s)) , \nabla \phi \times \m(s) \r\rangle_{L^2}|\, ds \\
				\leq &  C \mathbb{E^{\prime}} \sup_{t\in [0,T]}  |\m_n(s)|_{H^1} \sup_{s\in[0,T]} (|\m_n(s)|_{H^1} + \m(s)|_{H^1})^{\frac{1}{2}} \int_{0}^{T}     |\m_n(s) - \m(s)|_{L^2}^{\frac{1}{2}}\, ds \\
				&+ \mathbb{E^{\prime}} \int_{0}^{T} |\l\langle \nabla (\m_n(s) - \m(s)) , \nabla \phi \times \m(s) \r\rangle_{L^2}|\, ds.
			\end{align*}
		}
		
		\begin{align*}
			\bigg|\mathbb{E^{\prime}} \int_{0}^{T} & \l\langle \m_n(s) \times \Delta \m_n(s) , \phi(s) \r\rangle_{L^2}   - 	  \l\langle \m(s) \times \Delta \m(s) , \phi(s) \r\rangle_{L^2}\, ds\bigg| \\
			\leq& \bigg|\mathbb{E^{\prime}} \int_{0}^{T} \l\langle \m_n(s) \times \Delta \m_n(s) , \phi(s) \r\rangle_{L^2} - \l\langle \m(s) \times \Delta \m_n(s) , \phi(s) \r\rangle_{L^2} \, ds\bigg| \\
			& + \bigg|\mathbb{E^{\prime}} \int_{0}^{T} \l\langle \m(s) \times \Delta \m_n(s) , \phi(s) \r\rangle_{L^2}  - \l\langle \m(s) \times \Delta \m(s) , \phi(s) \r\rangle_{L^2}\, ds \bigg|\\
			=& \bigg|\mathbb{E^{\prime}} \int_{0}^{T} \l\langle \nabla \m_n , \nabla \phi(s) \times \m_n(s) \r\rangle_{L^2}  - \l\langle \nabla \m_n(s) , \nabla \phi(s) \times \m(s) \r\rangle_{L^2}\, ds \bigg|  \\
			&+ \bigg|\mathbb{E^{\prime}} \int_{0}^{T} \l\langle \nabla \m_n(s) , \nabla \phi(s) \times \m(s) \r\rangle_{L^2} - \l\langle \nabla \m , \nabla \phi(s) \times \m \r\rangle_{L^2}\, ds \bigg|\\
			=&  \bigg|\mathbb{E^{\prime}} \int_{0}^{T} \l\langle \nabla \m_n(s) , \nabla \phi(s) \times (\m_n(s) - \m(s)) \r\rangle_{L^2}\, ds \bigg|  \\
			&+ \bigg| \mathbb{E^{\prime}} \int_{0}^{T} \l\langle \nabla (\m_n(s) - \m(s)) , \nabla \phi(s) \times \m(s) \r\rangle_{L^2}|\, ds \bigg| \\	
			\leq&  \mathbb{E^{\prime}} \int_{0}^{T} |\nabla \m_n(s)|_{L^2} |\nabla \phi(s)|_{L^2} |\m_n(s) - \m(s)|_{L^{\infty}} ds \\
			&+ \bigg|\mathbb{E^{\prime}} \int_{0}^{T} \l\langle \nabla (\m_n(s) - \m(s)) , \nabla \phi(s) \times \m(s) \r\rangle_{L^2}\, ds \bigg| \\
			\leq&  C \mathbb{E^{\prime}} \sup_{t\in [0,T]}  |\m_n(s)|_{H^1} \int_{0}^{T}   |\m_n(s) - \m(s)|_{H^1}^{\frac{1}{2}}  |\m_n(s) - \m(s)|_{L^2}^{\frac{1}{2}} \l| \phi(s) \r|_{H^1} \, ds \\
			&+ \bigg|\mathbb{E^{\prime}} \int_{0}^{T} \l\langle \nabla (\m_n(s) - \m(s)) , \nabla \phi \times \m(s) \r\rangle_{L^2} \, ds \bigg| \\
			\leq&  C \mathbb{E^{\prime}} \sup_{t\in [0,T]}  |\m_n(s)|_{H^1} \int_{0}^{T}   (|\m_n(s)|_{H^1} + \m(s)|_{H^1})^{\frac{1}{2}}  |\m_n(s) - \m(s)|_{L^2}^{\frac{1}{2}} \l| \phi(s) \r|_{H^1} \, ds \\
			&+ \bigg| \mathbb{E^{\prime}} \int_{0}^{T} \l\langle \nabla (\m_n(s) - \m(s)) , \nabla \phi \times \m(s) \r\rangle_{L^2}\, ds \bigg| \\
			\leq &  C \mathbb{E^{\prime}} \sup_{t\in [0,T]}  |\m_n(s)|_{H^1} \sup_{s\in[0,T]} \l(|\m_n(s)|_{H^1} + \m(s)|_{H^1}\r)^{\frac{1}{2}} \l( \int_{0}^{T} \l| \phi(s) \r|_{H^1}^2 \, ds \r)^{\frac{1}{2}} \\
			& \quad \centerdot \l( \int_{0}^{T}     |\m_n(s) - \m(s)|_{L^2}\, ds  \r)^{\frac{1}{2}}\\
			&+ \bigg| \mathbb{E^{\prime}} \int_{0}^{T} \l\langle \nabla (\m_n(s) - \m(s)) , \nabla \phi \times \m(s) \r\rangle_{L^2}\, ds \bigg| .
		\end{align*}
		The bounds \eqref{bound on m_n prime} and \eqref{bound on nabla m_n prime} along with the convergence of $\m_n$ imply that the first term in the above inequality goes to $ 0 $ as $n$ goes to $\infty$.
		
		Due to the continuous embedding $H^1\hookrightarrow L^{\infty}$, there exists a constant $C>0$ such that
		\begin{align*}
			\mathbb{E}^\p\int_{0}^{T}\l|\nabla \phi(s) \times \m(s)\r|_{L^2}^2\, ds &\leq \mathbb{E}^\p\int_{0}^{T} \l|\nabla \phi(s)\r|_{L^2}^2\l|\m(s)\r|_{L^{\infty}}^2 \, ds \\
			& \leq C\mathbb{E}^\p\l[ \l(\sup_{t\in[0,T]}|\m(s)|_{H^1}^2\r)\int_{0}^{T}|\nabla \phi(s)|_{L^2}^2 \, ds \r]\\
			\leq & C\l[\mathbb{E}^\p \l(\sup_{t\in[0,T]}|\m(s)|_{H^1}^4\r)\r]^{\frac{1}{2}}
			\l[\mathbb{E}^\p\l(\int_{0}^{T}|\nabla \phi(s)|_{L^2}^2 \, ds\r)^2 \r]^{\frac{1}{2}}< \infty.
		\end{align*}
		
		The above inequality along with the bound on $ |\m|_{H^1} $ implies that the second term also goes to $0$ as $n$ goes to $\infty$.
		The right hand side of the above inequality goes to $0$ as $n$ goes to $\infty$, thus concluding the proof.

	\end{proof}
	
	\dela{Since $\phi\in L^4(\Omega^\p: L^4(0,T;H^1)) \hookrightarrow L^2(\Omega^\p: L^2(0,T;L^2))$, by \eqref{equation weak convergence of mn prime times delta mn prime to Y} we have
		\begin{equation}
			\lim_{n\rightarrow\infty} \mathbb{E^{\prime}} \int_{0}^{T} \l\langle \m_n(s) \times \Delta \m_n(s) , \phi(s) \r\rangle_{L^2}\, ds = 			\mathbb{E^{\prime}} \int_{0}^{T} \l\langle Y(s) , \phi(s) \r\rangle_{L^2}\, ds.
		\end{equation}
		By uniqueness of limits, we have $Y = \m \times \Delta m$. That $\m \times \Delta \m \in L^2(\Omega^\p: L^2(0,T;L^2))$ can be shown as in the proof of Lemma 6, 
	}
	
	\dela{Ideally the terms on the left hand side should also have a projection operator along with them. But the limits with and without the projection operator are same can be shown. Write a comment about it and try to give some reference, say \cite{ZB+BG+TJ_Weak_3d_SLLGE},\cite{LE_Deterministic_LLBE}.
		Try to write and justify briefly that
		\begin{equation}
			\lim_{n\rightarrow\infty}\mathbb{E^{\prime}} \int_{0}^{T} \l\langle P_n \l( \m_n(s) \times \u_n(s) \r), \phi \r\rangle_{L^2}\, ds  =  \lim_{n\rightarrow\infty}\mathbb{E^{\prime}} \int_{0}^{T} \l\langle \m_n(s) \times \u_n(s), \phi \r\rangle_{L^2}\, ds
		\end{equation}
	}
	\dela{Note that the bound on $\m$ here is for $p = 2$. Make sure that the initial data is regular enough. Either consider the data to be deterministic, or assume it to be $L^p(\Omega)$ for a sufficiently large $p$ ($p = 4$ suffices). Consider deterministic initial data?? \cite{ZB+BG+TJ_Weak_3d_SLLGE} considers deterministic initial data in $H^1$. \cite{D+M+P+V} also consider deterministic initial data in $W^{1,2}(D,\mathbb{S}^2)$.}
	\begin{lemma}\label{Lemma convergencem mn times mn times Laplacian mn}
		Let \dela{$\phi\in H^1$} $\phi\in L^4(\Omega^\p ; L^4(0,T ; H^1))$\dela{ $\phi\in L^4(\Omega^\p: H^1)$}. Then
		\item \begin{align*}
			\lim_{n\rightarrow\infty} &\mathbb{E^{\prime}} \int_{0}^{T}   \l\langle m_n(s) \times (\m_n(s) \times  \Delta \m_n(s)) , \phi \r\rangle_{L^2}\, ds \\
			& = \mathbb{E^{\prime}} \int_{0}^{T}  \l\langle \m(s) \times (\m(s) \times \Delta \m(s)) , \phi \r\rangle_{L^2}\, ds.
		\end{align*}
	\end{lemma}
	\begin{proof}[Proof of Lemma \ref{Lemma convergencem mn times mn times Laplacian mn}]
		
		\dela{
			\begin{align*}
				\mathbb{E^{\prime}}   & \int_{0}^{T} \b| \psi_0(m_n(s))\l\langle m_n(s) \times (\m_n(s) \times  \Delta \m_n(s)) , \phi \r\rangle_{L^2} ds \\
				&\quad- \psi(m(s))\l\langle \m(s) \times (\m(s) \times \Delta \m(s)) , \phi \r\rangle_{L^2} \b|\, ds  \\
				& \leq \mathbb{E^{\prime}}   \int_{0}^{T} | (\psi(m_n(s)) - \psi(m(s))) \l\langle m_n(s) \times (\m_n(s) \times  \Delta \m_n(s)) , \phi \r\rangle_{L^2} |\, ds  \\
				& \quad+ \mathbb{E^{\prime}}   \int_{0}^{T} \psi(m(s)) |\l\langle (\m_n(s) - \m(s)) \times (\m_n(s) \times  \Delta \m_n(s)) , \phi \r\rangle_{L^2}\, ds | \\
				&  \quad + \mathbb{E^{\prime}}   \int_{0}^{T} \psi(m(s)) |\l\langle \m(s) \times (\m_n(s) \times  \Delta \m_n(s) - \m(s) \times  \Delta \m(s)) , \phi \r\rangle_{L^2}|\, ds.
			\end{align*}
			Convergence of the first term on the right hand side of the above inequality follows from the previous Lemma \ref{Lemma convergence of the bump function}.\\
			The second term goes to 0 as $n$ goes to infinity as follows.
		}
		\dela{Some bound ensures weak convergence of the triple product to some term in square integrable processes. Then considering the test function to be possibly a simple function, we go on to show that the convergence is indeed to a term of the required form. It is necessary to show that the limit triple product term is square integrable, which can be shown by the available estimates.}
		By the triangle inequality, we have
		\begin{align}\label{eqn intermediate 1 convergencem mn times mn times Laplacian mn}
			\nonumber &\bigg|	\mathbb{E^{\prime}}   \int_{0}^{T}  \l\langle m_n(s) \times (\m_n(s) \times  \Delta \m_n(s)) , \phi \r\rangle_{L^2} ds - \l\langle \m(s) \times (\m(s) \times \Delta \m(s)) , \phi \r\rangle_{L^2} \, ds\bigg|  \\
			\nonumber \leq &  \bigg|\mathbb{E^{\prime}}   \int_{0}^{T}  \l\langle (\m_n(s) - \m(s)) \times (\m_n(s) \times  \Delta \m_n(s)) , \phi \r\rangle_{L^2}\, ds \bigg| \\
			&  + \bigg|\mathbb{E^{\prime}}   \int_{0}^{T}  \l\langle \m(s) \times (\m_n(s) \times  \Delta \m_n(s) - \m(s) \times  \Delta \m(s)) , \phi \r\rangle_{L^2}\, ds\bigg|.
		\end{align}
		The first term of \eqref{eqn intermediate 1 convergencem mn times mn times Laplacian mn} goes to 0 as $n$ goes to infinity as follows.
		\dela{
			\begin{align*}
				& \bigg|\mathbb{E^{\prime}}   \int_{0}^{T}  \l\langle (\m_n(s) - \m(s)) \times (\m_n(s) \times  \Delta \m_n(s)) , \phi \r\rangle_{L^2}\, ds \bigg| \\
				&\leq \mathbb{E^{\prime}}   \int_{0}^{T}  |\l\langle (\m_n(s) - \m(s)) \times (\m_n(s) \times  \Delta \m_n(s)) , \phi \r\rangle_{L^2}|\, ds \\
				&\leq \mathbb{E^{\prime}}   \int_{0}^{T}  |\m_n(s) - \m(s)|_{L^{\infty}}|\m_n(s) \times  \Delta \m_n(s)|_{L^2}|\phi|_{L^{2}}\, ds \\
				& \leq C \,\mathbb{E^{\prime}}   \int_{0}^{T}  |\m_n(s) - \m(s)|_{L^{2}} ^{\frac{1}{2}}|\m_n(s) - \m(s)|_{H^1}^{\frac{1}{2}} |\m_n(s) \times  \Delta \m_n(s)|_{L^2} |\phi(s)|_{L^{2}} \, ds  \\
				&\leq C\left( \mathbb{E^{\prime}} \sup_{s\in [0,T]} (|\m_n(s)|_{H^1} + |\m(s)|_{H^1})^2 \right)^{\frac{1}{4}} \left( \mathbb{E^{\prime}} \l(\int_{0}^{T} |\m_n(s) - \m(s)|^2_{L^{2}} ds \r)^2 \right)^{\frac{1}{4}} \centerdot \\
				& \adda{\quad \left( \mathbb{E^{\prime}} \sup_{t\in [0,T]} |\m_n(s) \times \Delta\m(s)|_{L^2}^2 \right)^{\frac{1}{4}} \l( \mathbb{E^{\prime}} \l(\int_{0}^{T}|\phi|_{L^{2}}^4 \, ds \r)^{\frac{1}{4}} \r)^{\frac{1}{4}}}\\
				& \leq C {E^{\prime}} \sup_{t\in [0,T]} |\m_n(s) - \m(s)|_{L^2}^4.
			\end{align*}
		}

		\begin{align}\label{eqn intermediate 2 convergencem mn times mn times Laplacian mn}
			\nonumber & \bigg|\mathbb{E^{\prime}}   \int_{0}^{T}  \l\langle (\m_n(s) - \m(s)) \times (\m_n(s) \times  \Delta \m_n(s)) , \phi(s) \r\rangle_{L^2}\, ds \bigg| \\
			\nonumber &\leq \mathbb{E^{\prime}}   \int_{0}^{T}  |\l\langle (\m_n(s) - \m(s)) \times (\m_n(s) \times  \Delta \m_n(s)) , \phi(s) \r\rangle_{L^2}|\, ds \\
			\nonumber & \leq \l( \mathbb{E^{\prime}} \int_{0}^{T} \l| \m_n(s) - \m(s)) \r|_{L^4}^4 \, ds \r)^{\frac{1}{4}}
			\l( \mathbb{E^{\prime}} \int_{0}^{T} \l| \phi(s) \r|_{L^4}^4 \, ds  \r)^{\frac{1}{4}}
			\l( \mathbb{E^{\prime}} \int_{0}^{T} \l| \m_n(s) \times  \Delta \m_n(s) \r|_{L^2}^2 \, ds \r)^{\frac{1}{2}} \\
			\dela{&\leq \mathbb{E^{\prime}}   \int_{0}^{T}  |\m_n(s) - \m(s)|_{L^{\infty}}|\m_n(s) \times  \Delta \m_n(s)|_{L^2}|\phi|_{L^{2}}\, ds \\
				\nonumber & \leq C \,\mathbb{E^{\prime}}   \int_{0}^{T}  |\m_n(s) - \m(s)|_{L^{2}} ^{\frac{1}{2}}|\m_n(s) - \m(s)|_{H^1}^{\frac{1}{2}} |\m_n(s) \times  \Delta \m_n(s)|_{L^2} |\phi(s)|_{L^{2}} \, ds  \\
				\nonumber &\leq C\left( \mathbb{E^{\prime}} \sup_{s\in [0,T]} (|\m_n(s)|_{H^1} + |\m(s)|_{H^1})^2 \right)^{\frac{1}{4}} \left( \mathbb{E^{\prime}} \l(\int_{0}^{T} |\m_n(s) - \m(s)|^2_{L^{2}} ds \r)^2 \right)^{\frac{1}{4}} \centerdot \\
				\nonumber & \adda{\quad \left( \mathbb{E^{\prime}} \sup_{t\in [0,T]} |\m_n(s) \times \Delta\m(s)|_{L^2}^2 \right)^{\frac{1}{4}} \l( \mathbb{E^{\prime}} \l(\int_{0}^{T}|\phi|_{L^{2}}^4 \, ds \r)^{\frac{1}{4}} \r)^{\frac{1}{4}}}\\}
			& \leq C \l( \mathbb{E}^{\prime} \int_{0}^{T} |\m_n(s) - \m(s)|_{L^4}^4 \, ds \r)^{\frac{1}{4}}.
		\end{align}

		By the convergence in 
		\eqref{eqn convergence of L4L4 norm of m_n prime} the right hand side of the above inequality \eqref{eqn intermediate 2 convergencem mn times mn times Laplacian mn} goes to 0 as $n$ goes to infinity.
		
		The above bound uses the inequality \eqref{Inequality L infinity leq L2 H1} followed by the use of the generalized H\"older inequality. More precisely,
		\begin{equation*}
			|v_1v_2v_3|_{L^1} \leq |v_1|_{L^4} |v_2|_{L^4} |v_3|_{L^2}\ \text{for}\ v_1,v_2\in L^4, v_3\in L^2.
		\end{equation*}
		
		For the second term, we have the following.
		\begin{align}\label{eqn intermediate 3 convergencem mn times mn times Laplacian mn}
			\nonumber &\bigg|\mathbb{E^{\prime}}   \int_{0}^{T}  \l\langle \m(s) \times (\m_n(s) \times  \Delta \m_n(s) - \m(s) \times  \Delta \m(s)) , \phi(s) \r\rangle_{L^2}\, ds\bigg|\\
			= & \bigg|\mathbb{E^{\prime}}   \int_{0}^{T}  \l\langle  (\m_n(s) \times  \Delta \m_n(s) - \m(s) \times  \Delta \m(s)) , \m(s) \times \phi(s) \r\rangle_{L^2}\, ds\bigg|.
		\end{align}
		For $v_1,v_2\in H^1$, using the continuous embedding $H^1\hookrightarrow L^{\infty}$, we can show that there exists a constant $C>0$ such that
		\begin{equation}
			\l| v_1 v_2 \r|_{H^1} \leq C \l| v_1 \r|_{H^1} \l| v_2 \r|_{H^1}.
		\end{equation}
		Therefore,
		\begin{align*}
			\mathbb{E}^\p\int_{0}^{T}\l| \m(s) \times \phi(s) \r|_{H^1}^2 \, ds 
			& \leq C \,  \mathbb{E}^\p \int_{0}^{T}\l| \m(s) \r|_{H^1}^2 \l| \phi(s) \r|_{H^1}^2 \, ds \\
			& \leq  C \,  \mathbb{E}^\p \l[ \sup_{s\in[0,T]} \l| \m(s) \r|_{H^1}^2 \int_{0}^{T} \l| \phi(s) \r|_{H^1}^2 \, ds \r] \\
			& \leq \l( \mathbb{E}^\p   \sup_{s\in[0,T]} \l|\m(s)\r|_{H^1}^4 \r)^{\frac{1}{2}}
			\l( \mathbb{E}^\p   \int_{0}^{T} \l| \phi(s) \r|_{H^1}^4 \, ds \r)^{\frac{1}{2}} \\
			& < \infty.
		\end{align*}
		The finiteness of the right hand side is due to the bound \eqref{bound on m prime H1} and the assumption on $\phi$.
		For this calculation, letting $\psi = \m \times \phi$,  the right hand side of \eqref{eqn intermediate 3 convergencem mn times mn times Laplacian mn} goes to 0 following the weak convergence \eqref{equation weak convergence of mn prime times delta mn prime to Y} of $\m_n \times \Delta \m_n$. This concludes the proof of \dela{Lemma \ref{Lemma convergencem mn times Laplacian mn and mn times mn times Laplacian mn}}Lemma \ref{Lemma convergencem mn times mn times Laplacian mn}.
	\end{proof}
	\dela{By using the bound \eqref{bound on m prime H1}, we can show that $\m \times \l( \m \times \Delta \m \r) \in L^2(\Omega^\p: L^2(0,T;L^2))$.}
	\begin{lemma}\label{Lemma convergence mn times un and mn times mn times un}
		For $\phi\in\dela{ H^1}L^4(\Omega^\p ; H^1)$, the following convergences hold.	
		
		\begin{equation*}
			\lim_{n\rightarrow\infty}\mathbb{E^{\prime}} \int_{0}^{T} \l\langle \m_n(s) \times \u_n(s), \phi \r\rangle_{L^2}\, ds = 		\mathbb{E^{\prime}} \int_{0}^{T} \l\langle \m(s) \times \u(s), \phi \r\rangle_{L^2}\, ds		
		\end{equation*}
		
		\begin{align*}
			&\lim_{n\rightarrow\infty} \mathbb{E^{\prime}}\int_{0}^{T} \psi(|\m_n(s)|_{L^{\infty}}) \l\langle \m_n(s) \times (\m_n(s) \times \u_n(s)) , \phi \r\rangle_{L^2}\, ds  \\
			& \quad 	=   \mathbb{E^{\prime}}\int_{0}^{T} \psi(|\m(s)|_{L^{\infty}})\l\langle \m(s) \times (\m(s) \times \u(s)) , \phi \r\rangle_{L^2}\, ds.
		\end{align*}
	\end{lemma}
	\begin{proof}[Proof of Lemma \ref{Lemma convergence mn times un and mn times mn times un}]
		We prove the second convergence. The first one can be shown similarly.	
		\begin{align}\label{equation intermediate 1 convergence of mn prime times mn times un prime}
			\nonumber &\mathbb{E^{\prime}}\int_{0}^{T} \psi(|\m_n(s)|_{L^{\infty}}) \l\langle \m_n(s) \times (\m_n(s) \times \u_n(s)) , \phi \r\rangle_{L^2}\, ds \\
			\nonumber & \quad - \mathbb{E^{\prime}}\int_{0}^{T} \psi(|\m(s)|_{L^{\infty}}) \l\langle \m(s) \times (\m(s) \times \u(s)) , \phi \r\rangle_{L^2}\, ds \\
			\nonumber & =  \mathbb{E^{\prime}}\int_{0}^{T} \bigl[ \psi(\m_n(s)) - \psi\bigl(\m(s)\bigr) \bigr] \l\langle \m_n(s) \times (\m_n(s) \times \u_n(s)) , \phi \r\rangle_{L^2}\, ds \\
			& \quad + \mathbb{E^{\prime}}\int_{0}^{T} \psi(\m(s)) \l\langle \l[ \m_n(s) \times \bigl(\m_n(s) \times \u_n(s)\bigr) - \m(s) \times \bigl(\m(s) \times \u(s)\bigr) \r] , \phi \r\rangle_{L^2}\, ds.
		\end{align}
		Combining Lemma \ref{Lemma convergence of the bump function} and the bound \eqref{bound on mn prime times mn prime times un prime} in Proposition \ref{Proposition bounds on m_n prime}, the first term on the right hand side of the equality \eqref{equation intermediate 1 convergence of mn prime times mn times un prime} goes to $0$ as $n$ goes to infinity.\dela{That the first term on the right hand side goes to zero as $n$ goes to infinity can be shown as done in the previous lemma. The main idea is to use the regularity of the cut-off function $\psi$.} For the second term, we have the following.
		\begin{align}\label{equation intermediate 2 convergence of mn prime times mn times un prime}
			\nonumber & \bigg|  \mathbb{E^{\prime}} \int_{0}^{T} \psi(\m(s)) \l\langle \m_n(s) \times (\m_n(s) \times \u_n(s))  -  \m(s) \times (\m(s) \times \u(s)) , \phi \r\rangle_{L^2} \, ds \bigg| \\
			\nonumber \leq &  \mathbb{E^{\prime}} \int_{0}^{T} | \psi(\m(s)) \l\langle (\m_n(s) - \m(s)) \times (\m_n(s) \times \u_n(s)) , \phi \r\rangle_{L^2}|\, ds \\
			\nonumber & + \mathbb{E^{\prime}} \int_{0}^{T} | \psi(\m(s))\l\langle \m(s) \times ((\m_n(s)  - \m(s)) \times \u_n(s)) , \phi \r\rangle_{L^2}|\, ds \\
			& + \l| \mathbb{E^{\prime}} \int_{0}^{T} \psi(\m(s))\l\langle \m(s) \times (\m(s) \times (\u_n(s) - \u(s))) , \phi \r\rangle_{L^2}\, ds \r| .
		\end{align}
		\textbf{Claim:} All the three terms on the right hand side of the above inequality go to $0$ as $n$ goes to infinity. We use the assumption on $\phi$ along with the fact that the space $H^1$ is continuously embedded into the space $L^{\infty}$. By \eqref{eqn convergence of L4L4 norm of m_n prime}, the sequence $\m_n$ converges to $\m$ in $L^4\l( \Omega^\p ; L^4\l(0,T;L^4\r)\r)$. Hence for the first term on the right hand side of \eqref{equation intermediate 2 convergence of mn prime times mn times un prime}, it is sufficient to show that $\l(\m_n \times \u_n\r) \times \phi \in L^{\frac{4}{3}}\l(\Omega^\p ; L^{\frac{4}{3}}\l(0,T;L^{\frac{4}{3}}\r)\r)$.
		
		Note that
		\begin{align*}
			& \mathbb{E}^{\prime}\int_{0}^{T} |\l( \m_n(s) \times \u_n(s) \r) \times \phi|_{L^{\frac{4}{3}}}^{\frac{4}{3}}\ ds \\
			&\leq \mathbb{E}^{\prime}\int_{0}^{T} |\m_n(s)|_{L^4}^{\frac{4}{3}} |u_n(s)|_{L^2}^{\frac{4}{3}} |\phi|_{L^{\infty}}^{\frac{4}{3}} \ ds \\
			& \leq C \mathbb{E}^{\prime} |\phi|_{H^1}^{\frac{4}{3}}\int_{0}^{T} |\m_n(s)|_{H^1}^{\frac{4}{3}} |\u_n(s)|_{L^2}^{\frac{4}{3}}\ ds \ ( \text{Since}\ H^1\hookrightarrow L^{\infty} \hookrightarrow L^4 ) \\
			& \leq C \l(\mathbb{E}^\p \l|\phi\r|_{H^1}^4\r)^{\frac{1}{3}}\l( \mathbb{E}^{\prime} \int_{0}^{T}\l|\m_n(s)\r|_{H^1}^4\, ds \r)^{\frac{1}{3}} \l( \mathbb{E}^{\prime} \l( \int_{0}^{T}\l|\u_n(s)\r|_{L^2}^2\, ds \r)^2 \r)^{\frac{1}{3}}.
		\end{align*}
		The right hand side of the above inequality is finite by the bounds \eqref{equation H1 bound on m prime} and \eqref{bound on u_n prime}. The second term follows similarly.\\		
		The third term goes to zero due to the cut-off function and the weak convergence \eqref{weak convergence of u_n prime to u in L2L2L2}.
		Hence all the three terms on the right hand side of the inequality \eqref{equation intermediate 2 convergence of mn prime times mn times un prime} go to $0$ as $n$ goes to infinity and the claim holds.
	\end{proof}
	The following proposition proves the convergence of the terms corresponding to $G_n(\m_n)$.

	\begin{lemma}\label{Lemma convergence of Gn}
		\begin{align*}
			\lim_{n\rightarrow\infty} \mathbb{E}^\p \sup_{s\in[0,T]} \l| G_n(\m_n(s)) - G(\m(s))\r|_{L^2}^2 = 0.
		\end{align*}
		
		\dela{ For $\phi\in H^1$, the following convergences hold as $n$ goes to infinity. We have the following convergences.
			\begin{enumerate}
				\item \begin{equation*}
					\mathbb{E}^{\prime}\sup_{s\in [0,T]} |P_n( \m_n(s) \times h) - \m(s) \times h|_{L^2}^4 \rightarrow 0
				\end{equation*}
				
				\item \begin{equation*}
					\mathbb{E}^{\prime}\sup_{s\in [0,T]} |\psi(|\m_n(s)|_{L^{\infty}})P_n(\m_n(s) \times (\m_n(s) \times h)) - \psi(|\m(s)|_{L^{\infty}})\m(s) \times (\m(s) \times h)|_{L^2}^2  \rightarrow 0
				\end{equation*}
			\end{enumerate}
		}
		\dela{
			\textbf{Note:} It would have been enough to show the convergence in the weak sense as shown for some other terms. But the convergence in $L^2(\Omega;(H^1)^{\prime})$ is required for the convergence of $M_n^{\prime}(t)$ to $M^{\prime}$ in Lemma \ref{convergence of Mn prime to M prime}.
			\adda{Why is the mentioned strong convergence required? Possibly for showing the convergence of the Stochastic integral.}
		}
	\end{lemma}
	\begin{proof}[Proof of Lemma \ref{Lemma convergence of Gn}]
		The proof follows from Lemma \ref{Lemma sup in t L2 in space convergene of m_n prime} and Lemma \ref{Lemma convergence of the bump function}.
		\dela{
			We prove the second convergence. The first follows similarly.
			Even though the terms contain the cut-off function, we do not mention it here for brevity. It can be handled as done previously, by adding and subtracting suitable terms, followed by using the H\"older's inequality and the convergence of $\m_n$ to $\m$.
			
			similarly.\\\begin{align*}
				|P_n(\m_n(s) &\times (\m_n(s) \times h)) - \m(s) \times (\m(s) \times h)|_{L^2} \\
				\leq& |P_n(\m_n(s) \times (\m_n(s) \times h)) - P_n(\m(s) \times (\m(s) \times h))|_{L^2} \\
				& + |P_n(\m(s) \times (\m(s) \times h)) - \m(s) \times (\m(s) \times h)|_{L^2} \\
				\leq& |P_n(\m_n(s) \times (\m_n(s) \times h) - \m(s) \times (\m(s) \times h))|_{L^2} \\
				& + |P_n(\m(s) \times (\m(s) \times h)) - \m(s) \times (\m(s) \times h)|_{L^2} \\
				\leq& |\m_n(s) \times (\m_n(s) \times h) - \m(s) \times (\m(s) \times h)|_{L^2} \\
				& + |P_n(\m(s) \times (\m(s) \times h)) - \m(s) \times (\m(s) \times h)|_{L^2} \\
				\leq&  |(\m_n(s) - \m(s)) \times (\m_n(s) \times h) |_{L^2} \\
				& + |\m(s) \times ((\m_n(s)- \m(s)) \times h) |_{L^2} \\
				& + |P_n(\m(s) \times (\m(s) \times h)) - \m(s) \times (\m(s) \times h)|_{L^2} \\
			\end{align*}
			\dela{
				Let $s\in[0,T]$ and $n\in\mathbb{N}$ be fixed.
				
				\begin{align*}
					|P_n(\m_n(s) &\times (\m_n(s) \times h)) - \m(s) \times (\m(s) \times h)|_{L^2} \\
					\leq& |P_n(\m_n(s) \times (\m_n(s) \times h)) - P_n(\m(s) \times (\m(s) \times h))|_{L^2} \\
					& + |P_n(\m(s) \times (\m(s) \times h)) - \m(s) \times (\m(s) \times h)|_{L^2} \\
					\leq& |P_n(\m_n(s) \times (\m_n(s) \times h) - \m(s) \times (\m(s) \times h))|_{L^2} \\
					& + |P_n(\m(s) \times (\m(s) \times h)) - \m(s) \times (\m(s) \times h)|_{L^2} \\
					\leq& |\m_n(s) \times (\m_n(s) \times h) - \m(s) \times (\m(s) \times h)|_{L^2} \\
					& + |P_n(\m(s) \times (\m(s) \times h)) - \m(s) \times (\m(s) \times h)|_{L^2} \\
					\leq&  |(\m_n(s) - \m(s)) \times (\m_n(s) \times h) |_{L^2} \\
					& + |\m(s) \times ((\m_n(s)- \m(s)) \times h) |_{L^2} \\
					& + |P_n(\m(s) \times (\m(s) \times h)) - \m(s) \times (\m(s) \times h)|_{L^2} \\
				\end{align*}
				
				Thus,	
			}
			\begin{align*}
				\mathbb{E^{\prime}} & \sup_{s\in[0,T]} |P_n(\m_n(s) \times (\m_n(s) \times h)) - \m(s) \times (\m(s) \times h)|_{L^2}^4 \\
				\leq&  	\mathbb{E^{\prime}} \sup_{s\in[0,T]} |\m_n(s) \times (\m_n(s) \times h) - \m(s) \times (\m(s) \times h)|_{L^2}^4 \\
				& + 	\mathbb{E^{\prime}} \sup_{s\in[0,T]} |P_n(\m(s) \times (\m(s) \times h)) - \m(s) \times (\m(s) \times h)|_{L^2}^4 \\
				& + 	\mathbb{E^{\prime}} \sup_{s\in[0,T]}  |\m(s) \times ((\m_n(s)- \m(s)) \times h) |_{L^2}^4 \\
				& + 	\mathbb{E^{\prime}} \sup_{s\in[0,T]} |P_n(\m(s) \times (\m(s) \times h)) - \m(s) \times (\m(s) \times h)|_{L^2}^4
			\end{align*}
			Using the Lemma \ref{Lemma sup in t L2 in space convergene of m_n prime}, the first three terms go to 0 as $n$ goes to infinity.
			
			For the fourth term,
			$$\lim_{n\rightarrow\infty}|P_n(\m(s) \times (\m(s) \times h)) - (\m(s) \times (\m(s) \times h))|_{L^2} = 0$$ for each $s\in[0,T]$ $\mathbb{P}^{\prime}$-a.s.\\
			\textbf{Claim:} $\m \times (\m \times h)$ is uniformly continuous on $[0,T]$ taking values in $L^2$ gives a uniform bound.
			
			The fourth term thus goes to 0 as a consequence of the above argument and the dominated convergence theorem.
			The lemma follows by the above arguments combined with Lemma \ref{Lemma convergence of the bump function}. This concludes the proof of the Lemma \ref{Lemma convergence of Gn}.\\
			\textbf{Proof for the claim that $\m \times (\m \times h)$ is uniformly continuous taking values in $L^2$:}
			$\m$ is uniformly continuous taking values in $L^2$ since it is continuous on a compact domain $[0,T]$.
			
			Now let $(t_n)_{n\in\mathbb{N}}$ be a sequence in $[0,T]$ that converges to $t$ in $[0,T]$. The bound in \eqref{bound on m prime H1} implies that $$\sup_{s\in [0,T]}|\m(s,\omega^{\prime})| < \infty.$$ Note here that this holds only for $ \mathbb{P}^{\prime} $ a.s. $\omega^{\prime}\in\Omega^{\prime}$. Fix this $\omega^{\prime}$.
			\begin{align*}
				|\m(t_n) \times (\m(t_n) \times h) - \m(t) \times ( \m(t)\times h)|_{L^2} \leq& |(\m(t_n) - \m(t)) \times (\m(t_n) \times h)|_{L^2} \\
				& + |\m(t) \times ((\m(t_n) - \m(t)) \times h)|_{L^2} \\
				\leq& |\m(t_n) - \m(t)|_{L^2}|\m(t_n)|_{L^{\infty}}|h|_{L^{\infty}} \\
				& + |\m(t_n) - \m(t)|_{L^2}|\m(t)|_{L^{\infty}}|h|_{L^{\infty}} \\
				\leq& C |\m(t_n) - \m(t)|_{L^2}|\m(t_n)|_{H^1}|h|_{L^{\infty}} \\&\quad (\text{Since}\ H^1\hookrightarrow L^{\infty}) \\
				& + C |\m(t_n) - \m(t)|_{L^2}|\m(t)|_{H^1}|h|_{L^{\infty}} \\
				\leq& C C(h) |\m(t_n) - \m(t)|_{L^2}.
			\end{align*}
			The right hand side of the above inequality goes to $0$ as $n$ goes to infinity. Thus we have shown the continuity of $\m\times (\m \times h)$.
			Compactness of $[0,T]$ concludes the uniform continuity.
		}
	\end{proof}

	
	Define the following $L^2$-valued random variables $\{M_n(t)\}_{t\in[0,T]}$ and $\{M_n^{\prime}(t)\}_{t\in[0,T]}$ on $(\Omega, \mathbb{F}, \mathbb{P})$ and $(\Omega^{\prime}, \mathbb{F}^{\prime}, \mathbb{P}^{\prime})$, respectively by
	\begin{align}\label{Definition of M_n}
		\nonumber M_n(t) := & m_n(t) -m_n(0) - \int_{0}^{t} \bigg[ F_n^1(m_n(s)) - \alpha \, F_n^2(m_n(s)) + F_n^3(m_n(s)) \\
		&+ \frac{1}{2}\psi\bigl(m_n(s)\bigr)^2 \l[DG\bigl(m_n(s)\bigr)\r]\l[G_n\bigl(m_n(s)\bigr)\r] \bigg] \, ds,
	\end{align}
	\dela{
		\begin{align}
			\nonumber M_n(t) &:= m_n(t) -m_n(0) - \int_{0}^{t} \bigg[ F_n^1(m_n(s)) - \alpha \, F_n^2(m_n(s)) + F_n^3(m_n(s)) + \frac{1}{2}\l[DG(m_n(s))\r]\l(G_n(m_n(s))\r) \bigg] \, ds.
			\nonumber - &  P_n(P_n(m_n(s) \times h) \times h) \\
			\nonumber + & \alpha \,  \psi(|m_n(s)|_{L^{\infty}})\psi(|P_n(m_n(s) \times h)|_{L^{\infty}}) \psi(|P_n(m_n(s) \times (m_n(s) \times h))|_{L^{\infty}}) \centerdot \\
			\nonumber & P_n(P_n(m_n(s) \times (m_n(s) \times h)) \times h) \\	\nonumber + &  \alpha \,  \psi(|m_n(s)|_{L^{\infty}}) \psi(|P_n(m_n(s) \times h)|_{L^{\infty}}) \psi(|P_n(m_n(s) \times (m_n(s) \times h))|_{L^{\infty}})  \centerdot \\
			\nonumber & P_n(P_n(m_n(s) \times h) \times (m_n(s) \times h))  \\
			\nonumber + & \alpha \, P_n(m_n(s) \times (P_n(m_n(s) \times h) \times h)) \\
			\nonumber - &  \alpha^2  \psi^2(|m_n(s)|_{L^{\infty}}) \psi(|P_n(m_n(s) \times h)|_{L^{\infty}})^2 \psi(|P_n(m_n(s) \times (m_n(s) \times h))|_{L^{\infty}})^2\centerdot \\
			\nonumber & P_n(P_n(m_n(s) \times (m_n(s) \times h)) \times (m_n(s) \times h)) \\
			- &  \alpha^2 P_n(m_n(s) \times (P_n(m_n(s) \times (m_n(s) \times h))\times h)) \bigg] \ ds
		\end{align}
	}
	and
	\dela{
		\begin{align}
			\nonumber M^{\prime}_n(t) &:= \m_n(t) -\m_n(0) - \int_{0}^{t}[F_n^1(\m_n(s)) - \alpha \, F_n^2(\m_n(s)) + F_n^3(\m_n(s))\\
			\nonumber & - P_n(P_n(\m_n(s) \times h) \times h) \\
			\nonumber &+ \alpha \,  \psi(|\m_n(s)|_{L^{\infty}})\psi(|P_n(\m_n(s) \times h)|_{L^{\infty}}) \psi(|P_n(\m_n(s) \times (\m_n(s) \times h))|_{L^{\infty}}) \centerdot \\
			\nonumber & P_n(P_n(\m_n(s) \times (\m_n(s) \times h)) \times h) \\	\nonumber & + \alpha \,  \psi(|\m_n(s)|_{L^{\infty}}) \psi(|P_n(\m_n(s) \times h)|_{L^{\infty}}) \psi(|P_n(\m_n(s) \times (\m_n(s) \times h))|_{L^{\infty}})  \centerdot \\
			\nonumber & P_n(P_n(\m_n(s) \times h) \times (\m_n(s) \times h))  \\
			\nonumber &+ \alpha \, P_n(\m_n(s) \times (P_n(\m_n(s) \times h) \times h)) \\
			\nonumber & - \alpha^2  \psi^2(|\m_n(s)|_{L^{\infty}}) \psi(|P_n(\m_n(s) \times h)|_{L^{\infty}})^2 \psi(|P_n(\m_n(s) \times (\m_n(s) \times h))|_{L^{\infty}})^2\centerdot \\
			\nonumber & P_n(P_n(\m_n(s) \times (\m_n(s) \times h)) \times (\m_n(s) \times h)) \\
			& - \alpha^2 P_n(\m_n(s) \times (P_n(\m_n(s) \times (\m_n(s) \times h))\times h)) ]\ ds.
		\end{align}
	}
	\begin{align}\label{Definition of M_n^{prime}}
		\nonumber M^{\prime}_n(t) := & \m_n(t) -\m_n(0) - \int_{0}^{t}[F_n^1(\m_n(s)) - \alpha \, F_n^2(\m_n(s)) + F_n^3(\m_n(s)) \\
		& + \frac{1}{2}\psi\bigl(\m_n(s)\bigr)^2\l[DG\bigl(\m_n(s)\bigr)\r] \l[G_n\bigl(\m_n(s)\bigr)\r] \bigg] \, ds,	
	\end{align}
	The aim here is to show that for each $t\in[0,T]$, $M^{\prime}_n(t)$ converges in some sense to $M^{\prime}(t)$, where $M^{\prime}(t)$ is defined as
	
	\dela{
		\begin{align}
			\nonumber &M^{\prime}(t) := \m(t) - \m_0 - \int_{0}^{t} [\m(s) \times \Delta \m(s) - \alpha \, \m(s) \times (\m(s) \times \Delta \m(s)) + \m(s)\times \u(s)   \\
			\nonumber& - \alpha \, \m(s) \times (\m(s) \times \u(s)) \\
			\nonumber & - (\m(s) \times h) \times h \\
			\nonumber &+ \alpha \,  \psi(|\m(s)|_{L^{\infty}})\psi(|P(\m(s) \times h)|_{L^{\infty}}) \psi(|\m(s) \times (\m(s) \times h)|_{L^{\infty}})\m(s) \times (\m(s) \times h) \times h \\
			\nonumber & + \alpha \,  \psi(|\m(s)|_{L^{\infty}}) \psi(|(\m(s) \times h)|_{L^{\infty}}) \psi(|(\m(s) \times (\m(s) \times h))|_{L^{\infty}}) (\m(s) \times h) \times (\m(s) \times h)  \\
			\nonumber &+ \alpha \, (\m(s) \times ((\m(s) \times h) \times h)) \\
			\nonumber & - \alpha^2  \psi^2(|\m(s)|_{L^{\infty}}) \psi(|(\m(s) \times h)|_{L^{\infty}})^2 \psi(|(\m(s) \times (\m(s) \times h))|_{L^{\infty}})^2 \centerdot \\
			\nonumber&\quad (\m(s) \times (\m(s) \times h)) \times (\m(s) \times h) \\
			& - \alpha^2 (m(s) \times ((\m(s) \times (\m(s) \times h))\times h)) ]\ ds.
		\end{align}
	}
	\begin{align}\label{definition of M prime}
		\nonumber &M^{\prime}(t) := \m(t) - \m_0 - \int_{0}^{t} \bigg[ \m(s) \times \Delta \m(s) - \alpha \, \m(s) \times (\m(s) \times \Delta \m(s)) + \m(s)\times \u(s)   \\
		& - \alpha \,\psi(\m(s)) \m(s) \times \bigl(\m(s) \times \u(s)\bigr) +	\frac{1}{2}\psi\bigl(\m(s)\bigr)^2 \l[DG\bigl(\m_n(s)\bigr)\r] \bigl[G\bigl(\m_n(s)\bigr)\bigr]\biggr] \, ds.
	\end{align}
	The main contents of the remainder of this section will be as follows:
	\begin{enumerate}
		\item Showing the convergence of $M^{\prime}_n(t)$ to $M^{\prime}(t)$ in some sense (Lemma \ref{convergence of Mn prime to M prime}).
		
		\item Showing that the process $W^{\prime}$, obtained as a limit of Wiener processes $W_n^{\prime}$ is a Wiener process (Lemma \ref{Lemma W prime is a Wiener process}).
		
		\item Showing that the limit $M^{\prime}$ is indeed an It\^o's integral (with respect to the process $W^{\prime}$) as required. This will be done in two steps: first we prove Lemma \ref{convergence of M_n prime to Ito integral}, which shows that $M^\p_n$ converges to the required stochastic integral and then comparing this with Lemma \ref{convergence of Mn prime to M prime} gives us the required result.
	\end{enumerate}
	\dela{Question: Note that the above calculations are for $\phi\in H^1$ and not for $\phi\in L^2(\Omega^\p:L^2(0,T;H^1))$. Therefore this cannot be compared to the weak convergence of the corresponding terms.
		
		Also, the terms in $M^\p_n$ have projection operators acting on them. Try to give a brief justification as to why the convergence shown initially is sufficient.}
	\begin{lemma}\label{convergence of Mn prime to M prime}
		For $\phi\in\dela{ H^1}L^4(\Omega ; H^1)$,
		and $t\in[0,T]$
		\begin{equation*}
			\mathbb{E^{\prime}} \l\langle M_n^{\prime}(t) , \phi \r\rangle_{L^2} \rightarrow \mathbb{E^{\prime}} \l\langle M^{\prime}(t) , \phi \r\rangle_{L^2}\ \text{as}\ n \to \infty.
		\end{equation*}
	\end{lemma}
	\begin{proof}[Proof of Lemma \ref{convergence of Mn prime to M prime}]
		
		We show the convergence of the terms individually. The previously stated lemmata, viz. Lemma \ref{Lemma sup in t L2 in space convergene of m_n prime}, Lemma \ref{Lemma convergence of the bump function}, Lemma \ref{Lemma convergencem mn times Laplacian mn and mn times mn times Laplacian mn}, Lemma \ref{Lemma convergence mn times un and mn times mn times un}, Lemma \ref{Lemma convergencem mn times mn times Laplacian mn}, Lemma \ref{Lemma convergence of Gn} show the convergence of some of the terms. The terms that remain are the ones corresponding to the Stratonovich to It\^o correction term. The convergence follows from the convergence described in Lemma \ref{Lemma sup in t L2 in space convergene of m_n prime} and Lemma \ref{Lemma convergence of the bump function}. \dela{ The convergence of these can be shown as follows.}

		We show the calculations for one term. Rest of the terms follow similarly.\\
		\textbf{Claim:}
		\begin{align*}
			\lim_{n\rightarrow\infty}& \mathbb{E^{\prime}}\int_{0}^{T}\bigg[\big|\psi^2(\m_n(s)) P_n\bigl(P_n(\m_n(s) \times (\m_n(s) \times h)) \times (\m_n(s) \times h)\bigr)\\
			& - \psi^2\bigl(\m(s)\bigr) \m(s) \times (\m(s) \times h) \times (\m(s) \times h)\big|_{(H^1)^{\prime}}^2 \bigg] \, ds = 0.
		\end{align*}
		Let $v_1,v_2,w_1,w_2\in H_n$. Then
		\begin{align*}
			\l| \psi(v_1)w_1 - \psi(v_1)w_1 \r|_{L^2} \leq \l| \l[\psi(v_1) - \psi(v_2)\r]w_1  \r|_{L^2} + \l| \psi(v_2)\l[w_1 - w_2\r]  \r|_{L^2}.
		\end{align*}
		The convergence in the claim can be seen into two parts, one with the convergence for the cut-off and one with the convergence for the remaining term. For the convergence of the cut-off function, we have Lemma \ref{Lemma convergence of the bump function}. We therefore continue with the remaining part. Note that the function $\psi$ need not be written here since it takes values in $[0,1]$ and hence does not affect the inequalities.\\
		The convergence can be split up into the following parts.
		\begin{align*}
			&|P_n(P_n(\m_n(s) \times (\m_n(s) \times h)) \times (\m_n(s) \times h)) - \m(s) \times (\m(s) \times h) \times (\m(s) \times h) |_{(H^1)^{\prime}} \\
			\leq & |P_n(P_n(\m_n(s) \times (\m_n(s) \times h)) \times (\m_n(s) \times h)) - P_n(\m(s) \times (\m(s) \times h) \times (\m(s) \times h)) |_{(H^1)^{\prime}} \\
			& + |P_n(\m(s) \times (\m(s) \times h) \times (\m(s) \times h)) - \m(s) \times (\m(s) \times h) \times (\m(s) \times h)|_{(H^1)^{\prime}} \\
			\leq & |P_n(P_n(\m_n(s) \times (\m_n(s) \times h)) \times (\m_n(s) \times h) - \m(s) \times (\m(s) \times h) \times (\m(s) \times h)) |_{(H^1)^{\prime}} \\
			& + |P_n(\m(s) \times (\m(s) \times h) \times (\m(s) \times h)) - \m(s) \times (\m(s) \times h) \times (\m(s) \times h)|_{(H^1)^{\prime}} \\
			\leq & |P_n(\m_n(s) \times (\m_n(s) \times h)) \times (\m_n(s) \times h) - \m(s) \times (\m(s) \times h) \times (\m(s) \times h) |_{(H^1)^{\prime}} \\
			& + |P_n(\m(s) \times (\m(s) \times h) \times (\m(s) \times h)) - \m(s) \times (\m(s) \times h) \times (\m(s) \times h)|_{(H^1)^{\prime}}.
		\end{align*}
		Thus,
		\begin{align*}
			\mathbb{E^{\prime}}&\int_{0}^{T}  |P_n(\m_n(s) \times (\m_n(s) \times h)) \times (\m_n(s) \times h) - \m(s) \times (\m(s) \times h) \times (\m(s) \times h) |_{(H^1)^{\prime}} ds \\
			\leq 	& \mathbb{E^{\prime}} \int_{0}^{T}  |P_n(\m_n(s) \times (\m_n(s) \times h)) \times (\m_n(s) \times h) - (\m(s) \times (\m(s) \times h)) \times (\m_n(s) \times h) \\
			& +  (\m(s) \times (\m(s) \times h)) \times (\m_n(s) \times h)  - \m(s) \times (\m(s) \times h) \times (\m(s) \times h) |_{(H^1)^{\prime}} ds \\
			\leq & \mathbb{E^{\prime}} \int_{0}^{T}  |(P_n(\m_n(s) \times (\m_n(s) \times h))  - (\m(s) \times (\m(s) \times h)) )\times (\m_n(s) \times h) |_{(H^1)^{\prime}} ds \\
			& + \mathbb{E^{\prime}} \int_{0}^{T}  |(\m(s) \times (\m(s) \times h)) \times (\m_n(s) \times h)  - \m(s) \times (\m(s) \times h) \times (\m(s) \times h)|_{(H^1)^{\prime}}
		\end{align*}
		Using the following inequality
		\begin{equation}\label{Inequality H -1 leq L2 L2 L infinity}
			|v_1v_2v_3|_{(H^1)^{\prime}} \leq C |v_1v_2v_3|_{L^1} \leq C |v_1|_{L^2}|v_2|_{L^2}|v_3|_{L^{\infty}},
		\end{equation}
		(for $v_1,v_2\in L^2$ and $v_3\in L^{\infty}$) we observe that for $s\in[0,T]$ and $n\in\mathbb{N}$,
		\begin{align*}
			&	|(P_n(\m_n(s) \times (\m_n(s) \times h)) - (\m(s) \times (\m(s) \times h)) )\times (\m_n(s) \times h) |_{(H^1)^{\prime}} \\
			&\leq |P_n(\m_n(s) \times (\m_n(s) \times h))  - (\m(s) \times (\m(s) \times h)) |_{L^2} |\m_n(s)|_{L^2}|h|_{L^{\infty}} \\
			& \leq C(h)\sup_{s\in[0,T]}|P_n(\m_n(s) \times (\m_n(s) \times h))  - (\m(s) \times (\m(s) \times h)) |_{L^2} \sup_{s\in[0,T]}|\m_n(s)|_{L^2}.
		\end{align*}
		The right hand side of the above inequality goes to $0$ as $n$ goes to infinity. This follows from the argument mentioned next along with the use of Lebesgue dominated convergence theorem, which is again justified in the following steps.
		
		Using the fact that $P_n$ is a projection operator on $L^2$ and the H\"older inequality, we get
		\begin{align*}
			\mathbb{E^{\prime}}\sup_{s\in[0,T]}|P_n(\m_n(s) \times (\m_n(s) \times h))|_{L^2} & \leq \mathbb{E^{\prime}}\sup_{s\in[0,T]}|\m_n(s) \times (\m_n(s) \times h)|_{L^2} \\
			& \leq \mathbb{E^{\prime}}\sup_{s\in[0,T]}|\m_n(s)|_{L^2} |\m_n(s)|_{L^{\infty}} |h|_{L^{\infty}} \\
			& \leq C \mathbb{E^{\prime}}\sup_{s\in[0,T]}|\m_n(s)|_{L^2} |\m_n(s)|_{H^1} |h|_{L^{\infty}} \\
			& \leq C |h|_{L^{\infty}}|m(0)|_{L^2}\mathbb{E^{\prime}} \sup_{s\in[0,T]}|\m_n(s)|_{H^1}.
		\end{align*}
		This along with the bound \eqref{bound on nabla m_n prime} give us a uniform bound for using the Lebesgue Dominated Convergence Theorem.
		\begin{align*}
			&|(\m(s) \times (\m(s) \times h)) \times (\m_n(s) \times h - \m(s) \times h)|_{(H^1)^{\prime}} \\
			&\leq |(\m(s) \times (\m(s) \times h))|_{L^2} |\m_n(s) - \m(s)|_{L^2} |h|_{L^{\infty}} \\
			& \leq \sup_{s\in [0,T]}|\m(s)|_{L^2} \sup_{s\in [0,T]}|\m(s)|_{L^{\infty}} \sup_{s\in[0,T]} |\m_n(s) - \m(s)|_{L^2} |h|_{L^{\infty}} \\
			& \leq C \sup_{s\in [0,T]}|\m(s)|_{L^2} \sup_{s\in [0,T]}|\m(s)|_{H^1} \sup_{s\in[0,T]} |\m_n(s) - \m(s)|_{L^2} |h|_{L^{\infty}} \\
			& \leq C C(h)|m_0|_{L^2}\sup_{s\in[0,T]} |\m_n(s) - \m(s)|_{L^2}.
		\end{align*}
		Thus,
		\begin{align*}
			\mathbb{E}^{\prime} &|(\m(s) \times (\m(s) \times h)) \times (\m_n(s) \times h - \m(s) \times h)|_{(H^1)^{\prime}}^2 \\
			&\leq  CC(h) |m_0|^2_{L^2}\mathbb{E^{\prime}}\sup_{s\in[0,T]} |\m_n(s) - \m(s)|_{L^2}^2.		
		\end{align*}
		The right hand side of the above inequality goes to $0$ by Lemma \ref{Lemma sup in t L2 in space convergene of m_n prime}.
		Hence
		
		\begin{align*}
			\lim_{n\rightarrow\infty}\mathbb{E^{\prime}} \l|(P_n(\m_n(s) \times (\m_n(s) \times h)) - (\m(s) \times (\m(s) \times h)) )\times (\m_n(s) \times h) \r|_{(H^1)^{\prime}} \, ds = 0
		\end{align*}
		and
		\begin{align*}
			\lim_{n\rightarrow\infty}\mathbb{E^{\prime}} \int_{0}^{T} |(\m(s) \times (\m(s) \times h)) \times (\m_n(s) \times h - \m(s) \times h)|_{(H^1)^{\prime}} \, ds = 0.
		\end{align*}
		Concerning the remaining term, the calculations can be done as follows.
		For $s\in[0,T]$,
		
		$$\lim_{n\rightarrow\infty}|P_n(\m(s) \times (\m(s) \times h)) - \m(s) \times (\m(s) \times h)| = 0.$$ The above pointwise convergence and the uniform bound
		\begin{align*}
			\mathbb{E^{\prime}}\int_{0}^{T}|\m(s) \times (\m(s) \times h)|_{(H^1)^{\prime}}\, ds \leq C(h) \mathbb{E^{\prime}} |m_0|_{L^2}^2
		\end{align*}
		together with the Lebesgue Dominated Convergence Theorem gives
		\begin{align*}
			\lim_{n\rightarrow\infty} &\mathbb{E^{\prime}} \int_{0}^{T} \l|P_n \bigl(\m(s) \times \l(\m(s) \times h\r) \times \bigl(\m(s) \times h\bigr) \bigr) - \m(s) \times \bigl(\m(s) \times h\bigr) \times \bigl(\m(s) \times h \bigr) \r|_{(H^1)^{\prime}} \,ds \\
			& = 0.
		\end{align*}
		Combining the above calculations with the Lemma \ref{Lemma convergence of the bump function} justifies the claim.

	\end{proof}
	We now show that the driving process $W^{\prime}$ is a Wiener process.
	
	\dela{That the driving process $W^{\prime}$ is a Wiener process can be shown following the proof of Lemma 5.2 in \cite{ZB+BG+TJ_Weak_3d_SLLGE}.}
	\begin{lemma}\label{Lemma W prime is a Wiener process}
		\dela{We proof follows on the lines of the proof of Lemma 5.2 in \cite{ZB+BG+TJ_Weak_3d_SLLGE}.} The process $W^{\prime}$ is a Wiener process on the space $(\Omega^{\prime}, \mathcal{F}^{\prime}, \mathbb{P}^{\prime})$.
		Also, $W_n^{\prime}(t) - W_n^{\prime}(s)$ is independent of the $\sigma$- algebra generated by $\m_n(r), \u(r), W_n^{\prime}(r)$ for $0\leq r\leq s <t$.
	\end{lemma}

	\begin{proof}[Proof of Lemma \ref{Lemma W prime is a Wiener process}]
		
		$W^{\prime}_n$ converges to $W^{\prime}$ in $C([0,T];\mathbb{R})$ $\mathbb{P}$-a.s. Hence, $W^{\prime}\in C([0,T];\mathbb{R})$ $\mathbb{P}$-a.s. That is, $W^{\prime}$ thus has almost surely continuous trajectories.	
		We proceed as follows: First show that $W_n^{\prime}$ is a Wiener process on $(\Omega^{\prime} , \mathcal{F}^{\prime}
		, \mathbb{P}^{\prime})$ for each $n\in\mathbb{N}$. Recall that the processes $W_n^\p$ and $W$ have the same laws on the space $C([0,T];\mathbb{R})$.

		Let $\phi_i , \zeta_i$, $i=1,\dots,k$ be  continuous and bounded real valued functions on $(H^1)^{\prime}$.
		
		Let $\psi,\psi_i$, $i=1,\dots,k$ be  continuous and bounded real valued functions on $\mathbb{R}$. Let $0< r_1<\dots<r_k\leq s\leq t$, $0< s_1<\dots<s_k\leq s\leq t$.
		
		Now for each $n\in\mathbb{N}$

		\begin{align*}
			\mathbb{E}^{\prime} &\l[ \prod_{j=1}^{k}\phi_j\big(\m_n(r_j)\big)\prod_{j=1}^{k}\zeta_j\big(\u_n(r_j)\big)\prod_{j=1}^{k}\psi_j\big(W_n^{\prime}(s_j)\big) \psi\big(W_n^{\prime}(t) - W_n^{\prime}(s)\big) \r] \\
			&= \mathbb{E} \l[ \prod_{j=1}^{k}\phi_j\big(m_n(r_j)\big)\prod_{j=1}^{k}\zeta_j\big(u_n(r_j)\big)\prod_{j=1}^{k}\psi_j\big(W(s_j)\big) \psi\big(W(t) - W(s)\big) \r] \\
			& =  \mathbb{E} \l[ \prod_{j=1}^{k}\phi_j\big(m_n(r_j)\big)\prod_{j=1}^{k}\zeta_j\big(u_n(r_j)\big)\prod_{j=1}^{k}\psi_j\big(W(s_j)\big)\r] \mathbb{E}\l[\psi\big(W(t) - W(s)\big) \r] \\
			& = \mathbb{E}^{\prime} \l[ \prod_{j=1}^{k}\phi_j\big(\m_n(r_j)\big)\prod_{j=1}^{k}\zeta_j\big(\u_n(r_j)\big)\prod_{j=1}^{k}\psi_j\big(W_n^{\prime}(s_j)\big)\r] \mathbb{E}^{\prime} \l[\psi\big(W_n^{\prime}(t) - W_n^{\prime}(s)\big)\r].
		\end{align*}
		Thus, $W_n^{\prime}(t) - W_n^{\prime}(s)$ is independent of the $\sigma$- algebra generated by $\m_n(r), \u_n(r), W_n^{\prime}(r)$ for $r\leq s$.
		
		Taking the limit as $n$ goes to infinity, we get

		\begin{align*}
			\lim_{n\rightarrow\infty} \mathbb{E}^{\prime} & \l[ \prod_{j=1}^{k} \phi_j \big( \m_n(r_j)\big) \prod_{j=1}^{k} \zeta_j \big( \u_n(r_j)\big) \prod_{j=1}^{k} \psi_j \big( W_n^{\prime}(s_j) \big) \psi \big( W_n^{\prime}(t) - W_n^{\prime}(s) \big) \r] \\
			&	= \lim_{n\rightarrow\infty}  \mathbb{E}^{\prime} \l[ \prod_{j=1}^{k} \phi_j \big( \m_n(r_j) \big) \prod_{j=1}^{k} \zeta_j \big( \u_n(r_j) \big) \prod_{j=1}^{k} \psi_j \big( W_n^{\prime}(s_j) \big) \r] \mathbb{E}^{\prime} \l[ \psi \big( W_n^{\prime}(t) - W_n^{\prime}(s) \big) \r].
		\end{align*}
		By Lebesgue dominated convergence theorem, we have

		\begin{align*}
			\mathbb{E}^{\prime} &\l[ \prod_{j=1}^{k}\phi_j \bigl(\m(r_j)\bigr)\prod_{j=1}^{k}\zeta_j(\u \big( r_j ) \big) \prod_{j=1}^{k} \psi_j \big( W^{\prime} (s_j) \big) \psi\big( W^{\prime}(t) - W^{\prime}(s)\big) \r] \\
			&	= \mathbb{E}^{\prime} \large[ \prod_{j=1}^{k} \phi_j \big( \m(r_j) \big) \prod_{j=1}^{k} \zeta_j \big( \u(r_j) \big) \prod_{j=1}^{k} \psi_j \big( W^{\prime}(s_j) \big) \large] \mathbb{E}^{\prime} \l[ \psi \big( W^{\prime}(t) - W^{\prime}(s) \big) \r].
		\end{align*}
		Thus, $W^{\prime}(t) - W^{\prime}(s)$ is independent of the $\sigma$- algebra generated by $\m(r), \u(r) , W^{\prime}(r)$ for $r\leq s \leq t$.\\
		Now, let $k\in\mathbb{N}$, $s_0 = 0 < s_1 < \dots < s_k \leq T$. For $(t_1,\dots t_k)\in\mathbb{R}^k$. Then for each $n\in\mathbb{N}$, we have
		
		\begin{align*}
			\mathbb{E}^{\prime}\l[e^{i\sum_{j=1}^{k}t_j \big( W^{\prime}_n(s_j) - W^{\prime}_n(s_{j-1}) \big)} \r] &=  \mathbb{E} \l[ e^{ i\sum_{j=1}^{k}t_j \bigl( W(s_j) - W(s_{j-1}) \bigr) } \r] \\
			& =  e^{ - \frac{1}{2} \sum_{j=1}^{k} t_j^2 \big( s_j - s_{j-1} \big) }.
		\end{align*}
		Thus
		\begin{align*}
			\lim_{n\rightarrow\infty} \mathbb{E}^{\prime}\l[ e^{ i \sum_{j=1}^{k} t_j \big( W^{\prime}_n(s_j) - W^{\prime}_n(s_{j-1}) \big) } \r] = \lim_{n\rightarrow\infty} e^{ - \frac{1}{2} \sum_{j=1}^{k} t_j^2 ( s_j-s_{j-1} ) }
		\end{align*}
		and by the Lebesgue dominated convergence theorem,
		\begin{align*}
			\mathbb{E}^{\prime} \l[ e^{ i \sum_{j=1}^{k} t_j \big( W^{\prime}(s_j) - W^{\prime}(s_{j-1}) \big) } \r] =
			e^{ - \frac{1}{2} \sum_{j=1}^{k} t_j^2 ( s_j-s_{j-1} )}.
		\end{align*}
		Hence, the increments are normally distributed.
	\end{proof}

	\begin{lemma}\label{convergence of M_n prime to Ito integral}
		
		For each $t\in[0,T]$,
		$M_n^{\prime}(t)$ converges to $\int_{0}^{t} \psi\l(\m(s)\r) G\big(\m(s)\big) \, dW^{\prime}(s)$ in $L^2(\Omega^{\prime} ; (H^1)^{\prime})$. In particular,
		\begin{equation}
			M^\p(t) =  \int_{0}^{t} \psi(\m(s)) G(\m(s)) \, dW^\p(s), \ \mathbb{P}^\p-a.s.
		\end{equation}
	\end{lemma}

	\begin{proof}[Idea of proof of Lemma \ref{convergence of M_n prime to Ito integral}]
		We first give a brief idea of the proof in mainly two steps. We then go on to justify the steps.
		\begin{enumerate}
			\item Let us choose and fix $t\in[0,T]$ and $n\in\mathbb{N}$. We show that
			\begin{equation}
				M_n^\p(t) = \int_{0}^{t} \psi(\m_n(s)) G \big( \m_n(s) \big) \, dW_n^{\prime}(s), \ \mathbb{P}^\p-a.s.
			\end{equation}
			\item Again, Let us choose and fix $t\in[0,T]$. Then, using step (1) we show that $M_n^\p(t)$ converges to $\int_{0}^{t} \psi(\m(s)) G(\m(s)) \, dW^\p(s)$ as $n\to\infty$ in $L^2(\Omega^\p ; (H^1)^\p)$, and hence, in particular, weakly in $L^{\frac{4}{3}}(\Omega^\p ; (H^1)^\p)$.
		\end{enumerate}
		From Lemma \ref{convergence of Mn prime to M prime}, we know that $M_n^\p(t)$ converges to $M^\p(t)$ weakly in $L^{\frac{4}{3}}(\Omega^\p ; (H^1)^\p)$. Combining this convergence with the convergence from step (2), we have,
		\begin{equation}
			M^\p(t) =  \int_{0}^{t} \psi(\m(s)) G(\m(s)) \, dW^\p(s), \ \mathbb{P}^\p-a.s.
		\end{equation}
	\end{proof}
	\begin{proof}[Proof of Lemma \ref{convergence of M_n prime to Ito integral}]
		\textbf{Proof of Step 1:}	Let $k,n\in\mathbb{N}$, and let $t\in[0,T]$.\\ 
		Let $\mathcal{P}_k := \l\{ s_j^k: s_j^k = \frac{jT}{k},j=0,\dots,k \r\}$ be a partition of $[0,T]$.\\
		\textbf{Claim}: \begin{equation}
			M_n^{\prime}(t) = \int_{0}^{t} \psi(\m_n(s)) G \big( \m_n(s) \big) \, dW_n^{\prime}(s).
		\end{equation}
		For $n\in\mathbb{N}$ and $t\in[0,T]$, consider the following random variables.
		
		\begin{align}\label{eqn rv1}
			M_n(t) - \sum_{j=0}^{k-1} \psi(m_n(s_j^k)) G_n\bigl(m_n(s_j^k)\bigr) \bigl(W(s_{j+1}^k \wedge t\bigr) - W(s_{j}^k \wedge t) \bigr),
		\end{align}
		and
		\begin{align}\label{eqn rv2}
			M_n^{\prime}(t) - \sum_{j=0}^{k-1} \psi(\m_n(s)) G_n \bigl( \m_n(s_j^k) \bigr) \bigl( W_n^{\prime}(s_{j+1}^k \wedge t) - W_n^{\prime}(s_{j}^k \wedge t) \bigr).
		\end{align}
		\textbf{Sub-claim:} For each $t\in[0,T]$ and $n\in\mathbb{N}$, we have the following convergence. The random variable
		\begin{align}
			\nonumber & \sum_{j=0}^{k-1} \psi(m_n(s_j^k \wedge t )) G_n\bigl(m_n(s_j^k \wedge t )\bigr) \bigl(W(s_{j+1}^k\bigr) - W(s_{j}^k) \bigr) \\
			& = \int_{0}^{t}  \chi_{[s^k_j , s^k_{j+1})}(s) \psi(m_n(s_j^k \wedge t)) G_n(m_n(s_j^k \wedge t)) \, dW(s),
		\end{align}
		converges to the random variable
		\begin{equation}
			\int_{0}^{t} \psi(m_n(s)) G_n(m_n(s)) \, dW(s),
		\end{equation}
		in the space $L^2(\Omega;L^2)$ as $k\to\infty$. By the equality in \eqref{Definition of M_n}, we, therefore, have the first variable to be $0$ (in the limit as $k\to\infty$) $\mathbb{P}^\p$-a.s.
		
		\begin{proof}[\textbf{Proof of the sub-claim.}] Firstly, for any $f\in C([0,T];H_n)$, we have the following
			\begin{equation}
				\lim_{k\to\infty} \int_{0}^{t} \l| \chi_{[s^k_j , s^k_{j+1})}(s) f( s_j^k \wedge t ) - f(s) \r|_{L^2}^2 \, ds = 0.
			\end{equation}
			Now, observe that $\psi(m_n) G_n(m_N) \in C([0,T];H_n)$. Therefore for $f(\cdot) = \psi(\cdot) G_n(m_n(\cdot)) \in C([0,T];H_n)$, we have
			\begin{equation}
				\lim_{k\to\infty} \int_{0}^{t} \l| \chi_{[s^k_j , s^k_{j+1})}(s) \psi(m_n(s_j^k \wedge t)) G_n(m_n(s_j^k \wedge t)) - \psi(m_n(s))G(m_n(s)) \r|_{L^2}^2 \, ds = 0,\ \mathbb{P}-a.s.
			\end{equation}
			Moreover, by Lemma \ref{bounds lemma 1}, there exists a constant $C$ independent of $k$ such that
			\begin{align}
				& \mathbb{E} \l[  \int_{0}^{t} \l| \chi_{(s^k_j , s^k_{j+1}]}(s) \psi(m_n(s_j^k \wedge t)) G_n(m_n(s_j^k \wedge t)) - \psi(m_n(s))G(m_n(s)) \r|_{L^2}^2 \, ds \r]^2 \\
				\leq &  4 \mathbb{E} \int_{0}^{t} \l| \chi_{[s^k_j , s^k_{j+1})}(s) \psi(m_n(s_j^k \wedge t)) G_n(m_n(s_j^k \wedge t)) \r|_{L^2}^4 \, ds 
				+ 4 \mathbb{E} \int_{0}^{t} \l|  \psi(m_n(s))G(m_n(s)) \r|_{L^2}^4 \, ds \leq C. 
			\end{align}
			Therefore by the Vitali Convergence Theorem, we have the following convergence.
			\begin{equation}
				\lim_{k\to\infty} \mathbb{E} \int_{0}^{t} \l| \chi_{[s^k_j , s^k_{j+1})}(s) \psi(m_n(s_j^k \wedge t)) G_n(m_n(s_j^k \wedge t)) - \psi(m_n(s))G(m_n(s)) \r|_{L^2}^2 \, ds = 0.
			\end{equation}
			In order to prove the claim, we consider the following difference. By the It\^o isometry, we have
			\begin{align}
				\nonumber & \mathbb{E} \l| \int_{0}^{t}  \chi_{[s^k_j , s^k_{j+1})}(s) \psi(m_n(s_j^k \wedge t)) G_n(m_n(s_j^k \wedge t)) - \psi(m_n(s)) G_n(m_n(s))  \, dW(s) \r|_{L^2}^2 \\ 
				& = \mathbb{E} \int_{0}^{t} \l| \chi_{[s^k_j , s^k_{j+1})}(s) \psi(m_n(s_j^k \wedge t)) G_n(m_n(s_j^k \wedge t)) - \psi(m_n(s)) G_n(m_n(s)) \r|_{L^2}^2 \, ds .
			\end{align}
			The right hand side, and hence the left hand side of the above inequality converges to $0$ as $k\to\infty$. This completes the proof of the sub-claim.
		\end{proof}
		
		Note that the two random variables in \eqref{eqn rv1} and \eqref{eqn rv2} are obtained by applying measurable transformations to $m_n,\m_n,W_n^{\prime}$ and $W$ and hence have the same distributions.
		Strong convergence of $M_n(t)$ implies convergence of the corresponding laws. Since the random variables in \eqref{eqn rv1} and \eqref{eqn rv2} have the same laws, the laws of $M_n^\p(t)$ also converge to the law of some random variable, the law of which is the same as that of the law of the limit of $M_n(t)$.
		But since $M_n(t) - \int_{0}^{t} \psi(m_n(s)) G_n(m_n(s)) \, dW(s) = 0,\ \mathbb{P}$-a.s. (because $m_n$ is a solution to \eqref{definition of solution Faedo Galerkin approximation}), we have
		\begin{align}
			\lim_{k\to\infty} \l[ M_n^\p(t) -  \int_{0}^{t} \chi_{[s^k_j , s^k_{j+1})}(s) \psi(m_n(s_j^k \wedge t)) G_n(m_n(s_j^k \wedge t)) \, dW_n^\p(s) \r] = 0,\ \mathbb{P}^\p-a.s.
		\end{align}
		Thus,
		
		\begin{equation*}
			M_n^{\prime}(t) = \int_{0}^{t} \psi(\m_n(s)) G\big(\m_n(s)\big) \, dW_n^{\prime}(s),\ \mathbb{P}^\p-a.s.
		\end{equation*}
		Hence the claim is shown. This concludes step 1.\\  
		\textbf{Proof of Step 2:}
		In the second step, we have to show the convergence of $M_n^\p(t)$ to the stochastic integral $\int_{0}^{t} \psi(\m(s)) G(\m(s)) \, dW^\p(s)$ as $n\to\infty$. In step 1, we have shown that $M_n^\p(t) = \int_{0}^{t} \psi(\m_n(s)) G(\m_n(s)) \, dW_n^\p(s), \mathbb{P}^\p$-a.s.
		
		Now, some standard adding and subtracting, along with the triangle inequality, gives us the following inequality.

		\begin{align}
			\nonumber \mathbb{E^{\prime}}& \l| \int_{0}^{t}  \psi(\m_n(s)) G_n\big(\m_n(s)\big)\, dW_n^{\prime}(s)  - \int_{0}^{t} \psi(\m(s)) G\big( \m(s\dela{_j^k})\big)\, dW^{\prime}(s) \r|_{(H^1)^{\prime}}^2   \\
			\nonumber \leq & \mathbb{E^{\prime}}  \l[ \l| \int_{0}^{t} \psi(\m_n(s)) G_n\big(\m_n(s)\big)\, dW_n^{\prime}(s) - \int_{0}^{t} \psi(\m(s)) G_n\big(\m_n(s)\big)\, dW_n^{\prime}(s)  \r|^2_{(H^1)^{\prime}}\r] \\
			\nonumber & + \mathbb{E^{\prime}}  \l[ \l| \int_{0}^{t} \psi(\m(s)) G_n\big(\m_n(s)\big)\, dW_n^{\prime}(s) - \int_{0}^{t} \psi(\m(s)) G\big(\m(s)\big)\, dW_n^{\prime}(s)  \r|^2_{(H^1)^{\prime}}\r] \\
			&+ \mathbb{E^{\prime}} \l[ \l| \int_{0}^{t} \psi(\m(s)) G\big(\m(s)\big)\, dW_n^{\prime}(s) - \int_{0}^{t} \psi(\m(s)) G\big(\m(s)\big) \, dW^{\prime}(s)  \r|^2_{(H^1)^{\prime}} \r].
		\end{align}
		The first term on the right hand side converges to $0$ as $n\to\infty$. This follows from using the convergences in Lemma \ref{Lemma convergence of the bump function} and some standard arguments.
		
		For the second term, note that since $L^2\hookrightarrow (H^1)^\p$, there exists a constant $C>0$ such that
		\begin{align*}
			&\mathbb{E^{\prime}}  \l[ \l| \int_{0}^{t} \psi(\m(s))G_n\big(\m_n(s)\big) \, dW_n^{\prime}(s) - \int_{0}^{t} \psi(\m(s)) G\big(\m(s)\big) \, dW_n^{\prime}(s)  \r|^2_{(H^1)^{\prime}} \r] \\
			= & \mathbb{E^{\prime}}  \l[ \l| \int_{0}^{t}  \psi(\m(s)) \l[ G_n\big(\m_n(s)\big) - G\big(\m(s)\big) \r] \, dW_n^{\prime}(s)  \r|^2_{(H^1)^{\prime}}\r] \\
			\leq &  C\mathbb{E^{\prime}}  \l[ \l| \int_{0}^{t} \psi(\m(s)) \l[ G_n\big(\m_n(s)\big) - G\big(\m(s)\big) \r] \, dW_n^{\prime}(s)  \r|^2_{L^2}\r] \\
			\leq & C \mathbb{E^{\prime}}  \l[  \int_{0}^{t}  \l|\l[ G_n\big(\m_n(s)\big) - G\big(\m(s)\big)\r] \r|^2_{L^2} \, ds  \r]. \\
		\end{align*}
		In the last inequality, we have used the fact that $\psi \leq 1$, along with the It\^o isometry.
		By the convergence in Lemma \ref{Lemma convergence of Gn}, the right hand side converges to 0 as $n \to \infty$. In particular, for every $\varepsilon>0$, we can choose $N_{\varepsilon}$ large enough so that the first term is bounded by $\frac{\varepsilon}{4}$ for each $n\geq N_{\varepsilon}$.
		For the third term, we approximate the integrals by finite sums.
		\begin{align*}
			&\mathbb{E^{\prime}} \l[ \l| \int_{0}^{t} \psi(\m(s)) G(\m(s))\, dW_n^{\prime}(s) - \int_{0}^{t} \psi(\m(s)) G(\m(s))\, dW^{\prime}(s)  \r|^2_{(H^1)^{\prime}} \r] \\
			\leq & \mathbb{E^{\prime}} \l[ \l| \int_{0}^{t} \psi(\m(s)) G(\m(s))\, dW_n^{\prime}(s) - \sum_{j=0}^{k-1} \psi(\m(s_j^k)) G(\m(s_j^k))\, \l(W_n^{\prime}(s_{j+1}^k) - W_n^{\prime}(s_{j+1}^k)\r)  \r|^2_{(H^1)^{\prime}} \r] \\
			& + \mathbb{E^{\prime}} \bigg[ \bigg|  \sum_{j=0}^{k-1} \psi(\m(s_j^k)) 
			G(\m(s_j^k))\, \l(W_n^{\prime}(s_{j+1}^k) - W_n^{\prime}(s_{j+1}^k)\r) \\
			& \quad - \sum_{j=0}^{k-1} \psi(\m(s_j^k)) G(\m(s_j^k))\, \l(W^{\prime}(s_{j+1}^k) - W^{\prime}(s_{j+1}^k)\r)  \bigg|^2_{(H^1)^{\prime}} \bigg] \\
			& + \mathbb{E^{\prime}} \l[ \l|  \sum_{j=0}^{k-1} \psi(\m(s_j^k)) G(\m(s_j^k))\, \l(W^{\prime}(s_{j+1}^k) - W^{\prime}(s_{j+1}^k)\r) - \int_{0}^{t} \psi(\m(s)) G(\m(s))\, dW^{\prime}(s)  \r|^2_{(H^1)^{\prime}} \r].
		\end{align*}
		Since the mentioned sums approximate the corresponding It\^o integrals in  $L^2(\Omega^\p ; (H^1)^\p)$, the first and the third term converge to 0 as $k\to\infty$. Convergence of the processes $W_n^\p$ to $W$ along with uniform integrability implies that the second term goes to 0 as $n$ goes to infinity.
		
		Combining the convergences concludes step 2, and hence the proof of the lemma.
		
		\dela{
			\begin{align*}
				\mathbb{E} \l|\sum_{j=0}^{k-1}\l[ W(s^k_{j+1}) - W(s^k_j) \r] \r|
				\leq & \sum_{j=0}^{k-1} \mathbb{E} \l|\l[ W(s^k_{j+1}) - W(s^k_j) \r] \r| \\
				\leq & C \sum_{j=0}^{k-1} \l(\mathbb{E} \l| W(s^k_{j+1}) - W(s^k_j)\r|^2 \r)^{\frac{1}{2}} \\
				\leq & C \sum_{j=0}^{k-1} \l| s_{j+1^k}  - s_j^k \r|^{\frac{1}{2}} < \infty.
			\end{align*}	
		}
		\dela{
			Let $\varepsilon > 0$ be given. Then there exists a partition $\mathcal{P}_k$ such that for every $n\in\mathbb{N}$\adda{Why uniform in $n$?}
			\begin{equation*}
				\mathbb{E^{\prime}}  \l[ \l| \int_{0}^{t} G(\m_n(s))\, dW_n^{\prime}(s) - \sum_{j=0}^{k-1} \int_{0}^{t} \chi_{(s^k_j , s^k_{j+1}]}(s)G(\m_n(s_j^k))\, dW_n^{\prime}(s) |^2_{(H^1)^{\prime}} \r|_{(H^1)^{\prime}}^2 \r] < \frac{\varepsilon}{4}	
			\end{equation*}
			and
			\begin{equation*}
				\mathbb{E^{\prime}} \l[ \l| \sum_{j=0}^{k-1} \int_{0}^{t} \chi_{(s^k_j , s^k_{j+1}]}(s)G(\m(s_j^k))\, dW^{\prime}(s) -
				\int_{0}^{t} G(\m(s))\, dW^{\prime}(s) \r|_{(H^1)^{\prime}}^2 \r] < \frac{\varepsilon}{4}.
			\end{equation*}
			
			Corresponding to the $\varepsilon$ mentioned above, there exists $N\in \mathbb{N}$ such that for each $n\geq N$
			\begin{equation*}
				\mathbb{E^{\prime}} \l[ \l| \sum_{j=0}^{k-1} \int_{0}^{t} \chi_{(s^k_j , s^k_{j+1}]}(s)G_n(\m_n(s_j^k))\,  dW_n^{\prime}(s) - \sum_{j=0}^{k-1} \int_{0}^{t} \chi_{(s^k_j , s^k_{j+1}]}(s)G(\m(s_j^k))\, dW_n^{\prime}(s) \r|^2_{(H^1)^{\prime}} \r] \leq \frac{\varepsilon}{4}
			\end{equation*}
			and
			\begin{equation*}
				\mathbb{E^{\prime}} \left[ \left| \sum_{j=0}^{k-1} \int_{0}^{t} \chi_{(s^k_j , s^k_{j+1}]}(s)G(\m(s_j^k)) \, dW_n^{\prime}(s) -
				\sum_{j=0}^{k-1} \int_{0}^{t} \chi_{(s^k_j , s^k_{j+1}]}(s)G(\m(s_j^k))\, dW^{\prime}(s) \right|^2_{(H^1)^{\prime}} \right] < \frac{\varepsilon}{4}.
			\end{equation*}
			Thus,
			\begin{equation*}
				\lim_{n\rightarrow\infty} \mathbb{E^{\prime}} \left|\int_{0}^{t}  G_n(\m_n(s)) dW_n^{\prime}(s) - \int_{0}^{t} G(\m(s))   dW^{\prime}(s)\right|^2_{(H^1)^{\prime}} = 0.
			\end{equation*}
		}
	\end{proof}

	\section{Continuation of the proof of Theorem \ref{Theorem Existence of a weak solution}: verification of the  constraint condition}\label{sec-The constraint condition section}
	After showing the existence of a solution to the equation, we now have to show that the obtained process $m$ satisfies the constraint condition \eqref{eqn-constraint condition}.	
	We use the It\^o formula version from the paper of Pardoux \cite{Pardoux_1979}, Theorem 1.2.\\
	For $t\in[0,T]$, consider the equation in $(H^1)^\p$
	\begin{align*}
		\nonumber	\m(t) = & \ \m_0 + \int_{0}^{t} \bigg[ \m(s) \times \Delta \m(s) - \alpha \, \m(s) \times \bigl(\m(s) \times \Delta \m(s) \bigr) + \m(s)\times u^{\prime}(s)   \\
		& - \alpha \, \psi(\m(s)) \m(s) \times \bigl(\m(s) \times \u(s) \bigr) + \frac{1}{2} \psi(\m(s))^2\l[ DG \bigl( \m(s) \bigr) \r] \bigl[ G\bigl(\m(s)\bigr)\bigr] \bigg]\, ds \\
		& + \int_{0}^{t} \psi(\m(s)) G\bigl(\m(s)\bigr)\, dW^{\prime}(s).
	\end{align*}
	Let $M^2(0,T ; L^2)$ be the space of all $L^2$ valued processes $v$ that satisfy
	\begin{equation*}
		\mathbb{E^{\prime}}\l[\int_{0}^{T} |v(t)|^2_{L^2}\,dt\r] < \infty.
	\end{equation*}
	That is, $M^2(0,T;L^2) = L^2(\Omega^\p;L^2(0,T;L^2))$.
	Similarly define $M^2(0,T;H^1)$, $M^2(0,T;(H^1)^{\prime})$.
	(For details see Section 1.3 in \cite{Pardoux_1979}).
	
	Let $\phi\in C_c^{\infty}(\mathcal{O})$, with $\phi$ taking values in $\mathbb{R}^+$.
	Define $\phi_4:L^2\rightarrow \mathbb{R}$ by
	\begin{equation*}
		\phi_4(v) = \frac{1}{2} \l\langle \phi v , v \r\rangle_{L^2}.
	\end{equation*}
	This can be written as
	\begin{align*}
		\phi_4(v) = \frac{1}{2} \int_{\mathcal{O}} \phi(x) \langle v(x) , v(x) \rangle_{\mathbb{R}^3} dx.
	\end{align*}
	First, we present the Fr\'echet derivatives $\phi_4^{\p},\phi_4^{\p\p}$ of $\phi_4$.
	Let $v_i\in L^2,i=1,2,3$. Then we have
	\begin{equation*}
		\phi_4^\p(v_1) (v_2) = \l\langle \phi v_1 , v_2 \r\rangle_{L^2}.
	\end{equation*}
	Similarly,
	\begin{equation*}
		\phi_4^{\p\p}(v_1)(v_2,v_3) = \l\langle \phi v_2 , v_3 \r\rangle_{L^2}.
	\end{equation*}

	Using the bound \eqref{bound on m prime H1} and the assumption on the initial data $m_0$, one can show that the following hold.
	\begin{enumerate}
		\item
		\begin{equation*}
			\m \in L^2(0,T ; H^1);
		\end{equation*}
		\item
		\begin{equation*}
			m_0^\p \in H^1;
		\end{equation*}
		\item
		\begin{equation*}
			\m \times \Delta \m \in M^2(0,T ; (H^1)^{\prime});
		\end{equation*}
		
		\item
		\begin{equation*}
			\m \times (\m \times \Delta \m) \in M^2(0,T ; (H^1)^{\prime});
		\end{equation*}
		\item
		\begin{equation*}
			m \times u^{\prime} \in M^2(0,T ; (H^1)^{\prime});
		\end{equation*}
		\item
		\begin{equation*}
			\m \times (\m \times \u) \in M^2(0,T ; (H^1)^{\prime});
		\end{equation*}
		
		\item
		\begin{equation*}
			\l[DG(\m)\r]\big(G(\m)\big) \in M^2(0,T ; (H^1)^{\prime});
		\end{equation*}
		\item
		\begin{equation*}
			G(\m) \in M^2(0,T ; L^2).
		\end{equation*}
	\end{enumerate}
	Thus, the It\^o formula can be applied to the function $\phi_4$ defined above.

	The calculations that follow are similar to the ones in the proof of Lemma \ref{bounds lemma 1}. On applying the It\^o formula, Theorem 1.2,  \cite{Pardoux_1979}, we get the next stated inequality because the terms that previously had the bump (cut-off) function cancel with the correction term arising because of the It\^o formula and the other terms are $0$, except for the terms stated in the following equation.
	For a similar computation, see Remark 3.2 in \cite{ZB+BG+TJ_LargeDeviations_LLGE}.
	An application of the It\^o formula thus yields
	
	\begin{align}\label{eqn for constraint condition eqn 1}
		\nonumber	\phi_4(\m(t)) = &  \phi_4(m_0) + \int_{0}^{t} \l\langle  \m(s) \times \Delta \m(s) , \m(s) \r\rangle_{L^2} \, ds \\
		\nonumber & - \alpha \, \int_{0}^{t} \l\langle \m(s) \times \bigl(\m(s) \times \Delta \m(s) \bigr) , \m(s) \r\rangle_{H^1} \, ds \\
		\nonumber & + \int_{0}^{t} \l\langle \m(s)\times u^{\prime}(s) , \m(s) \r\rangle_{H^1} \, ds \\
		\nonumber & - \alpha \, \int_{0}^{t} \l\langle \psi(\m(s)) \m(s) \times \bigl(\m(s) \times \u(s) \bigr) , \m(s) \r\rangle_{H^1} \, ds\\
		\nonumber & + \frac{1}{2} \int_{0}^{t} \l\langle \psi^2(\m(s)) \l[ DG \bigl( \m(s) \bigr) \r] \bigl[ G\bigl(\m(s)\bigr)\bigr]  , \m(s) \r\rangle_{L^2}\, ds \\
		\nonumber  & + \frac{1}{2} \int_{0}^{t} \psi^2(\m(s))\l[\phi_4^{\p\p}(\m(s))\r]\l\langle \big(G(\m(s)) , G(\m(s))\big) \r\rangle_{L^2} \, ds \\
		\nonumber & + \int_{0}^{t} \l\langle \psi(\m(s)) G\bigl(\m(s)\bigr) , \m(s) \r\rangle_{L^2} \, dW^{\prime}(s) \\
		= & \phi_4(m_0) + \sum_{i = 1}^{7} I_i(t).
	\end{align}
	Our first observation for the integrals on the right hand side of \eqref{eqn for constraint condition eqn 1} is that
	\begin{equation}
		I_i(t) = 0, \ \text{for}\ i = 1, 2, 3, 4,\ \text{and}\  7.
	\end{equation}
	We give a brief justification for the following. We mainly use the fact that for vectors $a,b\in\mathbb{R}^3$, we have
	\begin{equation*}
		\l\langle a \times b , a \r\rangle_{\mathbb{R}^3} = 0.
	\end{equation*}
	For any $p\geq1$, the above equality gives
	\begin{equation}
		\ _{L^{p^\p}}\l\langle a \times b , a \r\rangle_{L^p} = 0,
	\end{equation}
	with $\ _{L^{p^\p}}\langle\, \cdot , \cdot \rangle_{L^p}$ denoting the $L^p$ duality pairing.\\
	Observe that not all the inner products on the right hand side of \eqref{eqn For constraint condition} are the $L^2$ inner products. To use the above equality, we replace the $(H^1)^\p-H^1$ duality pairing by $L^p$ duality pairing for some convenient $p$.
	To see this, first, note that the space $H^1$ is compactly embedded into the spaces $L^4$ and $L^6$. Therefore, the $(H^1)^\p - H^1$ duality pairing can be appropriately replaced by the $(H^1)^\p - H^1$ duality pairing can be replaced by the $L^{\frac{4}{3}} - L^4$ (for $I_2, I_3$) and $L^{\frac{6}{5}} - L^6$ (for $I_4$) duality pairings.

	For the triple product term $m \times \l( \m \times \Delta \m \r)$ (inside the integral $I_2$), note that 
	$$ \l| \m \times \l( \m \times \Delta \m \r) \r|_{L^{\frac{4}{3}}} \leq C \l| \m \r|_{L^4} \l| \m \times \Delta \m \r|_{L^2}.$$
	Similar can be said about $\m \times \u$ for $I_3$.
	
	For $ \m \times \l( \m \times \u \r) $ (inside the integral $I_4$), note that $$\l| \m \times \l( \m \times \u \r) \r|_{L^{\frac{6}{5}}} \leq C \l| \m \r|_{L^6}^2 \l| \u \r|_{L^2}.$$
	
	Now, the terms that remain are $I_5,I_6$.
	Note that
	\begin{equation*}
		\l[\phi_4^{\p\p}(\m)\r](\l(G(\m) , G(\m)\r)) = \l\langle G(\m) , G(\m) \r\rangle_{L^2} = \l| G(\m) \r|_{L^2}^2.
	\end{equation*}
	Moreover, the following equality holds from Lemma B.2 in \cite{ZB+BG+TJ_LargeDeviations_LLGE}.
	\begin{equation*}
		\l\langle \l[DG(\m)\r]\big(G(\m)\big) , \m \r\rangle_{L^2} = - \l| G(\m) \r|_{L^2}^2.
	\end{equation*}
	Therefore,
	\begin{equation*}
		I_6(t) + I_7(t) = 0,\ \forall t\in[0,T].
	\end{equation*} 
	Hence, the equality \eqref{eqn for constraint condition eqn 1} is now 
	\begin{align*}
		\phi_4\big(\m(t)\big) = \phi_4(m_0),
	\end{align*}
	for each $t\in[0,T]$.
	That is
	\begin{equation}\label{eqn For constraint condition}
		\int_{\mathcal{O}}\phi(x)|\m(t,x)|_{\mathbb{R}^3}^2\, dx = \int_{\mathcal{O}}\phi(x)|m_0(x)|_{\mathbb{R}^3}^2\, dx.
	\end{equation}
	Now, the equality \eqref{eqn For constraint condition} holds for all $\phi\in C_c^{\infty}(\mathcal{O})$.
	Hence, we have the following
	\begin{equation}
		|\m(t,x)|_{\mathbb{R}^3}^2 = |m_0(x)|_{\mathbb{R}^3}^2 = 1, \ \text{Leb.a.a.} \  x\in \mathcal{O} \ \text{for all}\  t\in[0,T] \ \mathbb{P}^\p-\text{a.s.}
	\end{equation}
	Thus, the constraint condition \eqref{eqn-constraint condition} is satisfied.
	
	\begin{remark}
		Now that the constraint condition has been satisfied, we observe that the cut-off $\psi$ only takes the value $1$, and hence can be removed from the equation. This completes the proof of existence of a weak martingale solution to the problem \eqref{problem considered}, as per Definition \ref{Definition of Weak martingale solution}.
	\end{remark}
	
	\section{Proof of Theorems \ref{thm-pathwise uniqueness} and \ref{thm-existence of a strong solution} about the pathwise uniqueness and the existence of a unique strong solution}
	\label{Section Pathwise uniqueness}

	For this section, let us fix a probability space $\l(\Omega , \mathcal{F} , \mathbb{P}\r)$ and a Wiener process $W$ on this space, as in Definition \ref{Definition of Weak martingale solution}. The existence theorem (Theorem \ref{Theorem Existence of a weak solution}) states that the process $m$ satisfies the equation \eqref{problem considered} with the help of a test function. The following result, which is a corollary of Theorem \ref{Theorem Existence of a weak solution}, states that the equation also makes sense in the strong (PDE) form.

	\begin{corollary}\label{Strong form of weak martingale solution}
		Let us assume that the process $u$ is a control process such that \eqref{assumption on u} holds. Let $\l( \Omega , \mathcal{F} , \mathbb{P} , W , m , u \r)$ be a weak martingale solution of \eqref{problem considered} corresponding to the control process $u$, satisfying the properties stated in Theorem \ref{Theorem Existence of a weak solution}. Then the following equation is satisfied in the strong (PDE) sense in the space $L^2$ for each  $t\in[0,T]$.
		
		\begin{align}
			\nonumber	m(t) &= \int_{0}^{t}m(s) \times \Delta m(s)\, ds - \alpha \, \int_{0}^{t} m(s)\times(m(s)\times u(s))\, ds - \alpha \, \int_{0}^{t} m(s) \times \l(m(s) \times \Delta m(s)\r)\, ds \\
			& + \int_{0}^{t} m(s)\times u(s)\, ds  + \frac{1}{2}\int_{0}^{t} \l[DG\l(m(s)\r)\r]\l[G\big(m\l(s\r)\big)\r]\,  ds + \int_{0}^{t} G\big(m(s)\big)\, dW(s), \mathbb{P}-a.s.
		\end{align}
		
	\end{corollary}

	\begin{proof}[Proof of Corollary \ref{Strong form of weak martingale solution}]

		The proof of the above corollary follows once we note that each of the integrands of the equality lies in the space $L^2\l(\Omega ; L^2\l( 0,T; L^2 \r)\r)$. This can be verified by using the bounds established in the previous Section \ref{sec-The constraint condition section}, Lemma \ref{Proposition bounds on m_n prime} and Lemma \ref{Lemma extension of norms and lower semicontinuity}.

		By Theorem \ref{Theorem Existence of a weak solution}, the process $m\times \Delta m$ lies in the space $L^2\l(\Omega ; L^2\l( 0,T; L^2 \r)\r)$ $ \mathbb{P} $-a.s. By the constraint condition \eqref{eqn-constraint condition}
		\begin{align}
			\nonumber \mathbb{E}\int_{0}^{T}\l|m(t) \times (m(t) \times \Delta m(t))\r|_{L^2}^2\, dt
			\nonumber & \leq  C \mathbb{E}\int_{0}^{T}\l|m(t)\r|_{L^{\infty}}^2 \l| m(t) \times \Delta m(t)\r|_{L^2}^2\, dt \\
			& = \mathbb{E}\int_{0}^{T}\l|m(t) \times \Delta m(t)\r|_{L^2}^2\, dt <\infty.
		\end{align}
		Hence the process $m\times(m\times\Delta m)$ also lies in the space $L^2\l(\Omega ; L^2\l(0,T; L^2\r) \r)$ $\mathbb{P}$-a.s.
		Arguing similarly, we say that by the constraint condition \eqref{eqn-constraint condition} and part $(4)$ in the Assumption \ref{assumption on u} on the process $u$,
		\begin{align}
			\nonumber \mathbb{E}\int_{0}^{T}\l|m(t) \times u(t)\r|_{L^2}^2\, dt & \leq   \mathbb{E} \int_{0}^{T}\l|m(t)\r|^2_{L^{\infty}}\l|u(t)\r|_{L^2}^2\, dt \\
			& = \mathbb{E} \int_{0}^{T}\l|u(t)\r|_{L^2}^2\, dt < \infty.
		\end{align}
		Again from the constraint condition \eqref{eqn-constraint condition} and the above inequality,
		\begin{align}
			\nonumber \mathbb{E}\int_{0}^{T}\l|m(t) \times (m(t) \times u(t))\r|_{L^2}^2\, dt
			\nonumber& \leq  C \mathbb{E}\int_{0}^{T}\l|m(t)\r|^2_{L^{\infty}}\l|m(t) \times u(t)\r|_{L^2}^2\, dt \\
			\nonumber& =  C \mathbb{E}\int_{0}^{T}\l|m(t) \times u(t)\r|_{L^2}^2\, dt \ \text{By}\ \eqref{eqn-constraint condition} \\
			& <\infty.
		\end{align}		
		We recall that $$G(m) = m \times h - \alpha \, m \times ( m\times h ).$$ It is thus sufficient to verify the above inequality for the two terms individually. We also recall that $h$ is assumed to be in $H^1$. The continuous embedding $H^1\hookrightarrow L^{\infty}$ implies that there exists a constant $C>0$ such that
		\begin{equation*}
			\l|h\r|_{L^{\infty}} \leq C \l|h\r|_{H^1} < \infty.
		\end{equation*}
		Thus,
		\begin{align}
			\nonumber \mathbb{E}\int_{0}^{T}\l|m(t) \times h\r|_{L^2}^2\, dt
			\nonumber & \leq  \mathbb{E}\int_{0}^{T} \l|m(t)\r|_{L^2}^2 \l| h \r|_{L^{\infty}}^2\, dt \\
			& \leq  T\l| h \r|_{L^{\infty}}^2 \mathbb{E} \sup_{t\in [0,T]}\l|m(t)\r|_{L^2}^2 < \infty.
		\end{align}
		The right hand side of the last inequality is finite because of the constraint condition\dela{ and the continuous embedding $L^{\infty} \hookrightarrow L^2$}. Similarly,
		\begin{align}
			\mathbb{E}\int_{0}^{T}\l|m(t) \times (m(t) \times h)\r|_{L^2}^2\, dt
			& \leq  \mathbb{E}\int_{0}^{T} \l|m(t)\r|_{L^{\infty}} \l|m(t) \times h\r|_{L^2}^2\, dt  < \infty.
		\end{align}
		The right hand side of the above inequality is finite by the constraint condition \eqref{eqn-constraint condition} and the assumption on $h$.
		Hence $G(m)$ takes values in the space $L^2\l( 0,T; L^2 \r)$ $ \mathbb{P} $-a.s. What remains is to verify the bounds for the correction term, that is to show that the term
		$\l(DG(m)\r)\l(G(m)\r)$ also lies in the space $L^2\l( \Omega ; L^2 \l( 0,T; L^2 \r) \r)$, $\mathbb{P}$-a.s.

		Recall that Proposition \ref{prop-derivative} shows that the correction term is locally Lipschitz. Also, by the definition of the term $\l[DG(m)\r]\l(G(m)\r)$, we have
		\begin{equation*}
			\l[DG(0)\r]\l(G(0)\r) = 0.
		\end{equation*}
		The constraint condition \eqref{eqn-constraint condition} implies that the process $m$ takes values in the unit ball in the space $L^{\infty}$.
		Hence there exists a constant $C>0$ such that
		\begin{align*}
			\l| DG \big( m(t) \big) \big[ G \big(m(t) \big) \big] \r|_{L^2} \leq C \l| m(t) \r|_{L^2}.
		\end{align*}
		Hence
		\begin{align*}
			\mathbb{E} \int_{0}^{T} \l| DG\big(m(t)\big)\big[G\big(m(t)\big)\big] \r|_{L^2}^2 \, dt \leq C \mathbb{E} \int_{0}^{T} \l| m(t) \r|_{L^2}^2 \, dt < \infty.
		\end{align*}
		The right hand side of the last inequality is finite by Theorem \eqref{Theorem Existence of a weak solution}.
		This concludes the proof of Corollary \ref{Strong form of weak martingale solution}.		
	\end{proof}
	Before we start the proof of the Theorem \ref{thm-pathwise uniqueness}, we state a proposition, followed by a corollary that will be used for the proof.
	\begin{proposition}\label{Proposition m times m times Delta m equals Delta m plus gradient m squared m}
		Let $v\in H^1$. Further assume that
		\begin{equation}\label{Intermediate eqn 1 Proposition m times m times Delta m equals Delta m plus gradient m squared m}
			|v(x)|_{\mathbb{R}^3} = 1\ \text{for Leb. a.a.}\ x\in D.
		\end{equation}
		Then the following equality holds in $(H^{1})^{\p}$.
		\begin{align}\label{Intermediate eqn 2 Proposition m times m times Delta m equals Delta m plus gradient m squared m}
			v \times ( v \times \Delta v ) = - \Delta v - |\nabla v|_{\mathbb{R}^3}^2 v.
		\end{align}
	\end{proposition}
	\begin{proof}[Proof of Proposition \ref{Proposition m times m times Delta m equals Delta m plus gradient m squared m}]
		\dela{
			The proof follows from the following property: For $a,b,c\in\mathbb{R}^3$,
			\begin{equation}
				a \times \l( b \times c \r) = b \l(a.c\r) - c\l(a.b\r).
			\end{equation}
		}
		
		We begin by verifying that each side of equality
		\eqref{Intermediate eqn 2 Proposition m times m times Delta m equals Delta m plus gradient m squared m} belongs to the space $(H^1)^\prime$.
		By the equality in \eqref{Interpretation eqn 1}, we now show that $\Delta v$ takes values in the space $\l(H^1\r)^{\prime}$.
		Let $\phi\in H^1$. Then
		\begin{align*}
			\l|\ _{\l(H^1\r)^{\prime}}\l\langle \Delta v , \phi \r\rangle_{H^{1}} \r| &= \l|- \l\langle \nabla v , \nabla \phi \r\rangle_{L^2}\r| \\
			&= \l| \l\langle \nabla v , \nabla \phi \r\rangle_{L^2}\r| \\
			& \leq \l| \nabla v \r|_{L^2} \l|\nabla \phi \r|_{L^2}.\dela{\ (\text{By Cauchy-Schwartz inequality})}
		\end{align*}
		The assumptions on $v$ and $\phi$ imply that the right hand side, and hence the left hand side of the above inequality, is finite.\\
		The second term on the right hand side of \eqref{Intermediate eqn 2 Proposition m times m times Delta m equals Delta m plus gradient m squared m} is interpreted as follows:
		\begin{align}\label{Interpretation of nabla v R 3 squared v}
			\ _{(H^1)^{\prime}}\l\langle \l|\nabla v\r|_{\mathbb{R}^3}^2 v , \phi \r\rangle_{H^1} = \int_{\mathcal{O}} \l|\nabla v(x)\r|_{\mathbb{R}^3}^2 \l\langle v(x) , \phi(x) \r\rangle_{\mathbb{R}^3}\, dx.
		\end{align}
		To show that the right hand side of the above equality makes sense, we observe that $\phi\in H^1$ implies that $\phi\in L^{\infty}$\dela{ (due to the Sobolev embedding $H^1\hookrightarrow L^{\infty}$)}. This along with the equality \eqref{Intermediate eqn 1 Proposition m times m times Delta m equals Delta m plus gradient m squared m}
		\begin{align*}
			\l|\int_{\mathcal{O}} \l|\nabla v(x)\r|_{\mathbb{R}^3}^2 \l\langle v(x) , \phi(x) \r\rangle_{\mathbb{R}^3}\, dx\r| \leq C \int_{\mathcal{O}} \l|\nabla v(x)\r|_{\mathbb{R}^3}^2 \, dx.
		\end{align*}
		The right hand side of the above inequality is finite since $v\in H^1$.
		The left hand side of the equality \eqref{Intermediate eqn 2 Proposition m times m times Delta m equals Delta m plus gradient m squared m} is in $(H^1)^{{\prime}}$ by the way the triple product is understood in \eqref{Interpretation eqn 3}.\\
		Hence both the terms on the right hand side of the equality \eqref{Intermediate eqn 2 Proposition m times m times Delta m equals Delta m plus gradient m squared m} belong to the space $\l(H^1\r)^{\prime}$. We now proceed to show the equality.
			%
		%
		Let $\phi\in H^1$.
		The proof uses the following identity in $\mathbb{R}^3$:
		\begin{equation}
			a \times (b \times c) = b \l\langle a , c\r\rangle_{\mathbb{R}^3} - c \l\langle a , b \r\rangle_{\mathbb{R}^3},\ a,b,c\in\mathbb{R}^3.
		\end{equation}
		By \eqref{Interpretation eqn 3}, we have
		\begin{align*}
			\ _{(H^1)^{\prime}} \l\langle v \times ( v \times \Delta v ) , \phi \r\rangle_{H^{1}} &=  \l\langle v \times \nabla (\phi \times v) , \nabla v \r\rangle_{L^2} \\
			& =  \l\langle v \times (\nabla \phi \times v) , \nabla v \r\rangle_{L^2} +  \l\langle v \times  (\phi \times \nabla v) , \nabla v \r\rangle_{L^2} \\
			& =  \l\langle \nabla \phi |v|_{\mathbb{R}^3}^2  - v \l\langle v , \nabla \phi \r\rangle_{\mathbb{R}^3} , \nabla v \r\rangle_{L^2} +  \l\langle \l\langle \nabla v , v \r\rangle_{\mathbb{R}^3}\phi - \nabla v \l\langle \phi , v \r\rangle_{L^2} , \nabla v \r\rangle_{L^2} \\
			& =  \l\langle \nabla \phi |v|_{\mathbb{R}^3}^2 , \nabla v \r\rangle_{L^2} -  \l\langle \nabla v \l\langle \phi , v \r\rangle_{\mathbb{R}^3} , \nabla v 	\r\rangle_{L^2} \ (\text{By}\ \eqref{dot product in R3 of v and nabla v is 0})\\
			& =  \l\langle \nabla \phi  , \nabla v \r\rangle_{L^2} -  \l\langle \nabla v \l\langle \phi , v \r\rangle_{\mathbb{R}^3} , \nabla v \r\rangle_{L^2}.\ (\text{By}\ \eqref{Intermediate eqn 1 Proposition m times m times Delta m equals Delta m plus gradient m squared m})
		\end{align*}
		In view of the equalities \eqref{Interpretation eqn 1} and \eqref{Interpretation of nabla v R 3 squared v}, the right hand side of the above equality equals $$-\Delta v - \l|\nabla v\r|_{\mathbb{R}^3}^2 v$$ in $(H^1)^\p$.

		The following equality has been used in the calculations above:
		\begin{align}\label{dot product in R3 of v and nabla v is 0}
			\l\langle v , \nabla v \r\rangle_{\mathbb{R}^3} &= \frac{1}{2}\nabla |v|_{\mathbb{R}^3}^2 = 0.
		\end{align}
		The right hand side of the above equality is $0$ since by \eqref{Intermediate eqn 1 Proposition m times m times Delta m equals Delta m plus gradient m squared m}, $\l|v\r|_{\mathbb{R}^3}^2$ is constant.
		
		Hence
		\begin{align*}
			v \times ( v \times \Delta v ) = - \Delta v - |\nabla v|_{\mathbb{R}^3}^2 v.
		\end{align*}
		This concludes the proof of Proposition \ref{Proposition m times m times Delta m equals Delta m plus gradient m squared m}.		
	\end{proof}
	We have the following result as a corollary of the above proposition.
	
	\begin{corollary}\label{Corollary m times m times Delta m equals Delta m plus gradient m squared m}
		
		Let $\l( \Omega , \mathcal{F} , \mathbb{P} , W , m , u \r)$ be a weak martingale solution of \eqref{problem considered} corresponding to the control process $u$, as in Corollary \ref{Strong form of weak martingale solution}. Then the following equality holds in $\l(H^1\r)^{\prime}$ for every $ t\in[0,T] $
		\begin{align*}
			m(t) \times (m(t) \times \Delta m(t)) = - \Delta m(t) - |\nabla m(t)|_{\mathbb{R}^3}^2 m(t),\ \mathbb{P}-a.s.
		\end{align*}
	\end{corollary}
	\begin{proof}[Proof of Corollary \ref{Corollary m times m times Delta m equals Delta m plus gradient m squared m}]
		To prove this corollary, it is sufficient to show that the process $m$ satisfies the assumptions in Proposition \ref{Proposition m times m times Delta m equals Delta m plus gradient m squared m}. Theorem \ref{Theorem Existence of a weak solution} implies that, in particular, for each $t\in[0,T]$, $m(t)\in H^1$, $\mathbb{P}$-a.s. Also, the constraint condition \eqref{eqn-constraint condition} implies that $\l|m(t,x)\r|_{\mathbb{R}^3} = 1$, Leb-a.a. $x\in D$ for all $t\in[0,T]$,  $\mathbb{P}$-a.s. Hence the corollary follows by applying Proposition \ref{Proposition m times m times Delta m equals Delta m plus gradient m squared m} to $m(t)$ for each $t\in[0,T]$.
	\end{proof}
	Using the above mentioned corollary, we proceed to prove the pathwise uniqueness.
	
	\begin{proof}[Proof of Theorem \ref{thm-pathwise uniqueness}]
		Let us choose and fix a control process  $u$ satisfying Assumption \ref{assumption on u} and  two weak martingale solutions $(\Omega , \mathcal{F} , \mathbb{P} , W , m_1 , u)$ and $(\Omega , \mathcal{F} , \mathbb{P} , W , m_2 , u)$ corresponding to $u$  as in Definition \ref{Definition of Weak martingale solution} and satisfying the properties stated in Theorem \ref{Theorem Existence of a weak solution}.\\
		Let us first observe that in view of Corollary \ref{Corollary m times m times Delta m equals Delta m plus gradient m squared m}, for each $i=1,2$, the following identity holds in $(H^1)^\prime$:
		
		\begin{align}\label{eqn_m_i-strong}
			\nonumber	m_i(t) =& \,\alpha \, \int_{0}^{t}  \Delta m_i(s) \, ds   + \alpha \, \int_{0}^{t}  |\nabla m_i(s)|_{\mathbb{R}^3}^2 m_i(s) \, ds \\
			\nonumber& + \int_{0}^{t}m_i(s) \times \Delta m_i(s)\, ds + \int_{0}^{t} m_i(s)\times u(s)\, ds - \alpha \, \int_{0}^{t} m_i(s)\times(m_i(s)\times u(s))\, ds
			\\
			& + \frac{1}{2}\int_{0}^{t} \l[DG\bigl(m_i(s)\bigr)\r]\l[G\big(m_i\l(s\r)\big)\r]\,  ds + \int_{0}^{t} G\big(m_i(s)\big) \, dW(s),
		\end{align}
		for all $t\in[0,T]$, $\mathbb{P}$-a.s.
		\dela{The above equality makes sense because of Corollary \ref{Strong form of weak martingale solution} and the following calculations.\\}
		The above equation is same as the equation in Corollary \ref{Strong form of weak martingale solution}, except that the triple product term is expressed as a sum of two terms. The equality holds in $(H^{1})^\p$ and hence it should not make a difference to the equation. It is thus sufficient to show that individually both the integrands lie in the space $L^2\l(0,T;(H^{1})^\p\r)$.
		
		Following the arguments in Proposition
		\ref{Proposition m times m times Delta m equals Delta m plus gradient m squared m}, we can prove that for $v\in L^2(0,T;H^1)$ and $t\in[0,T]$,
		\begin{align*}
			\int_{0}^{t} \ _{(H^1)^{\prime}}\l\langle \Delta m_i(s) , v \r\rangle_{H^1}\, ds = -\int_{0}^{t} \l\langle \nabla m_i(s) , \nabla v(s) \r\rangle_{L^2}\, ds.
		\end{align*}
		Thus by the Cauchy-Schwartz inequality,
		\begin{align*}
			\l| \int_{0}^{t} \l\langle \nabla m_i(s) , \nabla v(s) \r\rangle_{L^2}\, ds \r| \leq \l( \int_{0}^{t} \l| m_i(s) \r|_{L^2}^2 \, ds \r)^{\frac{1}{2}} \l( \int_{0}^{t} \l| v(s) \r|_{L^2}^2 \, ds \r)^{\frac{1}{2}} < \infty.
		\end{align*}
		The  right hand side of the above inequality is finite because of the assumptions on $m_i$ and $v$.
		
		We now show that the remaining (second) term also takes values in the space $(H^{1})^\p$.
		
		\begin{align*}
			\int_{0}^{T}\l| \l|\nabla m(t)\r|_{\mathbb{R}^3}^2m_i(t)\r|_{(H^1)^{\prime}}^2\, dt &\leq \int_{0}^{T} \l|\nabla m_i(t)\r|_{L^2}^2 \l|\nabla m_i(t)\r|_{L^2}^2 \l| m_i(t)\r|^2_{L^{\infty}}\, dt \\
			& \leq \l(\int_{0}^{T} \l|\nabla m_i(t)\r|_{L^2}^4\, dt\r)^{\frac{1}{2}} \l(\int_{0}^{T} \l|\nabla m_i(t)\r|_{L^2}^4\, dt\r)^{\frac{1}{2}} \\
			\dela{	&\quad\text{(By Cauchy-Schwartz inequality)}\\
			}& = \int_{0}^{T} \l|\nabla m_i(t)\r|_{L^2}^4\, dt \\
			& \leq C T \sup_{t\in [0,T]} \l|\nabla m_i(t)\r|_{L^2}^4.
		\end{align*}
		Hence
		\begin{align*}
			\mathbb{E} \l[ 	\int_{0}^{T}\l| \l|\nabla m(t)\r|_{\mathbb{R}^3}^2m(t)\r|_{\l(H^1\r)^{\prime}}^2\, dt  \r] \leq C\mathbb{E} \l[ \sup_{t\in [0,T]} \l|\nabla m(t)\r|_{L^2}^4 \r] <\infty.
		\end{align*}
		The last inequality follows from the Theorem \ref{Theorem Existence of a weak solution}. This justifies the writing of equation \eqref{eqn_m_i-strong}.\\
		Define a process $m$ by
		\begin{equation*}
			m(t) = m_1(t) - m_2(t)\ \text{for}\ t\in[0,T].
		\end{equation*}
		We now consider the equation \eqref{eqn_m_i-strong} satisfied by each $m_i$ for $i=1,2$. To get the equation satisfied by the process $m$, take the difference of \eqref{eqn_m_i-strong} for $i=1$ and $i=2$. We then simplify it to get the equality in \eqref{pathwise uniqueness initial equality}.
		\dela{
			\begin{align}
				\nonumber &m(t) = \alpha \, \int_{0}^{t} \Delta m(s) \, ds + \alpha \, \int_{0}^{t} |\nabla m_1(s)|_{\mathbb{R}^3}^2 m(s) \, ds \\
				\nonumber & + \alpha \, \int_{0}^{t} (|\nabla m_1(s)|_{\mathbb{R}^3} - |\nabla m_2(s)|_{\mathbb{R}^3})(|\nabla m_1(s)_{\mathbb{R}^3}| + |\nabla m_2(s)_{\mathbb{R}^3}|)m_2(s) \, ds \\
				\nonumber & + \int_{0}^{t} m(s) \times \Delta m_1(s) \, ds + \int_{0}^{t} m_2(s) \times \Delta m(s) \, ds + \frac{1}{2}\int_{0}^{t}((m(s) \times h) \times h) \, ds  + \int_{0}^{t} m(s) \times u(s) \, ds\\
				\nonumber & + \int_{0}^{t} (m(s) \times h) \, ds - \alpha \, \bigg[ \int_{0}^{t} m(s) \times (m_1(s) \times u(s)) \, ds + \int_{0}^{t} m_2(s) \times (m(s) \times u(s)) \, ds\bigg] \\
				\nonumber & - \frac{1}{2}\alpha \,  \bigg[\int_{0}^{t} (m(s) \times (m_1(s)\times h)) \times h \, ds + \int_{0}^{t} (m_2(s) \times (m(s)\times h)) \times h \, ds \bigg] \\
				\nonumber & -\frac{1}{2}\alpha \, \bigg[ \int_{0}^{t} m(s) \times ((m_1(s) \times	 h) \times h) \, ds + \int_{0}^{t} m_2(s) \times ((m(s) \times	h) \times h) \, ds \bigg] \\
				\nonumber & + \frac{1}{2}\alpha^2 \bigg[ \int_{0}^{t} (m(s) \times (m_1(s) \times h)) \times (m_1(s) \times h) \, ds \\
				\nonumber &  + \int_{0}^{t} (m_2(s) \times (m(s) \times h)) \times (m_1(s) \times h) \, ds + \int_{0}^{t} (m_2(s) \times (m_2(s) \times h)) \times (m(s) \times h) \, ds \bigg] \\
				\nonumber & + \frac{1}{2} \alpha^2 \bigg[ \int_{0}^{t} m(s) \times ((m_1(s) \times (m_1(s) \times h)) \times h) \, ds  \\
				\nonumber &  + \int_{0}^{t} m_2(s) \times ((m(s) \times (m_1(s) \times h)) \times h) \, ds + \int_{0}^{t} m_2(s) \times ((m_2(s) \times (m(s) \times h)) \times h) \bigg]\,ds \\
				& + \int_{0}^{t} (m(s) \times h) \, dW(s) - \alpha \, \bigg[ \int_{0}^{t} m(s) \times (m_1(s) \times h) \, dW(s) + \int_{0}^{t} m_2(s) \times (m(s) \times h) \, dW(s) \bigg]
			\end{align}
		}
		
		\begin{align}\label{pathwise uniqueness initial equality}
			\nonumber &m(t) = \alpha \, \int_{0}^{t} \Delta m(s) \, ds + \alpha \, \int_{0}^{t} |\nabla m_1(s)|_{\mathbb{R}^3}^2 m(s) \, ds \\
			\dela{\nonumber & + \alpha \, \int_{0}^{t} (|\nabla m_1(s)|_{\mathbb{R}^3} - |\nabla m_2(s)|_{\mathbb{R}^3})(|\nabla m_1(s) |_{\mathbb{R}^3} + |\nabla m_2(s)_{\mathbb{R}^3}|)m_2(s) \, ds \\}
			\nonumber & + \alpha \, \int_{0}^{t} \l\langle \big(\nabla m_1(s) - \nabla m_2(s) \big) , \big( \nabla m_1(s)  + \nabla m_2(s) \big) \r\rangle_{\mathbb{R}^3} m_2(s) \, ds \\
			\nonumber & + \int_{0}^{t} m(s) \times \Delta m_1(s) \, ds + \int_{0}^{t} m_2(s) \times \Delta m(s) \, ds  + \int_{0}^{t} m(s) \times u(s) \, ds\\
			\nonumber &  - \alpha \, \bigg[ \int_{0}^{t} m(s) \times (m_1(s) \times u(s)) \, ds + \int_{0}^{t} m_2(s) \times \big(m(s) \times u(s)\big) \, ds \bigg] + \int_{0}^{t} V_n(s) \, ds \\
			& + \int_{0}^{t} (m(s) \times h) \, dW(s) - \alpha \, \bigg[ \int_{0}^{t} m(s) \times \big(m_1(s) \times h \big) \, dW(s) + \int_{0}^{t} m_2(s) \times \big(m(s) \times h \big) \, dW(s) \bigg],
		\end{align}
		where
		\begin{align}
			\nonumber & \int_{0}^{t} V_n(s) \, ds =     \int_{0}^{t} (m(s) \times h) \, ds + \frac{1}{2}\int_{0}^{t}((m(s) \times h) \times h) \, ds
			- \frac{1}{2}\alpha \,  \bigg[\int_{0}^{t} (m(s) \times (m_1(s)\times h)) \times h \, ds \\
			\nonumber & + \int_{0}^{t} (m_2(s) \times (m(s)\times h)) \times h \, ds \bigg] \\
			\nonumber & -\frac{1}{2}\alpha \, \bigg[ \int_{0}^{t} m(s) \times ((m_1(s) \times	 h) \times h) \, ds + \int_{0}^{t} m_2(s) \times ((m(s) \times	h) \times h) \, ds \bigg] \\
			\nonumber & + \frac{1}{2}\alpha^2 \bigg[ \int_{0}^{t} (m(s) \times (m_1(s) \times h)) \times (m_1(s) \times h) \, ds \\
			\nonumber &  + \int_{0}^{t} (m_2(s) \times (m(s) \times h)) \times (m_1(s) \times h) \, ds + \int_{0}^{t} (m_2(s) \times (m_2(s) \times h)) \times (m(s) \times h) \, ds \bigg] \\
			\nonumber & + \frac{1}{2} \alpha^2 \bigg[ \int_{0}^{t} m(s) \times ((m_1(s) \times (m_1(s) \times h)) \times h) \, ds  \\
			\nonumber &  + \int_{0}^{t} m_2(s) \times ((m(s) \times (m_1(s) \times h)) \times h) \, ds + \int_{0}^{t} m_2(s) \times ((m_2(s) \times (m(s) \times h)) \times h) \bigg]\,ds \\
		\end{align}
		
		\dela{Is the detailed derivation of this equation required? The main idea is to add and subtract terms and then simplify.}
		
		For convenience of notation, let us write equation \eqref{pathwise uniqueness initial equality} as
		\begin{align}\label{pathwise uniqueness equation short form}
			m(t) = \sum_{i=1}^{9}\int_{0}^{t} C_i\,z_i(s)\,ds + \sum_{i=10}^{12} \int_{0}^{t} C_i\,z_i(s)\,dW(s).
		\end{align}
		Here $C_i,\ i=1,\dots,12$ are constants accompanying the integrals.
		\dela{Note that $z_i, i=10,...,21$ is the simplification of the term $\l\langle DG(m_1(s)) - DG(m_2(s)) , m(s) \r\rangle_{L^2}$.\\}

		Consider the function $\phi_5: L^2 \to \mathbb{R}$ defined by 
		\begin{equation*}
			v\mapsto \frac{1}{2} |v|_{L^2}^2.
		\end{equation*}
		Consider the process $m$ defined above. We apply the It\^o formula \cite{Pardoux_1979} to $\phi_5$. That the integrands on the right hand side of the equation \eqref{pathwise uniqueness equation short form} satisfy the conditions mentioned in \cite{Pardoux_1979} can be verified as done in  section \ref{sec-The constraint condition section}.

		Applying the It\^o formula gives us the following equation:
		\begin{align}\label{equation Ito formula Pathwise uniqueness}
			\nonumber \frac{1}{2}\l| m(t) \r|_{L^2}^2 =& \frac{1}{2}\l| m(0) \r|_{L^2}^2 + \sum_{i=1}^{9}\int_{0}^{t} C_i \l\langle z_i(s) , m(s) \r\rangle_{L^2} \, ds + \sum_{i=10}^{12}\int_{0}^{t} C_i \l\langle z_i(s) , m(s) \r\rangle_{L^2} \, dW(s) \\
			& + \frac{1}{2}\int_{0}^{t} \l| G\big(m(s)\big)\r|_{L^2}^2 \, ds,
		\end{align}
		for all $t\in[0,T]$ $\mathbb{P}$-a.s.
		Let us denote the last term on the right hand side of the above equality by $Z_{13}$.
		Note that since $m_1$ and $m_2$ have the same initial data, $m(0) = 0$, $\mathbb{P}=a.s.$
		For the sake of simplicity, we write some calculations separately and then combining them gives the desired result.\\
		\textbf{Calculation for $z_1$.}\\
		For each $t\in[0,T]$, the following equality holds $\mathbb{P}^{\prime}$-a.s., see \eqref{Interpretation eqn 1}
		\begin{align*}
			\int_{0}^{t} \l\langle \Delta m(s) , m(s) \r\rangle_{(H^1)^{\prime}}\, ds = - \int_{0}^{t} \l| \nabla m(s) \r|^2\, ds.
		\end{align*}
		The negative sign here implies that this term goes to the left hand side of the equality with a positive coefficient and hence can be used to balance the other $\int_{0}^{t} |\nabla m(s)|_{L^2}^2\, ds$ terms coming from some of the other estimates.\\
		\textbf{Calculations for the terms $z_2$ and $z_3$.}\\
		\dela{
			We briefly descrobe how we get those terms. We add and subtract the term $\l|m_1(s)\r|_{\mathbb{R}^3}^2m_2(s)$ on the left part of the inner product to get the following equality.
			
			\begin{align*}
				\l\langle \l| \nabla m_1(s) \r|_{\mathbb{R}^3}^2 m_1(s) - \l| \nabla m_2(s) \r|_{\mathbb{R}^3}^2 m_2(s), m(s) \r\rangle_{L^2} =&  \l\langle \l| \nabla m_1(s) \r|_{\mathbb{R}^3}^2 m_1(s) - \l|m_1(s)\r|_{\mathbb{R}^3}^2m_2(s) , m(s) \r\rangle_{L^2} \\
				& + \l\langle \l|m_1(s)\r|_{\mathbb{R}^3}^2m_2(s)\l| \nabla m_2(s) \r|_{\mathbb{R}^3}^2 m_2(s), m(s) \r\rangle_{L^2} \\
				=&  \l\langle \l| \nabla m_1(s) \r|_{\mathbb{R}^3}^2 (m_1(s) - m_2(s)) , m(s) \r\rangle_{L^2}  \\
				& + \l\langle \l( \nabla m_1(s) + \nabla m_2(s) \r) (\nabla m_1(s)\\
				& - \nabla m_2(s)) m_1(s) , m(s) \r\rangle_{L^2} \\
				=& \l\langle \l| \nabla m_1(s) \r|_{\mathbb{R}^3}^2 m(s) , m(s) \r\rangle_{L^2} \\
				& + \l\langle \l( \nabla m_1(s) + \nabla m_2(s) \r)  \nabla m(s) m_1(s) , m(s) \r\rangle_{L^2}.
			\end{align*}
		}
		The bound on the terms is calculated below.\dela{ Let $s\in[0,T]$.}
		\dela{
			\begin{align*}
				\l\langle \l| \nabla m_1(s) \r|_{\mathbb{R}^3}^2 m(s) , m(s) \r\rangle_{L^2} &\leq C\l| \nabla m_1 \r|_{L^2}^2 \l| m \r|_{L^{\infty}}^2 \\
				& \leq C\l| \nabla m_1 \r|_{L^2}^2 \l| m \r|_{L^2}\l| m \r|_{H^{1}} \\
				& \leq C \l| \nabla m_1 \r|_{L^2}^2 \l| m \r|_{L^2} \l[ \l| m \r|_{L^2} + \l| \nabla m \r|_{L^2} \r] \\
				& \leq C \l| \nabla m_1 \r|_{L^2}^2 \l| m \r|_{L^2}^2  + C \l| \nabla m_1 \r|_{L^2}^2 \l| m \r|_{L^2}\l| \nabla m \r|_{L^2} \\
				& \leq C \l| \nabla m_1 \r|_{L^2}^2 \l| m \r|_{L^2}^2 + C^2\frac{C(\varepsilon)}{2} \l| \nabla m_1 \r|_{L^2}^4\l| m(s) \r|_{L^2}^2 + \frac{\varepsilon}{2}\l| \nabla m(s) \r|_{L^2}^2.
			\end{align*}
		}
		By H\"older's inequality,
		\begin{align*}
			\int_{0}^{t} \l\langle \l| \nabla m_1(s) \r|_{\mathbb{R}^3}^2 m(s) , m(s) \r\rangle_{L^2} \, ds &\leq C \int_{0}^{t}  \l| \nabla m_1 \r|_{L^2}^2 \l| m \r|_{L^{\infty}}^2  \, ds \\
			\text{(By Agmon's inequality)} \ & \leq C \int_{0}^{t} \l| \nabla m_1 \r|_{L^2}^2 \l| m \r|_{L^2}\l| m \r|_{H^{1}} \, ds \\
			& \leq C \int_{0}^{t} \l| \nabla m_1 \r|_{L^2}^2 \l| m \r|_{L^2} \l[ \l| m \r|_{L^2} + \l| \nabla m \r|_{L^2} \r] \, ds \\
			& \leq C \int_{0}^{t} \l| \nabla m_1 \r|_{L^2}^2 \l| m \r|_{L^2}^2 \, ds  + C \int_{0}^{t} \l| \nabla m_1 \r|_{L^2}^2 \l| m \r|_{L^2}\l| \nabla m \r|_{L^2} \, ds \\
			\text{(By Young's inequality)} \ & \leq C \int_{0}^{t} \l| \nabla m_1 \r|_{L^2}^2 \l| m \r|_{L^2}^2 \, ds + C^2\frac{C(\varepsilon)}{2} \int_{0}^{t} \l| \nabla m_1 \r|_{L^2}^4\l| m(s) \r|_{L^2}^2 \, ds\\
			&\quad + \frac{\varepsilon}{2}\int_{0}^{t} \l| \nabla m(s) \r|_{L^2}^2 \, ds.
		\end{align*}
		Here $ \varepsilon > 0 $ will be chosen later. The above sequence of inequalities uses the inequality \eqref{Inequality L infinity leq L2 H1} along with Young's inequality.

		\begin{align*}
			\int_{0}^{t} \l\langle \l\langle \nabla m_1(s) , \nabla m(s) \r\rangle_{\mathbb{R}^3} m_2(s)  , m(s) \r\rangle_{L^2} \, ds  & \leq 	 \int_{0}^{t} \l|\nabla m_1(s)\r|_{L^2} \l|m_2(s)\r|_{L^{\infty}} \l| \nabla m(s) \r|_{L^2} \l| m(s) \r|_{L^{\infty}} \, ds  \\
			\text{(Since $\l|m_2(s)\r|_{L^{\infty}} = 1$, Agmon's inequality)}& \leq C \int_{0}^{t} \l|\nabla m_1(s)\r|_{L^2} \l| \nabla m(s) \r|_{L^2} \l| m(s) \r|_{L^2}^{\frac{1}{2}} \l| m(s) \r|_{H^1}^{\frac{1}{2}} \, ds  \\
			& \leq C \int_{0}^{t} \l|\nabla m_1(s)\r|_{L^2} \l| \nabla m(s) \r|_{L^2} \l| m(s) \r|_{L^2}^{\frac{1}{2}} \bigg[ \l| m(s) \r|_{L^2}^{\frac{1}{2}} \\
			& \qquad + \l| \nabla m(s) \r|_{L^2}^{\frac{1}{2}} \bigg] \, ds  \\
			& \leq C \int_{0}^{t} \l|\nabla m_1(s)\r|_{L^2} \l| \nabla m(s) \r|_{L^2} \l| m(s) \r|_{L^2} \, ds \\
			&\quad + C  \int_{0}^{t} \l|\nabla m_1(s)\r|_{L^2} \l| m(s) \r|_{L^2}^{\frac{1}{2}} \l| \nabla m(s) \r|_{L^2}^{\frac{3}{2}} \, ds  \\
			(\text{By Young's inequality for}\ p=q=2) & \leq C\frac{C(\varepsilon)}{2}  \int_{0}^{t} \l|\nabla m_1(s)\r|_{L^2}^2 \l|m(s)\r|_{L^2}^2  \, ds \\
			& \quad	+ \frac{\varepsilon}{2} \int_{0}^{t} \l|\nabla m(s)\r|_{L^2}^2 \, ds \\
			(\text{By Young's inequality for}\ p=4, q=\frac{4}{3}) & \quad + C^4\frac{C(\varepsilon)}{4} \int_{0}^{t} \l| \nabla m_1(s)\r|_{L^2}^4\l|m(s)\r|_{L^2}^2 \, ds  \\
			& \quad + \frac{3\varepsilon}{4}  \int_{0}^{t}  \l|\nabla m(s)\r|_{L^2}^2 \, ds  \\
			& = C(\varepsilon)   \int_{0}^{t} \l|m(s)\r|_{L^2}^2 \l[ \l|\nabla m_1(s)\r|_{L^2}^2 + \l|\nabla m_1(s)\r|_{L^2}^4 \r] \, ds \\
			& \quad + \frac{5\varepsilon}{4}  \int_{0}^{t}  \l|\nabla m(s)\r|_{L^2}^2 \, ds .
		\end{align*}		
		Similarly,
		\begin{align*}
			\int_{0}^{t} \bigl\langle \l\langle \nabla m_2(s) , \nabla m(s) \r\rangle_{\mathbb{R}^3} m_2(s)  , m(s) \bigr\rangle_{L^2} \, ds \leq & C(\varepsilon) \int_{0}^{t}   \l|m(s)\r|_{L^2}^2 \l[ \l|\nabla m_2(s)\r|_{L^2}^2 + \l|\nabla m_2(s)\r|_{L^2}^4 \r] \,ds \\
			& + \frac{5\varepsilon}{4} \int_{0}^{t} \l|\nabla m(s)\r|_{L^2}^2 \, ds.
		\end{align*}
		\textbf{Note:} All the constants have been condensed into $C(\varepsilon)$.\\
		Hence
		\begin{align*}
			&\int_{0}^{t} \l\langle \l| \nabla m_1(s) \r|_{\mathbb{R}^3}^2 m_1(s) - \l| \nabla m_2(s) \r|_{\mathbb{R}^3}^2 m_2(s), m(s) \r\rangle_{L^2}\, ds \\
			& \leq C(\varepsilon) \int_{0}^{t} \l|m(s)\r|_{L^2}^2 \l[ \l|\nabla m_1(s)\r|_{L^2}^2 + \l|\nabla m_1(s)\r|_{L^2}^4 + \l|\nabla m_2(s)\r|_{L^2}^2 + \l|\nabla m_2(s)\r|_{L^2}^4 \r] \, ds \\
			& \quad + \frac{5\varepsilon}{2} \int_{0}^{t} \l|\nabla m(s)\r|_{L^2}^2 \, ds.
		\end{align*}
		\textbf{Calculation for the terms $z_4$ and $z_5$.}\\
		\dela{
			The terms are obtained as follows.
			\begin{align*}
				&m_1(s) \times \Delta m_1(s) - m_2(s) \times \Delta m_2(s) \\&= m_1(s) \times \Delta m_1(s) - m_2(s) \times \Delta m_1(s) + m_2(s) \times \Delta m_1(s) - m_2(s) \times \Delta m_2(s) \\
				& = m(s) \times \Delta m_1(s) + m_2(s) \times \Delta m(s).
			\end{align*}
			What now follows is the bounds for the above mentioned terms.
		}
		\dela{
			\begin{align*}
				&\l|\l\langle m_1(s) \times \Delta m_1(s) - m_2(s) \times \Delta m_2(s) , m(s) \r\rangle_{L^2}\r| \\
				&\leq \l|\l\langle m(s) \times \Delta m_1(s)\r| + \l|m_2(s) \times \Delta m(s) , m(s) \r\rangle_{L^2}\r| \\
				& = \l| \l\langle m_2(s) \times \Delta m(s) , m(s)  \r\rangle_{L^2}\r| \\
				& = \l| \l\langle \nabla m_2(s) \times m(s) , \nabla m(s) \r\rangle_{L^2}\r| \\
				& \leq \frac{C(\varepsilon)}{2} \l|\nabla m_2(s)\r|_{L^2}^2 \l|m(s)\r|_{L^{\infty}}^2 + \frac{\varepsilon}{2}\l|\nabla m(s)\r|_{L^2}^2 \\
				& \leq \frac{C(\varepsilon)}{2}  \l|\nabla m_2(s)\r|_{L^2}^2 \l| m(s) \r|_{L^2}\l| m(s) \r|_{H^{1}} + \frac{\varepsilon}{2}\l|\nabla m(s)\r|_{L^2}^2 \\
				& \leq \frac{C(\varepsilon)}{2} \l|\nabla m_2(s)\r|_{L^2}^2\l|m(s)\r|_{L^2}\l[\l|m(s)\r|_{L^2} + \l|\nabla m(s)\r|_{L^2}\r] + \frac{\varepsilon}{2}\l|\nabla m(s)\r|_{L^2}^2 \\
				& \leq \frac{C(\varepsilon)}{2} \l|\nabla m_2(s)\r|_{L^2}^2\l|m(s)\r|_{L^2}^2 + \frac{C(\varepsilon)}{2} \l|\nabla m_2(s)\r|_{L^2}^2\l|m(s)\r|_{L^2}\l|\nabla m(s)\r|_{L^2}\\
				& \quad+ \frac{\varepsilon}{2}\l|\nabla m(s)\r|_{L^2}^2 \\
				& \leq \frac{C(\varepsilon)}{2} \l|\nabla m_2(s)\r|_{L^2}^2\l|m(s)\r|_{L^2}^2 + \frac{C(\varepsilon)}{2}\frac{C(\varepsilon)^2}{4}\l|\nabla m_2(s)\r|_{L^2}^4\l|m(s)\r|_{L^2}^2 \\
				& \quad+ \frac{\varepsilon}{2}\l|\nabla m(s)\r|_{L^2}^2 + \frac{\varepsilon}{2}\l|\nabla m(s)\r|_{L^2}^2
			\end{align*}
		}
		\begin{equation*}
			\int_{0}^{t} \l\langle z_4(s) , m(s) \r\rangle_{L^2} \, ds = 0.
		\end{equation*}
		\begin{align*}
			\l|\int_{0}^{t} \l\langle z_5(s) , m(s) \r\rangle_{L^2} \, ds\r| \dela{&\l|\l\langle m_1(s) \times \Delta m_1(s) - m_2(s) \times \Delta m_2(s) , m(s) \r\rangle_{L^2}\r| \\
				&\leq \l|\l\langle m(s) \times \Delta m_1(s)\r| + \l|m_2(s) \times \Delta m(s) , m(s) \r\rangle_{L^2}\r| \\}
			& \leq \int_{0}^{t} \l| \l\langle m_2(s) \times \Delta m(s) , m(s)  \r\rangle_{L^2}\r| \, ds \\
			& = \int_{0}^{t} \l| \l\langle \nabla m_2(s) \times m(s) , \nabla m(s) \r\rangle_{L^2}\r| \, ds \\
			(\text{H\"older's and Young's inequalities})& \leq \frac{C(\varepsilon)}{2} \int_{0}^{t} \l|\nabla m_2(s)\r|_{L^2}^2 \l|m(s)\r|_{L^{\infty}}^2 \, ds  + \frac{\varepsilon}{2} \int_{0}^{t} \l|\nabla m(s)\r|_{L^2}^2 \, ds  \\
			(\text{By Agmon's inequality})& \leq \frac{C(\varepsilon)}{2}  \int_{0}^{t}\l|\nabla m_2(s)\r|_{L^2}^2 \l| m(s) \r|_{L^2}\l| m(s) \r|_{H^{1}} \, ds  \\
			& \quad + \frac{\varepsilon}{2} \int_{0}^{t} \l|\nabla m(s)\r|_{L^2}^2 \, ds  \\
			& \leq \frac{C(\varepsilon)}{2} \int_{0}^{t} \l|\nabla m_2(s)\r|_{L^2}^2\l|m(s)\r|_{L^2}\l[\l|m(s)\r|_{L^2} + \l|\nabla m(s)\r|_{L^2}\r] \, ds  \\
			&   \quad + \frac{\varepsilon}{2} \int_{0}^{t} \l|\nabla m(s)\r|_{L^2}^2 \, ds  \\
			(\text{By Young's inequality})& \leq \frac{C(\varepsilon)}{2} \int_{0}^{t} \l|\nabla m_2(s)\r|_{L^2}^2\l|m(s)\r|_{L^2}^2 \, ds  \\
			&  \quad + \frac{C(\varepsilon)}{2}  \int_{0}^{t} \l|\nabla m_2(s)\r|_{L^2}^2\l|m(s)\r|_{L^2}\l|\nabla m(s)\r|_{L^2} \, ds \\
			& \quad + \frac{\varepsilon}{2} \int_{0}^{t} \l|\nabla m(s)\r|_{L^2}^2 \, ds  \\
			(\text{By Young's inequality})& \leq \frac{C(\varepsilon)}{2} \int_{0}^{t} \l|\nabla m_2(s)\r|_{L^2}^2\l|m(s)\r|_{L^2}^2 \, ds \\
			&  \quad + \frac{C(\varepsilon)}{2} \frac{C(\varepsilon)^2}{4} \int_{0}^{t} \l|\nabla m_2(s)\r|_{L^2}^4\l|m(s)\r|_{L^2}^2 \, ds  \\
			& \quad+ \frac{\varepsilon}{2} \int_{0}^{t} \l|\nabla m(s)\r|_{L^2}^2 \, ds
			+ \frac{\varepsilon}{2} \int_{0}^{t} \l|\nabla m(s)\r|_{L^2}^2 \, ds.
		\end{align*}
		Here $ \varepsilon > 0 $ will be chosen later.
		\dela{One of the terms in the first line is $0$ since $\l\langle a \times b, a \r\rangle = 0$. }The second equality is basically the way $m\times \Delta m$ is interpreted (as an element of $(H^1)^{\prime}$). The fourth inequality comes from the use of Young's $\varepsilon$ inequality.
		
		Combining the constants into one constant $C(\varepsilon)$, we get \dela{From the previous step, we can say that}
		\begin{align}
			\nonumber 
			\bigg|\int_{0}^{t} \l\langle z_4(s) +  z_5(s) , m(s) \r\rangle_{L^2} &\, ds \bigg| \leq C(\varepsilon)\int_{0}^{t} \bigg[ \l| \nabla m_2(s) \r|_{L^2}^2  \\
			& + \l| \nabla m_2(s) \r|_{L^2}^4 \bigg] \l| m(s) \r|_{L^2}^2   \, ds 
			+ \varepsilon \int_{0}^{t} \l| \nabla m(s) \r|_{L^2}^2 \, ds.
		\end{align}
		Here the constants depending on $\varepsilon$ are combined into one constant suitable $C(\varepsilon)$.\\		
		\textbf{Calculation for $z_6$.}\\
		Concerning the first term with the control process $u$, that is $z_6$, we observe that
		
		\begin{align*}
			\int_{0}^{t} \l\langle z_6(s) , m(s) \r\rangle_{L^2} \, ds  = \int_{0}^{t} \l\langle m(s) \times u(s) , m(s) \r\rangle_{L^2} \, ds  = 0.
		\end{align*}
		\textbf{Calculation for $z_7,z_8$.}\\
		For the remaining terms (with the control process $u$), the following estimate can be done.		
		\dela{
			\begin{align*}
				&m_1(s) \times (m_1(s) \times u(s)) - m_2(s) \times (m_2(s) \times u(s))\\
				& = m_1(s) \times (m_1(s) \times u(s)) - m_2(s) \times (m_1(s) \times u(s)) \\
				&+ m_2(s) \times (m_1(s) \times u(s)) - m_2(s) \times (m_2(s) \times u(s)) \\
				& = m(s) \times ( m_1(s) \times u(s) ) - m_2(s) \times (m(s) \times u(s)).
			\end{align*}
			Thus,
			\begin{align*}
				| \l\langle m_1(s) &\times (m_1(s) \times u(s)) - m_2(s) \times (m_2(s) \times u(s)) , m(s) \r\rangle_{L^2} | \\
				&= \l| \l\langle m(s) \times ( m_1(s) \times u(s) ) - m_2(s) \times (m(s) \times u(s)) , m(s) \r\rangle_{L^2} \r| \\
				& = \l| \l\langle m_2(s) \times (m(s) \times u(s)) , m(s) \r\rangle_{L^2} \r| \\
				& \leq \l|m_2(s)\r|_{L^2} \l| m(s) \r|_{L^{\infty}}^2 \l| u(s) \r|_{L^2} \\
				& \leq C \l|m_s(s)\r|_{L^{\infty}} \l| m(s) \r|_{L^2}\l| m(s) \r|_{H^1} \l| u(s) \r|_{L^2} \ \text{Since}\ \l|\cdot\r|_{L^{\infty}}^2 \leq C \l|\cdot\r|_{L^2}\l|\r|_{H^1} \\
				& \leq C  \l| m(s) \r|_{L^2} \l[\l| m(s) \r|_{L^2} + \l| \nabla m(s) \r|_{L^2}\r] \l| u(s) \r|_{L^2} \\
				& \leq C \l| m(s) \r|_{L^2}^2\l| u(s) \r|_{L^2} + C  \l| m(s) \r|_{L^2} \l| \nabla m(s) \r|_{L^2} \l| u(s) \r|_{L^2} \\
				& \leq C \l| m(s) \r|_{L^2}^2 \l| u(s) \r|_{L^2} + \frac{C(\varepsilon)}{2} \l| m(s) \r|_{L^2}^2\l| u(s) \r|_{L^2}^2 + \frac{\varepsilon}{2} \l| \nabla m(s) \r|_{L^2}^2 \\
				& \leq \l(C \l| u(s) \r|_{L^2} + \frac{C(\varepsilon)}{2} \l| u(s) \r|_{L^2}^2 \r) \l| m(s) \r|_{L^2}^2 + \frac{\varepsilon}{2} \l| \nabla m(s) \r|_{L^2}^2 .
			\end{align*}
			The first term from the first equality is zero since $\l\langle a \times b , a \r\rangle_{L^2} = 0$ for $a,b,c\in L^2$.\\
		}
		By H\"older's inequality, followed first by Agmon's inequality and then by Young's inequality implies that for $\varepsilon>0$, there exists constants $C,C(\varepsilon)$ such that for $t\in[0,T]$,
		\begin{align*}
			\int_{0}^{t}& | \l\langle m_1(s) \times \big( m_1(s) \times u(s) \big) - m_2(s) \times \big( m_2(s) \times u(s) \big) , m(s) \r\rangle_{L^2} | \, ds \\
			&\leq C \int_{0}^{t} \l( 1  + \l| u(s) \r|_{L^2}^2 \r) \l| m(s) \r|_{L^2}^2  \, ds + \frac{\varepsilon}{2} \int_{0}^{t}  \l| \nabla m(s) \r|_{L^2}^2  \, ds.
		\end{align*}
		The terms that remain are the terms corresponding to the noise term, that is $G(m)$ ($z_{10},z_{11},z_{12}$), It\^o to Stratonovich correction term $DG(m)(G(m))$, i.e. ($z_9$), along with the last term on the right hand side of \eqref{equation Ito formula Pathwise uniqueness}, i.e. $Z_{13}$.\\
		\textbf{Calculations for the terms $z_{9}$ and $Z_{13}$.}\\
		By Lemma \ref{Lemma G is a polynomial map} and Proposition \ref{prop-derivative}, both $z_{9},Z_{13}$ are locally Lipschitz. Hence it is sufficient to show that the processes $m_1$ and $m_2$ lie in a ball in the space $L^2$. In this direction, by the continuous embedding $L^{\infty}\hookrightarrow L^2$ and the Theorem \ref{Theorem Existence of a weak solution}, there exists a constant $C>0$ such that
		\begin{equation}
			|m_i|_{L^2} \leq C|m_i|_{L^{\infty}} \leq 2C.
		\end{equation}
		for $i=1,2$. The processes $m_1(s)$ and $m_2(s)$ thus take values in a ball in $L^2$. Hence there exists a constant $C>0$ such that for each $s\in[0,T]$,
		\begin{align*}
			|G(m_1(s)) - G(m_2(s))|_{L^2} \leq C_1|m_1(s) - m_2(s)|_{L^2} = C_1|m(s)|_{L^2}.
		\end{align*}
		Similarly, there exists another constant $C_2$ such that for each $s\in[0,T]$,
		\begin{align*}
			\l| DG\big( m_1(s) \big)\l[ G(m_1)(s)\r] - DG\big(m_2(s)\big)\l[G\big(m_2(s)\big)\r] \r|_{L^2} &\leq C_2 	\l|G(m_1(s)) - G\big(m_2(s)\big) \r|_{L^2} \\
			&\leq C_1 C_2|m_1(s) - m_2(s)|_{L^2} \\
			&= C_1 C_2|m(s)|_{L^2}.
		\end{align*}
		Hence by the Cauchy-Schwartz inequality and the above estimate, we have
		\begin{align*}
			\int_{0}^{T} \l\langle G(m_1) - G(m_2) , m(s) \r\rangle_{L^2}\, ds &\leq \int_{0}^{T} \l| G(m_1) - G(m_2) \r|_{L^2}\l| m(s)\r|_{L^2}\, ds \\
			& \leq C_1 \int_{0}^{T} \l| m(s) \r|_{L^2}^2\, ds.
		\end{align*}
		Similarly,
		\begin{align*}
			&\int_{0}^{t} \l\langle DG\big(m_1(s)\big)\l[G\big(m_1(s)\big)\r] - DG\big(m_2(s)\big)\l[G\big(m_2(s)\big)\r] , m(s) \r\rangle_{L^2}\, ds \\
			&\leq \int_{0}^{t} \l| DG\big(m_1(s)\big)\l[G\big(m_1(s)\big)\r] - DG\big(m_2(s)\big)\l[G\big(m_2(s)\big)\r] \r|_{L^2} \l| m(s) \r|_{L^2} \\
			& \leq C_1 C_2\int_{0}^{t} \l| m(s) \r|_{L^2}^2\, ds.
		\end{align*}
		Regarding the correction term that appears after the use of the It\^o formula, by the locally Lipschitz continuity of $G$, there exists a constant $C>0$ such that
		\begin{align*}
			\int_{0}^{t} \l|G(m_1(s)) - G(m_2(s))\r|_{L^2}^2 \, ds \leq C \int_{0}^{t} \l|m(s)\r|_{L^2}^2 \, ds.
		\end{align*}
		
		\dela{Correct this term.
			\begin{align*}
				\l| m(t) \r|_{L^2}^2 &\leq \l| m(0)\r |_{L^2}^2 - \alpha \, \int_{0}^{t} \l| \nabla m(s) \r|_{L^2}^2\, ds + C\int_{0}^{t} \l| m(s) \r|_{L^2}^2  \l[ \l|m_2(s)\r|_{H^1}^4  + \l| m_1 \r|_{H^{1}}^2 + \l| m_2 \r|_{H^{1}}^2\r]\, ds \\
				& + \frac{\varepsilon}{2}\int_{0}^{t} \l|\nabla m(s)\r|_{L^2}^2 \, ds + C \int_{0}^{t} \l| m(s) \r|_{L^2}^2  \l[ \l|m_2(s)\r|_{H^1}^4  + \l| m_1 \r|_{H^{1}}^2 + \l| m_2 \r|_{H^{1}}^2\r]\, ds + \frac{\varepsilon}{2}\int_{0}^{t} \l|\nabla m(s)\r|_{L^2}^2 \, ds \\
				& + \int_{0}^{t} \l(C \l| u(s) \r|_{L^2} + \frac{C(\varepsilon)}{2} \l| u(s) \r|_{L^2}^2 \r) \l| m(s) \r|_{L^2}^2  \, ds + \frac{\varepsilon}{2} \int_{0}^{t}  \l| \nabla m(s) \r|_{L^2}^2  \, ds \\
				& + \frac{\varepsilon}{2} + C \int_{0}^{t} \l|m(s)\r|_{L^2}^2\, ds + C \int_{0}^{t} \l\langle G(m_1) - G(m_2) , m(s) \r\rangle_{L^2} \, dW(s) \\
				& +  C(\varepsilon) \int_{0}^{t} \l|m(s)\r|_{L^2}^2 \l[ \l|\nabla m_1(s)\r|_{L^2}^2 + \l|\nabla m_1(s)\r|_{L^2}^4 + \l|\nabla m_2(s)\r|_{L^2}^2 + \l|\nabla m_2(s)\r|_{L^2}^4 \r] \, ds \\
				& + \frac{5\varepsilon}{2} \int_{0}^{t} \l|\nabla m(s)\r|_{L^2}^2 \, ds \\
				& = \l( - \alpha \, + 4\varepsilon \r)\int_{0}^{t} \l| \nabla m(s) \r|_{L^2}^2\, ds + C\int_{0}^{t} \l| m(s) \r|_{L^2}^2  \l[ \l|m_2(s)\r|_{H^1}^4  + \l| m_1(s) \r|_{H^1}^2 + \l| m_2(s) \r|_{H^1}^2 +  \r]\, ds \\
				& +  C \int_{0}^{t} \l|m(s)\r|_{L^2}^2 \l[ \l|\nabla m_1(s)\r|_{L^2}^2 + \l|\nabla m_1(s)\r|_{L^2}^4 + \l|\nabla m_2(s)\r|_{L^2}^2 + \l|\nabla m_2(s)\r|_{L^2}^4 \r] \, ds \\
				& + C \int_{0}^{t} \l\langle G(m_1) - G(m_2) , m(s) \r\rangle_{L^2} \, dW(s).
			\end{align*}
			\adda{What about the correction term that comes after the It\^o formula?? $\frac{1}{2}\int_{0}^{t} \l|G(m_1(s)) - G(m_2(s))\r|_{L^2}^2 \, ds.$}

		}

		\dela{Recall the equality in \eqref{equation Ito formula Pathwise uniqueness}.} Now we combine \eqref{equation Ito formula Pathwise uniqueness} and the above mentioned estimates. We collect the integrals with similar integrands. While doing this, we also combine the corresponding constants for simplifying the presentation. Thus there exists a constant $C>0$ such that

		\begin{align*}
			\l|m(t)\r|_{L^2}^2 + \l(\alpha \, - 4\varepsilon\r) \int_{0}^{t} \l|\nabla m(s)\r|_{L^2}^2 \, ds \leq& \l|m(0)\r|_{L^2}^2 \\
			& + \int_{0}^{t} \l|m(s)\r|_{L^2}^2 \bigg[ C + C\bigg( \l|\nabla m_1(s)\r|_{L^2}^2 + \l|\nabla m_1(s)\r|_{L^2}^4 \\
			&+ \l|\nabla m_2(s)\r|_{L^2}^2 + \l|\nabla m_2(s)\r|_{L^2}^4  \bigg)
			+ \l|u(s)\r|_{L^2} + \l|u(s)\r|_{L^2}^2 \bigg] \, ds \\
			& + \int_{0}^{t} \l\langle G\big(m_1(s)\big) - G\big(m_2(s)\big) , m(s) \r\rangle_{L^2} \, dW(s).
		\end{align*}

		We choose $\varepsilon > 0$ such that $\l( \alpha \, - 4\varepsilon \r) < 0$.

		We recall that the processes $m_1$ and $m_2$ have the same initial condition $m_0$. Hence $\l|m(0)\r|_{L^2} = 0$. Also by the choice of $\varepsilon $, the term $\l(\alpha \, - 4\varepsilon\r) \int_{0}^{t} \l|\nabla m(s)\r|_{L^2}^2 \, ds$ is non-negative.
		
		Let $C>0$ be a constant. For $t\in[0,T]$, let
		\begin{align}
			\Phi_C(t) = C + C\l( \l|\nabla m_1(s)\r|_{L^2}^2 + \l|\nabla m_1(s)\r|_{L^2}^4 + \l|\nabla m_2(s)\r|_{L^2}^2 + \l|\nabla m_2(s)\r|_{L^2}^4  \r) + \l|u(s)\r|_{L^2}^2.
		\end{align}
		Hence
		\begin{align}
			\l|m(t)\r|_{L^2}^2 &\leq \int_{0}^{t} \Phi_C(t) \l|m(s)\r|_{L^2}^2 \, ds + \int_{0}^{t} \l\langle G\big(m_1(s)\big) - G\big(m_2(s)\big) , m(s) \r\rangle_{L^2} \, dW(s).
		\end{align}

		The application of the It\^o formula gives
		\begin{equation}\label{pathwise uniqueness exp intermediate inequality}
			\l|m(t)\r|_{L^2}^2e^{-\int_{0}^{t}\Phi_C(s) \, ds} \leq \int_{0}^{t} e^{-\int_{0}^{s}\Phi_C(r) \, dr} \l\langle G\big(m_1(s)\big) - G\big(m_2(s)\big) , m(s) \r\rangle_{L^2} \, dW(s).
		\end{equation}
		Some details of this calculation are given in the Appendix \ref{Section Proof for pathwise uniqueness intermediate inequality}. A similar idea has been used in \cite{ZB+BG+TJ_LargeDeviations_LLGE,Schmalfuss_Qualitative_Properties_SNE}.

		By the definition of $\Phi_C$, $\Phi_C(t) \geq 0$ for each $t\in [0,T]$ $\mathbb{P}-$a.s. and the bounds established in Theorem \ref{Theorem Existence of a weak solution} imply that for any $t\in[0,T]$,
		\begin{equation}
			\int_{0}^{t} \Phi_C(s) \, ds < \infty,\ \mathbb{P} -\text{a.s.}
		\end{equation}
		Hence $\mathbb{P}-$a.s.,
		\begin{equation}
			e^{-\int_{0}^{t}\Phi_C(s) \, ds} \leq 1.
		\end{equation}
		
		The mapping $G$ is Lipschitz on balls. The processes $m_1, m_2$ satisfy the constraint condition \eqref{eqn-constraint condition}, and hence are uniformly bounded. Hence the processes $m$ is also uniformly bounded. This implies that the stochastic integral on the right hand side of the inequality \eqref{pathwise uniqueness exp intermediate inequality} is a martingale. Thus taking the expectation on both the sides of the inequality \eqref{pathwise uniqueness exp intermediate inequality}, we get
		\begin{align}
			\mathbb{E} \l|m(t)\r|_{L^2}^2e^{-\int_{0}^{t}\Phi_C(s) \, ds} \leq \mathbb{E} \int_{0}^{t} e^{-\int_{0}^{s}\Phi_C(r) \, dr} \l\langle G\big( m_1(s) \big) - G\big( m_2(s) \big) , m(s) \r\rangle_{L^2} \, dW(s) = 0.
		\end{align}
		Hence
		\begin{equation*}
			\mathbb{E} \l|m(s)\r|_{L^2}^2e^{-\int_{0}^{t}\Phi_C(s) \, ds} \leq 0.
		\end{equation*}
		But for each $t\in [0,T]$, $e^{-\int_{0}^{t}\Phi_C(s) \, ds} \geq 0$.\\
		Hence
		\begin{equation}
			\l|m(t)\r|_{L^2}^2 = 0\ \mathbb{P}-\text{a.s.}
		\end{equation}
		This concludes the proof of Theorem \ref{thm-pathwise uniqueness}.
		
	\end{proof}	
	We now define what we mean by a strong solution to the problem \eqref{problem considered}.
	\begin{definition}[Strong solution]\label{Definition of strong solution}
		The problem \eqref{problem considered} is said to admit a strong solution if the following holds: Let $\l( \Omega , \mathbb{F} , \mathcal{F} , \mathbb{P} \r)$ be a filtered probability space along with initial data $m_0$ and a control process $u$ on the space, satisfying Assumption \ref{assumption on u}. Then there exists an $\mathbb{F}$-adapted process $m$ on the said probability space such that the tuple $\l( \Omega , \mathbb{F} , \mathcal{F} , \mathbb{P} , W , m , u \r)$ is a weak martingale solution to the problem \eqref{problem considered} according to Definition \ref{Definition of Weak martingale solution}.
	\end{definition}
	
	The existence of a strong solution now follows as a consequence, which is stated in the following result.

	\begin{theorem}\label{thm-existence of a strong solution}
		The problem \eqref{problem considered} for a given initial data $m_0$ and a control process $u$, both satisfying the assumptions mentioned in the Theorem \ref{Theorem Existence of a weak solution}, has a pathwise unique strong solution as defined in Definition \ref{Definition of strong solution}. Moreover, the strong solution is unique in law.
	\end{theorem}
	\begin{proof}[Proof of Theorem \ref{thm-existence of a strong solution}]

		To prove the existence of a strong solution,  we apply Theorem 2 from \cite{Ondrejat_thesis}, which is a special case of Theorem 12.1 in the same reference.
		
		First, Theorem \ref{Theorem Existence of a weak solution} ensures that the problem \eqref{problem considered} admits a weak martingale solution for initial data and control process satisfying Assumption \ref{assumption on u}. Further, Theorem \ref{thm-pathwise uniqueness} ensures that the obtained solution is pathwise unique in the following sense. Let $\l( \Omega , \mathbb{F} , \mathcal{F} , \mathbb{P} , m_1 , u , W\r)$ and $\l( \Omega , \mathbb{F} , \mathcal{F} , \mathbb{P} , m_2 , u , W\r)$ be two weak martingale solutions corresponding to the same initial data $m_0$ and control $u$, on the same probability space. Let $m_1$ and $m_2$ satisfy the bounds in $(5)$ of Definition \ref{Definition of Weak martingale solution}. Then for each $t\in[0,t]$, we have $m_1(t) = m_2(t),\ \mathbb{P}-a.s.$.
		
		Let $C_{0}([0,T];\mathbb{R})$ denote the space
		\begin{equation*}
			\l\{ v \in C([0,T];\mathbb{R}) : v(0) = 0 \r\}.    
		\end{equation*}
		By part $(3)$ of Theorem 12.1, Theorem 13.2 and Lemma E, \cite{Ondrejat_thesis}, there exists a Borel measurable map 
		\begin{equation*}
			J : C_{0}([0,T];\mathbb{R}) \to C([0,T];L^2) \cap L^2(0,T;H^1)
		\end{equation*}
		such that the following holds. Let $\l( \Omega, \mathcal{F}, \mathbb{F}, \mathbb{P} \r)$ be a given filtered probability space along with a control process $u$, all satisfying Assumption \ref{assumption on u}. Let $W = \l(W(t)\r)_{t\in[0,T]}$ be an arbitrary real valued Wiener process on the said space. Let $m = J \circ W$. That is,
		\begin{equation*}
			m : \Omega \ni \omega \mapsto J(W(\omega)) \in C([0,T];L^2) \cap L^2(0,T;H^1).
		\end{equation*}
		Then, the tuple $\l( \Omega, \mathcal{F}, \mathbb{F}, \mathbb{P} , W , m , u\r)$ is a weak martingale solution to the problem \eqref{problem considered} on the space $\l( \Omega, \mathcal{F}, \mathbb{F}, \mathbb{P} \r)$.

		Therefore, given a filtered probability space $\l( \Omega , \mathbb{F} , \mathcal{F} , \mathbb{P} \r)$ along with initial data $m_0$ and a control process $u$ on the space, satisfying Assumption \ref{assumption on u}, we have shown that there exists a $\mathbb{F}$-adapted process $m$ such that the tuple $\l( \Omega , \mathbb{F} , \mathcal{F} , \mathbb{P} , W , m , u \r)$ is a weak martingale solution to the problem \eqref{problem considered}, thus showing the existence of a strong solution according to Definition \ref{Definition of strong solution}.

	\end{proof}

	\section{Further regularity: Proof of Theorem \ref{Theorem Further regularity}}\label{Section Further regularity}
	
	So far we have shown that there exists a strong solution to the problem \eqref{problem considered} with the initial condition and the given control satisfying the assumptions given in Theorem \ref{Theorem Existence of a weak solution}. This section is dedicated to proving further regularity for the above mentioned strong solution.


	Recall that by definition 
	\begin{equation*}
		Av = - \Delta v \ \text{for}\  v\in D(A),
	\end{equation*}
	and
	\begin{equation*}
		D(A) = \l\{ v\in H^2 : \frac{\partial v}{\partial \nu} = 0\ \text{on}\ \partial \mathcal{O} \r\},
	\end{equation*}
	where $\nu$ denotes the outward pointing normal vector and $\partial \mathcal{O}$ denotes the boundary of $\mathcal{O}$. In other words, the domain of $A$ is the subspace of elements of $H^2$ that satisfy the Neumann boundary condition.
	
	We also recall that $$A_1 = I_{L^2} + A.$$
	Here $I_{L^2}$ denotes the identity operator on the space $L^2$. Thus showing the bound for $\Delta m$ should be enough since $m$ is already bounded in $L^2$.
	
	
	
	The existence of the process $m$ is guaranteed by Theorem \ref{thm-existence of a strong solution}. What remains to show is that $m$ satisfies the inequality \eqref{Further regularity}.

	Let $\{ e^{-tA}\}_{t\in[0,T]}$ denote the semigroup generated by the operator $A$. The solution $m$ to the problem \eqref{problem considered} can be written in mild form, see for example, Section 6 in \cite{Prato+Zabczyk}, or the proof of first part of Theorem 9.15 in \cite{Peszat+Zabczyk}, as
	
	\begin{align}\label{Mild formulation 2}
		\nonumber m(t)  =& e^{-\alpha \, tA}m_0 + \alpha \, \int_{0}^{t} e^{-\alpha(t-s)A} (|\nabla m(s)|_{\mathbb{R}^3}^2)m(s)\, ds + \int_{0}^{t} e^{-\alpha(t-s)A} \left( m(s) \times \Delta m(s) \right)\, ds\\
		\nonumber & - \alpha \,  \int_{0}^{t} e^{-\alpha(t-s)A} \left[  m(s) \times \big( m(s) \times u(s) \big)  \right] \, ds  + \int_{0}^{t} e^{-\alpha(t-s)A} \bigl[ \big( m(s) \times u(s) \big) \bigr]\, ds \\
		\nonumber & + \alpha \, \int_{0}^{t}e^{-\alpha(t-s)A} \big(m(s)\times (m(s) \times h) \big) \, dW(s) + \int_{0}^{t} e^{-\alpha(t-s)A} (m(s) \times h)\, dW(s) \\
		& + \frac{1}{2} \int_{0}^{t} e^{-\alpha(t-s)A} \l[DG\big(m(s)\big)\r]G\big(m(s)\big) \, ds.
	\end{align}
	
	%
	
	\dela{The idea for this is to show that a weak solution to the problem \eqref{problem considered} is also a mild solution. Here that argument is not required since we already have the existence of a strong solution and a strong solution is also a mild solution.\adda{  See \cite{Peszat+Zabczyk}, Theorem 9.15, page number 151 (First part of the theorem).}For simplicity, we set $\alpha \, = 1$ in the proof.}
	
	\textbf{Idea of the proof of \eqref{Further regularity}:} The proof will primarily consist of two steps. Step 1 shows the bound on the first term in the inequality \eqref{Further regularity}. We consider the above mentioned mild formulation \eqref{Mild formulation 2}. Instead of showing the bound directly on the process $m$, the bound will be shown on each term on the right hand side of \eqref{Mild formulation 2}.
	Step 2 will use the bound so obtained to show a bound on the second term in the inequality \eqref{Further regularity}.
	
	The following properties of the operators $A,A_1$ will be used throughout the proof.
	
	\begin{enumerate}
		\item $e^{-tA}$ is ultracontractive, see Section 7.2.7 in \cite{Arendt}.
		That is, for $1\leq p\leq q\leq \infty$, there exists a constant $C>0$ such that
		\begin{equation}\label{ultracontractive property}
			\l| e^{-tA} f \r|_{L^q} \leq \frac{C}{t^{\frac{1}{2}\l(\frac{1}{p} - \frac{1}{q}\r)}  } \l| f \r|_{L^p}\ \text{for}\ f\in L^p,\ t>0.
		\end{equation}
		
		\item $A$ has the maximal regularity property. Let $f\in L^2\l( 0,T;L^2 \r)$ and
		\begin{align*}
			v(t) = \int_{0}^{t}e^{-(t-s)A}f(s)\, ds,\,\quad t\in[0,T].
		\end{align*}
		Then we have
		\begin{align}\label{maximal regularity property}
			\int_{0}^{t}\l| A v(t) \r|_{L^2}^2\, dt \leq C \int_{0}^{t}\l| f(t) \r|_{L^2}^2\, dt.
		\end{align}
		
		\item The operator $A_1 = I + A$ generates a semigroup (denoted by $e^{-tA_1}$), see Theorem 1.1 in \cite{Pazy}.
		
		Thus using \eqref{ultracontractive property} for $f\in L^p$ and $t>0$, we get
		\begin{align}\label{Norm of e^-tA_1}
			\nonumber \l| e^{-tA_1} f \r|_{L^q} &= \l| e^{-tA} e^{-tI} f \r|_{L^q} \\
			\nonumber& \leq C\l| e^{-tA} f \r|_{L^q} \\
			& \leq \frac{C}{t^{\frac{1}{2}\l(\frac{1}{p} - \frac{1}{q}\r)}  } \l| f \r|_{L^p}.
		\end{align}
		
		\item The operators $A^{\delta}e^{-tA}$ and $A_1^{\delta}e^{-tA_1}$ are bounded on $L^2$, see Theorem 6.13 in \cite{Pazy}. Moreover, there exists a constant $C>0$ such that
		\begin{align}\label{Norm A delta e to the power A bound}
			\l| A^{\delta} e^{-tA} \r| \leq \frac{C}{t^{\delta}}
		\end{align}
		and
		\begin{align}\label{Norm A1 delta e to the power A1 bound}
			\l| A_1^{\delta} e^{-tA_1} \r| \leq \frac{C}{t^{\delta}}.
		\end{align}
		Here $\l| A^{\delta} e^{-tA} \r|$ and $\l| A_1^{\delta} e^{-tA_1} \r|$ denote the operator norms of $A^{\delta} e^{-tA}$ and $ A_1^{\delta} e^{-tA_1}$ respectively.
	\end{enumerate}

	\textbf{Step 1:} We show that
	\begin{align}\label{Step 1 bound}
		\mathbb{E}  \int_{0}^{T} | \nabla m(t) |_{L^4}^4\, dt < \infty.
	\end{align}
	The following Sobolev embedding holds for $\delta \in \l( \frac{5}{8} , \frac{3}{4} \r)$ , see Lemma \ref{Lemma Sobolev embedding ref Wong Zakai}.
	
	
	\begin{align*}
		X^{\delta} \hookrightarrow W^{1,4}.
	\end{align*}
	\dela{
		Hence there exists some constant $C>0$ such that
		\begin{align*}
			\l| . \r|_{W^{1,4}} \leq C\l| . \r|_{X^{\delta}}.
		\end{align*}
	}
	It is thus sufficient to prove the following stronger estimate to show \eqref{Step 1 bound}.
	
	\begin{align}\label{Step 1 stronger estimate}
		\mathbb{E} \int_{0}^{T} \l| A_1^{\delta} m(t) \r|_{L^2}^4\, dt < \infty.
	\end{align}
	
	We recall that for $v\in X^{\delta} = D(A_1^{\delta})$,
	\begin{equation*}
		\l|v\r|_{X^{\delta}} = \l| A_1^{\delta} v \r|_{L^2}.
	\end{equation*}
	

	The step will be further divided into 3 sub steps. The first dealing with the first two terms appearing in the equality \eqref{Mild formulation 2}. In the second sub step, we consider a function $f$ satisfying certain bounds and show the bounds for this $f$. The idea is that the remaining terms in \eqref{Mild formulation 2} (except the terms with the stochastic integral) fall into this category and hence it suffices to show the calculations for $f$. The third sub step deals with the terms that contain the stochastic integral.
	\textbf{Sub step 1:}\\
	\dela{
		Consider the first term $e^{-\alpha \, t A}m_0$.
		\begin{align*}
			|A_1^{\delta} e^{-\alpha \, t A}m_0|_{L^2}^4 &= |\l( I + A \r)^{\delta} e^{\alpha \, t I_{L^2}}e^{-\alpha \, t \l( A + I \r)}m_0|_{L^2}^4 \ (\text{Since}\ A_1 = I_{L^2} + A)\\
			& \leq Ce^t |A_1^{\delta} e^{-\alpha \, t A_1}m_0|_{L^2}^4 \ (\text{Since}\  \l|e^{\alpha \, t I_{L^2}}\r| \leq C e^{\alpha \, t})\\
			& \leq Ce^T |A_1^{\delta -\frac{1}{2}} e^{-\alpha \, t A_1} A^{\frac{1}{2}} m_0|_{L^2}^4 \ (\text{Since}\  \delta = \delta - \frac{1}{2} + \frac{1}{2})\\
			& \leq \frac{C}{t^{4(\frac{2\delta -1}{2})}}\l| A^{\frac{1}{2}} m_0 \r|_{L^2}^4\ (\text{By}\ \eqref{Norm A1 delta e to the power A1 bound})\\
			& \leq \frac{C}{t^{4\delta - 2}}\l|m_0\r|_{H^1}^4.\ (\text{Since}\ \l|A_1^{\frac{1}{2}}\centerdot\r|_{L^2} = \l|\centerdot\r|_{H^1})
		\end{align*}
	}
	Consider the first term $e^{- t A}m_0$.
	\begin{align*}
		|A_1^{\delta} e^{- t A}m_0|_{L^2}^4 &= |\l( I + A \r)^{\delta} e^{ t I_{L^2}}e^{- t \l( A + I \r)}m_0|_{L^2}^4 \ (\text{Since}\ A_1 = I_{L^2} + A)\\
		& \leq Ce^t |A_1^{\delta} e^{- t A_1}m_0|_{L^2}^4 \ (\text{Since}\  \l|e^{ t I_{L^2}}\r| \leq C e^{ t})\\
		& \leq Ce^T |A_1^{\delta -\frac{1}{2}} e^{- t A_1} A_1^{\frac{1}{2}} m_0|_{L^2}^4 \ (\text{Since}\  \delta = \delta - \frac{1}{2} + \frac{1}{2})\\
		& \leq \frac{C}{t^{4(\frac{2\delta -1}{2})}}\l| A_1^{\frac{1}{2}} m_0 \r|_{L^2}^4\ (\text{By}\ \eqref{Norm A1 delta e to the power A1 bound})\\
		& \leq \frac{C}{t^{4\delta - 2}}\l|m_0\r|_{H^1}^4.\ (\text{Since}\ \l|A_1^{\frac{1}{2}}\cdot\r|_{L^2} = \l|\cdot\r|_{H^1})
	\end{align*}	
	Hence
	\begin{align*}
		\int_{0}^{T} |A_1^{\delta} e^{- t A}m_0|_{L^2}^4 \, dt \leq \l|m_0\r|_{H^1}^4 \int_{0}^{T}\frac{C}{t^{4\delta - 2}} \, dt.
	\end{align*}	
	Since $\delta < \frac{3}{4}$, the integral on the right hand side of the above inequality is finite. Hence there exists a constant $C>0$ such that
	\begin{equation}
		\int_{0}^{T} |A_1^{\delta} e^{- t A}m_0|_{L^2}^4 \, dt \leq C.
	\end{equation}
	And hence
	\begin{equation}
		\mathbb{E}\int_{0}^{T} |A_1^{\delta} e^{- t A}m_0|_{L^2}^4 \, dt \leq C.
	\end{equation}
	For the second term, first we observe the following. Let $t\in [0,T]$.
	\begin{align*}
		\int_{\mathcal{O}} \l|\nabla m(t,x)\r|_{\mathbb{R}^3}^2\l|m(t,x)\r|_{\mathbb{R}^3}\, dx = \int_{\mathcal{O}} \l|\nabla m(t,x)\r|_{\mathbb{R}^3}^2 \, dx \leq \l| m(t) \r|_{H^1}^2.
	\end{align*}
	Hence
	\begin{align*}
		\sup_{t\in[0,T]} \int_{\mathcal{O}} \l|\nabla m(t,x)\r|_{\mathbb{R}^3}^2\l|m(t,x)\r|_{\mathbb{R}^3}\, dx \leq \sup_{t\in[0,T]} \l| m(t) \r|_{H^1}^2.
	\end{align*}
	For simplicity of notation, let $g(s) = \l|\nabla m(s)\r|_{\mathbb{R}}^2 m(s)$.
	\begin{align*}
		\l| A_1^{\delta} e^{-(t-s)A} g (s) \r|_{L^2} &\leq C \l|A_1^{\delta} e^{-(t-s)A_1}g(s)\r|_{L^2} \\
		& = C \l|A_1^{\delta} e^{-\frac{(t-s)}{2}A_1}e^{-\frac{(t-s)}{2}A_1} g(s)\r|_{L^2} \\
		& \leq C \l| A_1^{\delta} e^{-\frac{(t-s)}{2}A_1} \r| \l| e^{-\frac{(t-s)}{2}A_1} g(s) \r|_{L^2} \\
		& \leq C \l| A_1^{\delta} e^{-\frac{(t-s)}{2}A_1} \r| \frac{1}{(t-s)^{\frac{1}{4}}}\l| g(s) \r|_{L^1} \ (\text{By} \ \eqref{Norm of e^-tA_1}\ \text{with}\ p = 1, q = 2)\\
		& \leq \frac{C}{\l(t-s\r)^{\delta + \frac{1}{4}}}\l|g(s)\r|_{L^1}\ \text{By}\ \eqref{Norm A1 delta e to the power A1 bound} \\
		& \leq \frac{C}{\l(t-s\r)^{\delta + \frac{1}{4}}}\l|m(s)\r|_{H^1}^2.
	\end{align*}
	Therefore,
	
	\begin{align*}
		\int_{0}^{T}\l| \int_{0}^{t} A_1^{\delta} e^{-(t-s)A} g(s) \, ds \r|_{L^2}^4 \, dt &\leq C \sup_{t\in[0,T]}\l|m(s)\r|_{H^1}^8\int_{0}^{T}\int_{0}^{t}\l( \frac{1}{\l(t-s\r)^{\delta + \frac{1}{4}}} \, ds \r)^4 \, dt .
	\end{align*}
	Since $\delta < \frac{3}{4}$, that is $\delta + \frac{1}{4} < 1$, the integration on the right hand side is finite.\\
	Hence there exists a constant $ C>0 $ such that
	\begin{equation}
		\mathbb{E} \int_{0}^{T}\l| \int_{0}^{t} A_1^{\delta} e^{-(t-s)A} g(s) \, ds \r|_{L^2}^4 \, dt \leq C.
	\end{equation}
	\textbf{Sub step 2:}\\
	Consider a function $f\in L^4\l( \Omega; L^2\l( 0,T; L^2 \r) \r)$.\\
	There exists constants $C_1,C_2>0$ such that
	\begin{align*}
		\l| A_1^{\delta}e^{-(t-s)A} f(s) \r|_{L^2} &= \l| A_1^{\delta}e^{-(t-s)A_1} e^{(t-s) I_{L^2} } f(s) \r|_{L^2} \\
		& \leq \l| A_1^{\delta}e^{-(t-s)A_1} e^{(t-s) I_{L^2} } \r| \l| f(s) \r|_{L^2} \\
		& \leq C_1 \l| A_1^{\delta}e^{-(t-s)A_1} \r| \l| f(s) \r|_{L^2} \ (\text{Since}\ \l|e^{(t-s)I_{L^2}}\r| \leq C_1)\\
		& \leq \frac{C_2}{(t-s)^{\delta}} \l| f(s) \r|_{L^2}.\ (\text{By}\ \eqref{Norm A1 delta e to the power A1 bound})
	\end{align*}
	Therefore replacing the constants $C_1,C_2$ above by a suitable constant $C$, we get
	\begin{align*}
		\int_{0}^{T} \l( \int_{0}^{t} \l| A_1^{\delta}e^{-(t-s)A} f(s) \r|_{L^2}\, ds \r)^4 \, dt \leq C \int_{0}^{T} \l( \int_{0}^{t} \frac{1}{(t-s)^{\delta}} \l| f(s) \r|_{L^2}\, ds \r)^4 \, dt,
	\end{align*}
	Using Young's convolution inequality for $p=\frac{4}{3}$ and $q=2$, we get
	\begin{align*}
		\int_{0}^{T} \l( \int_{0}^{t} \frac{1}{(t-s)^{\delta}} \l| f(s) \r|_{L^2}\, ds \r)^4 \, dt \leq  \l( \int_{0}^{T} s^{-\frac{4\delta}{3}}\, ds \r)^{\l(\frac{3}{4}\r)\l(4\r)} \l( \int_{0}^{T} \l|f(s)\r|_{L^2}^2\, ds\r)^2.
	\end{align*}
	That $\delta<\frac{3}{4}$ implies $\frac{4\delta}{3}<1$. Hence the first integral on the right hand side of the above inequality is finite. Hence
	\begin{align*}
		\int_{0}^{T} \l( \int_{0}^{t} \frac{1}{(t-s)^{\delta}} \l| f(s) \r|_{L^2}\, ds \r)^4 \, dt \leq C \l( \int_{0}^{T} \l|f(s)\r|_{L^2}^2\, ds\r)^2.
	\end{align*}
	Therefore
	\begin{align*}
		\mathbb{E} \int_{0}^{T} \l( \int_{0}^{t} \frac{1}{(t-s)^{\delta}} \l| f(s) \r|_{L^2}\, ds \r)^4 \, dt  \leq C \mathbb{E} \l( \int_{0}^{T} \l|f(s)\r|_{L^2}^2\, ds\r)^2 < \infty.
	\end{align*}

	Now consider the remaining terms on the right hand side of the equality \eqref{Mild formulation 2}, except for the terms with the It\^o integral.
	\dela{ By Theorem \ref{Theorem Existence of a weak solution}, we can shown that each integrands (except the stochastic term) belong to the space $L^4(\Omega;L^2(0,T;L^4))$.} \\
	By Theorem \ref{Theorem Existence of a weak solution}, the solution $m$ takes values on the unit sphere in $\mathbb{R}^3$.
	By the bounds mentioned in Theorem \ref{Theorem Existence of a weak solution} and the Assumption \ref{assumption on u} on the control process $u$, we have
	\begin{align}
		m\times \Delta m \in L^4\l(\Omega ; L^2\l( 0,T; L^2\r)\r).
	\end{align}
	The constraint condition \eqref{eqn-constraint condition} implies that
	\begin{align}
		m \times \l(m\times \Delta m\r) \in L^4\l(\Omega ; L^2 \l(0,T;L^2\r)\r).
	\end{align}
	The assumption on $u$, viz. \ref{assumption on u} along with the constraint condition \eqref{eqn-constraint condition} and the assumption on the function $h$ implies that
	\begin{align}
		m\times u \in L^4\l(\Omega ; L^2\l( 0,T; L^2\r)\r),
	\end{align}
	
	\begin{align}
		m \times \l(m\times u\r) \in L^4\l(\Omega ; L^2 \l(0,T;L^2\r)\r),
	\end{align}
	and
	\begin{align}
		DG\l(m\r)\bigl(G(m)\bigr) \in L^4\l(\Omega ; L^2\l( 0,T; L^2\r)\r).
	\end{align}
	Note that the Assumption \ref{assumption on u} has been applied here for $p = 2$.\\
	Hence each of the integrands (except for the terms with the It\^o integral) takes values in\\
	$L^4\l( \Omega; L^2\l( 0,T; L^2 \r) \r)$. Hence by replacing $f$ in the above calculations by the integrands, one can show that each of the terms also satisfies the required bounds.\\
	\dela{By the bounds established in Theorem \ref{Theorem Existence of a weak solution}, each of the integrands (except for the terms with the It\^o integral) takes values in $L^4\l( \Omega; L^2\l( 0,T; L^2 \r) \r)$. Hence by replacing $f$ in the above calculations by the integrands, one can show that each of the terms also satisfies the required bounds.\\}
	\textbf{Sub step 3:}\\
	What remains now is the It\^o integral term.
	Recall that by Proposition \ref{prop-derivative} and the bound on the process $m$ in Theorem \ref{Theorem Existence of a weak solution},
	\begin{align}\label{eqn 1 Further regularity theorem}
		\mathbb{E}\int_{0}^{T}\l| G(m(t)) \r|_{H^1}^4\, dt \leq C \mathbb{E}\int_{0}^{T}\l| m(t) \r|_{H^1}^2 < \infty.
	\end{align}
	
	\begin{align*}
		\mathbb{E}\l| \int_{0}^{t} A_1^{\delta} e^{-(t-s)A_1}G(m(s))\,dW(s) \r|_{L^2}^4 &\leq C \mathbb{E} \l( \int_{0}^{t} \l| A_1^{\delta} e^{-(t-s)A_1}G(m(s)) \r|^2_{L^2}\, ds \r)^2 \\
		&\quad(\text{See Proposition 7.3 \cite{Prato+Zabczyk}}) \\
		& \leq C \mathbb{E} \l( \int_{0}^{t} \l| A_1^{\delta- \frac{1}{2}} e^{-(t-s)A_1} A_1^{\frac{1}{2}} G(m(s)) \r|^2_{L^2}\, ds \r)^2 \\
		&\quad(\text{Since}\ \delta = \delta - \frac{1}{2} + \frac{1}{2}) \\
		& \leq C \mathbb{E} \l( \int_{0}^{t} \frac{1}{(t-s)^{2\delta - 1}}\l| A_1^{\frac{1}{2}}G(m(s)) \r|^2_{L^2}\, ds \r)^2\ (\text{By}\ \eqref{Norm A1 delta e to the power A1 bound})\\
		& \leq C \mathbb{E} \l( \int_{0}^{t} \frac{1}{(t-s)^{2\delta - 1}}\l| G(m(s)) \r|^2_{H^1}\, ds \r)^2\\
		& \quad(\text{Since}\ \l|A_1^{\frac{1}{2}}\centerdot\r|_{L^2} = \l|\centerdot\r|_{H^1}) \\
		& \leq C \mathbb{E} \l( \int_{0}^{t} \frac{1}{(t-s)^{4\delta - 2}} \, ds\r) \mathbb{E} \l(\int_{0}^{t} \l| G(m(s)) \r|^4_{H^1}\, ds\r).\\
		&\quad\text{(Cauchy-Schwartz inequality)}
	\end{align*}
	
	Here $\delta < \frac{3}{4}$ implies that $4\delta - 2 < 1$. Hence the first integral is finite. The second integral is finite because of the inequality \eqref{eqn 1 Further regularity theorem}. Hence combining all the inequalities, the bound \eqref{Step 1 stronger estimate}, and hence \eqref{Step 1 bound} is shown.	
	\\
	\textbf{Step 2:}
	This step uses the following identity. Let $a,b\in\mathbb{R}^3$.
	\begin{align}\label{vector identity a times b plus a dot b}
		\l| a \times b \r|_{\mathbb{R}^3}^2 + \l| \l\langle a , b\r\rangle_{\mathbb{R}^3} \r| = \l|a\r|_{\mathbb{R}^3}^2 \l|b\r|_{\mathbb{R}^3}^2.
	\end{align}
	\textbf{Brief proof of the equality \eqref{vector identity a times b plus a dot b}:}
	\begin{align*}
		\l| a \times b \r|_{\mathbb{R}^3}^2 &= \l\langle a \times b , a \times b \r\rangle_{\mathbb{R}^3} 
		= \l\langle a  ,  b \times \l( a \times b \r) \r\rangle_{\mathbb{R}^3}.
	\end{align*}
	
	We expand the right hand side using the triple product formula and simplify to get the identity \eqref{vector identity a times b plus a dot b}.\\
	For Leb. a.a. $x\in\mathcal{O},t\in[0,T]$, the following equality holds $\mathbb{P}$-a.e.
	
	\begin{align*}
		\l| m(t,x) \times \Delta m(t,x) \r|_{\mathbb{R}^3}^2 + \l| \l\langle m(t,x) , \Delta m(t,x) \r\rangle_{\mathbb{R}^3} \r|^2 &= \l| m(t,x) \r|_{\mathbb{R}^3}^2 \l| \Delta m(t,x) \r|_{\mathbb{R}^3}^2 \\
		& = \l| \Delta m(t,x) \r|_{\mathbb{R}^3}^2.
	\end{align*}
	Hence to show the bound on the second term, it is sufficient to show the corresponding bound on the two terms on the left hand side of the above equality.
	
	For the second term,
	\begin{align*}
		\mathbb{E} \int_{0}^{T} \int_{\mathcal{O}} \l| \l\langle m(t,x) , \Delta m(t,x) \r\rangle_{\mathbb{R}^3} \r|^2\, ds\, dt = \mathbb{E} \int_{0}^{T} \int_{\mathcal{O}} \l| \nabla m(t,x) \r|_{\mathbb{R}^3}^4\, ds\, dt.
	\end{align*}
	The right hand side of the above equality is finite because of the bound \eqref{Step 1 bound} in Step 1.	
	This, along with the bound in Theorem \ref{Theorem Existence of a weak solution} (for the first term) concludes the proof of the bound on the second term.
	
	Hence the proof of Theorem \ref{Further regularity} is complete.

	\begin{lemma}\label{Lemma continuous in time with values in H1.}
		The process $m$ lies in the space $C\l(\l[0,T\r] ; H^1\r)$ $\mathbb{P}-$a.s..
	\end{lemma}
	
	We postpone the proof of this lemma to Appendix \ref{Section Proof of Lemma Lemma continuous in time with values in H1}.

	\dela{
		\adda{This lemma is not directly needed. Try to write the proof of the previous theorem in a different way. First prove one bound and then show the other bound.}
		
		\adda{	\begin{lemma}\label{Further regularity intermediate Lemma}
				Let there exist a constant $C_1>0$ such that
				\begin{equation*}
					\mathbb{E}  \int_{0}^{T} |\nabla  m(s)|_{L^4}^4\, ds \leq C_1.
				\end{equation*}
				Then there exists another constant $C_2>0$ such that
				\begin{equation}
					\mathbb{E} \int_{0}^{T} |A_1 m(s)|_{L^2}^2\, ds \leq C_2.
				\end{equation}
				\adda{\begin{equation*}
						\mathbb{E} \int_{0}^{T} |\Delta m(s)|_{L^2}^2\, ds \leq C_2.
					\end{equation*}
					In the proof state that proving this is sufficient to prove a bound on $A_1 m$.}
			\end{lemma}
			\begin{proof}[Proof of Lemma \ref{Further regularity intermediate Lemma}]
				
				By the hypothesis, there exists a constant $C>0$ such that
				\begin{align*}
					\mathbb{E} \int_{0}^{T}\left| \left|\nabla m(t) \right|_{\mathbb{R}^3}^2 m(t) \right| \leq C.
				\end{align*}
				The idea of the proof is to use the It\^o Lemma for the function $$v\mapsto \frac{1}{2}|\nabla v|_{L^2}^2,\ \text{for}\ v\in H^1.$$, as done previously. The difference between the previous and the following calculations is that the negative coefficient of $(\Delta m(t))$ enables us to take the term on the left hand side of the inequality and thus bound it. Note that apart from the $\Delta m(t)$ term, another term, viz. $|\nabla m(t)|_{\mathbb{R}^3}m(t)$ is also present, which was not explicitly present in the previous calculations.
				
				The term $\Delta m(t)$ is understood as an element of $(H^1)^{\prime}$. Thus, as done in the Theorem \ref{thm-pathwise uniqueness}, we apply the It\^o Lemma after applying the projection operator $\mathscr{P}_n$ for $n\in\mathbb{N}$. Then we obtain a bound on the right hand side that is independent of $n$.
				
				First we note the following calculations for $n\in\mathbb{N}$ and $t\in[0,T]$.
				\begin{align*}
					\int_{0}^{t}\l\langle -\Delta \mathscr{P}_n m(s) , \Delta \mathscr{P}_n m(s) \r\rangle_{L^2}\, ds = - \int_{0}^{t} |\Delta \mathscr{P}_n m(s)|_{L^2}^2\,ds
				\end{align*}
				Let $z_i$ denote the integrand on the right hand side of \adda{Equation after using the fact that $|m|_{\mathbb{R}^3} = 1$ and the triple product formula}. The above equality deals with one term. This term will be used to absorb the other $\int_{0}^{t} \Delta \mathscr{P}_n m(s)\,ds$ terms arising from other calculations.
				
				By the Young's inequality, for some $\varepsilon > 0 $,
				\begin{align*}
					\int_{0}^{t} |\l\langle \Delta \mathscr{P}_n m(s) , z_i(s) \r\rangle_{L^2}|\,ds \leq \frac{\varepsilon}{2}\int_{0}^{t}|\Delta \mathscr{P}_n m(s)|_{L^2}^2\, ds + \frac{C(\varepsilon)}{2}\int_{0}^{t} |z_i(s)|_{L^2}^2\, ds.
				\end{align*}
				Hence
				\begin{align*}
					\mathbb{E} \int_{0}^{t} |\l\langle \Delta \mathscr{P}_n m(s) , z_i(s) \r\rangle_{L^2}|\,ds \leq \frac{\varepsilon}{2} \mathbb{E} \int_{0}^{t}|\Delta \mathscr{P}_n m(s)|_{L^2}^2\, ds + \frac{C(\varepsilon)}{2} \mathbb{E} \int_{0}^{t} |z_i(s)|_{L^2}^2\, ds.
				\end{align*}

				The second term on the right hand side of the above inequality can be bounded by a constant (after taking the expectation.) The term that is to be dealt with differently is $\int_{0}^{T}\left| \left|\nabla m(t) \right|_{\mathbb{R}^3}^2 m(t) \right|$. But the hypothesis mentioned in the statement takes care of this term.
			\end{proof}
			Hence using the above lemma, to prove the Theorem \ref{Theorem Further regularity}, it is sufficient to show the bound only on the first term in \eqref{Further regularity}.
			\adda{Mention Lemmata corresponding to other way of writing $m\times (m\times \Delta m)$, existence of a mild solution, other bound Lemmata required, etc.}
			\begin{proof}[Proof of Theorem \ref{Theorem Further regularity}]
				Proof of existence of a unique strong solution is proved in Corollary \ref{thm-existence of a strong solution}. We now show the bound \eqref{Further regularity}.
	\end{proof}}}

	\section{Proof of Theorem \ref{Theorem existence of optimal control} : Optimal control}\label{Section Optimal control}
	The objective of this section is to show that there exists an optimal control to the problem \eqref{problem considered}, with an appropriate\dela{ cost functional and } admissibility criterion. We fix a probability space $(\Omega , \mathcal{F} , \mathbb{P})$ as in Section \ref{Section Statements of the main Results}.
	\begin{proof}[Outline of the section:]
		We start by giving an equivalent equation \eqref{problem considered for optimal control part} to equation \eqref{problem considered}. We follow it up with the definition of a \textit{strong martingale solution} to the problem in Definition \ref{Definition Strong martingale solution}. Assumption \ref{Assumption admisibility criterion} outlines the assumption that is required on the control processes. The class  $\mathcal{U}_{ad}(m_0,T)$ of admissible solutions is then defined. This is followed by a proof for Theorem \ref{Theorem existence of optimal control}.
	\end{proof}
	\dela{\\
		We recall the optimal control problem here for the reader's convenience.\\ We have $\bar{m}\in L^2(\Omega ; L^2(0,T; H^1))$, a given desired state with values on the unit sphere $\mathcal{S}^2$. The terminal cost given by the function $\Psi$, that is assumed to be continuous on $L^2$. For a fixed $0 < T < \infty$, our aim is to minimize the cost functional
		
		\begin{align}
			J(\pi) = \mathbb{E} \l[ \int_{0}^{T} \l( \l|m(t) - \bar{m}(t)\r|_{H^1}^2 + \l| u(t) \r|_{L^2}^2 \r) + \Psi(m(T))\r]
		\end{align}
		over the space of admissible solutions $\pi = \l(\Omega , \mathcal{F} , \mathbb{P} , W , m , u \r)$ to the problem \eqref{problem considered}. 	
	}
	For the remainder of this section, we will consider the following equation. For $t\in[0,T]$
	\begin{align}\label{problem considered for optimal control part}
		\nonumber	m(t) =& \int_{0}^{t}m(s) \times \Delta m(s)\, ds - \alpha \, \int_{0}^{t} m(s)\times(m(s)\times u(s))\, ds
		\\\nonumber
		&+ \alpha \, \int_{0}^{t}  \Delta m(s) \, ds   + \alpha \, \int_{0}^{t}  |\nabla m(s)|_{\mathbb{R}^3}^2 m(s) \, ds \\
		& + \int_{0}^{t} m(s)\times u(s)\, ds  + \frac{1}{2}\int_{0}^{t} \l[DG\l(m(s)\r)\r]\l(G(m\l(s\r))\r)\, ds + \int_{0}^{t} G(m(t))\, dW(t),\  \mathbb{P}-a.s.
	\end{align}
	Recall that by Corollary \ref{Corollary m times m times Delta m equals Delta m plus gradient m squared m}, the equation \eqref{problem considered} and the above equation \eqref{problem considered for optimal control part} are equivalent in $(H^1)^{\prime}$, since $m$ satisfies the constraint condition.

	\begin{definition}[Strong martingale solution]\label{Definition Strong martingale solution}
		Let the initial data $m_0$, the function $h$ and time $T$ be fixed. A strong martingale solution of \eqref{problem considered for optimal control part} is a tuple
		\begin{equation*}
			\pi = (\Omega, \mathcal{F}, \mathbb{F}, \mathbb{P}, W, m, u)
		\end{equation*} such that $\pi$ is a weak martingale solution as in Definition \ref{Definition of Weak martingale solution} and  the process $m$ satisfies the additional  regularity property  \eqref{Further regularity},  i.e. 
		\begin{align}\label{Further regularity bound definition of strong martingale solution}
			\mathbb{E} \left( \int_{0}^{T} | \nabla m(t) |_{L^4}^4\, dt  + \int_{0}^{T} |A_1 m(t)|_{L^2}^2\, dt\right) < \infty.
		\end{align}
	\end{definition}
	
	\begin{remark}\label{remark equivalence for strong martingale solution 1}
		A weak martingale solution is defined for the problem \eqref{problem considered}. By Corollary \ref{Corollary m times m times Delta m equals Delta m plus gradient m squared m} the equations \eqref{problem considered} and \eqref{problem considered for optimal control part} are equivalent in $(H^1)^{\prime}$. Hence the above definition makes sense.\\
		Hence Theorem \ref{thm-existence of a strong solution} implies that the problem \eqref{problem considered for optimal control part}, with the initial data $m_0$ has a strong solution corresponding to any control process satisfying \eqref{assumption on u}.
	\end{remark}

	\dela{\textbf{Regarding the Admissibility criterion:
		} An extra assumption ($p=4$) is required for the bound \eqref{Further regularity}. But this is not required for the optimal control part if the regularity in \eqref{Further regularity} is not used. One way to show the existence of an optimal control is to try and show direct convergence (possibly along a subsequence) of a minimizing sequence. The obtained limit satisfies the equation (limit of right hand sides of the corresponding equations) and is also a solution. Show the convergence in stronger sense, say in $C(0,T;H^1)$ and $L^2(0,T; D(A))$. Hence by uniqueness of solution, the obtained limit is a strong solution of \eqref{problem considered} and also satisfies the required admissibility criterion. The limit is a solution corresponding to the control process obtained as a limit of the control processes in the minimizing sequence.
		
		Another way is to use the Relaxed control setup. From there show the existence of weak martingale solution that minimizes the cost functional. The existence of a strong solution on a given probability space is known. Also, if the solution process have the same initial data and the control processes have the same laws then the solution will also have the same laws. Hence the solution corresponding to the control process having the same law as the obtained process also minimizes the cost functional. Since that is a strong solution, we have the existence of an optimal control. Can we proceed from here to show that the optimal control is unique? This does not seem that trivial because there can be another tuple of a control process and its corresponding solution that minimizes the cost functional.
		
	}
	
	\dela{The admissibility criterion will also change now. The solutions in this class need not be strong solutions. It is sufficient to have Strong martingale solutions as defined previously, along with the stronger assumptions required for the control process ($p$th moments assumed to be finite.)}
	
	\begin{assumption}[Admissibility criterion for the control process]\label{Assumption admisibility criterion}
		We say that a given control process $u$ satisfies the admissibility criterion if for $p\geq 1$ and a given constant $K_p > 0$,
		\begin{equation}\label{admissibility criterion}
			\mathbb{E}\l(\int_{0}^{T} \l| u(t) \r|_{L^2}^{2} \, dt\r)^p \leq K_p.
		\end{equation}	
		\dela{Which $p$ you need?}
		In particular, we assume \eqref{admissibility criterion} for $p=4$.
	\end{assumption}
	We now describe the class of admissible solutions over which the cost function will be minimized.
	Let us fix the law of the initial data $m_0$ such that it\dela{$m_0$} satisfies the assumptions in Theorem \ref{Theorem Existence of a weak solution}. Also fix the function $h\in H^1$. Fix $T<\infty$. Consider a tuple 	$ \pi = (\Omega, \mathcal{F}, \mathbb{F}, \mathbb{P}, W, m, u) $ which is a strong martingale solution to \eqref{problem considered for optimal control part} as defined in Definition \ref{Definition Strong martingale solution}. Let the control process $u$ also satisfy the Assumption \ref{Assumption admisibility criterion} for $p=4$. Hence the process $m$ satisfies the bounds mentioned in Theorem \ref{Theorem Existence of a weak solution}. Such a tuple $\pi$ will be called an \textit{admissible solution} and the space of all such admissible solutions will be denoted by $\mathcal{U}_{ad}(m_0,T)$.\\
	\begin{remark}\label{remark equivalence for strong martingale solution 2} Even if the tuples are strong martingale solutions, the equations still make sense in $(H^1)^{\prime}$, (and even in $L^2$, see Corollary \ref{Strong form of weak martingale solution}) due to the regularity proved in Theorem \ref{Theorem Further regularity}.
	\end{remark}
	
	\dela{	
		\textbf{Note:} Although the Assumption \ref{Assumption admisibility criterion} requires more regularity on the control process $u$, we do not incorporate that in the cost functional. If included, the proof for that will still remain the same as it uses the Fatou Lemma and lower semicontinuity etc.
	}

	We recall the optimal control problem here for the reader's convenience.\\
	The cost functional is defined as follows. Let $\pi = \l(\Omega , \mathcal{F} , \mathbb{F} , \mathbb{P} , W , m , u\r)\in \mathcal{U}_{ad}(m_0,T)$. Assume that the terminal cost $\Psi$ is continuous on $L^2$. For a given process (desired state) \dela{$\bar{m}\in L^2(\Omega ; L^2(0,T; H^1\adda{\mathcal{S}^2}))$ }$\bar{m}\in L^2(\Omega ; L^2(0,T; \mathcal{S}^2))$
	\begin{equation}\label{definition of cost functional}
		J(\pi) = \mathbb{E} \l[ \int_{0}^{T} \l( \l| m(t) - \bar{m}(t) \r|_{H^1}^2 + \l| u(t) \r|_{L^2}^2 \r) \, dt + \Psi\l(m(T)\r) \r].
	\end{equation}
	\dela{The control problem is to minimize the above mentioned cost functional over the space $\mathcal{U}_{ad}(m_0,T)$.}\\
	Our aim is to minimize the above mentioned cost functional over the space $\mathcal{U}_{ad}(m_0,T)$.\\
	Stated formally, the optimal control problem is to find an admissible solution $\pi^* \in \mathcal{U}_{ad}(m_0,T)$ such that
	\begin{equation}\label{Optimal control problem}
		J(\pi^*) = \inf_{\pi \in \mathcal{U}_{ad}(m_0,T)} J(\pi).
	\end{equation}
	Let us denote the infimum of the cost functional by $\Lambda$. That is
	\begin{equation}
		\inf_{\pi \in \mathcal{U}_{ad}(m_0,T)} J(\pi) = \Lambda.
	\end{equation}
	
	\begin{proof}[Idea of the proof  of Theorem \ref{Theorem existence of optimal control}]
		\dela{First we show that the set of admissible solutions is non-empty. Hence the infimum $\Lambda$ is finite. This implies the existence of a minimizing sequence. We then show that this sequence converges (possibly along a subsequence) to a process, which is a strong martingale solution of the problem \eqref{problem considered for optimal control part}. Then we show that the infimum is attained at this process obtained as a limit, thus showing the existence of an optimal control.}

		First, we show that the set of admissible solutions is non-empty. Hence the infimum $\Lambda$ is finite. This implies the existence of a minimizing sequence $\{\pi_n\}_{n\in\mathbb{N}}$. Lemma \ref{bounds lemma 1 minimizing sequence} and Lemma \ref{bounds lemma 2 minimizing sequence} show that the minimizing sequence $\{\pi_n\}_{n\in\mathbb{N}}$ is uniformly bounded. Lemma \ref{Lemma Further regularity minimizing sequence} shows that the minimizing sequence is bounded in the maximal regular space. Further, Lemma \ref{tightness lemma minimizing sequence} shows that the sequence of laws of $\l( m_n , u_n \r)$ are tight on the space $L^2(0,T;H^1) \cap C([0,T]; L^2) \times L^2_w(0,T;L^2)$. In Proposition \ref{Proposition use of Skorokhod theorem minimizing sequence}, we use the Jakubowski's version of the Skorohod Theorem to obtain another sequence $\{\l( \m_n , \u_n \r)\}_{n\in\mathbb{N}}$ of processes, along with random variables $\m,\u, W^\p$, possibly on a different probability space $\l(\Omega^\p  , \mathcal{F}^\p , \mathbb{F}^\p , \mathbb{P}^\p \r)$. As before, we denote the tuple $\{\pi_n^\p\}_{n\in\mathbb{N}} : = \l(\Omega^\p , \mathcal{F}^\p , \mathbb{F}^\p , \mathbb{P}^\p  , \m_n , \u_n , W^\p_n\r)$ and $\{\pi^\p\}_{n\in\mathbb{N}} : = \l(\Omega^\p , \mathcal{F}^\p , \mathbb{F}^\p , \mathbb{P}^\p  , \m , \u , W^\p\r)$. Proposition \ref{Proposition use of Skorokhod theorem minimizing sequence} further gives us pointwise convergence of the processes $\m_n , \u_n$ and $W_n^\p$ to their corresponding limits in $\pi^\p$, in appropriate spaces. Lemma \ref{lem-bounds on mn prime minimizing sequence}, Lemma \ref{lem-Further regularitiy mn prime minimizing sequence} and Lemma \ref{lem-bounds lemma for m prime minimizing sequence} establish uniform bounds on the newly obtained processes $\m_n,n\in\mathbb{N}$ and $\m$. Then arguing similarly to Section \ref{Section Proof of existence of a solution}, we show that the obtained tuple $\pi^\p$ is a strong martingale solution of the problem \eqref{problem considered for optimal control part}.
		A main difference in the calculations is that in Section \ref{Section Proof of existence of a solution} we consider processes that have values in finite dimensional spaces, whereas that cannot be assumed here. One needs to be careful while applying the Kuratowski Theorem.
		Some more details are given in Remark \ref{remark same bounds remark minimizing sequence}. Moreover, we go on to show that the obtained tuple $\pi^\p$ is an admissible solution. Then we show that the infimum for the cost $J$ is attained at $\pi^\p$, thus showing the existence of an optimal control and completing the proof.
	\end{proof}
	
	\begin{remark}\label{Remark Strong optimal control}
		Before we begin with the proof of Theorem \ref{Theorem existence of optimal control}, we make a small comment. Theorem \ref{thm-existence of a strong solution}, combined with Remark \ref{remark equivalence for strong martingale solution 1} gives us the existence of a strong solution for the problem \eqref{problem considered for optimal control part}, which is stated in Theorem \ref{thm-existence of a strong solution}. That is, given a filtered probability space $\l( \Omega, \mathcal{F}, \mathbb{F}, \mathbb{P} \r)$, a Wiener process $W$, an initial data and a control process $u$ on the given space, there exists a process $ m $ which is a solution of the problem \eqref{problem considered for optimal control part}. The optimization problem can then be posed by fixing the given probability space and Wiener process, and then finding a tuple $\l( m^* , u^*\r)$ such that:
		\begin{enumerate}
			\item $m^*$ is a solution of the problem \eqref{problem considered for optimal control part} corresponding to the control process $u^*$.
			\item The tuple $\l( m^* , u^*\r)$ minimizes the cost \eqref{definition of cost functional} on the given probability space.
		\end{enumerate}
		This could be one way of formulating the problem. But, as of now, this does not contribute significantly to the overall progression of the problem and hence has not been considered.
		
	\end{remark}
	\begin{proof}[Proof of Theorem \ref{Theorem existence of optimal control}]

		Theorem \ref{Theorem Existence of a weak solution} along with Theorem \ref{Theorem Further regularity} shows that the space $\mathcal{U}_{ad}(m_0 , T)$ is non-empty. Hence $\Lambda < \infty$. Hence there exists a minimizing sequence $\{\pi_n\}_{n\in\mathbb{N}}$ of strong martingale solutions,
		$$	\pi_n = (\Omega_n, \mathcal{F}_n, \mathbb{F}_n, \mathbb{P}_n, W_n, m_n, u_n).$$
		That is
		\begin{equation}
			\lim_{n\rightarrow\infty} J(\pi_n) = \Lambda.
		\end{equation}
		Since $\pi_n$ is a minimizing sequence, there exists a constant $R>0$ such that for each $n\in\mathbb{N}$,
		\begin{equation}
			J(\pi_n) \leq R.
		\end{equation}
		Hence there exists a constant $C>0$ such that for any $n\in\mathbb{N}$,
		
		\begin{equation}\label{bound on m_n minimizing sequence L2 H1}
			\mathbb{E}^n \int_{0}^{T} \l|m_n(t)\r|_{H^1}^2 \, dt \leq C
		\end{equation}
		and
		\begin{equation}\label{bound on u_n minimizing sequence L2 L2}
			\mathbb{E}^n \int_{0}^{T} \l|u_n(t)\r|_{L^2}^2\, dt \leq K_1.
		\end{equation}
		Here $\mathbb{E}^n$ denotes the expectation with respect to the probability space $\l(\Omega_n , \mathcal{F}_n , \mathbb{P}_n\r)$.
		
		Before we continue with the main line of the proof we formulate and prove some essential auxiliary results. 
		
	\end{proof}
	
	\begin{lemma}\label{bounds lemma 1 minimizing sequence}
		There exists a constant $C>0$ such that for each $n\in\mathbb{N}$, the following bounds hold.
		\begin{align}\label{bound on m_n H1 minimizing sequence}
			\mathbb{E}^n \int_{0}^{T} \l| m_n(t) \r|_{H^1}^2 \, dt \leq C,
		\end{align}
		\dela{What about the sup norm in t?? Should there be a power 4??}
		
		\begin{align}\label{bound on m_n L infty H1 minimizing sequence}
			\mathbb{E}^n \sup_{t\in[0,T]}\l| m_n(t) \r|_{H^1}^4 \leq C,
		\end{align}

		\begin{equation}\label{bound on m_n times Delta m_n minimizing sequence}
			\mathbb{E}^n \int_{0}^{T} \l|m_n(s) \times \Delta m_n(s)\r|_{L^2}^2 \, ds \leq C,
		\end{equation}
		
		\begin{equation}\label{bound on m_n times m_n times Delta m_n minimizing sequence}
			\mathbb{E}^n \int_{0}^{T} \l| m_n(s) \times \l(m_n(s) \times \Delta m_n(s)\r) \r|_{L^2}^2 \, ds \leq C,
		\end{equation}
		
		\begin{equation}\label{bound on m_n times u_n minimizing sequence}
			\mathbb{E}^n \int_{0}^{T} \l|m_n(s) \times u_n(s)\r|_{L^2}^2 \, ds \leq C,
		\end{equation}
		
		\begin{equation}\label{bound on m_n times m_n times u_n minimizing sequence}
			\mathbb{E}^n \int_{0}^{t} \l| m_n(s) \times \l(m_n(s) \times u_n(s)\r)\r|_{L^2}^2 \, ds \leq C.
		\end{equation}
	\end{lemma}
	\begin{proof}[Proof of Lemma \ref{bounds lemma 1 minimizing sequence}]
		The first inequality \eqref{bound on m_n H1 minimizing sequence} follows from the fact that $\pi_n$ is a minimizing sequence and the inequality \eqref{bound on m_n minimizing sequence L2 H1}.

		The following equation is satisfied by the process $m_n$ 
		for all $t\in[0,T]$
		\begin{align}\label{eqn satisfied by minimizing sequence}
			\nonumber m_n(t) &= m_n(0) + \int_{0}^{t} m_n(s) \times \Delta m_n(s)\, ds  - \alpha \, \int_{0}^{t} m_n(s) \times \l( m_n(s) \times \Delta m_n(s) \r) \, ds  \\
			\nonumber & + \int_{0}^{t} m_n(s) \times u_n(s)\, ds - \alpha \, \int_{0}^{t} m_n(s) \times \l( m_n(s) \times u_n(s) \r) \, ds \\
			& + \frac{1}{2}\int_{0}^{t} \l[DG(m_n(s))\r]\l(G(m_n(s))\r)\, ds + \int_{0}^{t} G(m_n(s)) \, dW_n(s),\  \mathbb{P}_n-a.s.
		\end{align}
		\dela{We recall the function $\phi_2$ defined in \eqref{definition of phi 2}.That function was from $H_n$ to $H_n$. Hence better not recall it here. Just consider the function from $H^1\to\mathbb{R}$.}
		Let $\bar{\phi} : H^1 \to \mathbb{R}$ be given by
		\begin{equation}
			\bar{\phi}(v) = \frac{1}{2} \l|\nabla v\r|_{L^2}^2.
		\end{equation}
		We now apply the It\^o Lemma for the above function. The calculations are similar to the proofs of Lemma \ref{bounds lemma 1} and Lemma \ref{bounds lemma}, and hence are skipped. A difference is that the calculations here are in infinite dimensions, for which we apply the It\^o formula from \cite{Pardoux_1979}. It is therefore sufficient to show that the integrands on the right hand side of the equality \eqref{eqn satisfied by minimizing sequence} lie in appropriate spaces, see \cite{Pardoux_1979}, so that the It\^o formula can be applied.		
		Theorem \ref{Theorem Existence of a weak solution} implies that the terms $ m_n\times \Delta m_n , m_n \times \l( m_n \times \Delta m_n  \r) \in M^2(0,T; L^2) $. For the definition of the space, see Section \ref{sec-The constraint condition section}, see also \cite{Pardoux_1979}. \dela{ Avoid this notation $M^2$. Try to write in full.}\\			
		By the constraint condition \eqref{eqn-constraint condition},
		\begin{align*}
			\mathbb{E}^n\int_{0}^{T} \l| m_n(t) \times u_n(t) \r|_{L^2}^2 \, dt &\leq \mathbb{E}^n\int_{0}^{T} \l| u_n(t) \r|_{L^2}^2 \, dt  < \infty.
		\end{align*}
		The last inequality holds by \eqref{bound on u_n minimizing sequence L2 L2}.\\
		Similarly, the constraint condition \eqref{eqn-constraint condition} implies that
		\begin{align*}
			\mathbb{E}^n\int_{0}^{T} \l|m_n(t) \times \bigl( m_n(t) \times u_n(t) \bigr) \r|_{L^2}^2 \, dt &\leq \mathbb{E}^n\int_{0}^{T} \l|m_n(t) \times u_n(t) \r|_{L^2}^2 \, dt  < \infty.
		\end{align*}
		By the assumption $h\in H^1$, the embedding $H^1\hookrightarrow L^{\infty}$ and by the constraint condition \eqref{eqn-constraint condition},we have
		\begin{align*}
			\mathbb{E}^n \int_{0}^{T} \l| \big[ DG(m_n(t)) \big] \bigl[ G\big(m_n(t)\big) \bigr] \r|_{L^2}^2 \, dt < \infty.
		\end{align*}
		Hence $m_n \times u_n ,  \ m_n \times (m_n \times u_n) ,  \ \l[ DG\big(m_n\big) \r] \l[ G\big(m_n\big) \r] \in M^2(0,T;L^2)$.\\	
		Also, by the constraint condition implies that
		\begin{align*}
			\mathbb{E}^n \int_{0}^{T} \l|  G(m_n(t)) \r|_{L^2}^2\, dt < \infty.
		\end{align*}
		Hence $G(m_n)\in M^2(0,T;L^2)$.  The inequalities \eqref{bound on m_n L infty H1 minimizing sequence}, \eqref{bound on m_n times Delta m_n minimizing sequence}  then follow by applying the It\^o formula.
		The inequalities \eqref{bound on m_n times m_n times Delta m_n minimizing sequence}, \eqref{bound on m_n times u_n minimizing sequence} and \eqref{bound on m_n times m_n times u_n minimizing sequence} follow from the assumption on $u_n$ and the constraint condition \eqref{eqn-constraint condition}.		
		\dela{
			We consider the following calculations and then collect them to apply the It\^o Lemma. Fix $s\in[0,T]$ and $n\in\mathbb{N}$.
			
			\adda{Write the equality that arises after applying the It\^o formula. Then write the calculations after numbering each of the terms. This proof is similar to the proof of Lemma \ref{bounds lemma 1}. Are all the details required? A major difference is that the latter is in infinite dimension. But the structure of the proof remains the same.}
			
			We use the following identity for the following calculations. For $a,b\in \mathbb{R}^3$
			\begin{align}
				\l\langle a , a \times b \r\rangle_{\mathbb{R}^3} = 0.
			\end{align}
			
			\begin{align*}
				\l\langle \nabla m_n(s) , \nabla \l( m_n(s) \times \Delta m_n(s) \r) \r\rangle_{L^2} &= - \l\langle \Delta m_n(s) ,  m_n(s) \times \Delta m_n(s) \r\rangle_{L^2} \\
				& = 0.
			\end{align*}
			
			\begin{align*}
				\l\langle \nabla m_n(s) , \nabla \l( m_n(s) \times \l( m_n(s) \times \Delta m_n(s) \r) \r) \r\rangle_{L^2} &= - \l\langle \Delta m_n(s) ,  m_n(s) \times \l( m_n(s) \times \Delta m_n(s) \r)  \r\rangle_{L^2}\\
				& = \l\langle  m_n(s) \times \Delta m_n(s) ,  m_n(s) \times \Delta m_n(s)   \r\rangle_{L^2} \\
				& = \l| m_n(s) \times \Delta m_n(s) \r|_{L^2}^2.
			\end{align*}
			Hence
			\begin{align*}
				\int_{0}^{t} \l\langle \nabla m_n(s) , \nabla \l( m_n(s) \times \l( m_n(s) \times \Delta m_n(s) \r) \r) \r\rangle_{L^2} \, ds = \int_{0}^{t} \l| m_n(s) \times \Delta m_n(s) \r|_{L^2}^2 \, ds
			\end{align*}
			Recall that in \eqref{eqn satisfied by minimizing sequence} the above term has a negative coefficient. Hence this term can be taken on the left hand side once the It\^o Lemma is applied.
			
			The other terms can be handled in a similar way as done in Step 2 of the proof of Lemma \ref{bounds lemma 1}. For the sake of completion, we present the calculations here.
			
			\begin{align*}
				\l\langle \nabla m_n(s) , \nabla \l(m_n(s) \times u_n(s) \r) \r\rangle_{L^2} &= - \l\langle \Delta m_n(s) , m_n(s) \times u_n(s)  \r\rangle_{L^2} \\
				& = \l\langle m_n(s) \times \Delta m_n(s) , u_n(s)  \r\rangle_{L^2}.
			\end{align*}
			Hence
			\begin{align*}
				\int_{0}^{t} \l\langle \nabla m_n(s) , \nabla \l(m_n(s) \times u_n(s) \r) \r\rangle_{L^2}\, ds & = \l\langle m_n(s) \times \Delta m_n(s) , u_n(s)  \r\rangle_{L^2} \\
				& \leq \int_{0}^{t} \l| m_n(s) \times \Delta m_n(s) \r|_{L^2} \l| u_n(s) \r|_{L^2} \, ds \\
				&\quad \text{By Cauchy-Schwarz inequlity} \\
				& \leq \frac{\varepsilon}{2}\int_{0}^{t} \l| m_n(s) \times \Delta m_n(s) \r|_{L^2}^2 \, ds + \frac{C(\varepsilon)}{2}\int_{0}^{t} \l| u_n(s) \r|_{L^2} \, ds. \\
				& ( \text{By Young's inequality} )
			\end{align*}
			Here the constant $\varepsilon$ will be chosen later.

			\begin{align*}
				\l\langle \nabla m_n(s) , \nabla \l( m_n(s) \times \l( m_n(s) \times u_n(s) \r) \r) \r\rangle_{L^2} & = - \l\langle \Delta m_n(s) ,  m_n(s) \times \l( m_n(s) \times u_n(s) \r)  \r\rangle_{L^2} \\
				& = \l\langle  m_n(s) \times \Delta m_n(s) , m_n(s) \times u_n(s) \r\rangle_{L^2}.
			\end{align*}
			
			Hence
			\begin{align*}
				&\int_{0}^{t} \l\langle \nabla m_n(s) , \nabla \l( m_n(s) \times \l( m_n(s) \times u_n(s) \r) \r) \r\rangle_{L^2} \\
				& = - \l\langle \Delta m_n(s) ,  m_n(s) \times \l( m_n(s) \times u_n(s) \r)  \r\rangle_{L^2} \, ds \\
				& = \int_{0}^{t} \l\langle  m_n(s) \times \Delta m_n(s) , m_n(s) \times u_n(s) \r\rangle_{L^2} \, ds \\
				& \leq \int_{0}^{t} \l| m_n(s) \times \Delta m_n(s)  \r|_{L^2} \l| m_n(s) \times u_n(s) \r|_{L^2} \, ds \\
				&\text{(By Cauchy-Schwarz inequlity)} \\
				& \leq \frac{\varepsilon}{2}\int_{0}^{t} \l| m_n(s) \times \Delta m_n(s)  \r|_{L^2}^2 \, ds + \frac{C(\varepsilon)}{2} \int_{0}^{t}  \l| m_n(s) \times u_n(s) \r|_{L^2}^2 \, ds \\
				& \leq \frac{\varepsilon}{2}\int_{0}^{t} \l| m_n(s) \times \Delta m_n(s)  \r|_{L^2}^2 \, ds + \frac{C(\varepsilon)}{2} \int_{0}^{t}  \l|m_n(s)\r|_{L^{\infty}}\l| u_n(s) \r|_{L^2}^2 \, ds \\
				& \leq \frac{\varepsilon}{2}\int_{0}^{t} \l| m_n(s) \times \Delta m_n(s)  \r|_{L^2}^2 \, ds + \frac{C(\varepsilon)}{2} \int_{0}^{t} \l| u_n(s) \r|_{L^2}^2 \, ds .
			\end{align*}
			The last inequality follows by Youngs inequality.
			
			There are two terms now that are yet to be approximated, viz. $\l[DG(m_n(s))\r]\l(G(m_n(s))\r)$ and $G(m_n(s))$.
			
			Recall that for each $n\in\mathbb{N}$, the process $m_n$ satisfies the constraint condition \eqref{eqn-constraint condition}. Also $h\in H^1$ and the continuous embedding $H^1\hookrightarrow L^{\infty}$ implies that there exists constants $C , C_1$ such that
			\begin{equation}
				\l| h \r|_{L^{\infty}} \leq C_1\l| h \r|_{H^1} \leq C.
			\end{equation}
			We use this to show that
			\begin{equation*}
				\int_{0}^{t} \l\langle \nabla m_n(s) , \nabla \l(\l[G(m_n(s))\r]\l(G(m_n(s))\r)\r) \r\rangle \, ds \leq C \int_{0}^{t}\l|  m_n(s) \r|_{H^1}^2 \, ds.
			\end{equation*}
			\dela{The reason for introducing the $H^1$ norm here is that when the gradient goes to the $h$ term, there is just one term $\nabla m_n(s)$. Hence there is no square. But there is $\l|m_n(s)\r|_{L^{\infty}}$ term which can be bounded by $\l|m_n(s)\r|_{H^1}$ and hence the $H^1$ norm. With the usual integral this should not be a major concern since we have the continuous embedding $L^2\hookrightarrow L^1$. Care has to be taken while working with the It\^o integral.}
			
			The main idea here is the following. Expand the derivative on the right side of the inner product (using the product rule for derivatives). Split the inner product over the addition/subtraction signs. Observe that exactly one term will have the derivative at once, along with the $\nabla m_n(s)$ term. Put $L^2$ norm on the terms with the derivative. Whenever there is no derivative, the $L^{\infty}$ norm of that term is bounded, either because the term is $h$ or because of the constraint condition \eqref{eqn-constraint condition}.
			
			We now do the calculations regarding the term $G(m_n(s))$ and the stochastic integral.
			By the Burkholder-Davis-Gundy inequality, there exists a constant $C>0$ such that
			\begin{align}
				\mathbb{E} \sup_{t\in[0,T]} \int_{0}^{t} \l\langle \nabla G(m_n(s)) , \nabla m_n(s) \r\rangle_{L^2} \, dW_n(s) \leq C \mathbb{E} \l( \int_{0}^{T} \l\langle \nabla G(m_n(s)) , \nabla m_n(s) \r\rangle_{L^2}^2 \, ds \r)^{\frac{1}{2}}
			\end{align}

			We now bound the right hand side of the above inequality.
			\begin{align*}
				\l\langle \nabla m_n(s) , \nabla \l( m_n(s) \times h \r) \r\rangle_{L^2} & = \l\langle \nabla m_n(s) , \nabla m_n(s) \times h  \r\rangle_{L^2} + \l\langle \nabla m_n(s) ,  m_n(s) \times \nabla h \r\rangle_{L^2} \\
				& \leq \l| \nabla m_n(s) \r|_{L^2}^2\l|h\r|_{L^{\infty}} + \l|\nabla m_n(s)\r|_{L^2} \l|m_n(s)\r|_{L^{\infty}} \l|\nabla h\r|_{L^2} \\
				& \leq C_1(h) \l| \nabla m_n(s) \r|_{L^2}^2 + C_2(h) \l| \nabla m_n(s) \r|_{L^2}
			\end{align*}

			
			Hence combining the constants $C_1(h)$ and $C_2(h)$ into a suitable constant $C(h)$, we get
			\dela{\begin{align*}
					\int_{0}^{t} \l\langle \nabla m_n(s) , \nabla \l( m_n(s) \times h \r) \r\rangle_{L^2} \, dW_n(s) \leq C(h) \int_{0}^{t} \l(\l| \nabla m_n(s) \r|_{L^2}^2 + \l| \nabla m_n(s) \r|_{L^2} \r) \, dW_n(s).
				\end{align*}
			}
			\begin{align*}
				\int_{0}^{T} \l\langle \nabla m_n(s) , \nabla \l( m_n(s) \times h \r) \r\rangle_{L^2}^2 \, ds \leq C(h) \int_{0}^{T} \l(\l| \nabla m_n(s) \r|_{L^2}^4 + \l| \nabla m_n(s) \r|_{L^2}^2 \r) \, ds.
			\end{align*}		
			
			Similarly, by H\"older's inequality, we have
			\begin{align*}
				\l\langle \nabla m_n(s) , \nabla \l(m_n(s) \times \l(m_n(s) \times h\r)\r) \r\rangle_{L^2}  =& \l\langle \nabla m_n(s) , \nabla m_n(s) \times \l(m_n(s) \times h\r) \r\rangle_{L^2} \\
				&+ \l\langle \nabla m_n(s) ,  m_n(s) \times \l( \nabla m_n(s) \times h\r) \r\rangle_{L^2} \\
				& + \l\langle \nabla m_n(s) ,  m_n(s) \times \l(m_n(s) \times \nabla h\r) \r\rangle_{L^2} \\
				\leq &  2\l|\nabla m_n(s)\r|_{L^2}^2 \l|m_n(s)\r|_{L^{\infty}} \l|h\r|_{L^{\infty}}\\
				&+ \l|\nabla m_n(s)\r|_{L^2} \l|m_n(s)\r|_{L^{\infty}}^2 \l|\nabla h\r|_{L^2} \\
				\leq&  C(h)\l( \l|\nabla m_n(s)\r|_{L^2}^2 + \l|\nabla m_n(s)\r|_{L^2} \r).
			\end{align*}		
			
			Hence
			\dela{\begin{align*}
					\int_{0}^{t} \l\langle \nabla m_n(s) , \nabla \l(m_n(s) \times \l(m_n(s) \times h\r)\r) \r\rangle_{L^2} \, dW_n(s) \leq C(h) \int_{0}^{t} \l( \l|\nabla m_n(s)\r|_{L^2}^2 + \l|\nabla m_n(s)\r|_{L^2} \r) \, dW_n(s) .
				\end{align*}
			}
			
			\begin{align*}
				\int_{0}^{T} \l\langle \nabla m_n(s) , \nabla \l(m_n(s) \times \l(m_n(s) \times h\r)\r) \r\rangle_{L^2}^2 \, ds \leq C(h) \int_{0}^{T} \l( \l|\nabla m_n(s)\r|_{L^2}^4 + \l|\nabla m_n(s)\r|_{L^2}^2 \r) \, ds .
			\end{align*}

			Thus, we combine the above two inequalities to get a constant $C>0$ such that
			\dela{	\begin{align*}
					int_{0}^{t} \l\langle \nabla m_n(s) , \nabla G(m_n(s)) \r\rangle_{L^2} \, dW_n(s) \leq C \int_{0}^{t} \l( \l|\nabla m_n(s)\r|_{L^2}^2 + \l|\nabla m_n(s)\r|_{L^2} \r) \, dW_n(s).		
				\end{align*}
			}
			
			\begin{align*}
				\mathbb{E}^n \sup_{t\in[0,T]}\int_{0}^{T} \l\langle \nabla m_n(s) , \nabla G(m_n(s)) \r\rangle_{L^2}^2 \, ds &\leq C\mathbb{E}^n \int_{0}^{T} \l( \l|\nabla m_n(s)\r|_{L^2}^4 + \l|\nabla m_n(s)\r|_{L^2}^2 \r) \, ds \\
				& \leq C \mathbb{E}^n \int_{0}^{T} \l|\nabla m_n(s)\r|_{L^2}^4  \, ds.		
			\end{align*}
			The last inequality follows from Jensen's inequality.

			We now proceed to apply the It\^o Lemma to the function $\phi_2$. The calculations are similar to the proof of \eqref{bound 2} in Lemma \ref{bounds lemma}.
			
			\begin{align}
				\nonumber  \frac{1}{2}\l|\nabla m_n(s)\r|_{L^2}^2 &- \frac{1}{2}\l|\nabla m_0\r|_{L^2}^2 \\
				\nonumber=& \int_{0}^{t} \l\langle \nabla m_n(s) \times \Delta m_n(s) , \nabla m_n(s) \r\rangle_{L^2}\, ds\\
				\nonumber& - \alpha \, \int_{0}^{t} \l\langle \nabla m_n(s) \times ( m_n(s) \times \Delta m_n(s)) , \nabla m_n(s) \r\rangle_{L^2}\, ds \\
				\nonumber & + \int_{0}^{t} \l\langle \nabla m_n(s) \times u_n(s) , \nabla m_n(s) \r\rangle_{L^2}\, ds\\
				\nonumber&- \alpha\int_{0}^{t} \l\langle \nabla m_n(s) \times (m_n(s) \times u_n(s)) , \nabla m_n(s) \r\rangle_{L^2}\, ds\\
				\nonumber& + \int_{0}^{t} \l\langle \nabla \l(DG(m_n(s))\r)(G(m_n(s))) , \nabla m_n(s) \r\rangle_{L^2}\, ds \\
				&+ \int_{0}^{t} \l\langle G(m_n(s)) , \nabla m_n(s) \r\rangle_{L^2} \, dW(s).
			\end{align}
			
			Hence combining the above equality with the previous calculations gives
			\begin{align*}
				\l|\nabla m_n(t)\r|_{L^2}^2 \leq& \l| \nabla m_n(0)\r|_{L^2}^2 + \l(\varepsilon - \alpha \, \r) \int_{0}^{t} \l|m_n(s) \times \Delta m_n(s)\r|_{L^2}^2 \, ds + \frac{C(\varepsilon)}{2}\int_{0}^{t} \l|u_n(s)\r|_{L^2}^2 \, ds \\
				& + C \int_{0}^{t} \l| \nabla m_n(s)\r|_{L^2}^2 \, ds + C \int_{0}^{t}  \l\langle \nabla G(m_n(s)) , \nabla m_n(s) \r\rangle  \, dW_n(s).
			\end{align*}

			We choose $\varepsilon$ small enough so that $\l(\varepsilon - \alpha\r) < 0$. Let us choose $\varepsilon = \frac{\alpha}{2}$. The resulting inequality is
			\begin{align*}
				&\l|\nabla m_n(t)\r|_{L^2}^2 + \frac{\alpha}{2} \int_{0}^{t} \l|m_n(s) \times \Delta m_n(s)\r|_{L^2}^2 \, ds \\
				\leq& \l| \nabla m_n(0)\r|_{L^2}^2 + \frac{C(\varepsilon)}{2}\int_{0}^{t} \l|u_n(s)\r|_{L^2}^2 \, ds  + C \int_{0}^{t} \l| \nabla m_n(s)\r|_{L^2}^2 \, ds \\
				& + C \int_{0}^{t}  \l\langle \nabla G(m_n(s)) , \nabla m_n(s) \r\rangle  \, dW_n(s).
			\end{align*}
			
			Now, we multiply the above inequality by suitable constants to make the coefficients 1 on the left hand side. Also, the resulting coefficients on the right hand side are condensed into generic constants $C_1,C_2$ and $C_3$.

			\begin{align}\label{Intermediate equation for obtaining H1 bound minimizing sequence}
				\nonumber &\l|\nabla m_n(t)\r|_{L^2}^2 + \int_{0}^{t} \l|m_n(s) \times \Delta m_n(s)\r|_{L^2}^2 \, ds \\
				\nonumber \leq& \l| \nabla m_n(0)\r|_{L^2}^2 + C_1\int_{0}^{t} \l|u_n(s)\r|_{L^2}^2 \, ds + C_2 \int_{0}^{t} \l| \nabla m_n(s)\r|_{L^2}^2 \, ds \\
				& + C_3 \int_{0}^{t} \l\langle \nabla G(m_n(s)) , \nabla m_n(s) \r\rangle \, dW_n(s).
			\end{align}

			Therefore

			\begin{align*}
				\l|\nabla m_n(t)\r|_{L^2}^2 \leq& \l| \nabla m_n(0)\r|_{L^2}^2 + C_1\int_{0}^{t} \l|u_n(s)\r|_{L^2}^2 \, ds + C_2 \int_{0}^{t} \l| \nabla m_n(s)\r|_{L^2}^2 \, ds \\
				& + C_3 \int_{0}^{t} \l\langle \nabla G(m_n(s)) , \nabla m_n(s) \r\rangle \, dW_n(s).
			\end{align*}
			
			We take $\sup_{t\in[0,T]}$ of both sides of the above inequality to get
			\begin{align*}
				\sup_{t\in[0,T]}\l|\nabla m_n(t)\r|_{L^2}^2 \leq& \l| \nabla m_n(0)\r|_{L^2}^2 + C_1\int_{0}^{T} \l|u_n(s)\r|_{L^2}^2 \, ds + C_2 \int_{0}^{T} \l| \nabla m_n(s)\r|_{L^2}^2 \, ds \\
				& + C_3 \sup_{t\in[0,T]}\int_{0}^{t} \l\langle \nabla G(m_n(s)) , \nabla m_n(s) \r\rangle \, dW_n(s).
			\end{align*}
			We now square both sides and use Jensen's inequality to get
			\begin{align*}
				\sup_{t\in[0,T]}\l|\nabla m_n(t)\r|_{L^2}^4 \leq& C\bigg(\l| \nabla m_n(0)\r|_{L^2}^4 + C_1\int_{0}^{T} \l|u_n(s)\r|_{L^2}^4 \, ds + C_2 \int_{0}^{T} \l| \nabla m_n(s)\r|_{L^2}^4 \, ds \\
				& + C_3 \sup_{t\in[0,T]} \l( \int_{0}^{t} \l\langle \nabla G(m_n(s)) , \nabla m_n(s) \r\rangle \, dW_n(s)\r)^2\bigg).
			\end{align*}

			We take the expectation of both sides. Also, replace $\l|\nabla m_n(s)\r|_{L^2}$ by $\sup_{r\in[0,s]}\l| m_n(r)\r|_{H^1}$.
			\dela{	\begin{align*}
					\mathbb{E}\sup_{s\in[0,t]} \l|\nabla m_n(s)\r|_{L^2}^2 \leq& \mathbb{E} \l| \nabla m_n(0)\r|_{L^2}^2 + \mathbb{E} \int_{0}^{t} \sup_{r\in[0,s]}  \l|m_n(s)\r|_{L^2}^2 \, ds + C_1\mathbb{E}\int_{0}^{t} \l|u_n(s)\r|_{L^2}^2 \, ds\\
					&+ C \mathbb{E} \int_{0}^{t} \l( \sup_{r\in[0,s]} \l|\nabla m_n(s)\r|_{L^2}^2 + \sup_{r\in[0,s]} \l|\nabla m_n(s)\r|_{L^2} \r) \, dW_n(s).
				\end{align*}
			}
			
			\begin{align*}
				\mathbb{E}\sup_{s\in[0,t]} \l|\nabla m_n(s)\r|_{L^2}^4 \leq& \mathbb{E} \l| \nabla m_n(0)\r|_{L^2}^4 + \mathbb{E} \int_{0}^{T} \sup_{r\in[0,s]}  \l|m_n(s)\r|_{H^1}^4 \, ds + C_1\mathbb{E}\int_{0}^{T} \l|u_n(s)\r|_{L^2}^4 \, ds\\
				&+ C \mathbb{E} \sup_{t\in[0,T]}\int_{0}^{t} \l\langle \nabla G\l(m_n(s)\r) , \nabla m_n(s) \r\rangle_{L^2} \, dW_n(s) \\
				& \leq  \mathbb{E} \l| \nabla m_n(0)\r|_{L^2}^2 + \mathbb{E} \int_{0}^{t} \sup_{r\in[0,s]}  \l|m_n(s)\r|_{L^2}^4 \, ds + C_1\mathbb{E}\int_{0}^{t} \l|u_n(s)\r|_{L^2}^4 \, ds\\
				&+ C \mathbb{E}^n \int_{0}^{T} \l|\nabla m_n(s)\r|_{L^2}^4  \, ds\ \text{By Burkholder-Davis-Gundy inequality}\\
				& \leq C(m_0 , K) + \mathbb{E} \int_{0}^{t} \sup_{r\in[0,s]}  \l|m_n(s)\r|_{L^2}^4 \, ds + C \mathbb{E}^n \int_{0}^{T} \l|\nabla m_n(s)\r|_{L^2}^4  \, ds.
			\end{align*}

			\dela{Write the It\^o integral properly. Since supremum has been taken, it is not a martingale. BDG should be used here. That might require us to raise the power of both sides by 2. This is same as the proof of Lemma \ref{bounds lemma 1}.}

			By Fubini's theorem,
			\begin{align*}
				\mathbb{E}\sup_{s\in[0,t]} \l|\nabla m_n(s)\r|_{L^2}^4 &\leq C(m_0 , K) +  \int_{0}^{t} \mathbb{E} \sup_{r\in[0,s]}  \l| \nabla m_n(s)\r|_{L^2}^4 \, ds.
			\end{align*}
			The constant $C(m_0 , K)$ is obtained by combining the constants arising from the initial data term and the term with the control process. For a bound on the control process, we recall \eqref{bound on u_n minimizing sequence L2 L2}.
			By the Gronwall Lemma, there exists a constant $C>0$ such that for each $n\in\mathbb{N}$,
			\begin{align}
				\mathbb{E}\sup_{s\in[0,t]} \l|\nabla m_n(s)\r|_{L^2}^4 \leq C.
			\end{align}
			Here $t\in[0,T]$ was fixed initially. The same calculations work for any $t\in[0,T]$.
			
			Combining the above obtained bound with the constraint condition \eqref{eqn-constraint condition}, there exists a constant $C>0$ such that for each $n\in\mathbb{N}$,
			\begin{align}
				\mathbb{E}^n \int_{0}^{T} \l| m_n(t) \r|_{H^1}^2 \, dt \leq C.
			\end{align}
			Hence this shows the bound \eqref{bound on m_n L infty H1 minimizing sequence}.

			We now go back to the inequality in \eqref{Intermediate equation for obtaining H1 bound minimizing sequence}. We take the expectation of both sides to get
			\begin{align*}
				\mathbb{E}^n \int_{0}^{t} \l|m_n(s) \times \Delta m_n(s)\r|_{L^2}^2 \, ds &\leq \mathbb{E}^n \l|\nabla m_n(0)\r|_{L^2}^2 + C_1 \mathbb{E}^n \int_{0}^{t} \l|u_n(t)\r|_{L^2}^2 \, dt + C_2 \mathbb{E}^n \int_{0}^{t} \l| \nabla m_n(t)\r|_{L^2}^2 \, dt \\
				& \leq C(m_0 , K) + C_2 \mathbb{E}^n \int_{0}^{t} \l| \nabla m_n(t)\r|_{L^2}^2 \, dt.
			\end{align*}
			\dela{This uniform bound is not directly required here. This $\{\pi_n\}_{n\in\mathbb{N}}$ is a minimizing sequence. Hence there is a uniform bound on the above norm for the control process. The assumption is required because a uniform bound is also required on the $p-th$ moment for the control process, which is given by the Assumption \eqref{assumption on u}.}
			By the bound obtained in \eqref{bound on m_n H1 minimizing sequence}, the right hand side of the above inequality is uniformly (in $n$) bounded. The argument holds for any $t\in[0,T]$. Hence there exists a constant $C>0$ such that for any $n\in\mathbb{N}$
			
			\begin{equation}
				\mathbb{E}^n \int_{0}^{T} \l|m_n(s) \times \Delta m_n(s)\r|_{L^2}^2 \, ds \leq C.
			\end{equation}

			By the constraint condition \eqref{eqn-constraint condition},
			\begin{align*}
				\mathbb{E}^n \int_{0}^{t} \l| m_n(s) \times \l(m_n(s) \times \Delta m_n(s)\r) \r|_{L^2}^2 \, ds \leq \mathbb{E}^n \int_{0}^{t} \l|m_n(s) \times \Delta m_n(s)\r|_{L^2}^2 \, ds.
			\end{align*}
			Hence there exists a constant $C>0$ such that for any $n\in\mathbb{N}$,
			\begin{equation}
				\mathbb{E}^n \int_{0}^{T} \l| m_n(s) \times \l(m_n(s) \times \Delta m_n(s)\r) \r|_{L^2}^2 \, ds \leq C.
			\end{equation}
			
			Working similarly, by the constraint condition \eqref{eqn-constraint condition} we have
			\begin{equation*}
				\mathbb{E}^n \int_{0}^{t} \l|m_n(s) \times u_n(s)\r|_{L^2}^2 \, ds \leq \mathbb{E} \int_{0}^{t} \l| u_n(s) \r|_{L^2}^2 \, ds.
			\end{equation*}
			Hence by \eqref{bound on u_n minimizing sequence L2 L2}, there exists a constant $C>0$ such that for each $n\in\mathbb{N}$
			\begin{equation}
				\mathbb{E}^n \int_{0}^{T} \l|m_n(s) \times u_n(s)\r|_{L^2}^2 \, ds \leq C.
			\end{equation}
			Also,
			\begin{align*}
				\mathbb{E}^n \int_{0}^{t} \l| m_n(s) \times \l(m_n(s) \times u_n(s)\r)\r|_{L^2}^2 \, ds \leq \mathbb{E} \int_{0}^{t} \l|m_n(s) \times u_n(s)\r|_{L^2}^2 \, ds.
			\end{align*}
			Hence there exists another constant $C>0$ such that for each $n\in\mathbb{N}$,
			\begin{equation}
				\mathbb{E}^n \int_{0}^{t} \l| m_n(s) \times \l(m_n(s) \times u_n(s)\r)\r|_{L^2}^2 \, ds \leq C.
			\end{equation}
		}
	\end{proof}

	\begin{lemma}\label{bounds lemma 2 minimizing sequence}
		Let $\gamma \in \l(0,\frac{1}{2}\r)$ and $p\geq 2$. Then there exists a constant $C>0$ such that for each $\mathbb{N}$, the following bound holds. \dela{specify $p$}
		\begin{equation}
			\mathbb{E}^n\l[\l|m_n\r|^2_{W^{\gamma , p}(0,T ; L^2)}\r] \leq C.
		\end{equation}
	\end{lemma}
	\dela{Change the Lemma. Use the same bound used in Lemma \ref{bounds lemma}.
		Then claim that uniform bound also exists $\mathbb{P}$-a.s.}
	\begin{proof}[Proof of Lemma \ref{bounds lemma 2 minimizing sequence}]
		The proof is similar to the proof of Lemma \ref{bounds lemma}.\\
		The idea of the proof is to show a stronger bound (in $W^{1,2}(0,T;L^2)$) for the terms without the stochastic intergral, as done in the proof of Lemma \ref{bounds lemma}. Then use the \dela{compact} embedding
		\begin{equation}
			W^{1,2}(0,T;L^2)\hookrightarrow W^{\gamma , p}(0,T ; L^2),
		\end{equation}
		to conclude the bound. For the stochastic integral, the proof is similar to the proof in Lemma \ref{bounds lemma}, using Lemma \ref{Lemma reference W alpha p bound for stochastic integral}.\dela{, which follows Lemma 2.1 from \cite{Flandoli_Gatarek}. This can also be done by Lemma A.1 in \cite{ZB+BG+TJ_Weak_3d_SLLGE}.}
	\end{proof}
	Combining the bound \eqref{bound on m_n L infty H1 minimizing sequence} in Lemma \ref{bounds lemma 1 minimizing sequence} along with the Lemma \ref{bounds lemma 2 minimizing sequence}, we have that the sequence $\{m_n\}_{n\in\mathbb{N}}$ is bounded in the space $L^2(\Omega ; L^{\infty}(0,T;H^1))\cap L^2(\Omega ; W^{\gamma , p}(0,T;L^2))$.
	
	That each $m_n$ satisfies \eqref{Further regularity} follows from Theorem \ref{Theorem Further regularity}. The aim here is to show that the bound is uniform in $n\in\mathbb{N}$.
	
	\begin{lemma}\label{Lemma Further regularity minimizing sequence}
		There exists a constant $C>0$ such that for all $n\in\mathbb{N}$,
		\begin{align}\label{Further regularity minimizing sequence}
			\mathbb{E} \left( \int_{0}^{T} | \nabla m_n(t) |_{L^4}^4\, dt  + \int_{0}^{T} |A_1 m_n(t)|_{L^2}^2\, dt\right) \leq C.
		\end{align}
	\end{lemma}
	
	\begin{proof}[Idea of the proof of Lemma \ref{Lemma Further regularity minimizing sequence}]
		That $m_n$ is a strong martingale solution for each $n\in\mathbb{N}$ implies that the left hand side of the inequality \eqref{Further regularity minimizing sequence} is finite for each $n\in\mathbb{N}$. The aim of this lemma is to show that the constant on the right hand side is independent of $n$. One can verify from the proof of Theorem \ref{Theorem Further regularity} that the bounds on the right hand side depends only on $\mathbb{E}\l| u \r|_{L^2(0,T;L^2)}^{2p}$, the initial data $m_0$ and the fixed time $T$. By the Assumption \ref{assumption on u} and the fact that $\{\pi_n\}_{n\in\mathbb{N}}$ is a minimizing sequence, we can conclude the lemma.
	\end{proof}
	\begin{proof}[An outline of the proof of Lemma \ref{Lemma Further regularity minimizing sequence}]
		\dela{
			
			To prove the lemma, we will follow Step 1 and Step 2 (Section \ref{Section Further regularity}) of the proof of Theorem \ref{Theorem Further regularity} and show that the bound on the right hand side does not depend on $n$. In that direction, first we show that the bound on $\mathbb{E
			}\sup_{t\in[0,T]}\l|m_n\r|_{H^1}^2$ (for the minimizing sequence) is independent of $n$. We recall the inequalities \eqref{bound using u 1}, \eqref{bound using u 2}. Notice that the constant on the right hand side depends only on $\mathbb{E}\l| u \r|_{L^2(0,T;:L^2)}^{2p}$ (i.e. $K_p$), the initial data $m_0$ and the fixed time $T$.\dela{ Hence the constant that appears after the application of Gronwalls Lemma \adda{when??} is also dependent only on the above mentioned quantities.} Similarly, the bound \eqref{bound on m_n L infty H1 minimizing sequence} depends on the initial data $m_0$, the time $T$ and $\mathbb{E}\l| u_n \r|_{L^2(0,T;:L^2)}^{2p}$. By the Assumption \ref{assumption on u}\dela{/by the fact that $\{m_n\}_{n\in\mathbb{N}}$\adda{$\{\pi_n\}_{n\in\mathbb{N}}$} is a minimizing sequence}, the bound on $\mathbb{E
			}\sup_{t\in[0,T]}\l|m_n\r|_{H^1}^2$ is therefore independent of $n$. \dela{write $\sup_{t\in[0,T]}$}
			
		}
		To prove the lemma, we will follow Step 1 and Step 2 (Section \ref{Section Further regularity}) of the proof of Theorem \ref{Theorem Further regularity} and show that the bound on the right hand side does not depend on $n$. In that direction, first we recall that by Lemma \ref{bounds lemma 1 minimizing sequence}\dela{\eqref{bound on m_n L infty H1 minimizing sequence}, the bound on $\mathbb{E
			}\sup_{t\in[0,T]}\l|m_n\r|_{H^1}^2$ (for the minimizing sequence) is independent of $n$.}, the bounds on $m_n,u_n$ are independent of $n$.
		
		\dela{
			Following a similar \dela{same equation \adda{which equation??}}line of argument, the term $m_n \times \Delta m_n$ is bounded by a constant that depends only on the initial data $m_0$, the time $T$ and the bounds on $\mathbb{E}\l| u \r|_{L^2(0,T;L^2)}^{2p}$ and \adda{ Find a better way of writing this $\mathbb{E
				}\l|m_n\r|_{L^{\infty}(0,T;H^1)}^{2p}$}.
			
		}

		\dela{
			
			Note that by the constraint condition \eqref{eqn-constraint condition},
			\begin{equation*}
				\mathbb{E} \int_{0}^{T} \l| m_n(t) \times \l( m_n(t) \times \Delta m_n(t) \r) \r|_{L^2}^2\, dt \leq \mathbb{E} \int_{0}^{T} \l| m_n(t) \times \Delta m_n(t) \r|_{L^2}^2\, dt.
			\end{equation*}
			Hence the bound again does not depend on $n\in\mathbb{N}$ since the right hand side is uniformly bounded by Lemma \ref{bounds lemma 1 minimizing sequence}.\\
			Again by the constraint condition \eqref{eqn-constraint condition},
			\begin{equation*}
				\mathbb{E} \int_{0}^{T} \l| m_n(t) \times u_n(t) \r|_{L^2}^2\, dt \leq \mathbb{E} \int_{0}^{T} \l| u_n(t) \r|_{L^2}^2\, dt
			\end{equation*}
			and
			\begin{align*}
				\mathbb{E} \int_{0}^{T} \l| m_n(t) \times \l( m_n(t) \times u_n(t) \r) \r|_{L^2}^2\, dt & \leq \mathbb{E} \int_{0}^{T} \l| m_n(t) \times u_n(t) \r|_{L^2}^2\, dt \\
				& \leq \mathbb{E} \int_{0}^{T} \l| u_n(t) \r|_{L^2}^2\, dt.
			\end{align*}
			For showing the bound similarly for the minimizing sequence, the other terms should be bounded (uniformly in $n$) in the respective spaces.
		}
		
		We now recall Step 1 in the proof of Theorem \ref{Theorem Further regularity}. The bound on $\mathbb{E} \int_{0}^{T} \l| A_1^{\delta} m_n(t) \r|_{L^2}^2\, dt$ depends only on the choice of $\delta$ and the $L^4\l(\Omega ; L^2\l( 0,T ; L^2 \r)\r)$ norm of the functions on the right hand side of \dela{\eqref{eqn satisfied by minimizing sequence}} \eqref{problem considered for optimal control part}. Following the above arguments, we can show that the required bounds do not depend on $n\in\mathbb{N}$.
		
		For the It\^o integral term, We observe that the bound depends on the time $T$, the choice of $\delta$ and the norm $\mathbb{E} \int_{0}^{T} \l|G(m_n(t))\r|_{H^1}^2\, dt$, which again depends on the norm $\mathbb{E} \int_{0}^{T} \l|m_n(t)\r|_{H^1}^2\, dt$, the constraint condition and the fixed function $h$. Hence, from the above arguments, this bound also does not depend on $n\in\mathbb{N}$.
		
		Going back to Step 2 of the proof of Theorem \ref{Theorem Further regularity}, we observe that it is sufficient to bound the term $m_n \times \Delta m_n$, along with Step 1 to complete the proof of \eqref{Further regularity minimizing sequence}. Hence combining the arguments above, we conclude that the bound \eqref{Further regularity minimizing sequence} is independent of $n\in\mathbb{N}$.		
	\end{proof}
	From the bounds established in Lemma \ref{Lemma Further regularity minimizing sequence}, we can prove that the sequence $\{m_n\}_{n\in\mathbb{N}}$ is bounded in the space
	$L^2(\Omega ; L^{2}(0,T;H^2 )\cap L^2(\Omega ; W^{\gamma , p}(0,T;L^2) )$.

	We use the uniform bounds to show that the sequence of laws of $m_n$ is tight on the space $L^2(0,T;H^1)  \cap C([0,T]; L^2)$. Similarly, we use the uniform bound on the sequence of control processes $u_n$ to talk about tightness of laws on a suitable space. This is outlined in the following lemma.
	
	\begin{lemma}\label{tightness lemma minimizing sequence}
		The sequence of laws of $\l\{ \l(m_n , u_n\r) \r\}_{n\in\mathbb{N}}$ is tight on the space $L^2(0,T;H^1) \cap C([0,T]; L^2) \times L^2_w(0,T;L^2)$.
	\end{lemma}

	\begin{proof}[Proof of Lemma \ref{tightness lemma minimizing sequence}]
		The proof will be similar to the proof of Lemma \ref{tightness lemma}. This lemma shows tightness on a smaller (more regular) space than the previous counterpart. For completion, we give some details here. We show calculations for the sequence $\{m_n\}_{n\in\mathbb{N}}$. Tightness for the sequence of laws of $\{u_n\}_{n\in\mathbb{N}}$ follows similar to Lemma \ref{tightness lemma}. The main idea is to show that the laws of $m_n, n\in\mathbb{N}$ are concentrated inside a ball in the space $L^{\infty}(0,T;H^1)\cap L^2(0,T;H^2)\cap W^{\gamma , p}(0,T;L^2)$, which is compactly embedded into the space $L^2(0,T;H^1)\cap C([0,T];L^2)$.

		Towards that, let $r\in\mathbb{R}$ be arbitrary and fixed.
		\begin{align}\label{equation 1 tightness lemma minimizing sequence}
			\nonumber&\mathbb{P}_n\l(\l|m_n\r|_{L^{\infty}(0,T;H^1)\cap L^{2}(0,T;H^2)\cap W^{\gamma , p}(0,T;L^2)} \geq r\r) \\
			\nonumber&\leq \mathbb{P}_n\l(\l|m_n\r|_{L^{\infty}(0,T;H^1)} \geq \frac{r}{3}\r) + \mathbb{P}_n\l(\l|m_n\r|_{L^{2}(0,T;H^2)} \geq \frac{r}{3}\r) + \mathbb{P}_n\l(\l|m_n\r|_{ W^{\gamma , p}(0,T;L^2)} \geq \frac{r}{3}\r)\\
			\nonumber& \leq \frac{9}{r^2}\mathbb{E}^n \l|m_n\r|_{L^{\infty}(0,T;H^1)}^2 + \frac{9}{r^2}\mathbb{E}^n \l|m_n\r|_{L^{2}(0,T;H^2)}^2 + \frac{9}{r^2}\mathbb{E}^n \l|m_n\r|_{W^{\gamma , p}(0,T;L^2)}^2 \\
			& \leq \frac{C}{r^2}.
		\end{align}
		
		The second last inequality follows from the Chebyshev inequality. For the last inequality, Lemma \ref{bounds lemma 2 minimizing sequence} and Lemma \ref{Lemma Further regularity minimizing sequence}  imply the existence of a constant $C>0$ used in the inequality. Observe that the right hand side of the above inequality, and hence the left hand side can be made as small as desired by choosing $r$ large enough.\\
		Let 
		\begin{align}
			\nonumber B_r : = \bigg\{ &v\in L^{\infty}(0,T;H^1)\cap L^2(0,T;H^2)\cap W^{\gamma , p}(0,T;L^2) \\
			&: \l| v \r|_{L^{\infty}(0,T;H^1)\cap L^2(0,T;H^2)\cap W^{\gamma , p}(0,T;L^2)} \geq r \bigg\}.
		\end{align}
		Let $\varepsilon>0$ be given. In order to show tightness of the laws, it suffices to show that there exists a compact set $B^{\varepsilon}\subset L^2(0,T;H^1)\cap C([0,T];L^2)$ such that for each $n\in\mathbb{N}$,
		\begin{equation}
			\mathbb{P}_n \l(B^{\varepsilon}\r) > 1 - \varepsilon.
		\end{equation}
		In \eqref{equation 1 tightness lemma minimizing sequence}, we choose $r$ such that $r^2 > \frac{C}{\varepsilon}$. Therefore
		\begin{equation}
			\mathbb{P}_n \l(  B_r  \r) \leq \frac{C}{r^2} < \varepsilon.
		\end{equation}
		Let $B^{\varepsilon}$ denote the closure of the complement of this $B_r$. Therefore for each $n\in\mathbb{N}$, we have
		\begin{equation}
			\mathbb{P}_n \l( B^{\varepsilon} \r) \geq 1 - \mathbb{P}_n \l( B_r \r) > 1 - \varepsilon.
		\end{equation}
		
		By Lemma \ref{compact embedding of intersection 1} and Lemma \ref{compact embedding of intersection 2}, for $ \gamma p > 1 $, the set $B^{\varepsilon}$ is a compact subset of $L^2(0,T;H^1)\cap C([0,T];L^2)$.
		Hence the sequence of laws $\l\{ \mathcal{L}(m_n)\r\}_{n\in\mathbb{N}}$ is tight on the space $L^2(0,T;H^1)\cap C([0,T];L^2)$.\\
		The proof for the tightness of the sequence of laws of $u_n$ on the space $L^2_w(0,T;L^2)$ is similar to the proof of Lemma \ref{tightness lemma}.
		\dela{
			By Lemma \ref{compact embedding into space of continuous functions} and Lemma \ref{compact embedding of intersection 1}\adda{ This Lemma is not sufficient. Probably This works. Check it once. Lemma \ref{compact embedding of intersection 2}} in Appendix \ref{Section Some embeddings} we have the following compact embedding holds for $\gamma p >1$.
			
			\begin{align}\label{compact embedding minimizing sequence}
				\adda{L^{\infty}(0,T;H^1) \cap }L^{2}(0,T;H^2)\cap W^{\gamma , p}(0,T;L^2) \hookrightarrow L^2(0,T;H^1) \cap C([0,T]; L^2).
			\end{align}
			\adda{The embedding into $C(0,T;L^2)$ may not be accurate.\\
				Include the $L^{\infty}(0,T;H^1)$ norm in the tightness argument with "$r$".}\\		
			Rest of the proof is along the lines of the proof of Lemma \ref{tightness lemma}.\\
			The proof for the tightness of the sequence of laws of $u_n$ on the space $L^2_w(0,T;L^2)$ is similar to the proof of Lemma \ref{tightness lemma}.
			
			\adda{Combining which two arguments? The tightness argument for $u$ has not been given here. Mention at the start that we give the arguments that are different from the aforementioned Lemma. Hence we give only the arguments for $m$. Also, should the tightness of $W$ be written in the same Lemma?? }Combining the above two arguments with the compact embedding in \eqref{compact embedding minimizing sequence} concludes that the sequence of laws of $\l(m_n , u_n\r)$ is tight on the space $L^2(0,T;H^1) \cap C([0,T]; L^2) \times L^2_w(0,T;L^2)$.
		}		
	\end{proof}
	
	Note that each strong martingale solution has its own Wiener process.
	The processes $W_n$ have the same laws on $C([0,T];\mathbb{R})$. Hence it is sufficient to show that the law of $W_n$ is tight on the space $C([0,T] ; \mathbb{R})$ for any $n\in\mathbb{N}$.
	
	Let $n\in\mathbb{N}$. Since the space $C ([0,T] ; \mathbb{R})$ is a Radon space, every probability measure is tight. Hence, given $\varepsilon > 0$  there exists $K_{\varepsilon}\subset C ([0,T] ; \mathbb{R})$ such that
	\begin{align}
		\mathbb{P}_n\l( W_n \in K_{\varepsilon} \r) \geq 1 - \varepsilon.
	\end{align}
	Since $W_n$ and $W_k$ have the same laws on the space $C([0,T] ; \mathbb{R})$, for any $n,k\in\mathbb{N}$,
	\begin{align}
		\mathbb{P}_n\l( W_n \in K_{\varepsilon} \r) = \mathbb{P}_k\l( W_k \in K_{\varepsilon} \r) \geq 1 - \varepsilon.
	\end{align}
	Hence the sequence of laws of $\{W_n\}_{n\in\mathbb{N}}$ is tight on the space $C([0,T] ; \mathbb{R})$.

	Now that we have shown the tightness, we proceed as done in Section \ref{Section Proof of existence of a solution}.

	\dela{	Hence in Proposition \ref{Proposition use of Skorokhod theorem minimizing sequence}, the following convergence can be given.
		\begin{align}
			\m_n \to \m\ \text{in}\ L^2(0,T;H^1)\cap C([0,T] ; (H^1)^\p)\ \mathbb{P}-a.s.
		\end{align}
		
		\textbf{Note:}This convergence is not used anywhere else in the proof except for the last part where we need the approximation of $\Psi(\m(T))$ in terms of $\Psi(\m_n(T))$.}
	
	\begin{proposition}\label{Proposition use of Skorokhod theorem minimizing sequence}
		There exists a probability space $\l(\Omega^\p , \mathcal{F}^\p , \mathbb{P}^\p\r)$ and a sequence $\l(\m_n , \u_n , W_n^\p\r)$ of $L^2(0,T;H^1)\cap C([0,T]; L^2) \times L^2_w(0,T;L^2) \times C([0,T]; \mathbb{R})$-valued random variables, along with random variables $\l(\m , \u \ W^\p\r)$ defined on $\Omega^\p$ such that for each $n\in\mathbb{N}$, the law of $\l(m_n , u_n , W_n\r)$ equals the law of $\l(\m_n , \u_n , W_n^\p\r)$ on $L^2(0,T;H^1)\cap C([0,T]; L^2\dela{(H^1)^\p}) \times L^2_w(0,T;L^2) \times C([0,T]; \mathbb{R})$ and the following convergences hold $\mathbb{P}^\p$-a.s. as $n$ goes to infinity.
		\begin{equation}\label{Equation Convergence of mn prime to m prime Skorohod Theorem minimizing sequence}
			\m_n \to \m\ \text{in}\ L^2(0,T;H^1)\cap C([0,T]; L^2) ,
		\end{equation}
		\begin{equation}\label{eqn-u_n^prime to u^prime}
			\u_n \to \u\ \text{in}\ L^2_w(0,T;L^2),
		\end{equation}
		\begin{equation}
			W_n^\p \to W^\p\ \text{in}\ C([0,T]; \mathbb{R}).
		\end{equation}
	\end{proposition}
	\begin{proof}[Proof of Proposition \ref{Proposition use of Skorokhod theorem minimizing sequence}]
		The proof, similar to the proof of Proposition \ref{Use of Skorohod theorem}, follows from the Jakubowski version of the Skorohod Theorem, see Theorem 3.11 in \cite{ZB+EM}.  
		
		\dela{\\\adda{How do we use \eqref{eqn-u_n^prime to u^prime}?}}
		
	\end{proof}

	\begin{remark}\label{Remark m prime, u prime progressively measurable minimizing sequence}
		The processes $\m$ and $\u$ obtained in Proposition \ref{Proposition use of Skorokhod theorem minimizing sequence} are Borel measurable. Let the filtration $\mathbb{F}^{\prime} = \mathcal{F}_{t\in[0,T]}^{\prime}$ be defined by
		
		\begin{align*}
			\mathcal{F}_t^{\prime} = \sigma \{ \m(s) , \u(s) , W^{\prime}(s): 0 \leq s \leq t\}.
		\end{align*}
		Hence $\m , \u$ are $\mathbb{F}^{\prime}$-adapted. Thus, the processes $\m$ and $\u$ have progressively measurable modifications, see Proposition 1.12, \cite{Karatzas+Steven_BrownianMotionStochacsticCalculus}. From now on, these progressively measurable modifications will be considered.
	\end{remark}

	\begin{remark}\label{remark same bounds remark minimizing sequence}
		This remark is written in the same spirit as that of Remark \ref{same bounds remark}. The main difference between Remark \ref{same bounds remark} and this Remark \ref{remark same bounds remark minimizing sequence} is that 
		we cannot use the finite dimensionality of the spaces $H_n$ here. Let us show how we need to modify the previous argument. First, we discuss the laws of $m_n,n\in\mathbb{N}$, and next we discuss the laws of $u_n,n\in\mathbb{N}$.

		\begin{enumerate}
			\item  Note that the spaces $C([0,T];L^2)$, $C([0,T];H^1)$, $L^2(0,T;H^1)$, $L^4(0,T;W^{1,4})$, and $L^2(0,T;H^2)$ are Polish spaces. In particular, since the embedding of $C([0,T];H^1)$ into the space $C([0,T];L^2) \cap L^2(0,T ; H^1)$ is continuous and injective, 
			by using the Kuratowski Theorem, Lemma \ref{Lemma Kuratowski}, we infer that the Borel subsets of $C([0,T];H^1)$ are also the Borel subsets of $C([0,T];L^2) \cap L^2(0,T ; H^1)$.
			Now, since by Lemma \ref{Lemma continuous in time with values in H1.} $\mathbb{P}_n\l\{ m_n \in C([0,T];H^1)\r\} = 1$ for each $n$ and $m_n$ and $\m_n$ have the same laws on $C([0,T];L^2) \cap L^2(0,T ; H^1)$ and $C([0,T];H^1)$ is a Borel subset of $C([0,T];L^2) \cap L^2(0,T ; H^1)$, we deduce the following
			\begin{equation*}
				\mathbb{P}^\p\l\{ \m_n \in C([0,T];H^1)\r\} = 1,\ \text{for each}\ n\in\mathbb{N}.
			\end{equation*}
			Arguing similarly (i.e. using the continuous embedding of the spaces $C([0,T];H^1)$, $L^2(0,T;H^1)$, $L^4(0,T;W^{1,4})$, and $L^2(0,T;H^2)$ into the space $C([0,T];L^2) \cap L^2(0,T ; H^1)$),
			we can prove that the processes $\m_n,n\in\mathbb{N}$ satisfy the same bounds as the processes $m_n,n\in\mathbb{N}$, in particular the bounds $(1)$, $(2)$ and $(3)$ in Lemma \ref{bounds lemma 1}.
			\item  
			Regarding the control processes corresponding to the processes $u_n$ and $u^\p_n$, we have the following. Firstly, the space $L^2_w(0,T;L^2)$ is the space $L^2(0,T;L^2)$ endowed with the weak topology, which is weaker than the norm topology. Therefore every open set in $L^2_w(0,T;L^2)$ is also an open set in $L^2(0,T;L^2)$. Therefore, the Borel sigma-algebra corresponding to $L^2_w(0,T;L^2)$ is contained in the Borel sigma-algebra corresponding to $L^2(0,T;L^2)$. In other words, Borel subsets of $L^2_w(0,T;L^2)$ are also Borel subsets of $L^2(0,T;L^2)$.
			Moreover, by Theorem 7.19 in \cite{Zizler_2003_Book_NonseparableBanachSpaces}, see also page number 112 in \cite{ZB+Ferrari_2019_StationarySolutions_DampedSNSE},  we infer that the Borel sigma algebras corresponding to $L^2_w(0,T;L^2)$ and $L^2(0,T;L^2)$ are equal. 
			\dela{

				The processes $u_n$ and $u^\p_n$ have the same laws on $L^2_w(0,T;L^2)$, for each $n\in\mathbb{N}$. Therefore
				we can prove the following for every Borel subset $B$ of $L^2_w(0,T;L^2)$:
				\begin{equation}\label{eqn same laws equality control sequence}
					\mathbb{P}\l\{ u_n \in B \r\} = \mathbb{P}^\p\l\{ u_n^\p \in B \r\}.
				\end{equation}
				\coma{Note the processes $u_n : \Omega \times [0,T] \to L^2$ is measurable. This induces a function $\Phi_{u_n} : \Omega \to L^2_{\text{w}}(0,T;L^2)$ given by
					\begin{equation*}
						\bigl( \Phi_{u_n}(\omega) \bigr) (t) = u_n(\omega , t).
					\end{equation*}
					Similar can be said for the processes $u_n^\p$, giving a function
					\begin{equation*}
						\bigl( \Phi_{u_n^\p}(\omega) \bigr) (t) = u^\p_n(\omega , t).
					\end{equation*}
					Now, the laws of the processes $u_n$ and $u_n^\p$ on the space $L^2_{\text{w}}(0,T;L^2)$ are equal. That is,
					\begin{equation*}
						\mathbb{P}\l( \bigl( \Phi_{u_n}(\omega) \bigr) (t) \in B \r) = \mathbb{P}^\p\l( \bigl( \Phi_{u_n^\p}(\omega) \bigr) (t) \in B \r).
					\end{equation*}
					Therefore, we can prove that
					\begin{equation}
						\mathbb{P}\l\{ u_n \in B \r\} = \mathbb{P}^\p\l\{ u_n^\p \in B \r\}.
					\end{equation}
					Therefore, we can prove in particular that for any constant $K>0$
					\begin{equation*}
						\mathbb{P} \l\{ \l| u_n \r|_{L^2}  \leq K  \r\} = \mathbb{P}^\p \l\{ \l| u^\p_n \r|_{L^2}  \leq K  \r\}.
					\end{equation*}
				}
				
				Note that the norm $\l|\cdot\r|_{L^2}$ is a measurable function (lower semicontinuous as an extended real valued function, and hence also Borel measurable) on the space $L^2_w(0,T;L^2)$.
				In particular, for every closed ball $B$ (in the space $L^2(0,T;L^2)$, the above equality \eqref{eqn same laws equality control sequence} holds.
				In particular, for any constant $K>0$, we deduce the following.
				\begin{equation*}
					\mathbb{P} \l\{ \l| u_n \r|_{L^2}  \leq K  \r\} = \mathbb{P}^\p \l\{ \l| u^\p_n \r|_{L^2}  \leq K  \r\}.
				\end{equation*}

			}
			By Proposition \ref{Proposition use of Skorokhod theorem minimizing sequence}, we infer that for each $n\in\mathbb{N}$, the law of the process $u_n^\p$ is equal to the law of the process $u_n$ on the space $L^2_{\text{w}}(0,T;L^2)$. In particular, the following holds for any constant $K>0$.
			\begin{equation*}
				\mathbb{P} \l\{ \l| u_n \r|_{L^2(0,T;L^2)}  \leq K  \r\} = \mathbb{P}^\p \l\{ \l| u^\p_n \r|_{L^2(0,T;L^2)}  \leq K  \r\}.
			\end{equation*}    
			Hence we infer that the processes $u^\p_n$ satisfy the same bounds as the processes $u_n$.
			
		\end{enumerate}
	\end{remark}
	
	The processes $\m_n$ and $\u_n$, therefore, satisfy the same bounds as the processes $m_n$ and $u_n$ respectively, for each $n\in\mathbb{N}$. We state this in the following two lemmata.
	
	\dela{\adda{Where do you use the Lemma below?}}
	\begin{lemma}\label{lem-bounds on mn prime minimizing sequence}
		There exists a constant $C>0$ such that for all $n\in\mathbb{N}$, the following bounds hold.
		
		\begin{align}\label{bound on m_n prime L infty H1 minimizing sequence}
			\mathbb{E}^\p \sup_{t\in[0,T]}\l| \m_n(t) \r|_{H^1}^2 \leq C,
		\end{align}

		\begin{equation}\label{bound on m_n prime times Delta m_n prime minimizing sequence}
			\mathbb{E}^\p \int_{0}^{T} \l|\m_n(s) \times \Delta \m_n(s)\r|_{L^2}^2 \, ds \leq C,
		\end{equation}

		\begin{equation}\label{bound on m_n prime times m_n prime times Delta m_n prime minimizing sequence}
			\mathbb{E}^\p \int_{0}^{T} \l| \m_n(s) \times \l(\m_n(s) \times \Delta \m_n(s)\r) \r|_{L^2}^2 \, ds \leq C,
		\end{equation}
		
		\begin{equation}\label{bound on m_n prime times u_n prime minimizing sequence}
			\mathbb{E}^\p \int_{0}^{T} \l|\m_n(s) \times \u_n(s)\r|_{L^2}^2 \, ds \leq C,
		\end{equation}
		
		\begin{equation}\label{eqn-bound on m_n prime times m_n prime times u_n prime minimizing sequence}
			\mathbb{E}^\p \int_{0}^{t} \l| \m_n(s) \times \l(\m_n(s) \times \u_n(s)\r)\r|_{L^2}^2 \, ds \leq C.
		\end{equation}
	\end{lemma}
	
	\begin{proof}[Proof of Lemma \ref{lem-bounds on mn prime minimizing sequence}]
		The proof of this Lemma is similar to the proof of Proposition \ref{Proposition bounds on m_n prime}. It follows from the bounds established in Lemma \ref{bounds lemma 1 minimizing sequence}.
	\end{proof}
	We now use Lemma \ref{Lemma Further regularity minimizing sequence} along with the Remark \ref{remark same bounds remark minimizing sequence} to get the following lemma.
	\begin{lemma}\label{lem-Further regularitiy mn prime minimizing sequence}
		There exists a constant $C>0$ such that for all $n\in\mathbb{N}$,
		\begin{align}
			\mathbb{E}^\p \left( \int_{0}^{T}\l|\m_n(t)\r|_{W^{1,4}}^4\, dt  + \int_{0}^{T} |\m_n(t)|_{H^2}^2\, dt\right) \leq C.
		\end{align}
	\end{lemma}
	\begin{proof}[Proof of Lemma \ref{lem-Further regularitiy mn prime minimizing sequence}]
		The proof follows from the Lemma \ref{Lemma Further regularity minimizing sequence} and Remark \ref{remark same bounds remark minimizing sequence}.
	\end{proof}
	\dela{We now show that the process $\m$ is a strong martingale solution to the problem \eqref{problem considered for optimal control part} corresponding to the control process $\u$ on the probability space $\l(\Omega^\p , \mathcal{F}^\p , \mathbb{P}^\p\r)$.}
	
	Having shown uniform estimates for the sequence $\{\m_n\}_{\mathbb{N}}$, we show similar bounds for the limit process $\m$.\dela{Firstly, we show that the process $\m$ satisfies the required bounds.}
	\begin{lemma}\label{lem-bounds lemma for m prime minimizing sequence}
		The process $\m$ satisfies the following bounds.
		
		\begin{enumerate}
			
			\item		\begin{equation}
				\sup_{0\leq t\leq T}\l|m^{\prime}(t)\r|_{L^2} \leq |m_0|_{L^2},\ \mathbb{P}^{\prime}-\text{a.s.}
			\end{equation}
			
			\item 
			\begin{equation}
				\mathbb{E}^{\prime}\sup_{0\leq t\leq T}\l|m^{\prime}(t)\r|^{4}_{H^1} < \infty,
			\end{equation}
			
			\item \begin{equation}
				\mathbb{E}^\p\int_{0}^{T}\l|\m(t)\r|_{W^{1,4}}^4 \, dt < \infty,
			\end{equation}

			\dela{\item \begin{equation}
					\mathbb{E}^\p\int_{0}^{T}\l|\m(t)\r|_{D(A_1)}^2 \, dt < \infty.
			\end{equation}}
			
			\item \begin{equation}
				\mathbb{E}^\p\int_{0}^{T}\l|\m(t)\r|_{H^2}^2 \, dt < \infty.
			\end{equation}
			
		\end{enumerate}
	\end{lemma}
	\begin{proof}
		The proof is essentially similar to the proof of Lemma \ref{Lemma extension of norms and lower semicontinuity}.\dela{ We give here a proof for the last two inequalities.} A sketch for the proofs of the last two inequalities is given here.

		For the last inequality, we first extend the norm $\l|\cdot\r|_{H^2}$ to the space $H^1$ as follows.
		\begin{align*}
			\l|v\r|_{H^2} = \begin{cases}
				&\l|v\r|_{H^2},\ \mbox{ if }\ v\in H^2, \\
				& \infty, \mbox{ if }\ v\in H^1\ \text{and}\ v \notin H^2.
			\end{cases}
		\end{align*}
		This extended norm is lower semicontinuous.
		\dela{
			\textbf{A brief proof of lower semicontinuity:}\adda{This proof can be referred to some other work. No need of giving details here.}
			This can be shown following the proof of Lemma \ref{Lemma extension of norms and lower semicontinuity}. The proof can be as follows.
			Let us fix some $r>0$. Consider the set $B = \{v \in H^1 : \l|v\r|_{H^2} \leq r\}$. To show lower semicontiuity, it is sufficient to show that $B$ is a closed subset of $H^1$. To this end, we consider a sequence $\l(v_n\r)_{n\in\mathbb{N}}$ such that $v_n$ converges to $v$ in $H^1$ for some $v\in H^1$. Our aim is to show that $v\in H^2$. By the definition of the set $B$, the sequence $\l(v_n\r)_{n\in\mathbb{N}}$ is bounded in $H^2$ and hence has a weakly convergent subsequence in $H^2$. Let this sequence be denoted by $v_{n_{k}}$ and let it converge weakly to some element $v^\p\in H^2$. Since the space $H^2$ is compactly embedded into the space $H^1$, the sequence $v_{n_{k}}$ has a further subsequence that converges to $v^\p$ strongly in $H^1$. Hence by uniqueness of limit, $v=v^\p$ and hence $ v\in H^2 $. It remains to show that $\l|v\r|_{H^2} \leq r$. The norm $\l|\centerdot\r|_{H^2}$ is weakly lower semicontinuous on the space $H^2$.
			Hence
			\begin{align*}
				\l|v\r|_{H^2} \leq \liminf_{k\rightarrow\infty}\l|v_{n_k}\r|_{H^2} \leq r.
			\end{align*}
			This concludes the proof for lower semicontinuity of the extended norm.
		}		
		Therefore the following holds for each $t\in[0,T]$.
		\begin{align*}
			\l|\m(t)\r|_{H^2}^2 \leq \liminf_{n\rightarrow\infty} \l|\m_n(t)\r|_{H^2}^2.
		\end{align*}
		Hence by the Fatou Lemma,
		\begin{align*}
			\mathbb{E}^\p \int_{0}^{T} 	\l|\m(t)\r|_{H^2}^2  \, dt \leq \liminf_{n\to\infty}\mathbb{E}^\p \int_{0}^{T} 	\l|\m_n(t)\r|_{H^2}^2  \, dt.
		\end{align*}
		The bound in Lemma \ref{lem-Further regularitiy mn prime minimizing sequence} implies that the right hand side of the above inequality is finite. This concludes the proof.
		
		For the second last inequality,we extend the norm $\l|\centerdot\r|_{L^4(0,T;W^{1,4})}$ to the space $L^2(0,T;L^2)$ as follows
		\begin{align*}
			\l|v\r|_{L^4(0,T;W^{1,4})} = \begin{cases}
				& \l|v\r|_{L^4(0,T;W^{1,4})},\ \text{if}\ v\in L^4(0,T;W^{1,4}), \\
				& \infty,\ \text{if}\ v\in L^2(0,T;L^2)\ \text{and}\ v\notin L^4(0,T;W^{1,4}).
			\end{cases}
		\end{align*}
		
		\dela{Another way to go about this is to first show that the obtained process $\m$ is a weak martingale solution to the problem \eqref{problem considered}. Then show that the process also satisfies the "Further/ Maximal Regularity" condition, and hence is a strong martingale solution.}
		\dela{Or Replace the convergence ($\mathbb{P}$-a.s.) in $L^2(0,T;H^1)$ by $L^2(0,T;L^2)$. The space $W^{1,4}$ is continuously embedded into the space $H^1$, which is compactly embedded into the space $L^2$. Hence the norm $\l|\cdot\r|_{W^{1,4}}$ is lower semicontinuous when extended to the space $L^2$.}
		The above defined map is lower semicontinuous.
			%
		Therefore the following holds for each $t\in[0,T]$ $\mathbb{P}^\p$-a.s.
		\begin{align*}
			\l|\m(t)\r|_{L^4(0,T;H^1)} \leq \liminf_{n\rightarrow\infty} \l|\m_n(t)\r|_{L^4(0,T;H^1)}.
		\end{align*}
		Hence by the Fatou Lemma,
		\dela{
			\begin{align*}
				\mathbb{E}^\p \int_{0}^{T} \l|\m(t)\r|^4_{L^4(0,T;H^1)} \, dt \leq \liminf_{n\rightarrow\infty} \mathbb{E}^\p \int_{0}^{T} \l|\m_n(t)\r|^4_{L^4(0,T;H^1)} < \infty.
			\end{align*}
		}
		\begin{align}
			\mathbb{E}^\p \int_{0}^{T} \l|\m(t)\r|^4_{L^4(0,T;W^{1,4})} \, dt \leq \liminf_{n\rightarrow\infty} \mathbb{E}^\p \int_{0}^{T} \l|\m_n(t)\r|^4_{L^4(0,T;W^{1,4})} < \infty.
		\end{align}
		This concludes the proof of the Lemma \ref{lem-bounds lemma for m prime minimizing sequence}.
	\end{proof}
	This concludes the Auxilliary results. We now use them to prove Theorem \ref{Theorem existence of optimal control}.
	
	\begin{proof}[Continuation of the proof of Theorem \ref{Theorem existence of optimal control}]
		
		We now show that the obtained limit is a strong martingale solution to the problem \eqref{problem considered for optimal control part}. For this aim, we first show that it is a weak martingale solution, for which we need to show that the process $\m$ satisfies \eqref{problem considered for optimal control part} with the corresponding probability space.
		
		The main difference between this proof and the proof of Theorem \ref{Theorem Existence of a weak solution} is that now the solutions no more have values in finite dimensional spaces. Previously, we employed results for the finite dimensional spaces $H_n$. Now the solutions are in infinite dimensional spaces. 
		The core of the arguments and the overall structure remains the same. Moreover, for the convergence arguments, the projection operators and the cut-off are absent. But this is more of a simplification.
		
		We have pointwise convergence of $\m_n$ to $\m$ from Proposition \ref{Proposition use of Skorokhod theorem minimizing sequence}. We now show that the convergence is in stronger sense. The bound in \eqref{bound on m_n prime L infty H1 minimizing sequence}, Lemma \ref{lem-bounds on mn prime minimizing sequence} gives us that the sequence $\{\m_n\}_{n\in\mathbb{N}}$ is uniformly integrable. Hence by the Vitali convergence theorem, see Theorem 4.5.4 in \cite{BogachevBook_MeasureTheor_2007}, we have the following convergence as $n$ goes to infinity.
		\begin{equation}\label{Strong convergence of mn prime to m prime minimizing sequence}
			\m_n \to \m\ \text{in}\ L^2(\Omega^\p ; L^2(0,T;H^1)).
		\end{equation}

		We give here some details of the convergence arguments for the terms containing $\u_n$. The arguments for other terms are similar to the previous arguments, see Section \ref{Section Proof of existence of a solution}.
		
		From \eqref{eqn-u_n^prime to u^prime}, we get the pointwise convergence of $\u_n$ to $\u$. To show that $\m$ satisfies the required equation, it suffices now to show the following assertions for each $t\in[0,T]$ and $\phi\in L^2(\Omega^\p:H^1)$.
		\begin{equation}\label{eqn-1001}
			\lim_{n\rightarrow\infty}\mathbb{E^{\prime}}\l[ \int_{0}^{t} \l\langle \bigl( \m_n(s) \times \u_n(s) - \m(s) \times \u(s) \bigr) , \phi \r\rangle_{L^2} \, ds \r]=0
		\end{equation}
		and 
		\begin{equation}\label{eqn-1002}
			\lim_{n\rightarrow\infty}\mathbb{E^{\prime}}\l[ \int_{0}^{t} \l\langle \bigl( \m_n(s) \times \bigl( \m_n(s) \times \u_n(s) \bigr) - \m(s) \times \bigl( \m(s) \times \u(s) \bigr) \bigr) , \phi \r\rangle_{L^2} \, ds \r]=0,
		\end{equation}
		
		We will give details of the proof of \eqref{eqn-1001}.  The proof of \eqref{eqn-1002} follows the suite.\\
		First, we have the convergence given in \eqref{eqn-u_n^prime to u^prime} in Proposition \ref{Proposition use of Skorokhod theorem minimizing sequence}. Fix $t\in[0,T]$ and $\psi\in L^2\l(\Omega^\p ; L^2\l( 0,T; L^2 \r)\r)$. For $n\in\mathbb{N}$, let us define auxilliary functions $f_n: \Omega^\p\to\mathbb{R}$ as follows.
		\begin{equation}
			f_n(\omega^\p) := \int_{0}^{t} \l\langle \u_n(s,\omega^\p) - \u(s,\omega^\p) , \psi(s,\omega^\p) \r\rangle_{L^2} \, ds,\ \omega^\p\in\Omega^\p.
		\end{equation}
		By Assumption \ref{Assumption admisibility criterion} \dela{\eqref{bound on u_n prime}} and \eqref{intermediate liminf bound for u prime for control part} (established in the subsequent calculations), each $f_n$ is well defined. Moreover by \eqref{eqn-u_n^prime to u^prime}, we have
		\begin{equation}
			f_n \to 0,\ \mathbb{P}^\p-a.s.
		\end{equation}
		Note that \eqref{eqn-u_n^prime to u^prime} gives weak convergence for the processes on the space $L^2(0,T;L^2)$. This result can be used for $t\in[0,T]$ instead of $T$ by considering $\chi_{[0,t]}\psi$ instead of $\psi$.
		The idea now is to use the Vitali convergence theorem \cite{BogachevBook_MeasureTheor_2007} to show the convergence (of $\{f_n\}_{n\in\mathbb{N}}$) in $L^{1}(\Omega^\p:\mathbb{R})$. To use the Vitali convergence theorem, it suffices to show that the sequence $\{f_n\}_{n\in\mathbb{N}}$ is uniformly bounded in $L^{\frac{4}{3}}(\Omega^\p:\mathbb{R})$.\dela{ This is one of the places where the assumption on $u$ for $p=2$ is used. COmment no more needed. Also, some details here can be skipped.}
		\begin{align*}
			\mathbb{E^{\prime}} \l| f_n(\omega^\p) \r|^\frac{4}{3} = & \mathbb{E^{\prime}} \l| \int_{0}^{t} \l\langle \u_n(s,\omega^\p) - \u(s,\omega^\p) , \psi(s,\omega^\p) \r\rangle_{L^2} \, ds \r|^{\frac{4}{3}} \\
			\leq & \mathbb{E^{\prime}} \l( \int_{0}^{t} \l| \l\langle \u_n(s,\omega^\p) - \u(s,\omega^\p) , \psi(s,\omega^\p) \r\rangle_{L^2}\r| \, ds \r)^{\frac{4}{3}} \\
			\leq & \mathbb{E^{\prime}} \l( \int_{0}^{t} \l|  \u_n(s,\omega^\p) - \u(s,\omega^\p) \r|_{L^2}^2 \, ds \r)^\frac{2}{3} \l( \int_{0}^{t} \l| \psi(s,\omega^\p) \r|_{L^2}^2 \, ds \r)^\frac{2}{3} \\
			\leq & \l[\mathbb{E^{\prime}} \Bigl[  \l( \int_{0}^{t} \l|  \u_n(s,\omega^\p) - \u(s,\omega^\p) \r|_{L^2}^2 \, ds \r)^{2} \Bigr] \r]^{\frac{1}{3}} \l[\mathbb{E^{\prime}}  \l(\int_{0}^{t} \l| \psi(s,\omega^\p) \r|_{L^2}^2 \, ds \r) \r]^{\frac{2}{3}} \leq C_{p=2},
		\end{align*}
		for some constant $C_{p=2}$ independent of $n\in\mathbb{N}$. The existence of such a constant $C_{p=2}$ is guaranteed by Assumption \ref{Assumption admisibility criterion} \dela{\eqref{bound on u_n prime}} and \eqref{intermediate liminf bound for u prime for control part}.
		Therefore the sequence of processes $\{f_n\}_{n\in\mathbb{N}}$ is uniformly bounded in $L^{\frac{4}{3}}(\Omega^\p:\mathbb{R})$, and hence uniformly integrable on $L^{1}(\Omega^\p:\mathbb{R})$. Therefore, by the Vitali convergence theorem, we have
		\dela{\begin{equation}
				\lim_{n\rightarrow\infty} \mathbb{E^{\prime}} \l| f_n \r|^{\frac{4}{3}} = 0.
			\end{equation}
			Moreover by the continuous embedding $L^q(\Omega^\p:\mathbb{R})\in L^{\frac{4}{3}}(\Omega^\p:\mathbb{R}),1\leq q < \frac{4}{3}$, we have
			\begin{equation}
				\lim_{n\rightarrow\infty} \mathbb{E^{\prime}} \l| f_n \r|^{q} = 0.
			\end{equation}
			In particular for $q = 1$,\dela{ for $\psi\in L^2$,}}
		\begin{align}\label{Equation Weak convergence un prime to u prime minimizing sequence}
			\lim_{n\rightarrow\infty} \mathbb{E^{\prime}} \l| f_n \r| = & \lim_{n\rightarrow\infty} \mathbb{E^{\prime}} \l| \int_{0}^{t} \l\langle \u_n(s,\omega^\p) - \u(s,\omega^\p) , \phi \r\rangle_{L^2} \, ds \r| = 0,
		\end{align}
		thereby giving
		\begin{equation}\label{eqn-1003}
			\u_n\to\u \text{ weakly in }L^2\l(\Omega^\p:L^2\l(0,T;L^2\r)\r)
		\end{equation}
		Let us choose and fix $\phi\in L^2(\Omega^\p:L^4(0,T;H^1))$. Then we have the following equality
		\begin{align}\label{Equation Intermediate for convergence of mn prime times un prime minimizing sequence}
			\nonumber \mathbb{E} \int_{0}^{t} \l\langle \m_n(s) \times \u_n(s) - \m(s)\times \u(s) , \phi \r\rangle_{L^2} \, ds = &
			\mathbb{E} \int_{0}^{t} \l\langle \bigl( \m_n(s) - \m(s) \bigr) \times \u_n(s) , \phi \r\rangle_{L^2} \, ds \\
			& + \mathbb{E} \int_{0}^{t} \l\langle \m(s) \times \bigl( \u_n(s) - \u(s) \bigr) , \phi \r\rangle_{L^2} \, ds.
		\end{align}
		In what follows we are going to prove that  the first term on the right hand side of equality \eqref{Equation Intermediate for convergence of mn prime times un prime minimizing sequence} converges to $0$. 
		For this aim  we have the following sequence of inequalities.
		\dela{\begin{align*}
				& \l| \mathbb{E} \int_{0}^{t} \l\langle \l( \m_n(s) - \m(s) \r) \times\u_n(s) , \phi \r\rangle_{L^2} \, ds \r| 
				\leq  \mathbb{E} \int_{0}^{t} \l| \l\langle \l( \m_n(s) - \m(s) \r) \times\u_n(s) , \phi(s) \r\rangle_{L^2} \r| \, ds \\
				& \leq  \mathbb{E} \int_{0}^{t} \l| \m_n(s) - \m(s) \r|_{L^4} \l| \u_n(s) \r|_{L^2}  \l| \phi(s) \r|_{L^4}  \, ds \\
				& \leq  \mathbb{E} \l[ \l(\int_{0}^{t} \l| \m_n(s) - \m(s) \r|_{L^4}^4 \, ds \r)^{\frac{1}{4}}
				\l(\int_{0}^{t} \l| \u_n(s) \r|_{L^2}^2 \, ds \r)^{\frac{1}{2}}
				\l( \int_{0}^{t} \l|\phi(s)\r|_{L^4}^4 \, ds \r)^{\frac{1}{4}} \r] \\
				\leq &  C \l[\mathbb{E} \l( \int_{0}^{t} \l| \m_n(s) - \m(s) \r|_{L^4}^4   \, ds \r)^2 \r]^{\frac{1}{4}}
				\l[\mathbb{E} \l(\int_{0}^{t}  \l| \u_n(s) \r|_{L^2}^2 \, ds \r)^2 \r]^{\frac{1}{4}}
				\l[\mathbb{E} \l( \int_{0}^{t} \l|\phi(s)\r|_{H^1}^4 \, ds \r)^{\frac{1}{2}} \r]^{\frac{1}{2}}\\
				\leq &  C(\phi) \l[\mathbb{E} \int_{0}^{t} \l| \m_n(s) - \m(s) \r|_{L^4}^2   \, ds  \r]^{\frac{1}{4}}
				\l[\mathbb{E} \l(\int_{0}^{t}  \l| \u_n(s) \r|_{L^2}^2 \, ds \r)^2 \r]^{\frac{1}{4}} \\
				\leq &  C(\phi) \l[\mathbb{E} \int_{0}^{t} \l| \m_n(s) - \m(s) \r|_{H^1}^2   \, ds  \r]^{\frac{1}{4}}
				\l[\mathbb{E} \l(\int_{0}^{t}  \l| \u_n(s) \r|_{L^2}^2 \, ds \r)^2 \r]^{\frac{1}{4}} \\
				\leq & C(K_2) C(\phi) \l[\mathbb{E} \int_{0}^{t} \l| \m_n(s) - \m(s) \r|_{H^1}^2   \, ds  \r]^{\frac{1}{4}}.
		\end{align*}}

		\begin{align*}
			& \l| \mathbb{E} \int_{0}^{t} \l\langle \l( \m_n(s) - \m(s) \r) \times\u_n(s) , \phi \r\rangle_{L^2} \, ds \r| 
			\leq  \mathbb{E} \int_{0}^{t} \l| \l\langle \l( \m_n(s) - \m(s) \r) \times\u_n(s) , \phi(s) \r\rangle_{L^2} \r| \, ds \\
			& \leq  \mathbb{E} \int_{0}^{t} \l| \m_n(s) - \m(s) \r|_{L^4} \l| \u_n(s) \r|_{L^2}  \l| \phi(s) \r|_{L^4}  \, ds \\
			& \leq  \mathbb{E} \l[ \l(\int_{0}^{t} \l| \m_n(s) - \m(s) \r|_{L^4}^4 \, ds \r)^{\frac{1}{4}}
			\l(\int_{0}^{t} \l| \u_n(s) \r|_{L^2}^2 \, ds \r)^{\frac{1}{2}}
			\l( \int_{0}^{t} \l|\phi(s)\r|_{L^4}^4 \, ds \r)^{\frac{1}{4}} \r] \\
			& \leq  C \l[\mathbb{E}  \int_{0}^{t} \l| \m_n(s) - \m(s) \r|_{L^4}^4   \, ds  \r]^{\frac{1}{4}}
			\l[\mathbb{E} \l(\int_{0}^{t}  \l| \u_n(s) \r|_{L^2}^2 \, ds \r) \r]^{\frac{1}{4}}
			\l[\mathbb{E}  \int_{0}^{t} \l|\phi(s)\r|_{H^1}^4 \, ds  \r]^{\frac{1}{2}}\\
			& \leq C \l[\mathbb{E}  \int_{0}^{t} \l| \m_n(s) - \m(s) \r|_{L^4}^4   \, ds  \r]^{\frac{1}{4}}.
		\end{align*}
		Hence our claim follows by applying earlier proven assertion \eqref{Strong convergence of mn prime to m prime minimizing sequence}.
		The last inequality is a consequence of our assumption \ref{Assumption admisibility criterion}\dela{\ref{???}}.		
		By \eqref{Strong convergence of mn prime to m prime minimizing sequence}, the right hand side, and hence the left hand side of the above inequality goes to 0 as $n$ goes to infinity.\\
		For the second term in \eqref{Equation Intermediate for convergence of mn prime times un prime minimizing sequence},
		\begin{align*}
			\mathbb{E} \int_{0}^{t} \l\langle \m(s) \times \l( \u_n(s) - \u(s) \r) , \phi \r\rangle_{L^2} \, ds = \mathbb{E} \int_{0}^{t} \l\langle  \l( \u_n(s) - \u(s) \r) , \m(s) \times \phi \r\rangle_{L^2} \, ds.
		\end{align*}
		By the constraint condition \eqref{eqn-constraint condition},
		\begin{equation*}
			\m \times \phi \in L^2(\Omega^\p ; L^2(0,T;L^2)).
		\end{equation*}
		Hence by \eqref{eqn-1003} we infer  the right hand side above converges  to $0$ as $n$ goes to infinity.\\
		In particular, for $\psi\in L^2(\Omega^\p:H^1)$,
		\begin{align*}
			\lim_{n\in\mathbb{N}}\mathbb{E}  \l[\int_{0}^{t} \l\langle \m_n(s) \times \u_n(s) - \m(s) \times \u(s) , \psi \r\rangle_{L^2} \r] = 0.
		\end{align*}	
		Following the arguments in Section \ref{Section Proof of existence of a solution}, we can show that the process $\m$ is a weak martingale solution to the problem \eqref{problem considered for optimal control part}. A couple of key differences are as follows. First of all, the solutions are no longer taking values in finite dimensional spaces. But the arguments can still be repeated with appropriate infinite dimensional spaces replacing the corresponding finite dimensional spaces. Secondly, the projection operator and the cut-off function are absent. But this is actually a simplification.  That $\m$ is a weak martingale solution to the problem \eqref{problem considered for optimal control part}, along with Lemma \ref{lem-bounds lemma for m prime minimizing sequence} implies that the process $\m$ is a strong martingale solution to the problem \eqref{problem considered for optimal control part} on the probability space $\l(\Omega^\p , \mathcal{F}^\p , \mathbb{P}^\p\r)$ corresponding to the control process $\u$ and Wiener process $W^\p$.
		We also need to show that the process $\u$ satisfies the Assumption \ref{Assumption admisibility criterion}. To show this, let $p\geq 1$.\\
		The sequence $\u_n$ converges to $\u$ in $L^2_w(0,T;L^2)$ $\mathbb{P}^{\prime}$-a.s.
		Hence $\mathbb{P}^{\prime}$-a.s.
		\begin{equation}
			\l|\u\r|_{L^2(0,T;L^2)} \leq \liminf_{n\rightarrow\infty} \l|\u_n\r|_{L^2(0,T;L^2)}
		\end{equation}
		Therefore
		\begin{align*}
			\l|\u\r|_{L^2(0,T;L^2)}^{2p} &\leq \l(\liminf_{n\rightarrow\infty} \l|\u_n\r|_{L^2(0,T;L^2)}\r)^{2p} \\
			& \leq \liminf_{n\rightarrow\infty} \l|\u_n\r|_{L^2(0,T;L^2)}^{2p}.
		\end{align*}
		Taking the expectation of both the sides gives
		\begin{equation*}
			\mathbb{E}^\p \l|\u\r|_{L^2(0,T;L^2)}^{2p} \leq \mathbb{E}^\p \l[ \liminf_{n\rightarrow\infty} \l|\u_n\r|_{L^2(0,T;L^2)}^{2p} \r].
		\end{equation*}
		Hence by the Fatou Lemma,
		\begin{align}\label{intermediate liminf bound for u prime for control part}
			\mathbb{E}^{\prime} \l|\u\r|_{L^2(0,T;L^2)}^{2p} \leq \liminf_{n\rightarrow\infty} \mathbb{E}^{\prime} \l|\u_n\r|_{L^2(0,T;L^2)}^{2p} \leq K_p.
		\end{align}
		Therefore the control process $\u$ satisfies the Assumption \ref{Assumption admisibility criterion}.\\
		What remains to show is that this solution minimizes the cost functional.
		We show that in the following steps.\\
		Using the strong convergence in \eqref{Strong convergence of mn prime to m prime minimizing sequence}, we can show that
		\begin{equation}\label{Cost functional convergence eqn 1}
			\mathbb{E}^\p \int_{0}^{T} \l|\m(t) - \bar{m}(t)\r|^2_{L^2(0,T ; H^1)} \, dt = \lim_{n\rightarrow\infty} \mathbb{E}^\p \int_{0}^{T} \l|\m_n(t) - \bar{m}(t)\r|^2_{L^2(0,T ; H^1)} \, dt.
		\end{equation}
		Similarly, \eqref{intermediate liminf bound for u prime for control part} implies that
		\begin{equation}\label{Cost functional convergence eqn 2}
			\mathbb{E}^\p \int_{0}^{T} \l|\u(t)\r|^2_{L^2(0,T ; L^2)} \, dt \leq \liminf_{n\rightarrow\infty} \mathbb{E}^\p \int_{0}^{T} \l|\u_n(t)\r|^2_{L^2(0,T ; L^2)} \, dt.
		\end{equation}
		Recall that the terminal cost function $\Psi$ is assumed to be continuous on $L^2$. We have, by the convergence in \eqref{Equation Convergence of mn prime to m prime Skorohod Theorem minimizing sequence} in Proposition \ref{Proposition use of Skorokhod theorem minimizing sequence}, that $\m_n \to \m\in C([0,T];L^2)$ $\mathbb{P}^\p$-a.s. Therefore we have $\Psi(\m_n(T))\to\Psi(\m(T))$ $\mathbb{P}^\p$-a.s. as $n$ goes to infinity. In particular, we have
		\begin{equation}\label{Cost functional convergence for psi}
			\mathbb{E^{\prime}}\Psi(\m(T)) \leq \liminf_{n\rightarrow\infty} \mathbb{E^{\prime}}\Psi(\m_n(T)).
		\end{equation}
		Combining \eqref{Cost functional convergence eqn 1}, \eqref{Cost functional convergence eqn 2} and \eqref{Cost functional convergence for psi}, we have
		\begin{equation}\label{Cost functional convergence J pi prime bounded by liminf J pin n prime}
			J(\pi^\p) \leq \liminf_{n\rightarrow\infty} J(\pi_n^\p)
		\end{equation}		
		\dela{Note Not Needed\textbf{Note:} When we say that $\Psi(\m_n(T)) \to \Psi(\m(T))$ in some sense, it is required that $\m_n$ takes values in \dela{say $C([0,T] ; (H^1)^\p)\adda{L^2}$}$C([0,T] ; L^2)$.
			Hence convergence in this space is also required. \\}
		Recall that by Proposition \ref{Proposition use of Skorokhod theorem minimizing sequence}, the laws of $m_n$ and $\m_n$ are equal on the space $L^2(0,T ; H^1)\cap C([0,T];L^2)$. Similarly, the laws of $u_n$ and $\u_n$ are equal on the space $L^2(0,T ; L^2)$.\dela{This equality can be skippedHence for each $n\in\mathbb{N}$,
			\begin{align}
				J(\pi_n^\p) = J(\pi_n).
			\end{align}
		}
		Hence \dela{from \eqref{Cost functional convergence eqn 1}, \eqref{Cost functional convergence eqn 2},}
		\begin{align}\label{Cost functional convergence eqn 3}
			\liminf_{n\to\infty} J(\pi^\p_n)
			\leq \liminf_{n\to\infty} J(\pi_n).
		\end{align}
		Therefore combining \eqref{Cost functional convergence J pi prime bounded by liminf J pin n prime} with \eqref{Cost functional convergence eqn 3}, we have
		\begin{align}
			\nonumber J(\pi^\p) & \leq \liminf_{n\to\infty} J(\pi^\p_n) \\
			& \leq \liminf_{n\to\infty} J(\pi_n) = \Lambda.
		\end{align}
		Since $\Lambda$ is the infimum, $J(\pi^\p) = \Lambda$. Hence $\pi^\p$ is an optimal control as defined in Definition \ref{definition of optimal control}. This concludes the proof of Theorem \ref{Theorem existence of optimal control}.		
	\end{proof}

	\noindent	
		\textbf{Acknowledgement}:
			   The authors would like to thank Dr. Manil T. Mohan, Indian Institute of Technology, Roorkee, for his rigorous reading of the manuscript and useful comments and suggestions.

	\appendix

	\section{Proof of Lemma \ref{Lemma continuous in time with values in H1.}} \label{Section Proof of Lemma Lemma continuous in time with values in H1}	
	In this section we give a proof for the Lemma \ref{Lemma continuous in time with values in H1.}
	
	\begin{proof}[Proof of Lemma \ref{Lemma continuous in time with values in H1.}]
		The idea of the proof is as follows. We show that there exists a sequence of functions in $C\l([0,T];H^1\r)$ that converges uniformly in $C\l([0,T];H^1\r)$ to the process $m$, hence showing that $m\in C\l([0,T];H^1\r)$.
		
		Let $P_n: L^2 \to H_n$ be the projection operator in Section \ref{Section Faedo Galerkin approximation}. Let us fix the following notation for this lemma. For $n\in\mathbb{N}$, $m^n := P_n(m)$.
		Fix two numbers $n,k\in\mathbb{N}$ such that $n \geq k$.
		The following equation is satisfied $\mathbb{P}$-a.s. by the projections $m^i$ for $i=n,k$.
		
		\begin{align}\label{eqn_mn mk-strong}
			\nonumber	m^i(t) =& P_i(m_0) + \int_{0}^{t} P_i \bigl(  m(s) \times \Delta m(s) \bigr) \, ds  - \alpha \, \int_{0}^{t}  P_i \bigl[ m(s) \times \bigl( m(s) \times \Delta m(s) \bigr) \bigr]  \, ds
			\\\nonumber
			&  + \int_{0}^{t} P_i \bigl( m(s)\times u(s) \bigr) \, ds - \alpha \, \int_{0}^{t} P_i \l[ m(s)\times \bigl( m(s)\times u(s) \bigr) \r] \, ds  \\
			&  + \frac{1}{2}\int_{0}^{t} P_i \l[DG\bigl(m(s)\bigr)\r] \l[G\bigl( m\l(s\r) \bigr) \r] \,  ds + \int_{0}^{t} P_i G\big(m(t)\big) \, dW(t).
		\end{align}
		Also,
		\begin{equation*}
			m^n - m^k = \l(P_n - P_k\r)m.
		\end{equation*}		
		\dela{
			\begin{align}
				\nonumber	m^i(t) &= P_i(m_0) + \int_{0}^{t}m^i(s) \times \Delta m^i(s)\, ds - \alpha \, \int_{0}^{t} m^i(s)\times(m^i(s)\times u(s))\, ds
				\\\nonumber
				&+ \alpha \, \int_{0}^{t}  \Delta m^i(s) \, ds   + \alpha \, \int_{0}^{t}  |\nabla m^i(s)|_{\mathbb{R}^3}^2 m^i(s) \, ds \\
				& + \int_{0}^{t} m^i(s)\times u(s)\, ds  + \frac{1}{2}\int_{0}^{t} \l[DG\l(m^i(s)\r)\r]\l(G(m^i\l(s\r))\r)\,  ds + \int_{0}^{t} G(m^i(t))\, dW(t).
			\end{align}
		}
		Hence, the difference $m^n - m^k$ satisfies the following equation $\mathbb{P}$-a.s.
		
		\begin{align}\label{Difference eqn_mn-mk-strong}
			\nonumber	m^n(t) - m^k(t) =& \l( P_n - P_k \r)m_0 + \int_{0}^{t} \l[ P_n - P_k \r] \l[  m(s) \times \Delta m(s) \r] \, ds\\
			& \nonumber - \alpha \, \int_{0}^{t}  \l( P_n - P_k \r) \bigl[ m(s) \times \bigl( m(s) \times \Delta m(s) \bigr) \bigr]  \, ds
			\\\nonumber
			&  + \int_{0}^{t} \l( P_n - P_k \r) \bigl( m(s)\times u(s) \bigr) \, ds 
			- \alpha \, \int_{0}^{t} \l( P_n - P_k \r) \l[ m(s)\times \bigl( m(s)\times u(s) \bigr) \r] \, ds  \\
			&  + \frac{1}{2}\int_{0}^{t} \l( P_n - P_k \r) \l[DG\bigl(m(s)\bigr)\r]\l[G\bigl(m\l(s\r)\bigr)\r] \,  ds 
			+ \int_{0}^{t} \l( P_n - P_k \r) G\bigl(m(t)\bigr) \, dW(t).
		\end{align}
		
		Consider a function $\phi_5: H^1 \to \mathbb{R}$ defined by
		\begin{equation}\label{definition of phi_5}
			\phi_5(v) = \frac{1}{2}\l|v\r|_{H^1}^2,\ \text{for}\ v\in H^1.
		\end{equation}
		Let $v , w_1 , w_2\in H^1.$ We observe that $\phi_5(v) = \frac{1}{2}\l|v\r|_{L^2}^2 + \frac{1}{2}\l|\nabla v\r|_{L^2}^2 $. We also recall that $A_1 = I_{L^2} + A$. Hence
		\begin{equation*}
			\phi_5^{\prime}(v)(w_1) = \l\langle A_1 v , w_1 \r\rangle_{L^2}.
		\end{equation*}
		We observe the following. Let $v , w\in H^1$.
		\begin{align*}
			\phi_5^\p(v)(w) &= \l\langle v , w \r\rangle_{L^2} + \l\langle \nabla v , \nabla w \r\rangle_{L^2} \\
			& =  \l\langle v , w \r\rangle_{L^2} + \l\langle -\Delta v , w \r\rangle_{L^2} \\
			& = \l\langle v + (-\Delta)v, w \r\rangle_{L^2} \\
			& = \l\langle A_1 v, w \r\rangle_{L^2}
		\end{align*}

		\begin{align*}
			\phi_5^{\prime \prime}(v)(w_1,w_2) = \l\langle w_1 , w_2 \r\rangle_{L^2} + \l\langle \nabla w_1 , \nabla w_2 \r\rangle_{L^2}.
		\end{align*}		
		We apply the It\^o formula to $\phi_5$ for the process $m^n - m^k$. Thus we get the following. For $t\in[0,T]$
		\begin{align}\label{uniform continuity intermediate equation 1}
			\nonumber \frac{1}{2} \l| m^n(t) - m^k(t) \r|_{H^1}^2  =& \frac{1}{2} \l|A_1 \l( P_n - P_k \r) m_0 \r|_{L^2}^2 \\
			\nonumber&+ \int_{0}^{t} \l\langle A_1\l( P_n - P_k \r) m(s) , \bigl(  m(s) \times \Delta m(s) \bigr) \r\rangle_{L^2}\, ds \\
			\nonumber& - \alpha \, \int_{0}^{t}  \l\langle A_1 \l( P_n - P_k \r) m(s)  , \bigl[ m(s) \times \bigl( m(s) \times \Delta m(s) \bigr) \bigr] \r\rangle_{L^2}  \, ds \\
			\nonumber& + \int_{0}^{t} \l\langle A_1\l( P_n - P_k \r) m(s)  , \bigl( m(s)\times u(s) \bigr) \r\rangle_{L^2}\, ds \\
			\nonumber& - \alpha \, \int_{0}^{t} \l\langle A_1 \l( P_n - P_k \r) m(s)  , \l[ m(s) \times \big( m(s) \times u(s) \big) \r] \r\rangle_{L^2}\, ds\\
			\nonumber& + \frac{1}{2}\int_{0}^{t} \l\langle A_1 \l( P_n - P_k \r) m(s)  , \l[DG\l(m(s)\r)\r]\l(G\big(m\l(s\r)\big)\r)\r\rangle_{L^2} \,  ds \\
			& \nonumber   + \frac{1}{2} \int_{0}^{t} \l| \l( P_n - P_k \r) G\big(m(t)\big) \r|_{H^1}^2 \, ds \\
			& + \int_{0}^{t} \l\langle A_1 \l( P_n - P_k \r) G\big(m(t)\big) m(s)  , G\big(m(t)\big) \r\rangle_{L^2} \, dW(s),\ \mathbb{P}-a.s.
		\end{align}
		
		\dela{
			\adda{This inequality may not be required.} Hence
			\begin{align}\label{uniform continuity intermediate equation 2}
				\nonumber\frac{1}{2}\l|m^n(t) - m^k(t)\r|_{H^1}^2  \leq & \frac{1}{2}\l|A_1\l( P_n - P_k \r) m_0\r|_{L^2}^2 \\
				\nonumber &+ \int_{0}^{t} \l|\l\langle A_1\l( P_n - P_k \r) m(s) , \l(  m(s) \times \Delta m(s) \r)\r\rangle_{L^2}\r|\, ds \\
				\nonumber& - \alpha \, \int_{0}^{t} \l|\l\langle A_1\l( P_n - P_k \r) m(s)  , \l( m(s)\times(m(s)\times u(s)) \r)\r\rangle_{L^2}\r|\, ds
				\\
				\nonumber& - \alpha \, \int_{0}^{t}  \l|\l\langle A_1\l( P_n - P_k \r) m(s)  , \l( m(s) \times \l( m(s) \times \Delta m(s) \r) \r)\r\rangle_{L^2}\r|  \, ds \\
				\nonumber& + \int_{0}^{t} \l|\l\langle A_1\l( P_n - P_k \r) m(s)  , \l( m(s)\times u(s) \r)\r\rangle_{L^2}\r|\, ds \\
				\nonumber& + \frac{1}{2}\int_{0}^{t} \l|\l\langle A_1\l( P_n - P_k \r) m(s)  , \l[DG\l(m(s)\r)\r]\l(G(m\l(s\r))\r)\r\rangle_{L^2}\r| \,  ds \\
				\nonumber &   + \frac{1}{2}\int_{0}^{t} \l| \l( P_n - P_k \r)G(m(s)) \r|_{H^1}^2 \, ds \\
				&+ \int_{0}^{t} \l|\l\langle A_1\l( P_n - P_k \r) m(s)  , G(m(s))\r\rangle_{L^2}\r|\, dW(s)
			\end{align}
		}
		We take $\sup_{t\in[0,T]}$ of both sides of the above inequality to get
		
		\dela{
			\begin{align}
				\nonumber\sup_{t\in[0,T]}\frac{1}{2}\l|m^n(t) - m^k(t)\r|_{H^1}^2 & \leq \frac{1}{2}\l|A_1\l( P_n - P_k \r) m_0\r|_{L^2}^2 \\
				\nonumber&+ \int_{0}^{T} \l|\l\langle A_1\l( P_n - P_k \r) m(s) , \l(  m(s) \times \Delta m(s) \r)\r\rangle_{L^2}\r|\, ds \\
				\nonumber& - \alpha \, \int_{0}^{T} \l|\l\langle A_1\l( P_n - P_k \r) m(s)  , \l( m(s)\times(m(s)\times u(s)) \r)\r\rangle_{L^2}\r|\, ds
				\\
				\nonumber& - \alpha \, \int_{0}^{T}  \l|\l\langle A_1\l( P_n - P_k \r) m(s)  , \l( m(s) \times \l( m(s) \times \Delta m(s) \r) \r)\r\rangle_{L^2}\r|  \, ds \\
				\nonumber& + \int_{0}^{T} \l|\l\langle A_1\l( P_n - P_k \r) m(s)  , \l( m(s)\times u(s) \r)\r\rangle_{L^2}\r|\, ds \\
				\nonumber& + \frac{1}{2}\int_{0}^{T} \l|\l\langle A_1\l( P_n - P_k \r) m(s)  , \l[DG\l(m(s)\r)\r]\l(G(m\l(s\r))\r)\r\rangle_{L^2}\r| \,  ds \\
				\nonumber&   + \frac{1}{2}\int_{0}^{T} \l| \l( P_n - P_k \r)G(m(s)) \r|_{H^1}^2 \, ds \\
				\nonumber&+ \sup_{t\in[0,T]}\int_{0}^{t} \l|\l\langle A_1\l( P_n - P_k \r) m(s)  , G(m(s))\r\rangle_{L^2}\r|\, dW(s)\\
				=& \frac{1}{2}\l|\l( P_n - P_k \r) m_0\r|_{H^1}^2 + \sum_{i=1}^{7}c_iI_i(T).
			\end{align}
		}
		
		\begin{align}\label{uniform continuity intermediate equation 3}
			\nonumber\sup_{t\in[0,T]}\frac{1}{2}\l|m^n(t) - m^k(t)\r|_{H^1}^2 & \leq \frac{1}{2}\l|A_1\l( P_n - P_k \r) m_0\r|_{L^2}^2 \\
			\nonumber&+ \int_{0}^{T} \l|\l\langle A_1\l( P_n - P_k \r) m(s) , \bigl(  m(s) \times \Delta m(s) \bigr) \r\rangle_{L^2}\r|\, ds \\
			\nonumber& - \alpha \, \int_{0}^{T}  \l|\l\langle A_1\l( P_n - P_k \r) m(s)  , \l[ m(s) \times \bigl( m(s) \times \Delta m(s) \bigr) \r] \r\rangle_{L^2}\r|  \, ds \\
			\nonumber& + \int_{0}^{T} \l|\l\langle A_1\l( P_n - P_k \r) m(s)  , \bigl( m(s)\times u(s) \bigr) \r\rangle_{L^2}\r|\, ds \\
			\nonumber& - \alpha \, \int_{0}^{T} \l|\l\langle A_1\l( P_n - P_k \r) m(s)  , \l[ m(s) \times \big( m(s)\times u(s) \big) \r] \r\rangle_{L^2}\r|\, ds
			\\
			\nonumber& + \frac{1}{2}\int_{0}^{T} \l|\l\langle A_1\l( P_n - P_k \r) m(s)  , \l[DG\bigl(m(s)\bigr)\r]\l[G\big(m\l(s\r)\big)\r] \r\rangle_{L^2}\r| \,  ds \\
			\nonumber&   + \frac{1}{2}\int_{0}^{T} \l| \l( P_n - P_k \r)G\big(m(s)\big) \r|_{H^1}^2 \, ds \\
			\nonumber&+ \sup_{t\in[0,T]}\int_{0}^{t} \l|\l\langle A_1\l( P_n - P_k \r) m(s)  , G\big(m(s)\big) \r\rangle_{L^2}\r|\, dW(s)\\
			=& \frac{1}{2}\l| A_1 \l( P_n - P_k \r) m_0\r|_{H^1}^2 + \sum_{i=1}^{7}c_iI_i(T).
		\end{align}
		We now take the expectation of both sides of the above inequality. We then fix $k$ and let $n$ go to infinity. We claim convergence of the right hand side by the monotone convergence theorem. After that we show that each term on the right hand side of the resulting inequality goes to $ 0 $ as $k$ goes to infinity.
		
		To simplify the presentation, we write the bounds on each term individually and combine them to get the desired result.
		We have the following inequalities by the Cauchy-Schwartz inequality.\\
		\textbf{Calculation for $I_1$}
		\begin{align*}
			&\mathbb{E}\int_{0}^{T}\l|\l\langle A_1\l( P_n - P_k \r) m(s) , \bigl(  m(s) \times \Delta m(s) \bigr) \r\rangle_{L^2}\r|\, ds \\
			&\leq \l(\mathbb{E} \int_{0}^{T} \l|A_1\l( P_n - P_k \r) m(s)\r|_{L^2}^2 \, ds\r)^{\frac{1}{2}} \l(\mathbb{E} \int_{0}^{T} \l| m(s) \times \Delta m(s) \r|_{L^2}^2 \, ds\r)^{\frac{1}{2}}.
		\end{align*}		
		\textbf{Calculation for $I_2$}
		\begin{align*}
			&\mathbb{E}\int_{0}^{T}  \l|\l\langle A_1\l( P_n - P_k \r) m(s) , m(s) \times \bigl( m(s) \times \Delta m(s) \bigr) \r\rangle_{L^2}\r| \, ds\\
			& \leq  \l( \mathbb{E}\int_{0}^{T}\l|A_1\l( P_n - P_k \r) m(s)\r|_{L^2}^2 \, ds \r)^{\frac{1}{2}} \l( \mathbb{E}\int_{0}^{T}\l| m(s) \times \l(  m(s) \times \Delta m(s) \r) \r|_{L^2}^2 \, ds \r)^{\frac{1}{2}} .
		\end{align*}
		\textbf{Calculation for $I_3$}
		\begin{align*}
			&\mathbb{E}\int_{0}^{T} \l|\l\langle A_1\l( P_n - P_k \r) m(s) , m(s)\times u(s) \r\rangle_{L^2}\r| \, ds \\ & \leq    \l( \mathbb{E}\int_{0}^{T} \l|A_1\l( P_n - P_k \r) m(s)\r|_{L^2}^2 \, ds \r)^{\frac{1}{2}} \l( \mathbb{E}\int_{0}^{T} \l| m(s)\times u(s) \r|_{L^2}^2 \, ds \r)^{\frac{1}{2}} .
		\end{align*}
		\textbf{Calculation for $I_4$}
		\begin{align*}
			&\mathbb{E}\int_{0}^{T}  \l|\l\langle A_1\l( P_n - P_k \r) m(s) , m(s) \times \big( m(s)\times u(s) \big) \r\rangle_{L^2}\r| \, ds \\ & \leq   \l( \mathbb{E}\int_{0}^{T}\l|A_1\l( P_n - P_k \r) m(s)\r|_{L^2}^2 \, ds \r)^{\frac{1}{2}} \l( \mathbb{E}\int_{0}^{T}\l| m(s) \times \big( m(s)\times u(s) \big) \r|_{L^2}^2 \, ds \r)^{\frac{1}{2}}.
		\end{align*}		
		\textbf{Calculation for $I_5$}\\
		Similarly,
		\begin{align*}
			&\mathbb{E}\int_{0}^{T}\l|\l\langle A_1\l( P_n - P_k \r) m(s) , \l[ DG \bigl(m(s)\bigr)\r]\l[G\big(m\l(s\r)\big)\r] \r\rangle_{L^2}\r| \, ds \\
			&\leq \l( \mathbb{E} \int_{0}^{T} \l|A_1\l( P_n - P_k \r) m(s)\r|_{L^2}^2 \, ds\r)^{\frac{1}{2}} \l( \mathbb{E} \int_{0}^{T} \l| \l[ DG\bigl( m(s) \bigr) \r] \l[ G\big(m\l(s\r)\big) \r] \r|_{L^2}^2 \, ds \r)^{\frac{1}{2}}.
		\end{align*}
		\textbf{Calculation for $I_7$}\\
		By the Burkholder-Davis-Gundy inequality, followed by the Cauchy-Schwartz inequality, there exists a constant $C_1>0$ such that
		
		\begin{align*}
			&\mathbb{E}\sup_{0\leq t\leq T}\l| \int_{0}^{t} \l|\l\langle A_1\l( P_n - P_k \r) m(s)  , G\big(m(s)\big)\r\rangle_{L^2}\r|\, dW(s) \r| \\
			&\leq C_1 \mathbb{E} \l[ \l( \int_{0}^{T} \l|\l\langle A_1\l( P_n - P_k \r) m(s)  , G\big(m(s)\big)\r\rangle_{L^2}\r|^2 \, ds \r)^{\frac{1}{2}} \r] \\
			& \leq C_1 \mathbb{E}  \l[ \l( \int_{0}^{T} \l| A_1\l( P_n - P_k \r) m(s) \r|_{L^2}^2 \l|G\big(m(s)\big)\r|_{L^2}^2 \, ds \r)^{\frac{1}{2}} \r].
		\end{align*}
		By the constraint condition and the assumption on the function $h$, there exists another constant $C>0$ such that
		\begin{align*}
			\mathbb{E}\sup_{0\leq t\leq T}\l| \int_{0}^{t} \l|\l\langle A_1\l( P_n - P_k \r) m(s)  , G\big(m(s)\big)\r\rangle_{L^2}\r|\, dW(t) \r| \leq C\l(\mathbb{E} \int_{0}^{T}\l|A_1\l( P_n - P_k \r) m(s)\r|_{L^2}^2 \, ds \r)^{\frac{1}{2}}.
		\end{align*}
		Thus combining the above mentioned inequalities gives
		\begin{align}\label{uniform continuity intermediate equation 4}
			\nonumber &\mathbb{E} \sup_{t\in[0,T]} \l|m^n(t) - m^k(t)\r|_{H^1}^2 \\
			\nonumber &\leq \frac{1}{2}\l|A_1\l( P_n - P_k \r) m_0\r|_{H^1}^2 \\
			\nonumber& + \l(\mathbb{E} \int_{0}^{T} \l|A_1\l( P_n - P_k \r) m(s)\r|_{L^2}^2 \, ds\r)^{\frac{1}{2}} \Bigg[ \l(\mathbb{E} \int_{0}^{T} \l| m(s) \times \Delta m(s) \r|_{L^2}^2 \, ds\r)^{\frac{1}{2}} \\
			\nonumber& + \l( \mathbb{E}\int_{0}^{T}\l| m(s) \times \l(  m(s) \times \Delta m(s) \r) \r|_{L^2}^2 \, ds \r)^{\frac{1}{2}} + \l( \mathbb{E} \int_{0}^{T} \l| \l[DG\l(m(s)\r)\r]\l(G(m\l(s\r))\r) \r|_{L^2}^2 \, ds \r)^{\frac{1}{2}} \\
			\nonumber& + \l( \mathbb{E}\int_{0}^{T} \l| m(s)\times u(s) \r|_{L^2}^2 \, ds \r)^{\frac{1}{2}} + \l( \mathbb{E}\int_{0}^{T}\l| m(s)\times(m(s)\times u(s)) \r|_{L^2}^2 \, ds \r)^{\frac{1}{2}} \Bigg] \\
			& + \mathbb{E} \int_{0}^{T} \l| \l( P_n - P_k \r)G(m(s)) \r|_{H^1}^2 \, ds  + C\l(\mathbb{E} \int_{0}^{T}\l|A_1\l( P_n - P_k \r) m(s)\r|_{L^2}^2 \, ds \r)^{\frac{1}{2}}.
		\end{align}
		In the above inequality, we fix $k$ and let $n$ go to infinity.
		For any $v\in L^2$, $$P_n v \rightarrow v$$ as $n$ goes to infinity. Recall that $n,k$ were chosen such that $n\geq k$.
		Since $P_n$ is an orthonormal projection on for each $n\in\mathbb{N}$,
		\begin{equation*}
			\l|P_{n_1 - k}v\r|_{L^2} \leq \l|P_{n_2 - k}v\r|_{L^2}
		\end{equation*}
		for any $n_1,n_2\in\mathbb{N}$ such that $n_1 \geq n_2 \geq k$.
		Hence by the monotone convergence theorem, we have the following inequality as $n$ goes to infinity in \eqref{uniform continuity intermediate equation 4}.
		
		\dela{
			\begin{align}
				\nonumber \mathbb{E} &\sup_{t\in[0,T]} \l|m(t) - m^k(t)\r|_{H^1}^2 \leq \frac{1}{2}\l|A_1\l( I_{L^2} - P_k \r) m_0\r|_{H^1}^2 \\
				\nonumber& + \l(\mathbb{E} \int_{0}^{T} \l|A_1\l( I_{L^2} - P_k \r) m(s)\r|_{L^2}^2 \, ds\r)^{\frac{1}{2}} \dela{\centerdot \\
					\nonumber&\quad }\Bigg[ \l(\mathbb{E} \int_{0}^{T} \l| m(s) \times \Delta m(s) \r|_{L^2}^2 \, ds\r)^{\frac{1}{2}} \\
				\nonumber& + \l( \mathbb{E}\int_{0}^{T}\l| m(s) \times \l(  m(s) \times \Delta m(s) \r) \r|_{L^2}^2 \, ds \r)^{\frac{1}{2}} \dela{\\
					\nonumber&} + \l( \mathbb{E} \int_{0}^{T} \l| \l[DG\l(m(s)\r)\r]\l(G(m\l(s\r))\r) \r|_{L^2}^2 \, ds \r)^{\frac{1}{2}} \\
				\nonumber& + \l( \mathbb{E}\int_{0}^{T} \l| m(s)\times u(s) \r|_{L^2}^2 \, ds \r)^{\frac{1}{2}} \dela{\\
					\nonumber&} + \l( \mathbb{E}\int_{0}^{T}\l| m(s)\times(m(s)\times u(s)) \r|_{L^2}^2 \, ds \r)^{\frac{1}{2}} \Bigg] \\
				& + \mathbb{E} \int_{0}^{T} \l| \l( I_{L^2} - P_k \r)G(m(s)) \r|_{H^1}^2 \, ds \dela{\\
					&} + C\l(\mathbb{E} \int_{0}^{T}\l|A_1\l( P_n - P_k \r) m(s)\r|_{L^2}^2 \, ds \r)^{\frac{1}{2}}.
			\end{align}
		}
		
		\begin{align}\label{uniform continuity intermediate equation 5}
			\nonumber \mathbb{E} &\sup_{t\in[0,T]} \l|m(t) - m^k(t)\r|_{H^1}^2 \leq \frac{1}{2}\l|A_1\l( I_{L^2} - P_k \r) m_0\r|_{H^1}^2 \\
			\nonumber& + \l(\mathbb{E} \int_{0}^{T} \l|A_1\l( I_{L^2} - P_k \r) m(s)\r|_{L^2}^2 \, ds\r)^{\frac{1}{2}} \dela{\centerdot \\
				\nonumber&\quad }\Bigg[ \l(\mathbb{E} \int_{0}^{T} \l| m(s) \times \Delta m(s) \r|_{L^2}^2 \, ds\r)^{\frac{1}{2}} \\
			\nonumber& + \l( \mathbb{E}\int_{0}^{T}\l| m(s) \times \l(  m(s) \times \Delta m(s) \r) \r|_{L^2}^2 \, ds \r)^{\frac{1}{2}} \dela{\\
				\nonumber&} + \l( \mathbb{E} \int_{0}^{T} \l| \l[DG\bigl(m(s)\bigr)\r] \l[G\big(m\l(s\r)\big)\r] \r|_{L^2}^2 \, ds \r)^{\frac{1}{2}} \\
			\nonumber& + \l( \mathbb{E}\int_{0}^{T} \l| m(s)\times u(s) \r|_{L^2}^2 \, ds \r)^{\frac{1}{2}} 
			+ \l( \mathbb{E}\int_{0}^{T}\l| m(s) \times \big( m(s)\times u(s) \big) \r|_{L^2}^2 \, ds \r)^{\frac{1}{2}} \Bigg] \\
			& + \mathbb{E} \int_{0}^{T} \l| \l( I_{L^2} - P_k \r) G\big(m(s)\big) \r|_{H^1}^2 \, ds \dela{\\
				&} + C\l(\mathbb{E} \int_{0}^{T}\l|A_1\l( P_n - P_k \r) m(s)\r|_{L^2}^2 \, ds \r)^{\frac{1}{2}}.
		\end{align}
		\dela{By Lemma \ref{bounds lemma}, the second terms on the right hand side of the above 5 inequalities are bounded by a constant.
			
			\begin{align*}
				\mathbb{E}\int_{0}^{T} &\l|\l\langle A_1\l( P_n - P_k \r) m(s) , m(s)\times u(s) \r\rangle_{L^2}\r| \, ds = \mathbb{E}\int_{0}^{T} \l|\l\langle \l( P_n - P_k \r)A_1 m(s) , m(s)\times u(s) \r\rangle_{L^2}\r| \, ds  .
		\end{align*}}
		By Lemma \ref{bounds lemma 1}, the terms on the right hand side of the above inequality are bounded. Hence
		The right hand side of the above inequality goes to $0$ as $k$ goes to infinity.
		Hence
		
		\begin{align}
			\lim_{k\rightarrow\infty}\mathbb{E} \sup_{t\in[0,T]} \l|m(t) - m^k(t)\r|_{H^1}^2 = 0.
		\end{align}
		Thus there exists a subsequence of $m^k$, again denoted by $m^k$ such that $\mathbb{P}$-a.s.
		\begin{align}
			\lim_{k\rightarrow\infty} \sup_{t\in[0,T]} \l|m(t) - m^k(t)\r|_{H^1}^2 = 0.
		\end{align}
		The paths of the process $m^k$ lie in $C([0,T];H^1)$. Hence the above convergence implies that $\mathbb{P}$-a.s. the process $m$ takes values in $C(0,T;H^1)$. This concludes the proof of Lemma \ref{Lemma continuous in time with values in H1.}.
		
	\end{proof}

	\section{Proof for \eqref{pathwise uniqueness exp intermediate inequality}}\label{Section Proof for pathwise uniqueness intermediate inequality}	
	Consider a real valued stochastic process $\{X(t)\}_{t\in[0,T]}$ given by
	\begin{equation}
		X(t) = \int_{0}^{t} a(s) \, ds + \int_{0}^{t} b(s) \, dW(s)\ \mathbb{P}\ \text{-a.s.}
	\end{equation}
	Consider the function $F$ given by
	\begin{equation}
		F(t,x) = x e^{-\int_{0}^{t}\Phi_C(s) \, ds},
	\end{equation}
	Let $F_t^\p, F_x^\p$ and $F_{xx}^{\p\p}$ denote its partial derivatives.
	Then
	\begin{equation*}
		F_t^\p(t,x) = -\Phi_C(t) x e^{-\int_{0}^{t}\Phi_C(s) \, ds}
	\end{equation*}
	and
	\begin{equation*}
		F_x^\p(t,x) = e^{-\int_{0}^{t}\Phi_C(s) \, ds}.
	\end{equation*}
	Also,
	\begin{equation*}
		F_{xx}^{\p\p}(t,x) = 0.
	\end{equation*}
	Applying the It\^o formula for the function $F$, we get
	\begin{align*}
		F(t,X(t)) &= \int_{0}^{t} \l[ F_t^\p\big(s,X(s)\big) + F_x^\p\big(s,X(s)\big) a(s) + \frac{1}{2}F_{xx}^{\p\p}\big(s,X(s)\big) b^2(s) \r] \, ds \\
		& + \int_{0}^{t} F_x^\p\big(s,X(s)\big) b(s) \, dW(s).
	\end{align*}
	Therefore,
	\begin{align*}
		X(t) e^{-\int_{0}^{t}\Phi_C(s) \, ds} &= \int_{0}^{t} -\Phi_C(t) X(s) e^{-\int_{0}^{s}\Phi_C(r) \, dr} +  e^{-\int_{0}^{s}\Phi_C(r) \, dr} a(s) \, ds \\
		& + \int_{0}^{t} e^{-\int_{0}^{s}\Phi_C(r) \, dr} b(s) \, dW(s).
	\end{align*}
	Going back to Section \ref{Section Pathwise uniqueness}, the application of the It\^o Lemma for the function
	\begin{equation*}
		v \mapsto \frac{1}{2}\l| v \r|_{L^2}^2,
	\end{equation*}
	gives a representation of the process $X(t) = \l|m(t)\r|_{L^2}^2$. Applying the It\^o forumla for the function $F$ as done above gives
	\begin{align*}
		\l|m(t)\r|_{L^2}^2 e^{-\int_{0}^{t}\Phi_C(s) \, ds} &= \int_{0}^{t} -\Phi_C(t) \l|m(s)\r|_{L^2}^2 e^{ -\int_{0}^{s}\Phi_C(r) \, dr } +  e^{-\int_{0}^{s}\Phi_C(r) \, dr} a(s) \, ds \\
		& + \int_{0}^{t} e^{-\int_{0}^{s}\Phi_C(r) \, dr} b(s) \, dW(s).
	\end{align*}
	The non-negative function $\Phi_C$ is chosen such that for each $t\in [0,T]$,
	\begin{align*}
		\int_{0}^{t} a(s) \, ds \leq \int_{0}^{t} \l|m(s)\r|_{L^2}^2 \Phi_C(s) \, ds.
	\end{align*}
	Moreover, if $\Phi_C$ is $\mathbb{P}$-a.s. integrable over $[0,T]$, then
	\begin{equation*}
		0 \leq e^{ - \int_{0}^{t} \Phi_C(s) \, ds} \leq 1,
	\end{equation*}
	for each $t\in[0,T]$.	
	Hence
	\begin{align*}
		\int_{0}^{t} e^{-\int_{0}^{s}\Phi_C(r) \, dr} a(s)  \, ds -\Phi_C(t) \l|m(s)\r|_{L^2}^2 e^{ -\int_{0}^{s}\Phi_C(r) \, dr } \leq 0.
	\end{align*}
	Therefore,
	\begin{align*}
		\l|m(t)\r|_{L^2}^2 e^{-\int_{0}^{t}\Phi_C(s) \, ds} = \int_{0}^{t} e^{-\int_{0}^{s}\Phi_C(r) \, dr} b(s) \, dW(s).
	\end{align*}
	This concludes the proof of the inequality \eqref{pathwise uniqueness exp intermediate inequality}.

	\dela{\section{Some Auxiliary Results}\label{Section Auxiliary Results}
		\begin{proposition}\label{Proposition Global solution for FG approximation}
			\adda{Omit this Section} Consider the SDE in $H_n$ given by
			\begin{equation}\label{SDE}
				X(t) = X_0 + \int_{0}^{t} a(X(s)) \, ds + \int_{0}^{t} b(X(s)) \, dW(s),\ t\in[0,T].
			\end{equation}
			Suppose that the following hold.
			\begin{enumerate}
				\item The coefficients $a,b$ are locally Lipschitz. (These coefficients correspond to the coefficients from equation \eqref{definition of solution Faedo Galerkin approximation}).
				\item There exists a constant $C>0$ such that
				\begin{equation}
					\mathbb{E} \int_{0}^{T} \l| X(t) \r|_{L^2}^2 \, dt \leq C.
				\end{equation}
			\end{enumerate}
			Then the SDE admits a unique global solution.
		\end{proposition}
		\begin{proof}[Proof of Proposition \ref{Proposition Global solution for FG approximation}]
			We give a sketch of the proof. $a,b$ are locally Lipschitz. That is whenever $v_1,v_2\in H_n$ such that $\l|v_1\r|_{L^2},\l|v_2\r|_{L^2} \leq R$, there exists a constant $K_R$ such that
			\begin{equation}
				\l| a(v_1) - a(v_2) \r|_{L^2} \leq K_R \l| v_1 - v_2 \r|_{L^2}
			\end{equation}
			and
			\begin{equation}
				\l| b(v_1) - b(v_2) \r|_{L^2} \leq K_R \l| v_1 - v_2 \r|_{L^2}.
			\end{equation}
			Let $0<R<\infty$.\dela{ and $\theta:\mathbb{R}\to[0,1]$ denote a smooth cut-off function that truncates at $R$. That is
				\begin{align}
					\theta(x)=
					\begin{cases}
						1&\ \text{if}\ \l|x\r| \leq R,\\
						0&\ \text{if}\ \l|x\r| \geq 2R.
					\end{cases}
				\end{align}
				Now we consider the following truncated SDE.
				\begin{equation}\label{SDE truncated}
					X_R(t) = X_0 + \int_{0}^{t} \theta(\l|X_R(s)\r|_{L^2})a(X_R(s)) \, ds + \int_{0}^{t} \theta(\l|X_R(s)\r|_{L^2})b(X(s)) \, dW(s),\ t\in[0,T].
				\end{equation}
			}
			Define a "time" $\tau_R$ by
			\begin{equation}
				\tau_R = \inf_{t\in[0,T]}\{ \l|X_R\r|_{L^2(0,T;L^2)} \geq R\}\wedge T.
			\end{equation}
			Now we consider the stopped equation
			\begin{equation}\label{SDE stopped}
				X_R(t\wedge \tau_R) = X_0 + \int_{0}^{t\wedge \tau_R} a(X_R(s)) \, ds + \int_{0}^{t\wedge \tau_R} \theta(\l|X_R(s)\r|_{L^2})b(X(s)) \, dW(s),\ t\in[0,T].
			\end{equation}
			Whenever $\l|X_R\r|_{L^2(0,t)}\leq R$, equation \eqref{SDE stopped} is same as \eqref{SDE}. Note that for $t\in[0,\tau_R]$, the coefficients of \eqref{SDE stopped} are globally Lipschitz. Therefore \eqref{SDE stopped} admits a unique strong solution $X_R$ on the interval $t\in[0,\tau_R]$.\\
			It can be shown that $\tau_R$ is a stopping time for each $R\in\mathbb{R}$. The solution $X_R$ is adapted to the given $\sigma$-algebra. In our case, the coefficients $a,b$ are continuous, and hence measurable functions of $X_R$. Also, the norm $\l|\cdot\r|_{L^2(0,T;L^2)}$ is continuous on the space $L^2(0,T;L^2)$. Therefore the set $\{\tau_R \leq t\}\in\mathcal{F}_t$, for each $t\in[0,T]$, showing that $\tau_R$ is a stopping time.\\
			In particular, the sequence $\{\tau_R\}_{R\in\mathbb{N}}$ is a non-decreasing sequence of stopping times \adda{proof}, bounded above by $T$. Therefore, the sequence converges. Let $\lim_{R\to\infty} = \tau_{\infty}$. Let $X_{\infty}$ denote the solution corresponding to $\tau_{\infty}$. Let $t\in[0,T]$.
			\begin{align*}
				\mathbb{P}\l( \tau_R < t \r) = & \mathbb{E} \l[\chi_{\tau_R < t} \r] \\
				= &  \frac{1}{R} \mathbb{E} \l[\chi_{\tau_R < t} R \r] \\
				\leq &  \frac{1}{R} \mathbb{E} \l[\chi_{\tau_R < t} \l|X_R\r|_{L^2(0,\tau_R:L^2)} \r].
			\end{align*}
			Assuming that $a,b$ are now the coefficients considered in \eqref{definition of solution Faedo Galerkin approximation}, the It\^o formula applied for $v\mapsto \frac{1}{2}\l| v \r|_{L^2}^2$ implies that there exists a constant $C>0$ (independent of $R$) such that
			\begin{equation}
				\mathbb{E} \l[ \sup_{t\in[0,\tau_R]}\l|X_R\r|_{L^2(0,\tau_R:L^2)}^2 \r] \leq C.
			\end{equation}
			Therefore
			\begin{equation}
				\mathbb{P}\l( \tau_R < t \r) \leq \frac{C}{R}.
			\end{equation}
			The right hand side, and hence the left hand side of the above inequality goes to 0 as $R$ goes to infinity. Since $\tau_R\uparrow\tau_{\infty}$,
			\begin{equation}
				\mathbb{P}\l( \tau_{\infty} < t \r) = 0.
			\end{equation}
			In particular,
			\begin{equation}
				\mathbb{P}\l( \tau_{\infty} = T \r) = 1.
			\end{equation}
			
		\end{proof}
	}

	\section{Section Some embeddings}\label{Section Some embeddings}

	\begin{lemma}[Result 1, Appendix A, \cite{ZB+UM+DM_WongZakai}]\label{Lemma Sobolev embedding ref Wong Zakai}
		We have the following continuous embedding for $\delta \in \l(\frac{5}{8} , \frac{3}{4}\r)$.
		\begin{equation}
			W^{2\delta , 2} \hookrightarrow W^{1,4}.
		\end{equation}
	\end{lemma}
	\dela{For a proof of this lemma, we refer to Appendix A, \cite{ZB+UM+DM_WongZakai}.\\
		The following is Lemma A.1 from \cite{ZB+BG+TJ_Weak_3d_SLLGE}.}
	\begin{lemma}[Lemma A.1 \cite{ZB+BG+TJ_Weak_3d_SLLGE}]\label{Lemma reference W alpha p bound for stochastic integral}
		Let $E$ be a separable Hilbert space. Let $2\leq p < \infty$ and $0 < \alpha \, < \frac{1}{2}$.
		
		Let $\zeta$ be a process $\zeta: [0,T] \times \Omega \to E$ such that
		\begin{equation}
			\mathbb{E} \int_{0}^{T}\l|\zeta(t)\r|_{E}^p \, dt < \infty.
		\end{equation}
		Define a process $I(\zeta)$ by
		\begin{equation}
			I(\zeta) = \int_{0}^{t} \zeta (s) \, dW(s) , \quad t\geq 0.
		\end{equation}
		Then for all such processes $\zeta$, there exists a constant $C$ that depends on $T,\alpha$ such that
		\begin{equation}
			\mathbb{E} \l|I(\zeta)\r|_{W^{\alpha \, , p}(0,T; E)}^p \leq C\mathbb{E} \int_{0}^{T}\l|\zeta(t)\r|_{E}^p \, dt.
		\end{equation}
		In particular, $\mathbb{P}-a.s.$the paths of $I(\zeta)$ belong to the space $W^{\alpha \, , 2}(0,T; E)$.
	\end{lemma}

	We now state the Young's convolution inequality.
	\begin{lemma}[Young's convolution inequality, Theorem 3.9.4 in \cite{BogachevBook_MeasureTheor_2007}]\label{lemma Young convolution inequality}
		Let $f\in L^p,g\in L^q$. Let $p,q,r\geq 1$ be such that
		$$\frac{1}{p} + \frac{1}{q} = \frac{1}{r} + 1.$$
		Then
		\begin{align*}
			\l|f*g\r|_{L^r} \leq \l|f\r|_{L^p}\l|g\r|_{L^q}.
		\end{align*}
	\end{lemma}

	\begin{lemma}[Theorem 17.7, \cite{OK_FoundationsOfModernProbability}]\label{BDG inequality upper bound}
		Let $p\in (0,\infty)$. Then there exists a constant $K_p>0$ with the following property. Let $T\in(0,\infty)$ and $(\Omega , \mathcal{F} , \mathbb{F} , \mathbb{P})$ be a filtered probability space satisfying the usual conditions. Let $\{W(t),t\geq 0\}$ be a Wiener process on this space. For any progressively measurable function
		\begin{equation*}
			F:[0,T] \times \Omega \rightarrow \mathbb{R}
		\end{equation*}
		such that $\mathbb{P}\l(\int_{0}^{T}F^2(s)dt<\infty \r)= 1$, the following holds
		\begin{equation*}
			\mathbb{E} \left[ \sup_{t\in [0,T]} \l|\int_{0}^{t} F(s) \, dW(s)\r|^{2p} \right] \leq K_p\mathbb{E} \left[ \l(\int_{0}^{T} F^2(s) \, ds \r)^p \right].
		\end{equation*}
	\end{lemma}

	\begin{lemma}[Theorem 1.4.8, \cite{HenryGeometricTheoryofSemilinearParabolicEquations}]\label{X gamma compact embedding}
		For $\gamma_1 > \gamma_2$, $X^{\gamma_2}$ is compactly embedded into the space $X^{\gamma_1}$.
	\end{lemma}

	\dela{The next theorem we refer to \cite{Flandoli_Gatarek} (Theorem 2.2).}
	
	\begin{lemma}[Theorem 2.2, \cite{Flandoli_Gatarek}]\label{compact embedding into space of continuous functions}
		If $B\subset \tilde{B}$ are two Banach spaces with compact embedding, and let the real numbers $\gamma \in (0,1)$, $p > 1$ satisfy
		\begin{equation*}
			\gamma p > 1.
		\end{equation*}
		Then the space $W^{\gamma , p}(0,T; B)$ is compactly embedded into the space $C(0,T;\tilde{B})$.
	\end{lemma}
	
	\dela{For the following lemma, see \cite{Flandoli_Gatarek}, Theorem 2.1.}
	\begin{lemma}[Theorem 2.1, \cite{Flandoli_Gatarek}]\label{compact embedding of intersection 1}
		Let $B_0\subset B \subset B_1$ be Banach spaces. Assume that $B_0$ and $B_1$ are reflexive. Further assume that the embedding $B_0\subset B$ is compact, the embedding of $B\hookrightarrow B_1$ being continuous. Let $q\in (1,\infty)$ and $\gamma\in (0,1)$. Then we have the following compact embedding
		\begin{equation*}
			L^p(0,T;B_0)\cap W^{\gamma , q}(0,T; B_1) \hookrightarrow L^p(0,T;B).
		\end{equation*}
	\end{lemma}

	\begin{lemma}[See Lemma 5, \cite{Simon_Compact_Sets}]\label{compact embedding of intersection auxilliary 1}
		Let $f\in W^{\sigma, r}(0,T;B)$, $0<\sigma<1$, $1\leq r \leq \infty$ and $p$ be such that
		\begin{enumerate}
			\item $p\leq \infty$ if $\sigma > \frac{1}{r}$,
			
			\item
			$p < \infty$ if $\sigma = \frac{1}{r}$,
			
			\item $p\leq r_{*} = \frac{r}{1-\sigma r}$ if $p < \infty$ if $\sigma < \frac{1}{r}$.
		\end{enumerate}
		Then $f\in L^{p}(0,T;B)$ and there exists a constant $C$ independent of $f$ such that for all $h > 0$,
		\begin{align*}
			\l|\tau_h f - f\r|_{L^p(0,T;B)}\leq
			\begin{cases}
				&ch^{ \sigma + \frac{1}{p} - \frac{1}{p} } \l|f\r|_{\dot{W}^{\sigma , r}(0,T;B)},\ \text{if}\ r\leq p <\infty \\
				&ch^{ \sigma }T^{ \frac{1}{p} - \frac{1}{p} } \l|f\r|_{\dot{W}^{\sigma , r}(0,T;B)},\ \text{if}\ 1\leq r \leq p.
			\end{cases}
		\end{align*}
		
	\end{lemma}
	
	\dela{For the next lemma, refer to Theorem 3, \cite{Simon_Compact_Sets}.}
	\begin{lemma}[Theorem 3, \cite{Simon_Compact_Sets}]\label{compact embedding of intersection 2}
		Let $X,B$ be Banach spaces with $X$ compactly embedded in $B$. Let $F\subset L^{p}(0,T;B)$, where $1\leq p \leq \infty$. Assume that
		\begin{enumerate}
			\item
			$F$ is bounded in $L^1_{\text{loc}}(0,T;X)$.
			
			\item
			$\|\tau_h f - f\|_{L^p(0,T-h:B)}\rightarrow 0$ as $h\rightarrow 0$ uniformly for $f\in F$.
		\end{enumerate}
		Then $F$ is relatively compact in $L^p(0,T;B)$ (and in $C([0,T];B)$ if $p=\infty$).
	\end{lemma}

	\dela{Kuratowski's theorem (see Theorem 1.1, \cite{Vakhania_Probability_distributions_on_Banach_spaces})}
	\begin{lemma}[Kuratowski, Theorem 1.1, \cite{Vakhania_Probability_distributions_on_Banach_spaces}]\label{Lemma Kuratowski}
		Let $X$ be a Polish space, $Y$ be a separable metric space and let $f:X\rightarrow Y$ be an injective  Borel mapping. Then $f(B)\in \mathcal{B}(Y)$ for any $B\in \mathcal{B}(X)$.
	\end{lemma}

	\dela{A compactness result. See Theorem 2.1, page number 184, \cite{Temam}.}
	\begin{lemma}[Theorem 2.1, \cite{Temam}]\label{Theorem 2.1 from Temam}
		Let $X_0,X,X_1$ be Banach spaces such that
		\begin{equation*}
			X_0 \hookrightarrow X \hookrightarrow X_1
		\end{equation*}
		where the injection is continuous. Further assume that $X_0,X_1$ are reflexive spaces and the injection
		\begin{equation*}
			X_0\hookrightarrow X
		\end{equation*}
		is compact.
		Let $\alpha_0,\alpha_1>0$. Consider the space
		\begin{equation*}
			\mathcal{Y} = \l\{ v\in L^{\alpha_0}\l(0,T;X_0\r): \frac{dv}{dt} \in L^{\alpha_1}\l(0,T; X_1\r) \r\}
		\end{equation*}
		The space is provided with the norm
		\begin{equation*}
			\l| v \r|_{\mathcal{Y}} = \l| v \r|_{L^{\alpha_0}\l(0,T;X_0\r)} + \l| \frac{dv}{dt}\r|_{L^{\alpha_1}\l(0,T; X_1\r)}.
		\end{equation*}
		Then the space $\mathcal{Y}$ is compactly embedded in the space $L^{\alpha_0}\l(0,T;X\r)$.
	\end{lemma}


	
	\dela{	The section enlists some results about the Young measures.

		\begin{lemma}
			Let $(\Omega, \mathcal{F}, \mathbb{P})$ be a probability space, and let $M$ be a Radon space. Let $\lambda:\Omega \rightarrow\mathcal{Y}(0,T;M)$ be such that for every $J\in\mathcal{B}([0,T]\times M)$, the mapping
			\begin{equation*}
				\Omega\ni\omega\mapsto \lambda(\omega)(J)\in[0,T]
			\end{equation*}
			is measurable. Then there exists a stochastic relaxed control $\{ q_t \}_{t\in[0,T]}$ such that for $\mathbb{P}$-a.s. $\omega\in\Omega$ we have
			\begin{equation*}
				\lambda(\omega,C\times B) = \int_{B}q_t(\omega,C)dt \qquad \forall C\in\mathcal{B}(M),B\in\mathcal{B}([0,T]).
			\end{equation*}
		\end{lemma}
		For proof see \cite{ZB+RS}(Lemma 2.3).

		\begin{lemma}\label{liminf for Young measures}
			Let $\mathcal{M}$ be a metrizable Suslin space. Let $F: H\times \mathcal{M}\rightarrow[0,\infty]$ be measurable in $t$ with respect to $(X,u)\in H\times \mathcal
			{M}$ and satisfies
			\begin{equation*}
				F\geq 0.
			\end{equation*}
			Then we have
			\begin{enumerate}
				\item If $\mu_n\rightarrow\mu$ stably in $\mathcal{Y}(0,T;\mathcal{M})$ then
				\begin{equation*}
					\int_{0}^{T}\int_{\mathcal{M}}F(t,X,u)\mu(du,dt) \leq \liminf_{n\rightarrow\infty} 	\int_{0}^{T}\int_{\mathcal{M}}F(t,X,u)\mu_n(du,dt)
				\end{equation*}
				
			\end{enumerate}
			
		\end{lemma}
		See \cite{UM+DM} Lemma A.14 Appendix A.
		
		Also see Proposition 2.1.12(Semicontinuity theorem) \cite{Castaing_Book}, page number 36.
		
	}

	\par\medskip\noindent
	\bibliographystyle{plain}
	\bibliography{References_GokhaleSoham}
\end{document}